\documentclass[12pt,amstex]{amsart}

\usepackage{xr}
\usepackage[shortlabels]{enumitem}
\usepackage[english]{babel}
\usepackage{amscd, amsfonts, amsmath, amssymb, amsthm, epsfig, float, geometry, graphicx, pstricks, pst-node, txfonts, enumitem}

\usepackage[all]{xy}
\usepackage{tikz}
\usepackage{multicol}

\usepackage{hyperref}
\hypersetup{
	colorlinks,
	citecolor=black,
	filecolor=black,
	linkcolor=black,
	urlcolor=black
}

\numberwithin{equation}{section}

\theoremstyle{definition}
\newtheorem{thm}{Theorem}[section]
\newtheorem{lem}[thm]{Lemma}
\newtheorem{defn}[thm]{Definition}
\newtheorem{rem}[thm]{Remark}
\newtheorem{prop}[thm]{Proposition}
\newtheorem{ques}[thm]{Question}
\newtheorem{coro}[thm]{Corollary}
\newtheorem{ex}[thm]{Example}
\newtheorem{conj}[thm]{Conjecture}
\newtheorem{conv}[thm]{Convention}

\newtheorem*{defn*}{Definition}

\def \C {\mathbb{C}}
\def \E {\mathbb{E}}
\def \F {\mathbb{F}}

\def \N {\mathbb{N}}
\def \O {\mathcal{O}}
\def \P {\bold{P}}
\def \Q {\mathbb{Q}}
\def \R {\mathbb{R}}
\def \T {\mathbb{T}}
\def \V {\mathbb{F}_{p}^{d}}
\def \X {\mathcal{X}}
\def \Z {\mathbb{Z}}

\def \+ {\hat{+}}
\def \- {\hat{-}}

\def \d {\delta}
\def \e {\epsilon}

\def \h {\bold{h}}

\def \l {\lambda}
\def \m {\bold{m}}
\def \n {\bold{n}}

\def \w {\bold{w}}
\def \x {\bold{x}}
\def \y {\bold{y}}

\def \CP {\{\text{pure, nice, independent, consistent}\}}

\def \Gow {\Box}

\def \lin {\text{lin}}
\def \Lip {\text{Lip}}

\def \nlin {\text{nlin}}

\def \poly {\text{poly}}
\def \pp {\perp_{M}}

\def \rank {\text{rank}}

\def \sp {\text{span}}

\title[Spherical higher order Fourier analysis over finite fields I]{Spherical higher order Fourier analysis over finite fields I: equidistribution for nilsequences}
\author{Wenbo Sun}
\address[Wenbo Sun]{Department of Mathematics, Virginia Tech, 225 Stanger Street, Blacksburg, VA, 24061, USA}
\email{swenbo@vt.edu}

\thanks{The author was partially supported by the NSF Grant DMS-2247331}


\begin{document}

\maketitle

\begin{abstract}
   This paper is the first part of the series \emph{Spherical higher order Fourier analysis over finite fields}, aiming to develop the higher order Fourier analysis method along spheres over finite fields, and to solve the geometric Ramsey conjecture in the finite field setting.
   
   In this paper, we prove a quantitative equidistribution theorem for polynomial sequences in a nilmanifold, where the average is taken along spheres instead of cubes. To be more precise, let $\Omega\subseteq \mathbb{Z}^{d}$ be the preimage of a sphere $\mathbb{F}_{p}^{d}$ under the natural embedding from $\mathbb{Z}^{d}$ to $\mathbb{F}_{p}^{d}$. We showed that if a rational polynomial sequence $(g(n)\Gamma)_{n\in\Omega}$   
   is not equidistributed on a nilmanifold $G/\Gamma$, then there exists a nontrivial horizontal character $\eta$ of $G/\Gamma$ such that $\eta\circ g \mod \mathbb{Z}$ vanishes on $\Omega$. 
   We also prove quantitative equidistribution and factorization theorems for more general sets arising from quadratic forms defined over $\mathbb{F}_{p}^{d}$.   These results will serve as   fundamental tools in later parts of the series to proof the spherical Gowers inverse theorem and the geometric Ramsey conjecture.
\end{abstract}	


\tableofcontents

\section{Introduction}
\subsection{Spherical higher order Fourier analysis}
Originated from 
\cite{Fur77,Gow01,HK05}, higher order Fourier analysis is a theory that uses modern analytic methods to study problems in combinatorics. The area of higher order Fourier analysis grew rapidly in the last decade, and it was quickly seen to have much in common with areas of theoretical computer science related to polynomiality testing, as well as the asymptotics for various linear patterns in the prime numbers \cite{GT08,GT10}.

The theory of higher order Fourier analysis is especially powerful in the study of Szemer\'edi type theorems concerning averages of linear forms of the form
$$\frac{1}{p^{d}}\sum_{n\in \V}f_{1}(L_{1}(n))\cdot\dots\cdot f_{k}(L_{k}(n)),$$
where   $f_{1},\dots,f_{k}\colon \F_{p}\to\C$ are bounded functions, and $L_{1},\dots,L_{k}\colon \V\to \F_{p}$ are linear transformations. 
However, nowadays very little is known for averages of linear forms along spheres:
\begin{equation}\label{savg}
\frac{1}{\vert\Omega\vert}\sum_{n\in \Omega}f_{1}(L_{1}(n))\cdot\dots\cdot f_{k}(L_{k}(n)),
\end{equation}
where $\Omega\subseteq \V$ is set of zeros of a quadratic polynomial (for example, one can take $\Omega$ to be the set of $n\in\V$ with $n\cdot n=r$ for some fixed radius $r\in\F_{p}$). This pattern pertains to many important questions in combinatorics, including the geometric Ramsey conjecture (we postpone the details of this topic to \cite{SunD}).  

In this paper as well as in the series \cite{SunB,SunC,SunD}, we develop the theory of higher order Fourier analysis in the ``spherical setting", which enables us to study patterns of the form (\ref{savg}) and to provide an affirmative answer to the geometric Ramsey conjecture in the finite field setting.

\subsection{Equidistribution for spherical averages on nilmanifolds}
This paper is the first part of the series \emph{Spherical higher order Fourier analysis over finite fields} \cite{SunB,SunC,SunD}.
The purpose of this paper is to provide equidistribution theorems we need for the study of spherical higher order Fourier analysis. In the last decade, it has come to be appreciated that nilmanifolds and nilsequences play a fundamental role in combinatorial number theory, ergodic theory and additive combinatorics (see for example \cite{GT08,GT10,GTZ11,GTZ12,HK05}). 
We begin with recalling some standard terminologies in this area:

\begin{defn*}[Filtered group]
	Let $G$ be a group. An \emph{$\N$-filtration} on $G$ is a collection $G_{\N}=(G_{i})_{i\in \N}$ of subgroups of $G$ indexed by $\N$   such that  the following holds:
	\begin{enumerate}[(i)]
		\item for all $i,j\in \N$ with $i\leq j$, we have that $G_{i}\supseteq G_{j}$;
		\item  for all $i,j\in \N$, we have $[G_{i},G_{j}]\subseteq G_{i+j}$.
	\end{enumerate}	
	For $s\in \N$, we say that $G$ is an \emph{($\N$-filtered) nilpotent group} of \emph{degree} at most $s$ (or of \emph{degree} $\leq s$) with respect to some $\N$-filtration  $(G_{i})_{i\in \N}$ if $G_{i}$ is trivial whenever $i>s$. \end{defn*}

\begin{defn*}[Nilmanifold]
	Let $\Gamma$ be a discrete and cocompact subgroup of a connected, simply-connected nilpotent Lie group $G$ with filtration $G_{\N}=(G_{i})_{i\in \N}$ such that $\Gamma_{i}:=\Gamma\cap G_{i}$ is a cocompact subgroup of $G_{i}$ for all $i\in \N$.
	Then we say that $G/\Gamma$ is an \emph{($\N$-filtered) nilmanifold}, and we use $(G/\Gamma)_{\N}$ to denote the collection $(G_{i}/\Gamma_{i})_{i\in \N}$ (which is called the \emph{$\N$-filtration} of $G/\Gamma$). We say that $G/\Gamma$ has degree $\leq s$  with respect to $(G/\Gamma)_{\N}$ if $G$ has degree $\leq s$  with respect to $G_{\N}$.
\end{defn*}

\begin{defn*}[Polynomial sequences]
	Let $d,\in\N_{+}$ and $G$ be a connected simply-connected nilpotent Lie group and  $(G_{i})_{i\in\N}$ be an $\N$-filtration of $G$. 
	Denote $\Delta_{h}g(n):=g(n+h)g(n)^{-1}$
 for all $n, h\in G$.
	A map $g\colon \Z^{d}\to G$ is an \emph{($\N$-filtered) polynomial sequence} if
		$$\Delta_{h_{m}}\dots \Delta_{h_{1}} g(n)\in G_{m}$$
		 for all $m\in\N$ and $n,h_{1},\dots,h_{m}\in \Z^{d}$.	
    The set of all $\N$-filtered polynomial sequences is denoted by $\poly(\Z^{d}\to G_{\N})$.
  \end{defn*}

  A fundamental question in higher order Fourier analysis is to study the equidistribution property of polynomial sequences.  
  Let $\Omega$ be a non-empty subset of $\Z^{d}$ and $G/\Gamma$ be an $\N$-filtered nilmanifold. A sequence $\O\colon\Omega\to G/\Gamma$ is \emph{$\d$-equidistributed} on $G/\Gamma$ if for all Lipschitz function $F\colon G/\Gamma\to\C$ (with respect to the metric $d_{G/\Gamma}$ to be defined in Section \ref{1:s:pp3}), we have that
 	$$\limsup_{N\to\infty}\Bigl\vert\frac{1}{\vert \Omega\cap [N]^{d}\vert}\sum_{n\in\Omega\cap [N]^{d}}F(\O(n))-\int_{G/\Gamma}F\,dm_{G/\Gamma}\Bigr\vert\leq \d\Vert F\Vert_{\Lip},$$
	where  $m_{G/\Gamma}$ is the Haar measure of $G/\Gamma$ and the \emph{Lipschitz norm} is defined as 
	$$\Vert F\Vert_{\Lip}:=\sup_{x\in G/\Gamma}\vert F(x)\vert+\sup_{x,y\in G/\Gamma, x\neq y}\frac{\vert F(x)-F(y)\vert}{d_{G/\Gamma}(x,y)}.$$
	    We say that $\O$ is \emph{$\d$-totally equidistributed} on $G/\Gamma$ if for all $r\in\N_{+}$ with $r<\d^{-1}$ and $m\in\Z^{d}$, the sequence  $(\O(n))_{n\in\Omega\cap (r\Z^{d}+m)}$ is $\d$-equidistributed on $G/\Gamma$.
  
  It is known that the equidistribution property of polynomial sequences is  connected to the \emph{horizontal characters} of $G/\Gamma$, i.e.  group homomorphisms $\eta\colon G\to \R$ with $\eta(\Gamma)\subseteq \Z$.
  It was proved by Green and Tao \cite{GT12b,GT14} that for any $\delta>0$ and any polynomial  sequence $g\in\poly(\Z^{d}\to G_{\N})$, either $(g(n)\Gamma)_{n\in\Z^{d}}$ is $\d$-equidistributed on $G/\Gamma$ or there exists a nontrivial horizontal character $\eta$ of $G/\Gamma$ whose complexity is comparable with $G/\Gamma$ (we will define complexity formally in Section \ref{1:s:pp3}) such that $\eta\circ g$ is a slowly varying function. This is a quantitative generalization of a result of Leibman \cite{Lei05}.

   In this paper, we study the quantitative equidistribution property for sequences of the form $(g(n)\Gamma)_{n\in\Omega}$, where $\Omega=\{n\in\Z^{d}\colon n\cdot n\equiv r\mod p\Z\}$ for some $r\in\Z$ (see \cite{KSS18,Sun18b} for the study of related averages).
    The following is the main result of this paper (we refer the readers to  Section \ref{1:s:pp3} for definitions):  
  
   \begin{thm}[Equidistribution for polynomial sequence along spheres]\label{1:sLei00}
  	Let $0<\d<1/2, C>0, d\in\N_{+},s\in\N$ with $d\geq s+13$ and $p\gg_{C,d} \d^{-O_{C,d}(1)}$ be a prime. 
	Denote $\Omega=\{n\in\Z^{d}\colon n\cdot n\equiv r\mod p\Z\}$ for some $r\in\Z$.
	   Let $G/\Gamma$ be an $s$-step $\N$-filtered nilmanifold of complexity at most $C$ and $g\in\poly(\Z^{d}\to G_{\N})$ be a rational polynomial  sequence. 
	 Then either  $(g(n)\Gamma)_{n\in \Omega}$ is   $\d$-equidistributed on $G/\Gamma$, or
  	there exists a nontrivial horizontal character $\eta$ of complexity at most $O_{\d,d}(1)$ such that $\eta\circ g \mod \Z$  is a constant on $\Omega$.
  \end{thm}

  In fact, we will prove a stronger version of Theorem \ref{1:sLei00}, which is Theorem \ref{1:sLei} in this paper.   
  \begin{rem}
  Note that in Theorem \ref{1:sLei00}, it is not possible to require $\eta\circ g\mod \Z$ to be constant on the entire space $\Z^{d}$. For example, let $G/\Gamma=\R/\Z$,  $g(n)=(n\cdot n)/p$ for all $n\in\Z^{d}$.  
  Since $g(n)\equiv r/p \mod \Z$ for all $n\in\Omega$, it is not hard to see that $(g(n)\Gamma)_{n\in \Omega}$ is not  $\d$-equidistributed on $G/\Gamma$. On the other hand, since any nontrivial  horizontal character of $G/\Gamma$ is of the form $\eta(x):=kx$ for some $k\in\Z\backslash\{0\}$, we have that $\eta\circ g(n) \mod \Z=k(n\cdot n)/p \mod \Z$ can not be a constant.
  \end{rem}
  
  It is worth noting that in order for Theorem \ref{1:sLei00} to be useful in \cite{SunC,SunD}, we must prove this result for rational polynomials instead of just for $p$-periodic ones. See Remark \ref{1:r:wwer2} for more details.
 
  Let $\tau\colon \V\to\{0,\dots,p-1\}^{d}$ denote the natural embedding of $\V$ in $\Z^{d}$.
    By taking $g$ to be a  polynomial from $\Z^{d}$ to $\Z/p$, $G/\Gamma=\R/\Z$ and $F(x)=\exp(x):=e^{2\pi i x}$, we have the following immediate corollary for exponential sums along spheres, which is of independent interest in number theory.
    
  \begin{coro}[Weyl's equisidtribution theorem along spheres over finite field]\label{1:sweyl}
  	Let $0<\d<1/2, d\in\N_{+},s\in\N$ with $d\geq s+13$ and $p$ be a prime.   Let $g\colon\Z^{d}\to\Z$ be an integer valued polynomial of degree $s$. Denote $\Omega=\{n\in\V\colon n\cdot n=r\}$ for some $r\in\F_{p}$.
	 If $p\gg_{d} \d^{-O_{d}(1)}$, then either $$\Bigl\vert\frac{1}{\vert\Omega\vert}\sum_{n\in\Omega}\exp(g\circ\tau(n)/p)\Bigr\vert<\d$$ or
  	  $g\circ\tau(n)/p \mod \Z$  is a constant on $\Omega$.\footnote{By Corollary \ref{1:noloop2}, we may further conclude that $g(n)=(n\cdot n -\tau(r))g_{1}(n)+pg_{2}(n)$ for some integer valued polynomials $g_{1}$ and $g_{2}$.}
  \end{coro}
   
   Motivated by \cite{CS14,GT12b,GT14},
   it is natural to ask whether Theorem \ref{1:sLei00} implies a factorization theorem, i.e. whether one can decompose $g$ as the product of a constant $\e$, a polynomial sequence $g'$ being equidistributed on a subnilmanifold of $G/\Gamma$, and a polynomial sequence $\gamma$ having strong periodic properties. This is indeed the case (we refer the readers to  Section \ref{1:s:pp3} for definitions):

  \begin{thm}[Strong factorization property along spheres]\label{1:facf30}
  		Let $d\in\N_{+},s\in\N$ with $d\geq s+13$, $C>0$, $\mathcal{F}\colon\R_{+}\to\R_{+}$ be a growth function,\footnote{A \emph{growth function} $\mathcal{F}\colon\R_{+}\to\R_{+}$ is a non-decreasing function with $\mathcal{F}(n)\geq n$ for all $n\in\R_{+}$.} and $p\gg_{C,d,\mathcal{F}} 1$ be a prime. Denote $\Omega=\{n\in\Z^{d}\colon n\cdot n\equiv r \mod p\Z\}$ for some $r\in\Z$ and let $n_{\ast}\in\Omega$.
  		Let $G/\Gamma$ be an $s$-step $\N$-filtered nilmanifold of  complexity at most $C$, and let $g\in \poly(\Z^{d}\to G_{\N})$ be a rational polynomial sequence on $\Omega$ .	
  	There 		
  		exist some $C\leq C'\leq O_{C,d,\mathcal{F}}(1)$, 
  		 a proper subgroup $G'$ of $G$ which is $C'$-rational relative to $\mathcal{X}$, and
  	 a factorization $$g(n)=\e g'(n)\gamma(n)  \text{ for all }  n\in \Z^{d}$$ 
	  such that 
	  $\e\in G$ is of complexity $O_{C'}(1)$,
  		$g'\in \poly(\Z^{d}\to G'_{\N})$ is  rational  
		and $(g'(n)\Gamma)_{n\in\Omega}$  is $\mathcal{F}(C')^{-1}$-totally equidistributed on $G'/\Gamma'$, where $\Gamma':=G'\cap \Gamma$, and that $\gamma\in\poly(\Z^{d}\to G_{\N})$ is 
		partially $r$-rational with base point $n_{\ast}$ and partially $r$-periodic on $\iota^{-1}(\Omega)$
		  for some $r\in\N_{+}$ with $r\leq O_{C,d,s}(C')$.	
  \end{thm}

    In fact, we will prove a stronger version of Theorem \ref{1:facf30}, which is Theorem \ref{1:facf3r} in this paper.
 

 It is natural to ask whether one can extend Theorems \ref{1:sLei00} and \ref{1:facf30} to more general sets $\Omega$. In fact, in the series \cite{SunB,SunC,SunD}, we will frequent work with those $\Omega$ which is the set of solutions of a system of quadratic equations (whereas the set $\Omega$ defined in Theorems \ref{1:sLei00} and \ref{1:facf30} is the set of solution for a single quadratic equation). In this paper, we provide such generalizations of Theorems \ref{1:sLei00} and \ref{1:facf30} in Theorems \ref{1:veryr} and \ref{1:facf3r} respectively (see also Theorem \ref{1:facf3} for a variation of Theorem \ref{1:facf3r}).

      \subsection{Organization of the paper} We first provide the background material for polynomials, nilmanifolds and quadratic forms in Sections \ref{1:s:pp2} (as well as Appendix \ref{1:s:AppA}), \ref{1:s:pp3} and \ref{1:s:pp5} respectively. Then in Section \ref{1:s:45},    
      we state the main equidistribution result of the paper (i.e. Theorem \ref{1:sLei}, which is a generalization of Theorem \ref{1:sLei00}) and explain the outline of its proof.  
      
      Sections  \ref{1:s4},  \ref{1:s5} and \ref{1:s:pp9} are devoted to the proof of Theorem \ref{1:sLei}. 
      In Section \ref{1:s4}, we provide the solutions to some relevant algebraic equations for   $p$-periodic polynomials.
      Among other 
      results, we solve a polynomial equation of the form (\ref{1:t43t}), which plays a central role in the proof of Theorem \ref{1:sLei00} for $p$-periodic polynomial sequences. Then in Section \ref{1:s5}, we extend all the results in  Section \ref{1:s4} to the non-periodic setting.  Finally, in Section \ref{1:s:pp9}, we combine the results from Sections \ref{1:s4} and \ref{1:s5} with the approach of Green and Tao \cite{GT12b,GT14}  to complete the proof of Theorem \ref{1:sLei}. We refer the readers to Section \ref{1:s:ooer} for a more detailed explanation of this method.

      For the study in \cite{SunC,SunD}, we need to extend Theorem \ref{1:sLei00} to a more general setting. In Section \ref{1:s:mset}, we defined a large family of sets arising from quadratic equations (call the \emph{$M$-sets}), and provide some irreducible properties for such sets. In Section \ref{1:ss91}, we extend Theorem \ref{1:sLei00} to a large class of $M$-sets (Theorem \ref{1:veryr}). For convenience, we introduce a concept called the \emph{Leibman dichotomy} which characterizes the phenomenon  described by Theorem \ref{1:sLei00}.
      In Sections \ref{1:ss92}, we provide two factorization theorems (Theorems \ref{1:facf3} and  \ref{1:facf3r}) for large class of $M$-sets, which will play important roles in \cite{SunC,SunD}.
      
      Unfortunately the proofs of many results in Sections \ref{1:s:mset} and \ref{1:ss91} are rather involved. In order for the readers to follow the ideas of the paper more easily, we postpone the proofs of some very technical (but also very important) results in Appendices \ref{1:s:AppB}, \ref{1:s:dec}, \ref{1:s:AppC}, and \ref{1:s:AppD}.
      
    In Section \ref{1:s:oq}, we collect some interesting open questions.

     \

\textbf{Acknowledgements.} We thank James Leng for helpful discussions.
 
  \subsection{Definitions and notations}\label{1:ssdn}
  
  \begin{conv}\label{titi} 
  Throughout this paper, we use
   $\tau\colon\F_{p}\to \{0,\dots,p-1\}$ to denote the natural bijective embedding, and use $\iota$ to denote the map from $\Z$ (or $\Z_{K}$ for any $K$ divisible by $p$) to $\F_{p}$ given by $\iota(n):=\tau^{-1}(n \mod p\Z)$.
	We also use 
	$\tau$ to denote the map from $\F_{p}^{k}$ to $\Z^{k}$ (or $\Z_{K}^{k}$) given by $\tau(x_{1},\dots,x_{k}):=(\tau(x_{1}),\dots,\tau(x_{k}))$,
	and
	$\iota$ to denote the map from $\Z^{k}$ (or $\Z_{K}^{k}$) to $\F_{p}^{k}$ given by $\iota(x_{1},\dots,x_{k}):=(\iota(x_{1}),\dots,$ $\iota(x_{k}))$. 
	When there is no confusion, we will not state the domain and range of $\tau$ and $\iota$ explicitly.
	
	We may also extend the domain of $\iota$ to all the rational numbers of the form $x/y$ with $(x,y)=1, x\in\Z, y\in\Z\backslash p\Z$ by setting $\iota(x/y):=\iota(xy^{\ast})$, where $y^{\ast}$ is any integer with $yy^{\ast}\equiv 1 \mod p\Z$.
  \end{conv}
  
  Below are the notations we use in this paper:

\begin{itemize}
	\item Let $\N,\N_{+},\Z,\Q,\R,\R+,\C$ denote the set of non-negative integers, positive integers, integers, rational numbers, real numbers, positive real numbers, and complex numbers, respectively. Denote $\T:=\R/\Z$. Let $\F_{p}$ denote the finite field with $p$ elements. Let $\Z_{K}$ denote the cyclic group with $K$ elements.
		\item Throughout this paper, $d$ is a fixed positive integer and $p$ is a prime number.
		\item Throughout this paper, unless otherwise stated, all vectors are assumed to be horizontal vectors.
		\item Let $\mathcal{C}$ be a collection of parameters and $A,B,c\in\R$. We write $A\gg_{\mathcal{C}} B$ if $\vert  A\vert\geq K\vert B\vert$ and $A=O_{\mathcal{C}}(B)$ if $\vert A\vert\leq K\vert B\vert$ for some $K>0$ depending only on the parameters in $\mathcal{C}$.
In the above definitions, we allow the set $\mathcal{C}$ to be empty. In this case $K$ will be a universal constant.
\item Let $[N]$ denote the set $\{0,\dots,N-1\}$.
	\item
	For $i=(i_{1},\dots,i_{k})\in\Z^{k}$, denote   $\vert i\vert:=\vert i_{1}\vert+\dots+\vert i_{k}\vert$.
	 For $n=(n_{1},\dots,n_{k})\in\Z^{k}$ and $i=(i_{1},\dots,i_{k})\in\N^{k}$, denote $n^{i}:=n_{1}^{i_{1}}\dots n_{k}^{i_{k}}$  and $i!:=i_{1}!\dots i_{k}!$. 
For $n=(n_{1},\dots,n_{k})\in\N^{k}$ and $i=(i_{1},\dots,i_{k})\in\N^{k}$, denote $\binom{n}{i}:=\binom{n_{1}}{i_{1}}\dots \binom{n_{k}}{i_{k}}$.	
	\item For any set $F$, let $F[x_{1},\dots,x_{k}]$ denote the set of all polynomials in the variables $x_{1},\dots,x_{k}$ whose coefficients are from $F$. 
    \item For $x\in\R$, let $\lfloor x\rfloor$ denote the largest integer which is not larger than $x$, and $\lceil x\rceil$ denote the smallest integer which is not smaller than $x$. Denote $\{x\}:=x-\lfloor x\rfloor$.
	\item Let  $X$ be a finite set and $f\colon X\to\C$ be a function. Denote $\E_{x\in X}f(x):=\frac{1}{\vert X\vert}\sum_{x\in X}f(x)$, the average of $f$ on $X$.
	\item We say that a set $\Omega\subseteq \Z^{k}$ is \emph{$Q$-periodic} if $\Omega=\Omega+Q\Z^{k}$. 
	\item If $\Omega\subseteq \Z^{k}$ is  a $Q$-periodic set and $f\colon \Z^{k}\to\C$ is such that $f(n)=f(n+Qm)$ for all $m,n\in\Z^{k}$ with $n,n+Qm\in\Omega$. Then we denote $\E_{x\in\Omega}f(x):=\E_{x\in\Omega\cap [Q]^{k}}f(x)$.
	\item For $F=\Z^{k}$ or $\F_{p}^{k}$, and $x=(x_{1},\dots,x_{k}), y=(y_{1},\dots,y_{k})\in F$, let $x\cdot y\in \Z$ or $\F_{p}$ denote the dot product given by
	$x\cdot y:=x_{1}y_{1}+\dots+x_{k}y_{k}.$
	\item Let $\exp\colon \R\to\C$ denote the function $\exp(x):=e^{2\pi i x}$.
	\item If $G$ is a connected, simply connected Lie group, then we use  $\log G$ to denote its Lie algebra. Let $\exp\colon \log G\to G$ be the exponential map, and $\log\colon G\to \log G$ be the logarithm map. For $t\in\R$ and $g\in G$, denote $g^{t}:=\exp(t\log g)$. 
	\item If $f\colon H\to G$ is a function from an abelian group $H=(H,+)$ to some group $(G,\cdot)$, denote $\Delta_{h} f(n):=f(n+h)\cdot f(n)^{-1}$ for all $n,h\in H$.
	\item We write affine subspaces of $\V$ as $V+c$, where $V$ is a subspace of $\V$ passing through $\bold{0}$, and $c\in\V$.
 \end{itemize}	
 
  Let $D,D'\in\N_{+}$ and $C>0$. 
Here are some basic notions of complexities:

	\begin{itemize}
		\item \textbf{Real and complex numbers:} a  number $r\in\R$ is of \emph{complexity} at most $C$ if $r=a/b$ for some $a,b\in\Z$ with $-C\leq a,b\leq C$. If $r\notin \Q$, then we say that the complexity of $r$ is infinity. A complex number is of \emph{complexity} at most $C$ if both its real and imaginary parts are of complexity at most $C$.
	   \item \textbf{Vectors and matrices:} a vector or matrix is of \emph{complexity} at most $C$ if all of its entries are of  complexity at most $C$.
	   \item \textbf{Subspaces:} a subspace of $\R^{D}$ is of \emph{complexity} at most $C$ if it is the null space of a matrix of complexity at most $C$.
	  \item \textbf{Linear transformations:} let $L\colon\C^{D}\to\C^{D'}$ be a linear transformation. Then $L$ is associated with an $D\times D'$ matrix $A$ in $\C$. We say that $L$ is of \emph{complexity} at most $C$ if  $A$ is of complexity at most $C$.
	  \item \textbf{Lipschitz functions:} a Lipschitz function is of \emph{complexity} at most $C$ if its Lipschitz norm is at most $C$.
	\end{itemize}

\section{Background material for polynomials}\label{1:s:pp2}

\begin{defn}[Polynomials in finite field]
	Let $\poly(\V\to\F^{d'}_{p})$ be the collection of all functions $f\colon \V\to\F_{p}$ of the form
	\begin{equation}\label{1:p1}
	f(n_{1},\dots,n_{d})=\sum_{0\leq a_{1},\dots,a_{d}\leq p-1, a_{i}\in\N}C_{a_{1},\dots,a_{d}}n^{a_{1}}_{1}\dots n^{a_{d-1}}_{d}
	\end{equation}	
	for some $C_{a_{1},\dots,a_{d}}\in \F_{p}^{d'}$. 
	Let $f\in\poly(\V\to\F^{d'}_{p})$ be a function which is not constant zero.
	The \emph{degree of $f$} (denoted as $\deg(f)$) is the largest $r\in\N$ such that $C_{a_{1},\dots,a_{d}}\neq \bold{0}$ for some $a_{1}+\dots+a_{d}=r$. We say that $f$ is \emph{homogeneous of degree $r$} if $C_{a_{1},\dots,a_{d}}\neq \bold{0}$ implies that $a_{1}+\dots+a_{d}=r$.	 We say that $f$ is a \emph{linear transformation} if $f$ is homogeneous of degree 1.
\end{defn}	

\begin{conv}\label{1:c00}
       For convenience, the \emph{degree} of the constant zero function is allowed to take any integer value. For example, 0 can be regarded as a homogeneous polynomial of degree 10, as a linear transformation, or as a polynomial of degree -1. Here we also adopt the convention that the  only  polynomial of negative degree is 0.  
\end{conv}

In this paper, we will work with polynomials which satisfy various types of periodicity properties. To this end, we introduce the following notations:

 \begin{defn}[Partially $p$-periodic polynomials]\label{1:polygeneral}
     For any subset $R$ of $\R$,
     let $\poly(\Z^{d}\to R)$  be the collection of all polynomials in $\Z^{d}$ taking values in $R$. For any prime $p$, subset $\Omega$ of $\Z^{d}$, and sets  $R'\subseteq R\subseteq \R$, let $\poly(\Omega\to R\vert R')$ denote the set of all $f\in\poly(\Z^{d}\to R)$ such that $f(n)\in R$ for all $n\in\Z^{d}$ and that $f(n)\in R'$ for all $n\in\Omega$.
     Let $\poly_{Q}(\Omega\to R\vert R')$ denote the set of all $f\in\poly(\Z^{d}\to R)$ such that $f(n)\in R$ for all $n\in\Z^{d}$ and that $f(n+Qm)-f(n)\in R'$ for all $m,n\in\Z^{d}$ with $n,n+Qm\in\Omega$.  
     
     A polynomial $f$ in $\poly_{p}(\Omega\to \R\vert\Z)$ is called a \emph{partially $p$-periodic polynomial on $\Omega$}.
 \end{defn}

    	We use some examples to illustrate the meanings of these notions:

    		\begin{ex}
    			The set
    			$\poly(\Z^{d}\to R\vert R')$ is simply the set of all $R'$-valued polynomials. In other words, $\poly(\Z^{d}\to R\vert R')=\poly(\Z^{d}\to R')$. 
    		\end{ex}

	\begin{ex}
    		Let $M\colon\Z^{d}\to\Z/p$ be a quadratic form, i.e. $M(n)=((nA)\cdot n+v\cdot n+b)/p$ for all $n\in\Z^{d}$ for some $d\times d$ integer matrix $A$, $v\in\Z^{d}$ and $b\in\Z$.
	Let	$V_{p}(M)$ denote the set of $n\in \Z^{k}$ such that $M(n+pm)\in \Z$ for all $m\in\Z^{k}$ (which is the same as the set of $n\in \Z^{k}$ such that $M(n)\in \Z$).	For  any set $F$, $\poly(V_{p}(M)\to F\vert \Z)$ is the set of all $F$-valued polynomials $f\colon\Z^{d}\to F$ such that $f(n)\in \Z$ for all $n\in\Z^{d}$ with $M(n)\in \Z$, and $\poly_{p}(V_{p}(M)\to F\vert\Z)$ is the set of all $F$-valued polynomials $f\colon\Z^{d}\to \Z/p$ such that $f(n+pm)-f(n)\in \Z$ for all $n\in\Z^{d}$ with $M(n)\in \Z$. 
    	\end{ex}
  		
	   	\begin{ex}	
    		A polynomial $f\in \poly(\Z^{d}\to \Z/p)$ belongs to $\poly_{p}(\Z^{d}\to \Z/p\vert\Z)$ if $f$ is \emph{$p$-periodic}, meaning that $f(n+pm)-f(n)\in \Z$ for all $m,n\in\Z^{d}$. 
        	\end{ex}		
		
		Clearly every polynomial with $\Z/p$-coefficients belongs to $\poly_{p}(\Z^{d}\to \Z/p\vert\Z)$, and $\poly_{p}(\Z^{d}\to \Z/p\vert\Z)$ is contained in $\poly(\Z^{d}\to \Z/p)$.
    	    On the other hand, $\poly(\Z^{d}\to \Z/p)$ is not contained in $\poly_{p}(\Z^{d}\to \Z/p\vert\Z)$
    	     (consider for example $f(n)=n^{p}/p^{2}, d=1$).
    		However, the following lemma shows that polynomials in $\poly(\Z^{d}\to \Z/p)$ with degree less than $p$ always belong to $\poly_{p}(\Z^{d}\to \Z/p\vert\Z)$.
		
	\begin{lem}\label{1:pp2}
	Every polynomial $f\in \poly(\Z^{d}\to \Z/p)$ of degree smaller than $p$  belongs to $f\in \poly_{p}(\Z^{d}\to \Z/p\vert\Z)$.
	\end{lem}	
	\begin{proof}
	Let $f\in \poly(\Z^{d}\to \Z/p)$ with $\deg(f)<p$. Since $p$ is a prime, by the multivariate polynomial
  interpolation, we may write $f(n)=\sum_{i\in\N^{d},\vert i\vert\leq s} \frac{a_{i}}{Qp}n^{i}$ for some $a_{i}\in\Z$ and $Q\in\Z, p\nmid Q$. Then for all $m,n\in\Z^{d}$, $f(n+pm)-f(n)\in\Z/Q$. On the other hand, since $f$ takes values in $\Z/p$, we have that $f(n+pm)-f(n)\in(\Z/Q)\cap(\Z/p)=\Z$, and thus $f\in \poly_{p}(\Z^{d}\to \Z/p\vert\Z)$.
	\end{proof}

    We provide another simple lemma for $p$-periodic polynomials which we will use in later sections.

  \begin{lem}\label{1:ivie}
  	Let $d\in\N_{+}$, $p$ be a prime, and $f\in\poly(\Z^{d}\to \Z)$ be an integer valued polynomial of degree smaller than $p$. There exist integer valued polynomial $f_{1}\in\poly(\Z^{d}\to \Z)$ and integer coefficient polynomial $f_{2}\in\poly(\Z^{d}\to \Z)$ both having degrees at most $\deg(f)$ such that $\frac{1}{p}f=f_{1}+\frac{1}{p}f_{2}$.
  \end{lem}	
  \begin{proof}
  	Denote $s:=\deg(f)<p$.
  	By the multivariate polynomial interpolation, we may write
  	$$f(n)=\sum_{i\in\N^{d},\vert i\vert\leq s}\frac{a_{i}}{Q}n^{i}$$
  	for some $a_{i}\in\Z$ and $Q\in\Z, p\nmid Q$. Let $Q^{\ast}\in\Z$ be such that $Q^{\ast}Q\equiv 1 \mod p\Z$ and set 
  	$$f_{2}(n):=\sum_{i\in\N^{d},\vert i\vert\leq s}Q^{\ast}a_{i}n^{i}.$$
  	Then $f_{2}$ is an integer coefficient polynomial of degree at most $s$. Note that  for all $n\in\Z^{d}$,
  	$$\frac{1}{p}(f(n)-f_{2}(n))=\sum_{i\in\N^{d},\vert i\vert\leq s}\frac{a_{i}}{Q}\frac{1-Q^{\ast}Q}{p}n^{i},$$
  	the left hand side of which belongs to $\Z/p$ and the right hand side of which belongs to $\Z/Q$. So $f_{1}(n):=\frac{1}{p}(f(n)-f_{2}(n))$ takes values in $(\Z/Q)\cap (\Z/p)=\Z$.
  \end{proof}

There is a natural correspondence between polynomials taking values in $\F_{p}$ and polynomials  taking values in $\Z/p$.   Let $F\in\poly(\V\to\F_{p}^{d'})$ and $f\in \poly(\Z^{d}\to (\Z/p)^{d'})$ be polynomials of degree at most $s$ for some $s<p$. If  $F=\iota\circ pf\circ\tau$, then we say that $F$ is \emph{induced} by $f$ and $f$ is a \emph{lifting} of $F$.\footnote{Note that since $pf\circ\tau$ takes values in $\Z^{d'}$, the coefficients of $pf\circ\tau$ belongs to $\Z^{d'}/s!$, and thus $\iota\circ pf\circ\tau$ is well defined by Convention \ref{titi}.}
     We say that $f$ is a  \emph{regular lifting} of $F$ if in addition $f$ has the same degree as $F$ and $f$ has $\{0,\frac{1}{p},\dots,\frac{p-1}{p}\}$-coefficients. We refer the readers to Appendix \ref{1:s:AppA} for the basic properties of liftings.

\section{Background material for nilmanifolds}\label{1:s:pp3}

\subsection{Nilmanifolds, filtrations and polynomial sequences}

We start with some basic definitions on nilmanifolds.
Let $G$ be a group and $g,h\in G$. Denote $[g,h]:=g^{-1}h^{-1}gh$. For subgroups $H,H'$ of $G$, let $[H,H']$ denote the group generated by $[h,h']$ for all $h\in H$ and $h'\in H'$.

\begin{defn}[Filtered group]
	Let $G$ be a group. An \emph{$\N$-filtration} on $G$ is a collection $G_{\N}=(G_{i})_{i\in \N}$ of subgroups of $G$ indexed by $\N$  with $G_{0}=G$ such that  the following holds:
	\begin{enumerate}[(i)]
		\item for all $i,j\in \N$ with $i\leq j$, we have that $G_{i}\supseteq G_{j}$;
		\item  for all $i,j\in \N$, we have $[G_{i},G_{j}]\subseteq G_{i+j}$.
	\end{enumerate}	
	For $s\in \N$, we say that $G$ is an \emph{($\N$-filtered) nilpotent group} of \emph{degree} at most $s$ (or of \emph{degree} $\leq s$) with respect to some $\N$-filtration  $(G_{i})_{i\in \N}$ if $G_{i}$ is trivial whenever $i>s$.  	\end{defn}

\begin{ex}
The \emph{lower central series} of a group $G$ is the $\N$-filtration $(G^{i})_{i\in\N}$ given by $G^{0}:=G$ and $G^{i+1}=[G^{i},G]$ for all $i\in\N$. We say that $G$ is a \emph{nilpotent group} of step (or degree) at most $s$ (with respect to the lower central series) if $G^{i}$ is trivial for all $i\geq s+1$.
\end{ex}

\begin{defn}[Nilmanifold]
	Let $\Gamma$ be a discrete and cocompact subgroup of a connected, simply-connected nilpotent Lie group $G$ with filtration $G_{\N}=(G_{i})_{i\in \N}$ such that $\Gamma_{i}:=\Gamma\cap G_{i}$ is a cocompact subgroup of $G_{i}$ for all $i\in \N$.\footnote{In some papers, such $\Gamma_{i}$ is called a \emph{rational} subgroup of $G$.}
	Then we say that $G/\Gamma$ is an \emph{($\N$-filtered) nilmanifold}, and we use $(G/\Gamma)_{\N}$ to denote the collection $(G_{i}/\Gamma_{i})_{i\in \N}$ (which is called the \emph{$\N$-filtration} of $G/\Gamma$). We say that $G/\Gamma$ has degree $\leq s$  with respect to $(G/\Gamma)_{\N}$ if $G$ has degree $\leq s$  with respect to $G_{\N}$.\footnote{Unlike the convention in literature, for our convenience, we do not require $G_{0}=G$ to be the same as $G_{1}$.}
\end{defn}

 We also need to define some special types of nilmanifolds.

\begin{defn}[Sub-nilmanifold]\label{1:id1}
	Let $G/\Gamma$ be an $\N$-filtered nilmanifold of degree $\leq s$ with filtration $G_{\N}$ and $H$ be a rational subgroup of $G$. Then $H/(H\cap \Gamma)$ is also  an $\N$-filtered nilmanifold of degree $\leq s$ with the filtration $H_{\N}$ given by $H_{i}:=G_{i}\cap H$ for all $i\in \N$ (see Example 6.14 of \cite{GTZ12}). We say that  $H/(H\cap \Gamma)$ is a \emph{sub-nilmanifold} of $G/\Gamma$, $H_{\N}$ (or $(H/(H\cap \Gamma))_{\N}$) is the filtration \emph{induced by} $G_{\N}$ (or $(G/\Gamma)_{\N}$).
\end{defn}	

\begin{defn}[Quotient nilmanifold]\label{1:id2}
	Let $G/\Gamma$ be an $\N$-filtered nilmanifold of degree $\leq s$ with filtration $G_{\N}$ and $H$ be a normal subgroup of $G$. Then $(G/H)/(\Gamma/(\Gamma\cap H))$ is also  an $\N$-filtered nilmanifold of most $\leq s$ with the filtration $(G/H)_{\N}$ given by $(G/H)_{i}:=G_{i}/(H\cap G_{i})$ for all $i\in \N$. We say that  $(G/H)/(\Gamma/(\Gamma\cap H))$ is the \emph{quotient nilmanifold} of $G/\Gamma$ by $H$ and that $(G/H)_{\N}$ is the filtration \emph{induced by} $G_{\N}$.
\end{defn}	

\begin{defn}[Product nilmanifold]\label{1:id4}
	Let $G/\Gamma$ and  $G'/\Gamma'$ be  $\N$-filtered nilmanifolds of degree $\leq s$ with filtration $G_{\N}$ and $G'_{\N}$. Then $G\times G'/\Gamma\times\Gamma'$ is also  an $\N$-filtered nilmanifold of most $\subseteq J$ with the filtration $(G\times G'/\Gamma\times\Gamma')_{\N}$ given by $(G\times G'/\Gamma\times\Gamma')_{i}:=G_{i}\times G'_{i}/\Gamma_{i}\times \Gamma'_{i}$ for all $i\in \N$. We say that  $G\times G'/\Gamma\times\Gamma'$ is the \emph{product nilmanifold} of $G/\Gamma$ and $G'/\Gamma'$ and that $(G\times G'/\Gamma\times\Gamma')_{\N}$ is the filtration \emph{induced by} $G_{\N}$ and $G'_{\N}$.
\end{defn}	

We remark that although we used the same terminology ``induce" in Definitions \ref{1:id1}, \ref{1:id2} and \ref{1:id4}, this will not cause any confusion in the paper as the meaning of ``induce" will be clear from the context.

\begin{defn}[Filtered homomorphism]	
	An \emph{($\N$-filtered) homomorphism} $\phi\colon G/\Gamma\to G'/\Gamma'$ between two $\N$-filtered nilmanifolds is a group homomorphism $\phi\colon G\to G'$  which maps $\Gamma$ to $\Gamma'$ and maps $G_{i}$ to $G'_{i}$ for all $i\in \N$.
\end{defn}

Every nilmanifold has an explicit algebraic description by using the Mal'cev basis:

\begin{defn} [Mal'cev basis]\label{1:Mal}
	Let $s\in\N_{+}$, $G/\Gamma$ be a  nilmanifold  of step at most $s$ with the $\N$-filtration  $(G_{i})_{i\in\N}$. Let $\dim(G)=m$ and $\dim(G_{i})=m_{i}$ for all $0\leq i\leq s$. A basis $\mathcal{X}:=\{X_{1},\dots,X_{m}\}$ for the Lie algebra $\log G$ of $G$ (over $\mathbb{R}$) is a \emph{Mal'cev basis} for $G/\Gamma$ adapted to the filtration $G_{\N}$ if
	\begin{itemize}
		\item for all $0\leq j\leq m-1$, $\log H_{j}:=\text{Span}_{\mathbb{R}}\{\xi_{j+1},\dots,\xi_{m}\}$ is a Lie algebra ideal of $\log G$ and so $H_{j}:=\exp(\log H_{j})$  is a normal Lie subgroup of $G;$
		\item $G_{i}=H_{m-m_{i}}$ for all $0\leq i\leq s$;
		\item the map $\psi^{-1}\colon \mathbb{R}^{m}\to G$ given by
		\begin{equation}\nonumber
		\psi^{-1}(t_{1},\dots,t_{m})=\exp(t_{1}X_{1})\dots\exp(t_{m}X_{m})
		\end{equation}	
		is a bijection;
		\item $\Gamma=\psi^{-1}(\Z^{m})$.
	\end{itemize}	
	We call $\psi$ the  \emph{Mal'cev coordinate map} with respect to the Mal'cev basis $\mathcal{X}$.
	If $g=\psi^{-1}(t_{1},\dots,t_{m})$, we say that $(t_{1},\dots,t_{m})$ are the \emph{Mal'cev coordinates} of $g$ with respect to $\mathcal{X}$. 
	
	We say that the Mal'cev basis $\mathcal{X}$  is \emph{$C$-rational} (or of \emph{complexity} at most $C$)  if all the structure constants $c_{i,j,k}$ in the relations
	   	$$[X_{i},X_{j}]=\sum_{k}c_{i,j,k}X_{k}$$
	   	are rational with complexity at most $C$.
\end{defn}

For any $h\in G$, there is a unique way to write $h$ as
  $h=\{h\}[h]$
such that $\psi(\{h\})\in [0,1)^{m}$ and $[h]\in\Gamma$. We adopt this notation throughout this paper.

It is known that for every filtration $G_{\bullet}$ which is rational for $\Gamma$, there exists a Mal'cev basis adapted to it. See for example the discussion on pages 11--12 of \cite{GT12b}.

We use the following quantities to describe the complexities of the objected defined above.

 \begin{defn}[Notions of complexities for nilmanifolds]
Let $G/\Gamma$ be a nilmanifold with an $\N$-filtration $G_{\N}$ and a Mal'cev basis $\mathcal{X}=\{X_{1},\dots,X_{D}\}$ adapted to it. 
We say that  $G/\Gamma$ is of \emph{complexity} at most $C$ if the  Mal'cev basis $\mathcal{X}$ is $C$-rational and $\dim(G)\leq C$. 

 An element $g\in G$ is of \emph{complexity} at most $C$ (with respect to the Mal'cev coordinate map $\psi\colon G/\Gamma\to\R^{m}$) if $\psi(g)\in [-C,C]^{m}$.

Let $G'/\Gamma'$ be a nilmanifold  endowed with the Mal'cev basis  $\mathcal{X}'=\{X'_{1},\dots,X'_{D'}\}$ respectively. Let $\phi\colon G/\Gamma\to G'/\Gamma'$ be a filtered homomorphism, we say that $\phi$ is of  \emph{complexity} at most $C$ if the map $X_{i}\to\sum_{j}a_{i,j}X'_{j}$ induced by $\phi$ is such that all $a_{i,j}$ are of complexity at most $C$.

 Let $G'\subseteq G$  be a closed connected subgroup. We say that  $G'$ is \emph{$C$-rational} (or of \emph{complexity} at most $C$) relative to $\mathcal{X}$ if the Lie algebra $\log G$ has a basis consisting of linear combinations $\sum_{i}a_{i}X_{i}$ such that $a_{i}$ are rational numbers of complexity at most $C$.
\end{defn}

\begin{conv}
 In the rest of the paper, all nilmanifolds are assumed to have a fixed filtration, Mal'cev basis and a smooth Riemannian metric induced by the Mal'cev basis. Therefore, we will simply say that a nilmanifold, Lipschitz function, sub-nilmanifold etc. is of complexity $C$ without mentioning the reference filtration and Mal'cev basis.
\end{conv}

\begin{defn}[Polynomial sequences]
	Let $k\in\N_{+}$ and $G$ be a connected simply-connected nilpotent Lie group. Let    $(G_{i})_{i\in\N}$ be an $\N$-filtration of $G$. A map $g\colon \Z^{k}\to G$ is a \emph{($\N$-filtered) 
	polynomial sequence} if
		$$\Delta_{h_{m}}\dots \Delta_{h_{1}} g(n)\in G_{m}$$
		 for all $m\in\N$ and $n,h_{1},\dots,h_{m}\in \Z^{k}$.\footnote{Recall that $\Delta_{h}g(n):=g(n+h)g(n)^{-1}$
 for all $n, h\in H$.}	
	
    The set of all $\N$-filtered polynomial sequences is denoted by $\poly(\Z^{k}\to G_{\N})$.
 %
  \end{defn}

By Corollary B.4 of \cite{GTZ12}, $\poly(\Z^{k}\to G_{\N})$ is a  group with respect to the pointwise multiplicative operation.
We refer the readers to Appendix B of \cite{GTZ12} for more properties for   polynomial sequences.

\subsection{Different notions of  polynomial sequences}


  We start with some conventional definitions for polynomial sequences.
	
   \begin{defn}[null, rational and periodic polynomial sequences]\label{1:defr}
       Let $k,s\in\N_{+}$, $G/\Gamma$ be a nilmanifold of step at most $s$, 
       and   $g\in\poly(\Z^{k}\to G_{\N})$. We say that $g$ is 
       \begin{itemize}
          \item \emph{null} if $g(n)\in\Gamma$ for all $n\in\Z^{k}$;
          \item \emph{rational} if 
          $g(Qn)\in\Gamma$ 
          for all $n\in\Z^{k}$ for some $Q\in\N_{+}$, in which case we also say that $g$ is \emph{$Q$-rational};
          \item \emph{periodic} if $g(n+Qm)^{-1}g(n)\in\Gamma$ for all $m,n\in\Z^{k}$ for some $Q\in\N_{+}$, in which case we also say that $g$ is \emph{$Q$-periodic}.
       \end{itemize}
      \end{defn}

    We refer the readers to Appendix A of \cite{GT12b} for properties of rational/periodic polynomial sequences.
 
 In this paper, since we will work with polynomial sequences which are null, rational or periodic on certain subsets of $\Z^{k}$, we introduce the following notations:

\begin{defn}[Partially null, rational and periodic polynomial sequences]
	Let  $k,Q\in\N_{+}$, $G/\Gamma$ be an $\N$-filtered nilmanifold,   $\Omega$ be a subset of $\Z^{k}$, and $n_{\ast}\in\Omega$. 
	 \begin{itemize}
          \item We use $\poly(\Omega\to G_{\N}\vert\Gamma)$ to denote the set of all $g\in \poly(\Z^{k}\to G_{\N})$ which is \emph{partially null on $\Omega$,} meaning that  $g(n)\in \Gamma$ for all $n\in\Omega$.
           \item We use $\poly_{\approx Q,n_{\ast}}(\Omega\to G_{\N}\vert\Gamma)$ to denote the set of all $g\in \poly(\Z^{k}\to G_{\N})$ which is \emph{partially $Q$-rational on $\Omega$ with base point $n_{\ast}$} meaning that $g(n_{\ast}+Qm)\in \Gamma$ for all $m\in\Z^{k}$ with $n_{\ast}+Qm\in\Omega$.
          \item We use $\poly_{Q}(\Omega\to G_{\N}\vert\Gamma)$ to denote the set of all $g\in \poly(\Z^{k}\to G_{\N})$ which is \emph{partially $Q$-periodic on $\Omega$,} meaning that $g(n+Qm)^{-1}g(n)\in \Gamma$ for all $m,n\in\Z^{k}$ with $n,n+Qm\in\Omega$.
       \end{itemize}
 \end{defn}	
 
 %
 




 \begin{ex}
 	The set
 	$\poly(\Z^{k}\to G_{\N}\vert \Gamma)$ is simply the set of all $\Gamma$-valued (i.e. null) polynomials. In other words, $\poly(\Z^{k}\to \Gamma_{\N})=\poly(\Z^{k}\to G_{\N}\vert \Gamma)$.  The set $\poly_{\approx Q,\bold{0}}(\Z^{k}\to G_{\N}\vert \Gamma)$  is the set of $Q$-rational polynomial sequences in $\poly(\Z^{k}\to G_{\N})$. The set $\poly_{p}(\Z^{k}\to G_{\N}\vert \Gamma)$  is the set of $p$-periodic polynomial sequences in $\poly(\Z^{k}\to G_{\N})$.
 \end{ex}
 
 \begin{ex}
 Let $(\R/\Z)_{\N}$ be the nilmanifold  with the filtration given by $G_{0}=\dots=G_{s}=\R$ and $G_{i}=\{id_{G}\}$ for $i>s$. Then $\poly(\Omega\to \R_{\N}\vert\Z)$ is the collection of all polynomials in $\poly(\Omega\to \R\vert\Z)$ (defined in Definition \ref{1:polygeneral}) of degree at most $s$, and $\poly_{p}(\Omega\to \R_{\N}\vert\Z)$ is the collection of all polynomials in $\poly_{p}(\Omega\to \R\vert\Z)$ (defined in Definition \ref{1:polygeneral}) of degree at most $s$.
 \end{ex}


 \begin{ex}\label{1:notagroup2}  
 By Corollary B.4 of \cite{GTZ12},   
  it is not hard to see that $\poly(\Omega\to G_{\N}\vert \Gamma)$ and $\poly_{\approx Q,n_{\ast}}(\Z^{k}\to G_{\N}\vert \Gamma)$ are groups.
       However, we caution the readers that $\poly_{p}(\Z^{k}\to G_{\N}\vert \Gamma)$ is not necessarily a group. Let $k=1$ and $G=(\R^{3},\ast)$ be the Heisenberg group with the group action given by $$(x,y,z)\ast(x',y',z'):=(x+x',y+y',z+z'+xy')$$
       with the lower central series as its $\N$-filtration, meaning that $G_{0}=G_{1}=G$, $G_{2}=\{0\}\times\{0\}\times\R$ and $G_{s}=\{\bold{0}\}$ for $s\geq 3$.
       Indeed, let $d=k=1$, $\Gamma=\Z^{3}$, $g_{1},g_{2}\in \Gamma$ and $f_{i}(n)=g_{i}^{n/p}$ for all $n\in\Z$ and $1,2$. Then $f_{1},f_{2}\in \poly_{p}(\Z\to G_{\N}\vert \Gamma)$. However,
       $$f_{2}(n+p)^{-1}f_{1}(n+p)^{-1}f_{1}(n)f_{2}(n)=g_{2}^{-1-\frac{n}{p}}g_{1}^{-1}g_{2}^{\frac{n}{p}}=g_{2}^{-1}[g_{2}^{\frac{n}{p}},g_{1}]g_{1}^{-1}=[g_{2},g_{1}]^{\frac{n}{p}}g_{2}^{-1}g_{1}^{-1}.$$
       In general, $[g_{2},g_{1}]^{\frac{n}{p}}$ does not always belong to $\Gamma$ if we choose $g_{1}$ and $g_{2}$ properly. This means that $f_{1}f_{2}$ does not need to belong to $\poly_{p}(\Z\to G_{\N}\vert \Gamma)$. 
  \end{ex}

Although 	$\poly_{p}(\Omega\to G_{\N}\vert \Gamma)$ is not closed under multiplication, we have the following:

\begin{lem}\label{1:grg}
	Let  $k,Q\in\N_{+}$, $G/\Gamma$ be an $\N$-filtered nilmanifold,  and  $\Omega$ be a  subset of $\Z^{k}$. For all $f\in\poly_{Q}(\Omega\to G_{\N}\vert \Gamma)$ and $g\in\poly(\Omega\to G_{\N}\vert \Gamma)$, we have that $fg\in \poly_{Q}(\Omega\to G_{\N}\vert \Gamma)$.
\end{lem}
\begin{proof}
By Corollary B.4 of \cite{GTZ12}, $\poly(\Omega\to G_{\N})$ is a group. So $fg\in \poly(\Omega\to G_{\N})$. So it suffices to show that for all $m,n\in\Z^{k}$ with $n,n+Qm\in\Omega$, we have that 
$$(f(n+Qm)g(n+Qm))^{-1}(f(n)g(n))=g(n+Qm)^{-1}(f(n+Qm)^{-1}f(n))g(n)$$
belongs to $\Gamma$. Since $f\in\poly_{Q}(\Omega\to G_{\N}\vert \Gamma)$, we have that $f(n+Qm)^{-1}f(n)\in\Gamma$. Since $g\in\poly(\Omega\to G_{\N}\vert \Gamma)$, we have that $g(n+Qm)^{-1},g(n)\in\Gamma$. We are done.
\end{proof}


    The following lemma allows one to convert a partially periodic sequence to a rational one.

     \begin{lem}\label{1:oehfr}
    Let $k,Q,s\in\N_{+}$,  $\Omega\subseteq \Z^{k}$ be a non-empty $Q$-periodic set,  $G/\Gamma$ be a nilmanifold of step at most $s$, and $g\in\poly_{Q}(\Omega\to G_{\N}\vert\Gamma)$. 
    Then for any  $n_{\ast}\in\Omega$, the map $g\in\poly(\Z^{k}\to G_{\N})$ given by 
    $$g'(n):=\{g(n_{\ast})\}^{-1}g(n)[g(n_{\ast})]^{-1}$$
     is a $(s!)^{k}Q^{s}$-rational map in $\poly_{Q}(\Omega\to G_{\N}\vert\Gamma)$. 
 \end{lem}
 \begin{proof}
  It is easy to check that $g'\in\poly_{Q}(\Omega\to G_{\N}\vert\Gamma)$. For all $m\in\Z^{k}$, we have that $g'(n_{\ast}+Qm)^{-1}g'(n_{\ast})\in\Gamma$. Since $g'(n_{\ast})=id_{G}$, we have that  $g'(n_{\ast}+Qm)\in\Gamma$ for all $m\in\Z^{k}$. 
  Let $\psi\colon G\to \Z^{m}$ be the Mal'cev coordinate map. By interpolation, this implies that all the coefficients of  $\psi(g'(n_{\ast}+Q\cdot))$ are in $\Z^{k}/(s!)^{k}$, and so all the coefficients of  $\psi\circ g'$ are in $\Z^{k}/(s!)^{k}Q^{s}$, which implies that $g'$ is $(s!)^{k}Q^{s}$-rational.
 \end{proof}


In this paper, we also need the following proposition for partially rational polynomial sequences.

\begin{prop}\label{1:oehf}
      Let $k,Q,s\in\N_{+}$, $p\gg_{k,Q,s} 1$ be a prime, $\Omega\subseteq \Z^{k}$ be a non-empty $p$-periodic set, $n_{\ast}\in\Omega$, $G/\Gamma$ be a nilmanifold of step at most $s$ and complexity at most $Q$, and $g\in\poly_{\approx Q,n_{\ast}}(\Omega\to G_{\N}\vert\Gamma)$. Let $\psi\colon G\to\R^{m}$ be the Mal'cev coordinate map. Then there exists $Q'\in\N_{+}$ with $Q'\ll_{k,s} Q^{O_{k,s}(1)}$ such that  $g$ belongs to $\poly_{Q'}(\Omega\to G_{\N}\vert\Gamma)$. 
    \end{prop}
    
    In order to prove Proposition \ref{1:oehf}, we need the following result:
    
    \begin{prop}\label{1:polyGT}
Let $k,s\in\N_{+}$, $\d>0$ and $P\in\poly(\Z^{k}\to\Q)$ be a rational polynomial of degree at most $s$. Suppose that  $P$ is $K$-periodic for some $K\in\N_{+}$ and suppose that there exists a  subset $W$ of $[K]^{k}$ of cardinality at least $\d K^{k}$ such that $P(W)\subseteq \Z$. Then the coefficients of $P$ belong to $\Z/r$ for some $r\in\N_{+}$ with  $r\ll_{k,s}\d^{-O_{k,s}(1)}$.
\end{prop}
\begin{proof}
Let 
 $F\colon \T\to [0,1]$ be a Lipschitz function supported on $(-\d/10,\d/10)$ with $F(0)=1$ and with Lipschitz norm at most $100\d^{-1}$. 
    Then $\E_{n\in\Z^{k}} F(P(n) \mod \Z)\geq \d$, and $\int_{\T}F\,dm\leq \d/2$ (where $m$ is the Lebesgue measure on $\T$). So 
    $$\frac{\vert\E_{n\in\Z^{k}} F(P(n) \mod \Z)-\int_{\T}F\,dm\vert}{\Vert F\Vert_{\Lip}}\geq \d^{2}/200$$
    and thus the sequence $(P(n) \mod\Z)_{n\in \Z^{k}}$ is not $\d^{2}/200$-equidistributed in $\T$. By Theorem 8.6 of \cite{GT12b} (see also Theorem \ref{1:Lei0}), there exists $k\in\Z\backslash\{0\}$ with  $0<\vert k\vert\ll_{k,s}\d^{-O_{k,s}(1)}$ such that $kP(n)\equiv a\mod \Z$ for some $a\in\R$ for all $n\in \Z^{k}$. Since $W$ is non-empty, we may take $a=0$. So by interpolation, the coefficients of $P$ belong to $\Z/r$ for some $r\in\N_{+}$ with  $r\ll_{k,s}\d^{-O_{k,s}(1)}$.
\end{proof}
    
\begin{proof}[Proof of Proposition \ref{1:oehf}]

 Let $\psi=(\psi_{1},\dots,\psi_{m})\colon G\to\R^{m}$ be the Mal'cev coordinate map of $\psi$. Then each $\psi_{i}(g(n))$ is a rational polynomial of degree at most $s$.  
 For convenience denote $f_{i}:=\psi_{i}\circ g$.
  Since $p\gg_{Q} 1$, for any $x\in\Omega$, there exists $v_{x},u_{x}\in\Z^{k}$ such that $n_{\ast}+Qv_{x}+pu_{x}=x$. 
 Then for all $y\in\Z^{k}$, we have that $n_{\ast}+Qv_{x}+pQy=n_{\ast}+Q(v_{x}+py)\in \Omega$ and thus
   $f_{i}(n_{\ast}+Qv_{x}+pQy)\in\Z$. By interpolation, all  the coefficients of the polynomial $f_{i}(n_{\ast}+Qv_{x}+pQ\cdot)$ belong to $\Z/(s!)^{k}$ and thus all  the coefficients of the polynomial $f_{i}(n_{\ast}+Qv_{x}+p\cdot)$ belong to $\Z/Q^{s}(s!)^{k}$. So we have that
  \begin{equation}\label{1:wp9gu0}
\begin{split}
    f_{i}(x)\in\Z/Q^{s}(s!)^{k} \text{ for all } x\in\Omega.
\end{split}
\end{equation}

 

 For all $n',n\in \Z^{k}$ and $q\in\N_{+}$, by Lemma A.3 of \cite{GT12b}, for all $1\leq i\leq m$, it is not hard to see that we may write $\psi_{i}(g(n'+qn)^{-1}g(n'))$ in the form
\begin{equation}\nonumber
\begin{split}
 \psi_{i}(g(n'+qn)^{-1}g(n'))=P_{i}(\bold{f}(n'))
 +\sum_{i'=1}^{i}(f_{i'}(n'+qn)-f_{i'}(n'))P_{i,i'}(\bold{f}(n'+qn),\bold{f}(n')),
\end{split}
\end{equation}
where $\bold{f}(n):=(f_{1}(n),\dots,f_{i}(n))$ and $P_{i}, P_{i,i'}$ are some polynomials of degree $O_{k,s}(1)$ with coefficients in $\Z/q_{1}$ for some $q_{1}\in\N_{+}$ with $q_{1}\ll_{k,s} Q^{O_{k,s}(1)}$.
 Taking $n=\bold{0}$, we see that $P_{i}(\bold{f}(\cdot))\equiv 0$. So
 \begin{equation}\label{1:mcalfor}
\begin{split}
 \psi_{i}(g(n'+qn)^{-1}g(n'))=\sum_{i'=1}^{i}(f_{i'}(n'+qn)-f_{i'}(n'))P_{i,i'}(\bold{f}(n'+qn),\bold{f}(n')).
\end{split}
\end{equation}

 Since $g$ is partially $Q$-rational with base point $n_{\ast}$ on $\Omega$,
   we have that $f_{i}(n_{\ast}+Qn)\in\Z$ whenever $n_{\ast}+Qn\in\Omega$.
Since $\Omega$ is $p$-periodic, we have that  $f_{i}(n_{\ast}+pQn)\in\Z$ for all $n\in\Z^{k}$ since $\Omega$ is $p$-periodic. Since $f_{i}$ is rational and thus is periodic,
 by Proposition \ref{1:polyGT}, the coefficients of the polynomial $f_{i}(n_{\ast}+pQ\cdot)$ belong to $\Z/rp^{O_{d,s}(1)}$ for some $r\ll_{d,s} Q^{O_{d,s}(1)}$. 
 So by interpolation, all  the coefficients of the polynomial $f_{i}$ belong to $\Z/Q^{s}rp^{O_{k,s}(1)}$.
 
 Moreover, for any $x,y\in\Z^{k}$ and $Q'\in\N_{+}$, since all  the coefficients of the polynomial $f_{i}$ belong to $\Z/Q^{s}rp^{O_{k,s}(1)}$, we have that $f_{i}(x+Q'Q^{s}ry)-f_{i}(x)\in Q'\Z/p^{O_{k,s}(1)}$. If in addition $x,x+Q'Q^{s}ry\in \Omega$, then it follows from (\ref{1:wp9gu0}) that 
 $$f_{i}(x+Q'Q^{s}ry)-f_{i}(x)\in (\Z/Q^{s}(s!)^{k})\cap(Q'\Z/p^{O_{k,s}(1)})=Q'\Z.$$
 Therefore, it follows from (\ref{1:mcalfor}) that there exists some $Q'\in\N_{+}$ with $Q'\ll_{k,s} Q^{O_{k,s}(1)}$ such that 
 $\psi_{i}(g(n'+Q'Q^{s}rn)^{-1}g(n'))\in\Z$ for all $n,n'\in\Z^{k}$ with $n',n'+Q'Q^{s}rn\in \Omega$. This means that $g$ is partially $Q'Q^{s}r$-periodic on $\Omega$. This completes the proof.
 \end{proof}

 Since $\Z^{k}$ is $p$-periodic, as a consequence of Proposition \ref{1:oehf}, we recover Lemma A.12 of \cite{GT12b}:
 
   \begin{coro}[Lemma A.12 of \cite{GT12b}]\label{1:r222p}
        Let $k,s\in\N_{+}$ and $G/\Gamma$ be a nilmanifold of step at most $s$. Then every rational polynomial sequence in $\poly(\Z^{k}\to G_{\N})$ is periodic.
    \end{coro}

    \subsection{Polynomial sequences over finite fields}
    
    One can also extend the domain of polynomial sequences from $\Z^{k}$ to $\F_{p}^{k}$.

	\begin{defn}[Polynomial sequences in $\F_{p}^{k}$]
	Let  $k,Q\in\N_{+}$, $G/\Gamma$ be an $\N$-filtered nilmanifold, and $p$ be a prime. Let
 $\poly(\F_{p}^{k}\to G_{\N})$ denote the set of all functions of the form $f\circ\tau$ for some $f\in \poly(\Z^{k}\to G_{\N})$.
 For a subset $\Omega$ of $\F_{p}^{k}$, let
 $\poly_{p}(\Omega\to G_{\N}\vert\Gamma)$ denote the set of all functions of the form $f\circ\tau$ for some $f\in \poly_{p}(\iota^{-1}(\Omega)\to G_{\N}\vert\Gamma).$ We define $\poly(\Omega\to G_{\N}\vert\Gamma)$  similarly. 
  We call functions in $\poly_{p}(\Omega\to G_{\N}\vert\Gamma)$ \emph{partially $p$-periodic polynomial sequences on $\Omega$.}
 \end{defn}	
 
 \begin{ex}\label{1:notagroup}    
 By Corollary B.4 of \cite{GTZ12}, it is not hard to show that $\poly(\F_{p}^{k}\to G_{\N})$ is a group. However, unlike the conclusion of  Corollary B.5 of \cite{GTZ12}, $\poly(\F_{p}^{k}\to G_{\N})$ is not closed under translations and transformations.  To see this, let $k=1$ and $G/\Gamma=\R/\Z$ be the torus with $G_{1}=G$ and $G_{2}=\{id_{G}\}$. Let $f(n)=n/p$ for all $n\in\Z$. Then $g:=f\circ\tau\in \poly(\F_{p}\to G_{\N}\vert \Gamma)$. However, $g':=g(\cdot +1)$ does not belong to $\poly(\F_{p}\to G_{\N}\vert \Gamma)$. 
       
       To see this, suppose on the contrary that $g'\in \poly(\F_{p}\to G_{\N}\vert \Gamma)$. Then we may write $g'=f'\circ \tau$ for some $f'\in\poly(\Z\to G_{\N}\vert \Gamma)$.
       By the construction of the filtration $G_{\N}$, we must have $f'(n)=an+b$ for some $a,b\in\R$ for all $n\in\Z$.
        Then 
       \begin{equation}\label{1:ffgg}
       a\tau(n)+b=f'\circ\tau(n)=g'(n)=g(n+1)=f\circ\tau(n+1)=\tau(n+1)/p
       \end{equation}
       for all $n\in\F_{p}$. Setting $n=0$ and 1 in (\ref{1:ffgg}), we deduce that $a=b=1/p$. Setting $n=p-1$ in (\ref{1:ffgg}), we have that $1=a(p-1)+b=0$, a contradiction.
       \end{ex}

  As is explained in Example \ref{1:notagroup}, partially $p$-periodic polynomial sequences are not closed under translations and transformations. However, the following proposition asserts that partially $p$-periodic  are indeed closed under translations and transformations modulo $\Gamma$.

       \begin{prop}\label{1:BB}
       	Let $k\in\N_{+}$,  $G/\Gamma$ be an $\N$-filtered nilmanifold,  $p$ be a prime, $\Omega$ be a subset of $\F_{p}^{k}$, and $g\in\poly_{p}(\Omega\to G_{\N}\vert \Gamma)$. 
       	\begin{enumerate}[(i)]
       		\item For any $h\in\F_{p}^{k}$, there exists $g'\in\poly_{p}((\Omega-h)\to G_{\N}\vert \Gamma)$   such that $g(n+h)\Gamma=g'(n)\Gamma$ for all $n\in\Omega-h$.
       		\item For any $k'\in\N_{+}$ and  linear transformation $L\colon \F_{p}^{k'}\to \F_{p}^{k}$, there exists  $g''\in\poly_{p}(L^{-1}(\Omega)\to G_{\N}\vert \Gamma)$ such that   $g(L(n))\Gamma=g''(n)\Gamma$ for all $n\in L^{-1}(\Omega)$.	
       	\end{enumerate}	
\end{prop}	
       \begin{proof}
 Assume that $g=f\circ \tau$ for some $f\in\poly_{p}(\iota^{-1}(\Omega)\to G_{I}\vert \Gamma)$.
 We first prove Part (i). 
       Let $f'\colon \Z^{k}\to G$ be the function given by $f':=f(\cdot+\tau(h))$. 
       Since $f$ belongs to $\poly(\Z^{k}\to G_{\N})$, so does $f'$ by Corollary B.5 of \cite{GTZ12} (or by definition). 
        Note that $\tau(n+h)\equiv\tau(n)+\tau(h) \mod p\Z$.
       For all $n\in\iota^{-1}(\Omega-h)=\iota^{-1}(\Omega)-\tau(h)$ and $m\in\Z^{k}$, since $n+\tau(h),n+\tau(h)+pm\in\iota^{-1}(\Omega)$, we have that
       $f'(n+pm)^{-1}f'(n)=f(n+pm+\tau(h))^{-1}f(n+\tau(h))\in\Gamma.$ So $f'$ belongs to $\poly_{p}(\iota^{-1}(\Omega-h)\to G_{\N}\vert \Gamma)$ and $g':=f'\circ\tau$ belongs to  $\poly_{p}((\Omega-h)\to G_{\N}\vert \Gamma)$.   
       On the other hand, for all $n+h\in \Omega$, since $\tau(n+h), \tau(n)+\tau(h)\in \iota^{-1}(\Omega)$, we have that
       $$g(n+h)=f(\tau(n+h))\equiv f(\tau(n)+\tau(h))=f'\circ\tau(n)=g'(n) \mod \Gamma.$$ 
       
       We now prove part (ii).
       It is not hard to see that there exists a  linear transformation $L'\colon \Z^{k'}\to\Z^{k}$ such that $L'\circ\tau(n)\equiv\tau\circ L(n) \mod p\Z^{k}$ for all $n\in\Z^{k'}$. 
       Let $f''\colon\Z^{k'}\to G$ be the function given by $f''=f\circ L'$. Since        $f$ belongs to $\poly(\Z^{k}\to G_{\N})$, it is not hard to see that $f''$ belongs to $\poly(\Z^{k'}\to G_{\N})$ by Corollary B.5 of \cite{GTZ12} (or by definition).
      For all $n\in\tau(L^{-1}(\Omega))$ and $m\in\Z^{k'}$, we may write $n=\tau(n')$ for some $n'\in L^{-1}(\Omega)$. So
      \begin{equation}\label{1:ooss}
      \begin{split}
    &\quad  f''(n+pm)^{-1}f''(n)=f(L'(\tau(n'))+L'(pm))^{-1}f(L'(\tau(n')))
    \\&=\Bigl(f(L'(\tau(n'))+L'(pm))^{-1}f(\tau(L(n'))+L'(pm))\Bigr)\\&\qquad\cdot\Bigl(f(\tau(L(n'))+L'(pm))^{-1}f(\tau(L(n')))\Bigr)\cdot\Bigl(f(\tau(L(n')))^{-1}f(L'(\tau(n')))\Bigr).
      \end{split}
      \end{equation}
      Note that $L'(\tau(n'))\equiv\tau(L(n')) \mod p\Z^{k}$, $L'(pm)=pL'(m)\in p\Z^{k}$ and $\tau(L(n'))\in\iota^{-1}(\Omega)$. It follows from the fact $f\in\poly(\Z^{k}\to G_{I}\vert \Gamma)$ that (\ref{1:ooss}) belongs to $\Gamma$.
       Since $\iota^{-1}(L^{-1}(\Omega))=\tau(L^{-1}(\Omega))+p\Z^{k'}$, we have that 
         $f''$ belongs to $\poly_{p}(\iota^{-1}(L^{-1}(\Omega))\to G_{\N}\vert \Gamma)$ and $g'':=f''\circ\tau$ belongs to  $\poly_{p}(L^{-1}(\Omega)\to G_{\N}\vert \Gamma)$. 
     On the other hand, for all $n\in L^{-1}(\Omega)$, since $\tau\circ L(n), L'\circ\tau(n)\in \iota^{-1}(\Omega)$, we have that
       $$g\circ L(n)=f\circ\tau\circ L(n)\equiv f\circ L'\circ\tau(n)=f''\circ\tau(n)=g''(n) \mod \Gamma.$$ 
       \end{proof}



\subsection{The Baker-Campbell-Hausdorff formula}	

The material of this section comes from Appendix C of \cite{GT10b}. We write it down for completeness.

Let $G$ be a group, $t\in\N_{+}$ and $g_{1},\dots,g_{t}\in G$. The \emph{iterated commutator} of $g_{1}$ is defined to be $g_{1}$ itself. Iteratively, we define an \emph{iterated commutator} of $g_{1},\dots,g_{t}$ to be an element of the form $[w,w']$, where $w$ is an iterated commutator of $g_{i_{1}},\dots,g_{i_{r}}$, $w'$ is an iterated commutator of $g_{i'_{1}},\dots,g_{i'_{r'}}$ for some $1\leq r,r'\leq t-1$ with $r+r'=t$ and $\{i_{1},\dots,i_{r}\}\cup \{i'_{1},\dots,i'_{r'}\}=\{1,\dots,t\}$.

Similarly, let $X_{1},\dots,X_{t}$ be elements of a Lie algebra. The \emph{iterated Lie bracket} of $X_{1}$ is defined to be $X_{1}$ itself. Iteratively, we define an \emph{iterated Lie bracket} of $X_{1},\dots,X_{t}$ to be an element of the form $[w,w']$, where $w$ is an iterated Lie bracket of $X_{i_{1}},\dots,X_{i_{r}}$, $w'$ is an iterated Lie bracket of $X_{i'_{1}},\dots,X_{i'_{r'}}$ for some $1\leq r,r'\leq t-1$ with $r+r'=t$ and $\{i_{1},\dots,i_{r}\}\cup \{i'_{1},\dots,i'_{r'}\}=\{1,\dots,t\}$.

Let $G$ be a connected and simply connected nilpotent Lie group. The \emph{Baker-Campbell-Hausdorff formula} asserts that for all $X_{1},X_{2}\in\log G$, we have
$$\exp(X_{1})\exp(X_{2})=\exp\Bigl(X_{1}+X_{2}+\frac{1}{2}[X_{1},X_{2}]+\prod_{\alpha}c_{\alpha}X_{\alpha}\Bigr),$$	
where $\alpha$ is a finite set of labels, $c_{\alpha}$ are real constants, and $X_{\alpha}$ are iterated Lie brackets with $k_{1,\alpha}$ copies of $X_{1}$  and $k_{2,\alpha}$ copies of $X_{2}$ for some $k_{1,\alpha},k_{2,\alpha}\geq 1$ and $k_{1,\alpha}+k_{2,\alpha}\geq 3$. One may use this formula to show that for all $g_{1},g_{2}\in G$ and $x\in \R$, we have that
\begin{equation}\nonumber\label{1:C1}
(g_{1}g_{2})^{x}=g_{1}^{x}g_{2}^{x}\prod_{\alpha}g_{\alpha}^{Q_{\alpha}(x)},
\end{equation}	
where $\alpha$ is a finite set of labels,   $g_{\alpha}$ are iterated commutators with $k_{1,\alpha}$ copies of $g_{1}$  and $k_{2,\alpha}$ copies of $X_{2}$ for some $k_{1,\alpha},k_{2,\alpha}\geq 1$, and $Q_{\alpha}\colon\R\to\R$ are polynomials of degrees at most $k_{1,\alpha}+k_{2,\alpha}$ without constant terms.

Similarly, one can show that for any $g_{1},g_{2}\in G$ and $x_{1},x_{2}\in \R$, we have that	
\begin{equation}\nonumber
[g_{1}^{x_{1}},g_{2}^{x_{2}}]=[g_{1},g_{2}]^{x_{1}x_{2}}\prod_{\alpha}g_{\alpha}^{P_{\alpha}(x_{1},x_{2})},
\end{equation}	
where $\alpha$ is a finite set of labels,   $g_{\alpha}$ are iterated commutators with $k_{1,\alpha}$ copies of $g_{1}$  and $g_{2,\alpha}$ copies of $X_{2}$ for some $k_{1,\alpha},k_{2,\alpha}\geq 1$, $k_{1,\alpha}+k_{2,\alpha}\geq 3$, and $P_{\alpha}\colon\R^{2}\to\R$ are polynomials of degrees at most $k_{1,\alpha}$ in $x_{1}$ and at most $k_{2,\alpha}$ in $x_{2}$ which vanishes when $x_{1}x_{2}=0$.
	
\subsection{Type-I horizontal torus and character}

Horizontal torus and character are important terminologies which characterize the equidistribution properties of nilsequences. In literature, there are at least three different types of horizontal  toruses and characters. In this paper as well as  the series of work \cite{SunB,SunC,SunD}, we refer to them as  \emph{type-I}, \emph{type-II} and \emph{type-III} respectively. In this paper, we use the type-I horizontal torus and character defined below:\footnote{We will define the type-II horizontal torus and character in \cite{SunC}, and the type-III horizontal torus and character in \cite{SunD}.}

\begin{defn}[Type-I horizontal torus and character]
Let $G/\Gamma$ be a nilmanifold endowed with a Mal'cev basis $\mathcal{X}$.  The \emph{type-I horizontal torus} of $G/\Gamma$ is $G/[G,G]\Gamma$.
	A \emph{type-I horizontal character} is a continuous homomorphism $\eta\colon G\to \R$ such that $\eta(\Gamma)\subseteq \Z$.
	When written in the coordinates relative to $\mathcal{X}$, we may write $\eta(g)=k\cdot \psi(g)$ for some unique $k=(k_{1},\dots,k_{m})\in\Z^{m}$, where $\psi\colon G\to \R^{m}$ is the coordinate map with respect to the Mal'cev basis  $\mathcal{X}$. We call the quantity $\Vert\eta\Vert:=\vert k\vert=\vert k_{1}\vert+\dots+\vert k_{m}\vert$ the \emph{complexity} of $\eta$ (with respect to $\mathcal{X}$). 
\end{defn}

It is not hard to see that any type-I horizontal character mod $\Z$ vanishes on $[G,G]\Gamma$ and thus descent to a continuous homomorphism between the type-I horizontal torus $G/[G,G]\Gamma$ and $\R/\Z$. Moreover, $\eta \mod \Z$ is a well defined map from $G/\Gamma$ to $\R/\Z$. 

Type-I horizontal torus and character are used to characterize whether a nisequence is equidistributed on a nilmanifold \cite{GT12b,Lei05}.  We provide some properties for later uses.
The proof of the following lemma is similar to that of Lemma 6.7 of \cite{GT12b}, Lemma B.9 of  \cite{GTZ12} and Lemma 2.8 of \cite{CS14}. We omit the details. 
\begin{lem}\label{1:B.9}
	Let $k\in\N_{+}$, $s\in\N$ 
	 and $G$ be an $\N$-filtered group of degree at most $s$. A function $g\colon \Z^{k}\to G$ belongs to $\poly(\Z^{k}\to G_{\N})$ if and only if 
	 for all $i\in\N^{k}$ with $\vert i\vert\leq s$,
	 there exists  $X_{i}\in \log(G_{\vert i\vert})$  such that 
	$$g(n)=\prod_{i\in\N^{k}, \vert i\vert\leq s}\exp\Bigl(\binom{n}{i}X_{i}\Bigr).$$
	Moreover, if $g\in\poly(\Z^{k}\to G_{\N})$, then the choice of $X_{i}$ are unique.
	
	Furthermore, if $G'$ is a subgroup of $G$ and $g$ takes values in $G'$, then $X_{i}\in \log(G')$ for all $i\in\N^{k}$. 
\end{lem}	

An equivalent way of saying that a function $g\colon \Z^{k}\to G$ belongs to $\poly(\Z^{k}\to G_{\N})$ is that 
\begin{equation}\label{1:tl}
g(n)=\prod_{i\in\N^{k}, \vert i\vert\leq s}g_{i}^{\binom{n}{i}}
\end{equation}
for some $g_{i}\in G_{\vert i\vert}$ (with any fixed order in $\N^{k}$).
We call $g_{i}$ the \emph{($i$-th) type-I Taylor  coefficient} of $g$, and (\ref{1:tl}) the \emph{type-I Taylor expansion} of $g$.

The following lemma is obvious and we omit the proof:

\begin{lem}\label{1:projectionpp}
	Let $d,s\in\N_{+}$, $p$ be a prime, $\Omega\subseteq \Z^{d}$, and $\eta$ be a type-I horizontal character on some $\N$-filtered   degree $\leq s$ nilmanifold $G/\Gamma$. If $g\in\poly_{p}(\Omega\to G_{\N}\vert\Gamma)$, then $\eta\circ g\in \poly_{p}(\Omega\to \R\vert\Z)$ and is of degree at most $s$.
\end{lem}	

We can say more about the coefficients of $\eta\circ g$ for some partially $p$-periodic polynomial sequence $g$ on $\Omega$.

\begin{lem}\label{1:goodcoordinates}
Let $d,s\in\N_{+}$, $p\gg_{d,s} 1$ be a prime, $\Omega$ be a non-empty $p$-periodic subset of $\Z^{d}$, and $\eta$ be a type-I horizontal character on some $\N$-filtered degree $\leq s$ nilmanifold $G/\Gamma$.
Then for all $g\in\poly_{p}(\Omega\to G_{\N}\vert\Gamma)$, $p^{s}\eta\circ g(n)$  takes values in $\Z+C$ for some $C\in\R$.
\end{lem}
\begin{proof}
	Since $g\in\poly_{p}(\Omega\to G_{\N}\vert\Gamma)$, we have that $\eta\circ g\in \poly_{p}(\Omega\to \R\vert\Z)$ and is of degree at most $s$ by Lemma \ref{1:projectionpp}. 
Since $\Omega$ is $p$-periodic, for all $n\in\Omega$ and $m\in \Z^{d}$, we have that 
 	$$\eta\circ g(n+pm)-\eta\circ g(n)\in\Z.$$ 
 	Since $\Omega$ is non-empty, by interpolation, it is not hard to see  that
	we may write $$Q\eta\circ g(m)-C=\frac{1}{p^{s}}f(m)$$
	for some $C\in\R$, $Q\in \N$ with $p\nmid Q$, and some integer valued polynomial $f$. It remains to show that $Q\vert f(n)$ for all $n\in\Z^{d}$.

		Fix any $n_{0}\in \Omega$. 
		For any $n\in\Z^{d}$, there exists $m\in\Z^{d}$ such that $n_{0}+pm\equiv n \mod s!Q\Z^{d}$. 
		Since $f$ is integer valued, we have that
		$f(n)\equiv f(n_{0}+pm) \mod Q\Z$. On the other hand, we have $$\eta\circ(n_{0}+pm)-\eta\circ(n_{0})=(f(n_{0}+pm)-f(n_{0}))/Qp^{s}\in\Z.$$ So $f(n)\equiv f(n_{0}+pm)\equiv f(n_{0}) \mod Q\Z$ for all $n\in\Z^{d}$. Replacing $C$ by $C+\frac{1}{p^{s}}f(n_{0})$ if necessary, we have that $Q\vert f(n)$ for all $n\in\Z^{d}$.  
		This finishes the proof.
\end{proof}

 We end this section with a technical lemma to be used in later sections.

 \begin{prop}\label{1:normalizegamma}
 	Let $d,s\in\N_{+}$, $p$ be a prime, and  $\Omega$ be a non-empty $p$-periodic subset of $\Z^{d}$.
 	Let $G/\Gamma$ be a degree $\leq s$  nilmanifold of complexity at most $C$ with a filtration $G_{\N}$ and a Mal'cev basis $\mathcal{X}$. Let $\eta$ be a type-I horizontal character of $G$ of complexity at most $C$ and $g\in\poly_{p}(\Omega\to G_{\N}\vert\Gamma)$. Suppose that $\eta\circ g$ is the sum of an integer valued polynomial and a polynomial of degree $s'$ for some $0\leq s'\leq s$. If   $p\gg_{C,s} 1$, then there exists $\gamma\in\poly(\Z^{d}\to G_{\N}\vert\Gamma)$  
	with $\gamma(\bold{0})=id_{G}$ such that $\eta\circ g -\eta\circ \gamma$ is a polynomial of degree at most $s'$.
 \end{prop}	
 \begin{proof}
 	Since $G/\Gamma$ is of complexity at most $C$,	it is not hard to see that there exists a Mal'cev basis $\mathcal{X}'$ of $G/[G,G]$ such that $(G/[G,G])/(\Gamma/(\Gamma\cap [G,G]))$ is an abelian nilmanifold (endowed with the quotient filtration) of complexity $O_{C}(1)$ with respect to $\mathcal{X}'$, and that $\eta$ induces a type-I horizontal character of $G/[G,G]$ of complexity $O_{C}(1)$.
 	Note that $\eta$ annihilates $[G,G]$.
 	Moreover, we have that $g \mod [G,G]$ belongs to $\poly_{p}(\Omega\to (G/[G,G])_{\N}\vert\Gamma/(\Gamma\cap [G,G]))$.
 	So by the discussion above, it suffices to prove this proposition under the assumption that $G$ is abelian. So from now on we assume that $G$ is abelian.
 	
 	Denote $m_{i}=\dim(G_{i}), 0\leq i\leq s$ and $m=\dim(G)$. It is convenient to work with the additive notation and write
 	$$g(n)=\sum_{i\in\N^{d},0\leq\vert i\vert\leq s}g_{i}\binom{n}{i}$$
 	for some $g_{i}\in\{0\}^{m-m_{\vert i\vert}}\times\R^{m_{\vert i\vert}}$. Assume that $\eta(g)=k\cdot g$ for all $g\in G$ for some $k=(k_{1},\dots,k_{m})\in\Z^{m}$ with $\vert k\vert\leq C$. 
 	Fix any $i\in\N^{d},s'<\vert i\vert\leq s$.
 	Since $\eta\circ g \mod \Z$ is the sum of an integer valued polynomial and a polynomial of degree $s'$, we have that 
 	$\eta(g_{i})=k\cdot g_{i}\in\Z.$
 	Since  $g\in\poly_{p}(\Omega\to G_{\N}\vert\Gamma)$, by Lemma \ref{1:goodcoordinates}, we have that $p^{s}g_{i}\in\Z^{m}$. Letting $c_{i}$ denote the greatest common divisor of $k_{m-m_{\vert i\vert}+1},\dots,k_{m}$, we have that 
 	$$p^{s}\eta(g_{i})=k\cdot (p^{s} g_{i})\in k\cdot (\{0\}^{m-m_{\vert i\vert}}\times\Z^{m_{\vert i\vert}})\subseteq c_{i}\Z.$$
 	On the other hand,
 	$p^{s}\eta(g_{i})\in p^{s}\Z$. So if $p\gg_{C,s} 1$, then
 	$$p^{s}\eta(g_{i})\in (c_{i}\Z)\cap (p^{s}\Z)=p^{s}c_{i}\Z,$$
 	which implies that $\eta(g_{i})\in c_{i}\Z$. By the definition of $c_{i}$, there exists $\gamma_{i}\in \{0\}^{m-m_{\vert i\vert}}\times\Z^{m_{\vert i\vert}}$ such that $\eta(g_{i})=\eta(\gamma_{i})$.
 	Denote $$\gamma(n):=\sum_{i\in\N^{d},s'<\vert i\vert\leq s}\gamma_{i}\binom{n}{i},$$
 	we have that $\gamma$ belongs to $\poly(\Z^{d}\to G_{\N})$ and takes values in $\Gamma$, that $\gamma(\bold{0})=id_{G}$, and that $\eta\circ g(n)-\eta\circ \gamma(n)=\sum_{i\in\N^{d},0\leq\vert i\vert\leq s'}\eta(g_{i})\binom{n}{i}$ is a polynomial of degree at most $s'$.
 \end{proof}

\subsection{Vertical torus and character}

The vertical torus is an important concept which allows us to do Fourier decompositions on nilmanifolds. 

 \begin{defn}[Vertical torus and character]\label{1:vtc}
	Let  $s\in\N$,  $G/\Gamma$ be a nilmanifold with an $\N$-filtration $(G_{i})_{i\in \N}$ of degree $\leq s$ with a Mal'cev basis $\mathcal{X}$ adapted to it. Then $G_{s}$ lies in the center of $G$. The \emph{vertical torus} of $G/\Gamma$ is the set $G_{s}/(\Gamma\cap G_{s})$.
	A \emph{vertical character} of $G/\Gamma$ (with respect to the filtration $(G_{i})_{i\in \N}$) is a continuous homomorphism $\xi\colon G_{s}\to\R$ such that $\xi(\Gamma\cap G_{s})\subseteq \Z$ (in particular, $\xi$ descents to a continuous homomorphism from the vertical torus $G_{s}/(\Gamma\cap G_{s})$ to $\T$). Let $\psi\colon G\to \R^{m}$ be the Mal'cev coordinate map and denote $m_{s}=\dim(G_{s})$. Then there exists $k\in\Z^{m_{s}}$ such that $\xi(g_{s})=(\bold{0},k)\cdot\psi(g_{s})$ for all $g_{s}\in G_{s}$ (note that the first $\dim(G)-m_{s}$ coordinates of $\psi(g_{s})$ are all zero). We call the quantity $\Vert\xi\Vert:=\vert k\vert$ the \emph{complexity} of $\xi$  (with respect to $\mathcal{X}$). 
\end{defn}

  The following lemma is useful in the study of equidistribution properties on nilmanifolds.
 \begin{lem}\label{1:3.7}
 	Let 
	$\Omega$ be a non-empty subset of $\Z^{d}$ and $G/\Gamma$ be a nilmanifold with an  $\N$-filtration $G_{\N}$ of degree $s$. Let $m_{s}=\dim(G_{s})$ and $0<\d<1/2$.  If  a sequence $\O\colon\Omega\to G/\Gamma$ is not $\d$-equidistributed on $G/\Gamma$, then there exists a Lipschitz function $F\in\Lip(G/\Gamma\to \C)$ with
	a vertical character  of complicity at most $O_{m_{s}}(\d^{-O_{m_{s}}(1)})$ such that 
 	$$\limsup_{N\to\infty}\Bigl\vert\E_{n\in \Omega\cap [N]^{d}}F(\O(n))-\int_{G/\Gamma} F\,dm_{G/\Gamma}\Bigr\vert\gg_{m_{s}} \d^{O_{m_{s}}(1)}\Vert F\Vert_{\Lip},$$
	where  $m_{G/\Gamma}$ is the Haar measure of $G/\Gamma$.
 \end{lem}

 The proof of Lemma \ref{1:3.7} is a straightforward extension of Lemma 3.7 of \cite{GT12b}, we omit the details.

\section{Background materials for quadratic forms}\label{1:s:pp5}
\subsection{Basic properties for quadratic forms}

\begin{defn}[Quadratic forms on $\V$]
	We say that a function $M\colon\V\to\F_{p}$ is a \emph{quadratic form} if 
	$$M(n)=(nA)\cdot n+n\cdot u+v$$
	for some $d\times d$ symmetric matrix $A$ in $\F_{p}$, some $u\in \F_{p}^{d}$ and $v\in \F_{p}$.
We say that $A$ is the matrix \emph{associated to} $M$.\footnote{Sometimes in literature a quadratic form is defined to be a map $B\colon \V\times\V\to \V$ with $B(x,y)=(xA)\cdot y$ for some symmetric matrix $A$. We define quadratic forms differently for the convenience of this paper.} 

We say that $M$ is \emph{pure} if $u=\bold{0}$.
We say that $M$ is \emph{homogeneous} if $u=\bold{0}$ and $v=0$. We say that $M$ is \emph{non-degenerate} if $M$ is of rank $d$, or equivalently, $\det(A)\neq 0$.
\end{defn}

A quadratic form we are particularly interested in is the one induced by the dot product, namely $M(n):=n\cdot n$, which is associated with the identity matrix.

Let $M\colon\F_{p}^{d}\to\F_{p}$ be a quadratic form.
We use $\rank(M):=\rank(A)$ to denote the \emph{rank} of $M$.
 One can also define quadratic forms whose domain is an affine subspace of $\V$.
Let $V+c$ be an affine subspace of $\V$ of dimension $r$. There exists a (not necessarily unique) bijective linear transformation $\phi\colon \F_{p}^{r}\to V$.
We say that a function $M\colon V+c\to \F_{p}$ is a \emph{quadratic form} if there exists a quadratic form $M'\colon \F_{p}^{r}\to \F_{p}$ and a  bijective linear transformation $\phi\colon \F_{p}^{r}\to V$ such that $M(n)=M'(\phi^{-1}(n-c))$ for all $n\in V+c$ (or equivalently, $M'(m)=M(\phi(m)+c)$ for all $m\in \F_{p}^{r}$). We define the \emph{rank} $\rank(M\vert_{V+c})$ of $M$ restricted to $V+c$ as the rank of $M'$. Note that if $M(n)=M'(\phi^{-1}(n-c))=M''({\phi'}^{-1}(n-c))$ for some quadratic forms $M',M''$ associated to the matrices $A',A''$, then $A'=BA''{B}^{T}$, where $B$ is the unique $r\times r$ invertible matrix such that $\phi^{-1}\circ\phi(m)=mB$ for all $m\in\F_{p}^{r}$.
So $\rank(M\vert_{V+c})$ is independent of the choice of $M'$ and $\phi'$.

The following is a straightforward property for quadratic forms, which we record for later uses.

\begin{lem}\label{1:or}
	Let $M\colon\V\to\F_{p}$ be a quadratic form associated with the matrix $A$, and $x,y,z\in \V$. Suppose that $M(x)=M(x+y)=M(x+z)=0$. Then $M(x+y+z)=0$ if and only if $(yA)\cdot z=0$.
\end{lem}	
\begin{proof}
    The conclusion follows from the equality
    $$M(x+y+z)-M(x+y)-M(x+z)+M(x)=2(yA)\cdot z.$$
\end{proof}

The next lemma says that using translations and linear transformations, any quadratic form can be transformed into a ``standard" form:	

\begin{lem}\label{1:cov} 
Let $d\in\N_{+}$, $d'\in\N$ and $p$ be a prime.
	For any quadratic form $M\colon\V\to \F_{p}$ of rank $d'$, there exist $c,c',\l\in\F_{p}$ and $v\in\V$ with $c\neq 0$ and a $d\times d$ invertible matrix $R$ in $\F_{p}$ such that writing $M'(n):=M(nR+v)$, we have that
	\begin{equation}\label{1:ded}
	M'(n_{1},\dots,n_{d})=cn^{2}_{1}+n_{2}^{2}+\dots+n^{2}_{d'}+c'n_{d'+1}-\l.
	\end{equation}
	Moreover, if $M$ is homogeneous, then we may further require $c'=\l=0$ in (\ref{1:ded}); if $M$ is pure, then we we may further require $v=\bold{0}$; if $M$ is non-degenerate, then we may further require that $d'=d$.
\end{lem}

\begin{rem}
Throughout this paper, whenever we write a quadratic form $M'$ in the form (\ref{1:ded}), if $d'=d$, then we consider the term $c'n_{d'+1}$ as non-existing.
\end{rem}

\begin{proof}
	 Similar to the real valued quadratic forms, by a suitable substitution $n\mapsto nR$ for some upper triangular invertible matrix $R$ and some $v\in\V$, it suffices to consider the case when the matrix $A$ associated to $M$ is diagonal.
	Since $\rank(M)=d'$, we may assume without loss of generality that the first $d'$ entries of the diagonal of $A$ is non-zero.
	 Let $c\in\{2,\dots,p-1\}$ be the smallest element such that $c\neq x^{2}$ for all $x\in\F_{p}$ (such an element obviously exists).
	 Under a suitable change of variable $(n_{1},\dots,n_{d})\mapsto (c_{1}n_{1},\dots,c_{d'}n_{d'},n_{d'+1},\dots,n_{d})$ for some $c_{1},\dots,c_{d'}\in\F_{p}\backslash\{0\}$ and a change of the orders of $n_{1},\dots,n_{d'}$, we may reduce the problem to the case when
	 $$M(n_{1},\dots,n_{d})=n_{1}^{2}+\dots+n_{k}^{2}+cn_{k+1}^{2}+\dots+cn_{d'}^{2}+a_{1}n_{1}+\dots+a_{d}n_{d}-\lambda$$
	 for some $0\leq k\leq d'$.
	By a suitable substitution $n\mapsto n+m$ for some $m\in\V$, 
	 we may reduce the problem to the case when
	 $$M(n_{1},\dots,n_{d})=n_{1}^{2}+\dots+n_{k}^{2}+cn_{k+1}^{2}+\dots+cn_{d'}^{2}+a_{d'+1}n_{d'+1}+\dots+a_{d}n_{d}-\lambda.$$
	 By another change of variable, 
	 we may reduce the problem to the case when
	 $$M(n_{1},\dots,n_{d})=n_{1}^{2}+\dots+n_{k}^{2}+cn_{k+1}^{2}+\dots+cn_{d'}^{2}+c'n_{d'+1}-\lambda.$$ 
	By the minimality of $c$, we have that $a^{2}+1\equiv c \mod p\Z$ for some $1\leq a\leq p-1$. If $k\leq d-2$, then
	$$cn_{k+1}^{2}+cn_{k+2}^{2}=(an_{k+1}+n_{k+2})^{2}+(n_{k+1}-an_{k+2})^{2}.$$
	So if $k+2\leq d'$, then under another change of variable, we may use another change of variable to replace $cn_{k+1}^{2}+cn_{k+2}^{2}$ by $n_{k+1}^{2}+n_{k+2}^{2}$ in the expression of $M$. So inductively, we are reduced to the case $k=d'-1$ or $d'$. So (\ref{1:ded}) holds by switching the position of $n_{1}$ and $n_{d'}$.

	The ``moreover" part follows immediately through the proof. 
	\end{proof}

\subsection{$M$-isotropic subspaces}

Let $M\colon\V\to\F_{p}$ be a quadratic form associated with the matrix $A$ and $V$ be a subspace of $\V$. Let $V^{\pp}$ denote the set of $\{n\in\V\colon (mA)\cdot n=0 \text{ for all } m\in V\}$.

\begin{prop}\label{1:23}
    Let $M\colon\V\to\F_{p}$ be a quadratic form associated with the matrix $A$ and  $V$ be a subspace of $\V$.
    \begin{enumerate}[(i)]
    	\item  $V^{\pp}$ is a subspace of $\V$, and $(V^{\pp})^{\pp}=V$.
    	\item $\rank(M)=d$ if and only if $\det(A)\neq 0$.
    	\item  Let $\phi\colon\V\to\V$ be a bijective linear transformation, $v\in\V$ and $c\in\F_{p}$. Let $M'\colon\V\to\F_{p}$ be the quadratic form given by $M'(n):=M(\phi(n)+v)+c$.
    	 We have that $\rank(M)=\rank(M')$ and $\phi(V^{\perp_{M'}})=(\phi(V))^{\pp}$.
    	\item $\rank(M)$ equals to $d-\dim((\V)^{\pp})$.	
    \end{enumerate}		
\end{prop}	
\begin{proof}
	Part (i) is easy to check by definition and the symmetry of $A$.
 Part (ii) follows from the fact  that $\rank(M)=d\Leftrightarrow \rank(A)=d\Leftrightarrow \det(A)\neq 0$. For Part (iii), assume that $\phi(x)=xB$ for some $d\times d$ invertible matrix $B$, then $\rank(M)=\rank(A)=\rank(BAB^{T})=\rank(M\circ \phi)$. So $\rank(M)= \rank(M\circ \phi)=\rank(M(\phi(n)+v))$ since shifting by $v$ does not affect the rank of 
a quadratic form. The part $\phi(V^{\perp_{M'}})=(\phi(V))^{\pp}$ is easy to check by definition.
	 
	We now prove Part (iv). 
	Let $r=\rank(M)$.
	By Lemma \ref{1:cov}, we may assume that $M(n)=M'(\phi(n)+v)$
	for some bijective linear transformation $\phi\colon\V\to\V$, $v\in\V$, $c,c',\lambda\in\F_{p}, c\neq 0$, and some quadratic form $M'\colon\V\to\F_{p}$ of the form
	 \begin{equation}\nonumber
	 M'(n_{1},\dots,n_{d})=cn^{2}_{1}+n_{2}^{2}+\dots+n^{2}_{r}+c'n_{r+1}-\l.
	 \end{equation}
	 It is clear that $\rank(M')=d-\dim((\V)^{\perp_{M'}})=r=\rank(M)$. By Part (iii), $(\V)^{\perp_{M}}=\phi^{-1}((\phi(\V))^{\perp_{M'}})=\phi^{-1}((\V)^{\perp_{M'}})$. So $\dim((\V)^{\perp_{M}})=\dim((\V)^{\perp_{M'}})$ and thus $\rank(M)=d-\dim((\V)^{\pp})$.
\end{proof}

 For $n\in\R^{d}$, the dot product induces the Euclidean norm $\vert n\vert^{2}:=n\cdot n$. However, in $\V$, the dot product (as well as general quadratic forms) does not induce a norm in $\V$, due to the existence of  isotropic subspaces.

\begin{defn}[$M$-isotropic subspaces]
	Let $M\colon \V\to\F_{p}$ be a quadratic form associated with the matrix $A$. 
		A subspace $V$ of $\V$ is \emph{$M$-isotropic} if $V\cap V^{\pp}\neq\{\bold{0}\}$. 
	We say that a tuple $(h_{1},\dots,h_{k})$ of vectors in $\V$ is \emph{$M$-isotropic} 
	if the span of $h_{1},\dots,h_{k}$ is an $M$-isotropic subspace.
	We say that a subspace or tuple of vectors is \emph{$M$-non-isotropic} if it is not $M$-isotropic.
\end{defn}	 

Let $h_{1},\dots,h_{k}$ be a basis of a subspace $V$ of $\V$ and $A$ be the matrix associated to $M$. It is not hard to see that $V$ is $M$-isotropic if and only if the determinant of the matrix $((h_{i}A)\cdot h_{j})_{1\leq i,j\leq k}$ is 0.

The next lemma ensures the existence of $M$-non-isotropic subspaces:
  
  \begin{lem}\label{1:wfejpo}
      Let $M\colon\V\to\F_{p}$ be a non-degenerate quadratic form. Then for all $0\leq k\leq \rank(M)$, there exists a $M$-non-isotropic subspace of $\V$ of dimension $k$.
  \end{lem}
  \begin{proof}
  Since  $\{\bold{0}\}$ is $M$-non-isotropic, it suffices to consider the case $k\geq 1$. Let $d'=\rank(M)$.
  By Lemma \ref{1:cov}, under a change of variable, it suffices to consider the case when the associated matrix $A$ of $M$ is a diagonal matrix whose diagonal can be written as  $(a_{1},\dots,a_{d'},0,\dots,0)$ for some $a_{1},\dots,a_{d'}\neq 0$. Let $e_{1},\dots,e_{d}$ denote the standard unit vectors. Clearly, for $1\leq k\leq d'$, the matrix $((e_{i}A)\cdot e_{j})_{1\leq i,j\leq k}$ is the upper left $k\times k$ block of $A$, which is of determinant $a_{1}\cdot\ldots\cdot a_{k}\neq 0$. This implies that $\sp_{\F_{p}}\{e_{1},\dots,e_{k}\}$ is a $M$-non-isotropic subspace of $\V$ of dimension $k$.
    \end{proof}

We summarize some basic properties of $M$-isotropic subspaces for later uses.

\begin{prop}\label{1:iissoo}
	Let $M\colon \V\to\F_{p}$ be a quadratic form and $V$ be a subspace of $\V$ of co-dimension $r$, and $c\in\V$.
	\begin{enumerate}[(i)]
		\item We have $\dim(V\cap V^{\pp})\leq \min\{d-\rank(M)+r,d-r\}$.
		\item The rank of $M\vert_{V+c}$ equals to $d-r-\dim(V\cap V^{\pp})$ (i.e. $\dim(V)-\dim(V\cap V^{\pp})$). 
		\item The rank of $M\vert_{V+c}$ is at most $d-r$ and at least $\rank(M)-2r$.
		\item $M\vert_{V+c}$ is non-degenerate (i.e. $\rank(M\vert_{V+c})=d-r$) if and only if $V$ is not an $M$-isotropic subspace.
	\end{enumerate}	
\end{prop}
 \begin{proof}
 	Part (i).
 	Let $v_{1},\dots,v_{d-r}$ be a basis of $V$. Let $B=\begin{bmatrix}
 	v_{1}\\
 	\dots\\
 	v_{d-r}
 	\end{bmatrix}$. The dimension of $\dim(V^{\pp})$ equals to $d$ minus the rank of the matrix $BA$. Since $\rank(BA)\geq \rank(B)+\rank(A)-d=\rank(M)-r$, we have that $\dim(V^{\pp})\leq d-\rank(M)+r$. 
 	So Part (i) holds because $\dim(V\cap V^{\pp})\leq \min\{\dim(V),\dim(V^{\pp})\}$.
 	
 	Part (ii). Let $\phi\colon\F_{p}^{d-r}\to V$ be a bijective linear transformation and denote $M'(m):=M(\phi(m)+c)$. Then the rank of $M\vert_{V+c}$ equals to the rank of $M'$. 
 	Assume that $\phi(m)=mB$ for some $(d-r)\times d$ matrix $B$ of rank $d-r$. Then the matrix associated to $M'$ is $BAB^{T}$.
 	Note that for all $m\in \F_{p}^{r}$,
 	\begin{equation}\nonumber
 	\begin{split}
 	& m\in (\F_{p}^{r})^{\perp_{M'}}
 	\Leftrightarrow ((mB)A)\cdot (nB)=0 \text{ for all } n\in\F_{p}^{r}
 	\\&\Leftrightarrow ((mB)A)\cdot v=0 \text{ for all } v\in V
 	\Leftrightarrow mB\in V\cap V^{\pp}.
 	\end{split}
 	\end{equation}
 	By Proposition \ref{1:23},  
 	$$\rank(M')=d-r-\dim(\F_{p}^{r})^{\perp_{M'}}=d-r-\dim(V\cap V^{\pp}).$$
 	This proofs Part (ii).
 	
 	Parts (iii) and (iv) follows directly from Parts (i) and (ii).
 \end{proof}

 An annoying issue when working with subspaces of $\V$ is that  the restriction of a non-degenerate quadratic form $M$ to an $M$-isotropic subspace $V$ of $\V$ is not necessarily non-degenerate. Fortunately, with the help of Proposition \ref{1:iissoo} (iii), we have that the nullity of $M$  does not increase too much when restricted to $V$, provided that the co-dimension of $V$ is small.
 
 The following is a consequence of Proposition \ref{1:iissoo}, which will be used in \cite{SunB}.

 \begin{lem}\label{1:cbn}
 	Let $M\colon \V\to\F_{p}$ be a non-degenerate quadratic form, $V$ be a subspace of $\V$ of dimension $r$, and $V'$ be a subspace of $\V$ of  dimension $r'$. Suppose that $\rank(M\vert_{V^{\pp}})=d-r$. Then $\rank(M\vert_{(V+V')^{\pp}})\geq d-r-2r'$.
 \end{lem}	
 \begin{proof}
 	Since $V^{\pp}$ is of co-dimension $r$,
 	by Proposition \ref{1:iissoo} (ii), $\rank(M\vert_{V^{\pp}})=d-r$ implies that $\dim(V\cap V^{\pp})=0$.
	Since $\dim(V+V')\leq r+r'$ and $M$ is non-degenerate, we have that $\dim((V+V')^{\pp})\geq d-r-r'$.
	 By Proposition \ref{1:iissoo} (ii),
 	$$\rank(M\vert_{(V+V')^{\pp}})\geq d-r-r'-\dim((V+V')\cap(V+V')^{\pp})\geq d-r-r'-\dim((V+V')\cap V^{\pp}).$$
	
So it remains to show that $\dim((V+V')\cap V^{\pp})\leq r'$. To see this, note that for any $v_{1},\dots,v_{r'+1}\in (V+V')\cap V^{\pp}$, there exist $c_{1},\dots,c_{r'+1}$	 not all equal to 0 such that $c_{1}v_{1}+\dots+c_{r'+1}v_{r'+1}\in V$. Since  $c_{1}v_{1}+\dots+c_{r'+1}v_{r'+1}\in V^{\pp}$ and $V\cap V^{\pp}=\{\bold{0}\}$, we have that $c_{1}v_{1}+\dots+c_{r'+1}v_{r'+1}=\bold{0}$. This means that $\dim((V+V')\cap V^{\pp})\leq r'$ are we are done.
  \end{proof}

\subsection{Counting the zeros of quadratic forms}

 	For a polynomial $P\in\poly(\F_{p}^{k}\to\F_{p})$, let $V(P)$ denote the set of $n\in\F_{p}^{k}$ such that $P(n)=0$. 
	For a family of polynomials $\mathcal{J}=\{P_{1},\dots,P_{k}\}\subseteq\poly(\F_{p}^{k}\to\F_{p})$, denote $V(\mathcal{J}):=\cap_{i=1}^{k}V(P_{i})$.
The purpose of this section is to provide estimates on the cardinality of $V(\mathcal{J})$  for families of polynomials 
arising from quadratic forms.

We start with some basic estimates. 
The proof of the following lemma is similar to Lemma 2.4 of \cite{GT14}. We omit the details.

\begin{lem}\label{1:ns}
	Let $P\in\poly(\V\to\F_{p})$ be  of degree at most $r$.
	Then $\vert V(P)\vert\leq O_{d,r}(p^{d-1})$ unless $P\equiv 0$.
\end{lem}

As an application of Lemma \ref{1:ns}, we have

\begin{lem}\label{1:iiddpp}
	Let $d,k\in\N_{+}$, $p$ be a prime, and $M\colon\V\to\F_{p}$ be a non-degenerate quadratic form.
	\begin{enumerate}[(i)]
		\item The number of tuples $(h_{1},\dots,h_{k})\in (\V)^{k}$ such that $h_{1},\dots,h_{k}$ are linearly  dependent is at most $kp^{(d+1)(k-1)}$.
		\item The number of $M$-isotropic tuples $(h_{1},\dots,h_{k})\in (\V)^{k}$ is at most $O_{d,k}(p^{kd-1})$.
	\end{enumerate}	
\end{lem}	
\begin{proof} 
	Part (i). If $h_{1},\dots,h_{k}$ are linearly dependent, then there exists $1\leq i\leq k$ such that $h_{i}=\sum_{1\leq j\leq k, j\neq i}c_{j}h_{j}$ for some $c_{j}\in\F_{p}$. For each $j$, there are $p$ choices of $c_{j}$ and $p^{d}$ choices of $h_{j}$. Since there are $k$ choices of $i$, in total there are at most $k\cdot (p\cdot p^{d})^{k-1}=kp^{(d+1)(k-1)}$
possibilities.

	Part (ii). Note that $h_{1},\dots,h_{k}$ are $M$-isotropic if and only if 
	there exist $c_{1},\dots,c_{k}\in\F_{p}$ not all equal to zero such that $(h_{i}A)\cdot(c_{1}h_{1}+\dots+c_{k}h_{k})=0$ for all $1\leq i\leq k$, which is equivalent of saying that 
	the determinant of the matrix $((h_{i}A)\cdot h_{j})_{1\leq i,j\leq k}$ is zero. By viewing $((h_{i}A)\cdot h_{j})_{1\leq i,j\leq k}$ as a polynomial in $(h_{1},\dots,h_{k})$ (which is certainly not constant zero), we get the conclusion from Lemma \ref{1:ns}.
\end{proof}	

A very useful application of Lemma \ref{1:iiddpp} is as follows. Whenever we have a collections $X$ of $k$-tuples $(h_{1},\dots,h_{k})\in(\V)^{k}$ of positive density $\d>0$, if the dimension $d$ is sufficiently large compared with $k$, and if $p$ is sufficiently large compared with $\d,d,k$, then by passing to a subset of $X$, we may further require all the tuples in $X$ to be linearly independent/$M$-non-isotropic with $X$ still having a positive density.

The following lemma was proved in \cite{IR07}. We provide its details for completeness.

\begin{lem}\label{1:counting}
	Let $d\in\N_{+}$ and $p$ be a prime number. Let  $M\colon\V\to\F_{p}$ be a  quadratic form of rank $r$. Suppose that $r\geq 3$.  
	Then $$\vert V(M)\vert=p^{d-1}(1+O(p^{-\frac{r-2}{2}})),$$ and for all $\xi=(\xi_{1},\dots,\xi_{d})\in\Z^{d}$,
	we have that 
	\begin{equation}\label{1:c0}
	\begin{split}
	\E_{n\in V(M)}\exp\Bigl(\frac{\xi}{p}\cdot n\Bigr)=\bold{1}_{\xi=\bold{0}}+O(p^{-\frac{r-2}{2}}).
	\end{split}
	\end{equation}
	 Here we slightly abuse the notation in $\frac{\xi}{p}\cdot n$ by identifying $n\in\V$ with its embedding $\tau(n)\in\Z^{d}$.
\end{lem}

\begin{proof}
	By Lemma \ref{1:cov}, it is easy to see we only need to prove this lemma for the case when 
	$$M(n_{1},\dots,n_{d})=a_{1}n^{2}_{1}+\dots+a_{r}n^{2}_{r}+a_{r+1}n_{r+1}+\dots+a_{d}n_{d}-\l$$
	for some $a_{1},\dots,a_{r}\in\F_{p}\backslash\{0\}, a_{r+1},\dots,a_{d}\in\F_{p}$ and $\l\in\F_{p}$.\footnote{Here we do not need to use the full strength of Lemma \ref{1:cov}.}
		Note that
	\begin{equation}\nonumber
	\begin{split}
	&\quad\sum_{n\in V(M)}\exp\Bigl(\frac{\xi}{p}\cdot n\Bigr)=\sum_{n\in \V}\E_{j\in\F_{p}}\exp\Bigl(\frac{\xi}{p}\cdot n+\frac{j}{p}M(n)\Bigr)
	\\&=p^{d-1}\bold{1}_{\xi=\bold{0}}+\frac{1}{p}\sum_{j\in\F_{p}\backslash\{0\}}\exp(-j\lambda)\sum_{n\in \V}\exp\Bigl(\frac{\xi}{p}\cdot n+\frac{j}{p}\sum_{i=1}^{r}a_{i}n^{2}_{i}+\frac{j}{p}\sum_{i=r+1}^{d}a_{i}n_{i}\Bigr).
	\end{split}
	\end{equation}
	By a change of coordinate, we have that 
	\begin{equation}\nonumber
	\begin{split}
	&\quad\Bigl\vert\frac{1}{p}\sum_{j\in\F_{p}\backslash\{0\}}\exp(-j\lambda)\sum_{n\in \V}\exp\Bigl(\frac{\xi}{p}\cdot n+\frac{j}{p}\sum_{i=1}^{r}a_{i}n^{2}_{i}+\frac{j}{p}\sum_{i=r+1}^{d}a_{i}n_{i}\Bigr)\Bigr\vert
	\\&\leq \frac{1}{p}\sum_{j\in\F_{p}\backslash\{0\}}\Bigl\vert\sum_{n\in \V}\exp\Bigl(\frac{\xi}{p}\cdot n+\frac{j}{p}\sum_{i=1}^{r}a_{i}n^{2}_{i}+\frac{j}{p}\sum_{i=r+1}^{d}a_{i}n_{i}\Bigr)\Bigr\vert
	\\&=\frac{1}{p}\sum_{j\in\F_{p}\backslash\{0\}}\Bigl\vert\sum_{n\in \F_{p}^{d}}\exp\Bigl(\frac{j}{p}\sum_{i=1}^{r}a_{i}n^{2}_{i}+\frac{j}{p}\sum_{i=r+1}^{d}a_{i}n_{i}\Bigr)\Bigr\vert
	\\&=\frac{1}{p}\sum_{j\in\F_{p}\backslash\{0\}}\prod_{i=1}^{r}\Bigl\vert\sum_{n\in \F_{p}}\exp(ja_{i}n^{2}/p)\Bigr\vert\cdot\prod_{i=r+1}^{d}\Bigl\vert\sum_{n\in \F_{p}}\exp(ja_{i}n/p)\Bigr\vert
	\\&\leq \frac{1}{p}\sum_{j\in\F_{p}\backslash\{0\}}p^{d-r}\prod_{i=1}^{r}\Bigl\vert\sum_{n\in \F_{p}}\exp(ja_{i}n^{2}/p)\Bigr\vert 
	\end{split}
	\end{equation}	
	By the Gauss sum equality, for all $j\in\F_{p}$,
	$$\sum_{n\in\F_{p}}\exp\Bigl(\frac{j}{p}\cdot n^{2}\Bigr)=Q\sqrt{p}\eta(j),$$
	where $Q=\pm 1$ or $\pm i$ is a constant depending on $p$, and $\eta$ is the Legendre symbol of $\F_{p}\backslash\{0\}$.
	So
	\begin{equation}\label{1:c1}
	\begin{split}
	\sum_{n\in V(M)}\exp\Bigl(\frac{\xi}{p}\cdot n\Bigr)= p^{d-1}\bold{1}_{\xi=\bold{0}}+O(p^{d-\frac{r}{2}}).
	\end{split}
	\end{equation}
	Setting $\xi=\bold{0}$, we have that $\vert V(M)\vert=p^{d-1}(1+O(p^{-\frac{r-2}{2}}))$. Dividing both sides of (\ref{1:c1}) by $\vert V(M)\vert$, we get (\ref{1:c0}).
\end{proof}

	We remark that Lemma \ref{1:counting} fails if $r\leq 2$. For example, if $M(n_{1},n_{2})=n_{1}^{2}+an_{2}^{2}$ for some $a\in\F_{p}$ such that $x^{2}\neq -a$ for all  $x\in\F_{p}$, then $V(M)$ is an empty set.

The following is an application of Lemma \ref{1:counting}.
We say that $t$ is a \emph{quadratic root} of $\F_{p}$ if there exists $n\in\F_{p}$ such that $n^{2}=t$. It is clear that if $p>2$, then the number of quadratic root of  $\F_{p}$ is $\frac{p+1}{2}$. We have

\begin{lem}\label{1:sqrcount}
	Let $d,r\in\N_{+}$, $p$ be a prime number and $M\colon\V\to \F_{p}$ be a   quadratic form of rank $r$. Then the number of $n\in\V$ such that $M(n)$ is a quadratic root is $\frac{1}{2}p^{d}(1+O(p^{-\frac{r-2}{2}}))$ if $r\geq 3$, and $\frac{1}{2}p^{d}(1+O(p^{-\frac{1}{2}}))$ for $r=2$.
\end{lem}	 
\begin{proof}
	
	Let $S=\{s_{1},\dots,s_{(p+1)/2}\}$ be the set of quadratic roots of $\F_{p}$. 
	We first consider the case $r\geq 3$.
	By Lemma \ref{1:counting}, the number of $n\in\V$ such that $M(n)\in S$ is
	$$\sum_{i=1}^{(p+1)/2}\vert V(M-s_{i})\vert=(p+1/2)\cdot p^{d-1}(1+O(p^{-\frac{r-2}{2}}))=\frac{1}{2}p^{d}(1+O(p^{-\frac{1}{2}})).$$ 
	
	We now consider the case $r=2$. By Theorem 5A on page 52 of \cite{Sch76},  the number of $n\in\F_{p}^{2}$ such that $M(n)=s_{i}$ is $p(1+O(p^{-1/2}))$. So the total number of  $n\in\V$ such that $M(n)$ is a quadratic root is  $\frac{1}{2}p(p+1)(1+O(p^{-1/2}))$. We are done.
\end{proof}	

We provide some generalizations of Lemma \ref{1:counting}. 
As an immediate corollary of Lemma \ref{1:counting}, we have the following estimate on the intersection of $V(M)$ with an affine subspace of $\V$.

\begin{coro}\label{1:counting01}
	Let $d\in\N_{+},r\in\N$ and $p$ be a prime number. Let  $M\colon\V\to\F_{p}$ be a   quadratic form   and $V+c$ be an affine subspace of $\V$ of co-dimension $r$.  
	\begin{enumerate}[(i)]
	\item 	If $s:=\rank(M\vert_{V+c})\geq 3$, then
	$$\vert V(M)\cap (V+c)\vert=p^{d-r-1}(1+O(p^{-\frac{s-2}{2}})).$$	
	\item If $\rank(M)-2r\geq 3$, then $$\vert V(M)\cap (V+c)\vert=p^{d-r-1}(1+O(p^{-\frac{1}{2}})).$$	
	\end{enumerate}
\end{coro}
\begin{proof}
	We start with Part (i).
	Let $\phi\colon \F_{p}^{d-r}\to V$ be a bijective linear transformation and denote $M'(m):=M(\phi(m)+c)$ for all $m\in \F_{p}^{d-r}$.
	Then $\rank(M')\geq 3$.
	 Note that $n\in V(M)\cap (V+c)$ if and only if $n=\phi(m)+c$ for some $m\in\F_{p}^{d-r}$ such that $M'(m)=0$. So by Lemma \ref{1:counting}, 
	$$\vert V(M)\cap (V+c)\vert=\vert V(M')\vert=p^{d-r-1}(1+O(p^{-\frac{s-2}{2}})).$$

	If $\rank(M)-2r\geq 3$, then $\rank(M\vert_{V+c})\geq 3$ by Proposition \ref{1:iissoo} (iii). So Part (ii) follows from Part (i).
\end{proof}

The following sets are studied extensively in this paper. 

\begin{defn} \label{1:vh}
Let $r\in\N_{+}$, $h_{1},\dots,h_{r}\in \V$ and $M\colon\V\to\F_{p}$ be a quadratic form. Denote $$V(M)^{h_{1},\dots,h_{r}}:=\cap_{i=1}^{r}(V(M(\cdot+h_{i}))\cap V(M).$$
\end{defn}

Lemma \ref{1:counting} can also be used to study the cardinality of $V(M)^{h_{1},\dots,h_{r}}$.
	
\begin{coro}\label{1:counting02}
	Let $d,r\in\N_{+}$ and $p$ be a prime number. Let  $M\colon\V\to\F_{p}$ be a non-degenerate   quadratic form  and $h_{1},\dots,h_{r}$ be linearly independent vectors.  
		If $d-2r\geq 3$, then $$\vert V(M)^{h_{1},\dots,h_{r}}\vert=p^{d-r-1}(1+O(p^{-\frac{1}{2}})).$$	
\end{coro}
\begin{proof}
Since  $h_{1},\dots,h_{r}$ are linearly independent and $M$ is non-degenerate, it is not hard to see that $V(M)^{h_{1},\dots,h_{r}}$ is the intersection of $V(M)$ with an affine subspace $V+c$ of $\V$ of co-dimension $r$.  
The conclusion follows from Corollary \ref{1:counting01}.
\end{proof}

We also need to study the following sets in this paper:

\begin{defn}[Gowers set]\label{defgow}
Let $\Omega$ be a subset of $\V$ and $s\in\N$.  Let 
$\Gow_{s}(\Omega)$ denote the set of $(n,h_{1},\dots,h_{s})\in(\V)^{s+1}$ such that $n+\e_{1}h_{1}+\dots+\e_{s}h_{s}\in\Omega$ for all $(\e_{1},\dots,\e_{s})\in\{0,1\}^{s}$.
Here we allow $s$ to be 0, in which case $\Gow_{0}(\Omega)=\Omega$.
We say that $\Gow_{s}(\Omega)$ is the \emph{$s$-th Gowers set} of $\Omega$.
\end{defn}

It is worth notion that the Gowers set will also be used for the study of the spherical Gowers norms in \cite{SunC}.
We have the following description for Gowers sets.

\begin{lem}\label{1:changeh}
	Let $s\in\N$, $M\colon\V\to\F_{p}$ be a quadratic form associated with the matrix $A$, and $V+c$ be an affine subspace of $\V$. For $n,h_{1},\dots,h_{s}\in\V$, we have that $(n,h_{1},\dots,h_{s})\in \Gow_{s}(V(M)\cap (V+c))$ if and only if 
	\begin{itemize}
		\item $n\in V+c$, $h_{1},\dots,h_{s}\in V$;
		\item $n\in V(M)^{h_{1},\dots,h_{s}}$;
		\item $(h_{i}A)\cdot h_{j}=0$ for all $1\leq i,j\leq s, i\neq j$.
	\end{itemize}	
	
	In particular, let $\phi\colon\F_{p}^{r}\to V$ be any bijective linear transformation, $M'(m):=M(\phi(m)+c)$. We have that $(n,h_{1},\dots,h_{s})\in \Gow_{s}(V(M)\cap (V+c))$ if and only if 
	$(n,h_{1},\dots,h_{s})=(\phi(n')+c,\phi(h'_{1}),\dots,\phi(h'_{s}))$
	for some $(n',h'_{1},\dots,h'_{s})\in \Gow_{s}(V(M'))$.	
\end{lem}	
\begin{proof}
	If $(n,h_{1},\dots,h_{s})\in \Gow_{s}(V(M)\cap (V+c))$, then clearly $n\in V(M)^{h_{1},\dots,h_{s}}$. Since $n,n+h_{1},\dots,n+h_{s}\in V+c$, we have that $h_{1},\dots,h_{s}\in V$. Finally, the fact that $(h_{i}A)\cdot h_{j}=0$ for all $1\leq i,j\leq s, i\neq j$ follows from the definition of $\Gow_{s}(V(M)\cap (V+c))$ and Lemma \ref{1:or}.
	
	Conversely, if $n\in V+c$, $h_{1},\dots,h_{s}\in V$, then $n+\e_{1}h_{1}+\dots+\e_{s}h_{s}\in V+c$ for all $\e_{1},\dots,\e_{s}\in\{0,1\}$. Since $n\in V(M)^{h_{1},\dots,h_{s}}$, we have that $M(n+\e_{1}h_{1}+\dots+\e_{s}h_{s})=0$ whenever $\vert\e\vert:=\e_{1}+\dots+\e_{s}=0$ or 1. Suppose we have shown that $M(n+\e_{1}h_{1}+\dots+\e_{s}h_{s})=0$ whenever $\vert\e\vert:=\e_{1}+\dots+\e_{s}\leq i$ for some $1\leq i\leq s-1$, then it follows from Lemma \ref{1:or} and the fact that $(h_{i}A)\cdot h_{j}=0$ for all $1\leq i,j\leq s, i\neq j$ that $M(n+\e_{1}h_{1}+\dots+\e_{s}h_{s})=0$ whenever $\vert\e\vert:=\e_{1}+\dots+\e_{s}\leq i+1$. In conclusion, we have that $(n,h_{1},\dots,h_{s})\in \Gow_{s}(V(M)\cap (V+c))$.
	
	The proof of the ``in particular" part is left to the interested readers. 
 \end{proof}	

We summarize some basic properties for Gowers sets.

\begin{prop}[Some preliminary properties for Gowers sets]\label{1:ctsp}
	Let $d,s\in \N_{+}$ and $p$ be a prime. Let $M\colon\V\to\F_{p}$ be a quadratic form with $\rank(M)\geq s^{2}+s+3$. Then
	\begin{enumerate}[(i)]
		\item We have $\vert \Gow_{s}(V(M))\vert=p^{(s+1)d-(\frac{s(s+1)}{2}+1)}(1+O_{s}(p^{-1/2}))$;
		\item For all but at most $s(s+1)p^{d+s-\rank(M)}$ many $h_{s}\in\V$, the set of $(n,h_{1},\dots,h_{s-1})\in(\V)^{s}$ with $(n,h_{1},\dots,h_{s})\in\Gow_{s}(V(M))$ is at most  $p^{ds-(\frac{s(s+1)}{2}+1)}$. 
		\item 
		If $s=1$, then for any function $f\colon \Gow_{1}(V(M))\to\C$ with norm bounded by 1,   we have that
		\begin{equation}\nonumber
		\E_{(n,h)\in\Gow_{1}(V(M))}f(n,h)=\E_{h\in\V}\E_{n\in V(M)^{h}}f(n,h)+O(p^{-1/2}).
		\end{equation}
	\end{enumerate} 
	\end{prop}	
 
Proposition \ref{1:ctsp} can  either be proved by direct computation plus the counting properties developed in Section \ref{1:s:pp5}, or 
be derived by Example \ref{1:rreepp} and Theorem \ref{1:ct}, the Fubini's theorem for more general set. So we omit the proof of Proposition \ref{1:ctsp}.

 \section{The main equidistribution theorem}\label{1:s:45}
 \subsection{Statement of the main theorem and outline of the proof}
 
 With the help of the preparation in the previous sections, we are now ready to state the main result in this paper.



 The following is the main result of this paper, which is 
  a generalization of Theorem \ref{1:sLei00}:
    
  \begin{thm}[Equidistribution theorem on $V(M)$]\label{1:sLei}
  	Let $0<\d<1/2, d,k,m\in\N_{+},s,r\in\N$ with $d\geq r$, and $p$ be a prime. Let $M\colon\V\to\F_{p}$ be a quadratic form  and $V+c$ be an affine subspace of $\V$ of co-dimension $r$. Suppose that $\rank(M\vert_{V+c})\geq s+13$.  Let $G/\Gamma$ be an $s$-step $\N$-filtered nilmanifold of dimension $m$, equipped with an $\frac{1}{\d}$-rational Mal'cev basis $\mathcal{X}$, and that  $g\in \poly(\Z^{d}\to G_{\N})$ be a rational polynomial sequence. If $p\gg_{d,m} \d^{-O_{d,m}(1)}$ and $(g(n)\Gamma)_{n\in \iota^{-1}(V(M)\cap(V+c))}$ is not $\d$-equidistributed on $G/\Gamma$,
	 then  there exists a nontrivial type-I horizontal character $\eta$ with $0<\Vert\eta\Vert\ll_{d,m} \d^{-O_{d,m}(1)}$ (independent of $k,M$ and $p$) such that $\eta\circ g \mod \Z$  is a constant on $\iota^{-1}(V(M)\cap(V+c))$.\footnote{It follows from Corollary \ref{1:counting01} that $V(M)\cap(V+c)$ is non-empty.}
  \end{thm}

 In order to illustrate the idea for the proof of Theorem \ref{1:sLei}, we present an outline of the proof of Corollary \ref{1:sweyl} (which is a special case of Theorem \ref{1:sLei}). Assume for simplicity that $g$ is $p$-periodic. Suppose we have proven Corollary \ref{1:sweyl} for the case $s-1$. 
 Denote $\Omega=V(M)$.
 Assume on the contrary that 
 \begin{equation}\label{gnnmp0}
 \Bigl\vert\frac{1}{\vert\Omega\vert}\sum_{n\in\Omega}\exp(g\circ\tau(n)/p)\Bigr\vert>\d.
 \end{equation}
 Taking squares on both sides and under a change of variables, we have that
  $$\frac{1}{\vert\Omega\vert^{2}}\sum_{n,n+h\in\Omega}\exp((g\circ\tau(n+h)-g\circ\tau(n))/p)>\d^{2}.$$
By the Funibi's Theorem (Proposition \ref{1:ctsp} or Theorem \ref{1:ct}), we have that 
  $$\frac{1}{p^{d}}\sum_{h\in\V}\Bigl\vert\frac{1}{\vert\Omega\cap (\Omega-h)\vert}\sum_{n\in \Omega\cap (\Omega-h)}\exp((g\circ\tau(n+h)-g\circ\tau(n))/p)\Bigr\vert>\d^{2},$$
  where the cardinality of the sets $\Omega\cap (\Omega-h), h\in\V$ are comparable with each other. By the Pigeonhole Principle, for ``many" $h\in\V$, we have that 
 $$\Bigl\vert\frac{1}{\vert\Omega\cap (\Omega-h)\vert}\sum_{n\in \Omega\cap (\Omega-h)}\exp((g\circ\tau(n+h)-g\circ\tau(n))/p)\Bigr\vert\gg_{d} 1.$$
 Informally, one can view $\Omega\cap (\Omega-h)$  as a sphere whose dimension is lower than $\Omega$ by 1. So we may then apply the induction hypothesis to the degree $(s-1)$ polynomial sequence $g\circ\tau(n+h)-g\circ\tau(n)$ to conclude that 
 \begin{equation}\label{gnnmp}
 \text{  $(g\circ\tau(n+h)-g\circ\tau(n))/p \mod \Z$ is a constant on $\Omega\cap (\Omega-h)$ for ``many" $h\in \V$. }
 \end{equation}

  For convenience denote $Q:=\iota\circ pg\circ\tau$, then $Q\in\poly(\V\to\F_{p})$ is a polynomial of degree at most $s$. It then follows from (\ref{gnnmp}) that
   \begin{equation}\label{gnnmp2}
 \text{  $\Delta_{h_{2}}\Delta_{h_{1}}Q(n)=0$   for ``many" $(n,h_{1},h_{2})\in \Gow_{2}(\Omega)$,}
 \end{equation}
    where $\Gow_{2}(\Omega)$ is the second Gowers set of $\Omega$ defined in
    Definition \ref{defgow}.
   
   We have now reduced Corollary \ref{1:sweyl} to a problem of solving the polynomial equation (\ref{gnnmp2}).
    If $\Omega$ were the entire space $\V$, then using some basic knowledge for polynomials, it is not hard to show that (\ref{gnnmp2}) would imply that $Q$ is of degree at most $s-1$, which would complete the proof of Corollary \ref{1:sweyl}  by induction.
    When $\Omega$ is a sphere, the difficulty of the analysis of (\ref{gnnmp2}) increases significantly. Therefore, solving equation (\ref{gnnmp2}) (as well as its various generalizations) is a central topic in this paper.

 \subsection{An overview for Sections \ref{1:s4}, \ref{1:s5} and \ref{1:s:pp9}}\label{1:s:ooer}
 
 It turns out that equation (\ref{gnnmp2}) is an over simplification for our purpose. In order to prove Theorem \ref{1:sLei}, following the idea of Green and Tao \cite{GT12b,GT14}, one can reduce the problem to solving the following generalized versions of  (\ref{gnnmp2}) (see also (\ref{1:force001})):
  \begin{equation}\label{gnnmp3}
\Delta_{h_{3}}\Delta_{h_{2}}\Delta_{h_{1}}P(n)+\Delta_{h_{2}}\Delta_{h_{1}}Q(n)=0, \text{ for ``many" } (n,h_{1},h_{2},h_{3})\in\Gow_{3}(\Omega),
\end{equation}
 where $Q\in\poly(\V\to\F_{p})$ comes from the composition of $g$ with some type-I horizontal character of $G/\Gamma$ (so $\deg(Q)\leq s$), and $P\in\poly(\V\to\F_{p})$ is a polynomial of degree at most $s-1$ arising from the ``non-linear" part of of $g$.
 
 The bulk of Section \ref{1:s4} is devoted to solving the following generalized version of  (\ref{gnnmp3}) (see also equation (\ref{1:t43t})):
  \begin{equation}\label{gnnmp3d}
\Delta_{h_{s-1}}\dots\Delta_{h_{1}}P(n)+\Delta_{h_{s}}\dots\Delta_{h_{1}}Q(n)=0, \text{ for ``many" } (n,h_{1},\dots,h_{s})\in\Gow_{s}(\Omega).\footnote{The study of such a generalization of  (\ref{gnnmp3}) is necessary for the work in Section \ref{1:s:idf} on the intrinsic definitions for polynomials.}
\end{equation}
 We provide a solution to (\ref{gnnmp3d}) in Proposition \ref{1:packforce0}.
 In the same section, we also prove some other properties for polynomial equations in $\poly(\V\to\F_{p})$ for later uses.

 One can then combine the solution to equation (\ref{gnnmp3}) together with the machinery of Green and Tao \cite{GT12b,GT14}  to prove Theorem \ref{1:sLei} for the case when $g$ is $p$-periodic. 
 However, in order to deal with a partially $p$-periodic sequence $g$, we need to deal with the following even more general version of  equation (\ref{gnnmp3d}):
  \begin{equation}\label{gnnmp4}
\Delta_{h_{s-1}}\dots\Delta_{h_{1}}P(n)+\Delta_{h_{s}}\dots\Delta_{h_{1}}Q(n)\in\Z \text{ for ``many" } (n,h_{1},\dots,h_{s})\in \iota^{-1}(\Gow_{s}(\Omega)),
\end{equation}
  where $P,Q\in\poly(\Z^{d}\to\Q)$ are partially $p$-periodic polynomials on $\iota^{-1}(\Omega)$. The bulk of
 Section \ref{1:s5} is devoted to solving equation (\ref{gnnmp4}) (and we eventually solve equation (\ref{gnnmp4}) in Proposition \ref{1:att30}). 
 In order to solve (\ref{gnnmp4}),
we develop an approach called the \emph{$p$-expansion trick}, which allows us  to lift the solutions to (\ref{gnnmp4}) from $\Z/p$-valued polynomials to $\Z/p^{s}$-valued polynomials. In the same section, we also use the same trick to extend other results in Section \ref{1:s4} to the partially $p$-periodic case.

 Finally, in Section \ref{1:s:pp9}, we combine the algebraic preparations from Sections \ref{1:s4} and \ref{1:s5} with the approach of Green and Tao \cite{GT12b,GT14}  to complete the proof of Theorem \ref{1:sLei}.

 \begin{rem}\label{1:r:wwer2}
     If we are satisfied with equidistribution results for $p$-periodic polynomials, then, as is mentioned above, the entire Section \ref{1:s5} will be unnecessary, and the proofs in Appendix \ref{1:s:AppD} can also be simplified significantly.
     However, such results are less useful in \cite{SunC,SunD} since we do not have the factorization theorem in the $p$-periodic setting. 
     This is the reason why we must prove the equidistribution results for  partially rational polynomials.
  \end{rem}

\section{Algebraic properties for quadratic forms in $\F_{p}$}\label{1:s4}		
		
 \subsection{Irreducible properties for quadratic forms}

   We begin with a dichotomy on the intersection of the sets of solutions of a polynomial and a quadratic form.
   
  \begin{lem}\label{1:bzt}
  	Let $d\in\N_{+},s\in\N$  and $p$ be a prime such that $p\gg_{d,s} 1$. Let $M\colon \V\to\F_{p}$ be a  quadratic form of rank at least 3.
  	For any $P\in\poly(\V\to\F_{p})$  of degree at most $s$,  either $\vert V(P)\cap V(M)\vert\leq O_{d,s}(p^{d-2})$ or $V(M)\subseteq V(P)$. 
  \end{lem}		
  \begin{proof}
  	Denote $r=\rank(M)$.
  		We first assume that   $M$ and $P$ are homogeneous.
  	By Lemma \ref{1:cov}, it suffices to consider the special case when $$M(n)=c_{1}n^{2}_{1}+\dots+c_{r}n_{r}^{2}$$
	for some $c_{1},\dots,c_{r}\in\F_{p}\backslash\{0\}$.
  	Let  $U:=V(P)\cap V(M)$. 
  	By the Pigeonhole Principle, there exists $1\leq i\leq r$ such that the $i$-th coordinate of at least $(\vert U\vert-p^{d-r})/r$ many $n\in U$ is nonzero. We may assume without loss of generality that $i=1$. 
  	Write $n=(t,y)$, where $y\in\F_{p}^{d-1}$ and $t\in\F_{p}$.  Since $\deg(f)=s\ll p$, by the long division algorithm, we may write
  	\begin{equation}\label{1:411}
  	P(t,y)=M(t,y)R(t,y)+N_{1}(y)t+N_{0}(y)
  	\end{equation}
  	for some $R\in\poly(\V\to\F_{p}), N_{1},N_{0}\in\poly(\F_{p}^{d-1}\to\F_{p})$ such that the degrees of $R, N_{1}$ and $N_{0}$ are at most $s-2$, $s-1$ and $s$ respectively. Denote 
	$$\Delta(y):=M(0,y)N_{1}(y)^{2}+c_{1}N^{2}_{0}(y).$$
  	
  	We first assume that $\Delta(y)$ is not constant 0.
  	Let $U'$ be the set of $(t,y)\in U$ such that $N_{1}(y)=0$. Then (\ref{1:411}) implies that $N_{0}(y)=0$ and so $\Delta(y)=0$. Since we assumed that $\Delta(y)$ is not constant 0 and $\deg(\Delta(y))\leq 2s$, by Lemma \ref{1:ns}, we have that $\vert U'\vert\leq O_{d,s}(p^{d-2})$.

  	Suppose that $(t,y)\in U\backslash U'$ and $t\neq 0$. Then $M(t,y)=N_{1}(y)t+N_{0}(y)=0$. So $t=-N_{0}(y)N_{1}(y)^{-1}$, and so  $M(t,y)=0$ implies that $M(0,y)=-c_{1}N^{2}_{0}(y)N_{1}(y)^{-2}$, or equivalently $\Delta(y)=0$. Since $\Delta(y)$ is not constant 0, by Lemma \ref{1:ns}, we have that $\vert (U\backslash U')\cap \{(t,y)\colon t\neq\bold{0}\}\vert\leq O_{d,s}(p^{d-2})$. 
	Since we assumed that $\vert U\cap \{(t,y)\colon t\neq\bold{0}\}\vert\geq \frac{\vert U\vert-p^{d-r}}{r}$,
	we deduce that 
  	$\vert U\vert\leq O_{d,s}(p^{d-2})$ since $r\geq 3$.
  	
  	We now consider the case that $\Delta(y)\equiv 0$.   	Let $W$ be the set of $y\in\F_{p}^{d-1}$ such that $-c_{1}M(0,y)$ is not a square root in $\F_{p}$. Since $\rank(M(0,\cdot))\geq 2$, by Lemma \ref{1:sqrcount}, $\vert W\vert=\frac{1}{2}p^{d-1}(1+O(p^{-1/2}))$. 	Since $\Delta(y)\equiv 0$, for all $y\in W$, we have that $N_{1}(y)=0$. By Lemma \ref{1:ns}, $\vert W\vert\leq O_{s}(p^{d-2})$ unless $N_{1}(y)\equiv 0$. So if $p\gg_{d,s} 1$, then $N_{1}(y)\equiv 0$. Since $\Delta(y)\equiv 0$, we also have that $N_{2}(y)\equiv 0$. So (\ref{1:411}) implies that $P=MR$, which implies that $V(M)\subseteq V(P)$.

   Finally, we consider the general case when $M$ and $P$ are not necessarily homogeneous.
   Then there exist homogeneous polynomials $\tilde{M},\tilde{P}\colon \F_{p}^{d+1}\to\F_{p}$ of degrees at most 2 and $s$ respectively such that $M=\tilde{M}(\cdot,1)$ and $P=\tilde{P}(\cdot,1)$. Then $\tilde{M}$ is a quadratic form of rank at least 3. By the previous case, either $\vert V(\tilde{P})\cap V(\tilde{M})\vert\leq O_{d,s}(p^{d-1})$ or $V(\tilde{M})\subseteq V(\tilde{P})$. In the former case, since
   \begin{equation}\label{1:flf}
   \begin{split}
   &\quad\vert V(\tilde{P})\cap V(\tilde{M})\vert\geq \sum_{i=1}^{p-1}\vert V(\tilde{P}(\cdot,i))\cap V(\tilde{M}(\cdot,i))\vert
   \\&=(p-1)\vert V(\tilde{P}(\cdot,1))\cap V(\tilde{M}(\cdot,1))\vert=(p-1)\vert V(P)\cap V(M)\vert,
   \end{split}
   \end{equation}
   we have that $\vert V(P)\cap V(M)\vert\leq O_{d,s}(p^{d-2})$. In the later case, we have that $V(\tilde{M}(\cdot,1))\subseteq V(\tilde{P}(\cdot,1))$, which implies that $V(M)\subseteq V(P)$.
  \end{proof}

Let $P\colon\V\to \F_{p}$ be a polynomial with $\deg(P)<p$ and $V(M)\subseteq V(P)$. A natural question to ask is whether $P$ is divisible by $M$, namely $P=MQ$ for some polynomial $Q$ with $\deg(Q)=\deg(P)-2$. To answer this question, we start with some definitions.
  Let $P\colon\V\to \F_{p}$ be a polynomial of degree $s$ for some $s<p$ and $M\colon\V\to\F_{p}$ be a quadratic form associated with the symmetric matrix $A=(a_{i,j})_{1\leq i,j\leq d}$. If $a_{1,1}\neq 0$, then viewing $a_{1,1}n_{1}^{2}$ as the leading coefficient of $M$, we may use the long division algorithm to write
  $$P(n)=M(n)Q(n)+n_{1}R_{1}(n')+R_{0}(n')$$
  for some polynomials $Q$, $R_{1}$ and $R_{0}$ of degrees at most $s-2$, $s-1$ and $s$ respectively, where we write $n:=(n_{1},n')$ for some $n_{1}\in\F_{p}$ and $n'\in\F_{p}^{d-1}$. We call this algorithm the \emph{standard} long division algorithm. 
  It is clear that all the coefficients of $Q,R_{1}$ and $R_{0}$ can be written in the form $F/a_{1,1}^{s}$ with $F$ being a polynomial function with respect to the coefficients of $P$ and $M$ of degree at most $s$.

  We now consider the case when $a_{1,1}$ is not necessarily non-zero. In this case, we can make $a_{1,1}$ to be non-zero using a change of variable. To be more precise, 
for all $1\leq i,j\leq d$, 
   	let 
   	$$\phi_{i,i}(n_{1},\dots,n_{d}):=(n_{i},n_{1},\dots,n_{i-1},n_{i+1},\dots,n_{d})$$
   	and 
   	$$\phi_{i,j}(n_{1},\dots,n_{d}):=((n_{i}+n_{j})/2,(n_{i}-n_{j})/2,n_{1},\dots,n_{i-1},n_{i+1},\dots,n_{j-1},n_{j+1},\dots,n_{d})$$
   	if $i<j$.
   	Let $B_{i,j}$ be the symmetric $d\times d$ matrix induced by the bijective linear transformation $\phi_{i,j}$, namely $\phi_{i,j}(n)=nB_{i,j}$. 
	 
   \begin{lem}\label{1:2d2d}
  For any symmetric nonzero $d\times d$ matrix $A$ in $\F_{p}$, there exist $1\leq i,j\leq d$ such that the upper left entry of the matrix $B_{i,j}AB_{i,j}^{T}$ is nonzero.
   \end{lem}
   \begin{proof}  	
   	Assume that $A=(a_{i,j})_{1\leq i,j\leq d}$. If $a_{i,i}\neq 0$ for some $1\leq i\leq d$, then the upper left entry of the matrix $B_{i,i}AB_{i,i}^{T}$ is $a_{i,i}\neq 0$. If $a_{1,1}=\dots=a_{d,d}=0$ and $a_{i,j}\neq 0$ for some $1\leq i<j\leq d$, then it is not hard to compute that the upper left entry of the matrix $B_{i,j}AB_{i,j}^{T}$ is $a_{i,j}\neq 0$.   
	\end{proof}

  Let $B_{i,j}$ be defined as above and suppose that upper left entry of the matrix $B_{i,j}AB_{i,j}^{T}$ is nonzero, where $A$ is the matrix associated to $M$. 
  Then we may use the standard  long division algorithm to write
  $$P(nB_{i,j})=M(nB_{i,j})Q(n)+n_{1}R_{1}(n')+R_{0}(n').$$
  So
  \begin{equation}\label{1:BM}
  P(n)=M(n)Q(nB_{i,j}^{-1})+(nB_{i,j}^{-1})_{1}R_{1}((nB_{i,j}^{-1})')+R_{0}((nB_{i,j}^{-1})'),
  \end{equation}
  where we write $nB_{i,j}^{-1}:=((nB_{i,j}^{-1})_{1},(nB_{i,j}^{-1})')$. We call this decomposition (\ref{1:BM}) of $P$ the \emph{$B_{i,j}$-standard} long division algorithm. 
  
  It is clear that all the coefficients of $Q,R_{1}$ and $R_{0}$ can be written in the form $F/c_{1,1}^{s}$ with $F$ being a polynomial function (dependent on $B_{i,j}$) with respect to the coefficients of $P$ and $M$ of degree at most $s$, where $c_{1,1}$ is the coefficient of the $n_{1}^{2}$ term of $M(nB_{i,j})$.
  
  We say that $M$ \emph{divides $P$ with respect to the $B_{i,j}$-standard long division algorithm}    if  $R_{1}=R_{0}=0$ in the $B_{i,j}$-standard long division algorithm (\ref{1:BM}).  We have the following Hilbert Nullstellensatz type result for $V(M)$:
  
  \begin{prop}[Hilbert Nullstellensatz for $V(M)$]\label{1:noloop}
  	Let  $d\in\N_{+},s\in\N$, $p\gg_{d,s} 1$ be a prime number, 
  	$P\in \poly(\V\to\F_{p})$ be of degree at most $s$, and $M\colon\V\to\F_{p}$ be a quadratic form of rank at least 3 associated with the matrix $A$.  If $V(M)\subseteq V(P)$, then for all $1\leq i,j\leq d$ such that the upper left entry of the matrix $B_{i,j}AB_{i,j}^{T}$ is nonzero,
	$M$ divides $P$ with respect the $B_{i,j}$-standard long division algorithm.
	In particular, $P=MR$ for some $R\in \poly(\V\to\F_{p})$ of degree at most $s-2$.
  \end{prop}

  It is tempting to prove Proposition \ref{1:noloop} by using Hilbert Nullstellensatz.
  Let $F_{i}(n_{1},\dots,n_{d})$ $:=n_{i}^{p}-n_{i}$ for all $1\leq i\leq d$.
  By using Hilbert Nullstellensatz in the finite field (see \cite{Gho19}),  if $V(M)\subseteq V(P)$, then one can show that 
  $$P=MR+\sum_{i=1}^{d}F_{i}R_{i}$$
  for some polynomials $R,R_{1},\dots,R_{d}$. Unfortunately, this is not helpful for Proposition \ref{1:noloop}, since  there is no restriction for the degree of $R$, and it is unclear how to remove the polynomials $R_{1},\dots,R_{d}$. So we have to prove Proposition \ref{1:noloop} with a different method relying on the long division algorithm.

  \begin{proof}[Proof of Proposition \ref{1:noloop}]  
   Denote $r=\rank(M)$.
  Under a change of variable,  it suffices to show that for any $P\in \poly(\V\to\F_{p})$  of degree at most $s$ and any quadratic form $M\colon\V\to\F_{p}$  of rank at least 3 associated with the matrix $A$ whose upper left entry  is nonzero.  If $V(M)\subseteq V(P)$, 
	$M$ divides $P$ with respect the standard long division algorithm.

   		We first assume that   $M$ and $P$ are homogeneous.  
  	Write $n=(t,y)$, where $t\in\F_{p}$ and $y\in\F_{p}^{d-1}$.
  	Since the coefficient of the $t^{2}$ term of $M(t,y)$ is nonzero,
  	  we may apply the standard long division algorithm to write
  	\begin{equation}\label{1:412}
  	P(t,y)=M(t,y)R(t,y)+N_{1}(y)t+N_{0}(y)
  	\end{equation}
  	for some $R, N_{1},N_{0}\in \poly(\V\to\F_{p})$   of degrees at most $s-2,s-1$ and $s$ respectively. Our goal is to show that $N_{1},N_{0}\equiv 0$.
  By Lemma \ref{1:cov} and a change of variable, 
  it suffices to show that if (\ref{1:412}) holds for $$M(n)=c_{1}n^{2}_{1}+\dots+c_{r}n_{r}^{2}$$
	for some $c_{1},\dots,c_{r}\in\F_{p}\backslash\{0\}$, and that $V(M)\subseteq V(P)$, then $N_{0}=N_{1}\equiv 0$.

  	Let
  	$$\Delta(y):=M(0,y)N_{1}(y)^{2}+c_{1}N^{2}_{0}(y).$$
  	If $\Delta(y)$ is not constant zero, then by the proof of Lemma \ref{1:bzt}, the number of $n\in\V$ such that $M(n)=P(n)=0$ is at most $O_{d,s}(p^{d-2})$. However, since $V(M)\subseteq V(P)$, by Lemma \ref{1:counting}, the number of $n\in\V$ such that $M(n)=P(n)=0$ is $p^{d-1}(1+O(p^{-\frac{r-2}{2}}))$, a contradiction.
  	
  	So we must have that  $\Delta(y)\equiv 0$. Following the proof of Lemma \ref{1:bzt}, we have that $N_{1}(y)=N_{2}(y)\equiv 0$, and so $f=MR$. Since the degree of $R$ is at most $s-2$, we are done.

	 We now consider the general case when $M$ and $P$ are not necessarily homogeneous.
   Then there exist homogeneous polynomials $\tilde{M},\tilde{P}\colon \F_{p}^{d+1}\to\F_{p}$ of degrees at most 2 and $s$ respectively such that $M=\tilde{M}(\cdot,1)$ and $P=\tilde{P}(\cdot,1)$. Since $\vert V(P)\cap V(M)\vert=\vert V(M)\vert=p^{d-1}(1+O(p^{-1/2}))$ by Lemma \ref{1:counting},
  it follows from (\ref{1:flf}) that $\vert V(\tilde{P})\cap V(\tilde{M})\vert=p^{d}(1+O(p^{-1/2}))$. By Lemma \ref{1:bzt}, we have that $V(\tilde{M})\subseteq V(\tilde{P})$. So by the previous case, we have that $\tilde{M}$ divides $\tilde{P}$ with respect the standard long division algorithm. Since the $n_{1}^{2}$ term of $M$ is non-zero, it is not hard to see that $M$ also divides $P$ with respect the standard long division algorithm.
  \end{proof}

 As an immediate corollary of Lemma \ref{1:bzt} and Proposition \ref{1:noloop}, we have (see Corollary \ref{1:att4}  for a similar result):

 \begin{coro}\label{1:noloop3}
 	Let  $d,k\in\N_{+}$, $s\in\N$, $p\gg_{d,k,s} 1$ be a prime number, 
 	$P\in \poly_{p}(\V\to\F_{p})$ be a polynomial of degree at most $s$, $M\colon\V\to\F_{p}$ be a non-degenerate quadratic form, $c\in\V$, and $V$ be a subspace of $\V$ of dimension $k$ with a basis $h_{1},\dots,h_{k}$. Suppose that $\rank(M\vert_{V^{\pp}})\geq 3$. 
 	 Then either $\vert V(M)\cap (V^{\pp}+c)\cap V(P)\vert\leq O_{d,k,s}(p^{d-k-2})$ or $V(M)\cap (V^{\pp}+c)\subseteq V(P)$.
 	
 	Moreover, if $V(M)\cap (V^{\pp}+c)\subseteq V(P)$, then 
 	\begin{equation}\label{1:pmm}
 	P(n)=M(n)P_{0}(n)+\sum_{i=1}^{k}((h_{i}A)\cdot (n-c))P_{i}(n)
 	\end{equation}
 	for some   polynomials $P_{0},\dots,P_{k}\in \poly_{p}(\V\to\F_{p})$ with $\deg(P_{0})\leq s-2$ and $\deg(P_{i})\leq s-1, 1\leq i\leq k$.
 	
 	In particular, the conclusion of this corollary holds if $d\geq k+\dim(V\cap V^{\pp})+3$, or if $d\geq 2k+3$.\footnote{Some properties in this paper are not used in the proof of the main results, but they will be used later in the series of work \cite{SunB,SunC,SunD}.  Corollary \ref{1:noloop3} is such an example.}
 \end{coro}
   \begin{proof}
    By a change of variables $n\mapsto n+m$ for some $m\in\V$, it suffices to consider the special case when $M(n)=(nA)\cdot n-\lambda$ for some invertible symmetric matrix $A$ and some $\lambda\in\F_{p}$.
  Since 
   $A$ is invertible, $V^{\pp}$ is a subspace of $\V$ of dimension $d-k$. 
   Let $\phi\colon \F_{p}^{d-k}\to V^{\pp}$ be any bijective linear transformation. Denote $M'(m):=M(\phi(m)+c)$ and $P'(m):=P(\phi(m)+c)$ for all $m\in \F_{p}^{d-k}$.  It is not hard to see  that $V(M)\cap (V^{\pp}+c)=\phi(V(M'))+c$. Moreover, $$V(M)\cap (V^{\pp}+c)\cap V(P)=\phi(V(M')\cap V(P'))+c.$$
   Since $\deg(M')=\rank(M\vert_{V^{\pp}})\geq 3$, by Lemma \ref{1:bzt}, either $\vert V(M')\cap V(P')\vert\leq O_{d,k,s}(p^{d-k-2})$ or $V(M')\subseteq V(P')$. In the former case we have that  $$\vert V(M)\cap (V^{\pp}+c)\cap V(P)\vert=\vert\phi(V(M')\cap V(P))\vert=\vert V(M')\cap V(P')\vert\leq O_{d,k,s}(p^{d-k-2}).$$
   In the later case, we have that 
   $$V(M)\cap (V^{\pp}+c)=\phi(V(M'))+c\subseteq \phi(V(P'))+c\subseteq V(P).$$

   \

   We now assume that $V(M)\cap (V^{\pp}+c)\subseteq V(P)$ and show that (\ref{1:pmm}) holds. By a change of variables $n\to nB$ for some invertible matrix $B$, we may assume without loss of generality that $h_{i}A=e_{d-k+i}$ for $1\leq i\leq k$. Then $V^{\pp}=\sp_{\F_{p}}\{e_{1},\dots,e_{d-k}\}$. 
   We may then assume without loss of generality that $c=(0,\dots,0,c_{d-k+1},\dots,c_{d})$.
   Then $x\in V(M)\cap (V^{\pp}+c)$ if and only if $x=(x_{1},\dots,x_{d-k},c_{d-k+1},\dots,c_{d})\in V(M)$ for some $x_{1},\dots,x_{d-k}\in\F_{p}$. 
   Since  $V(M)\cap (V^{\pp}+c)\subseteq V(P)$ and $\rank(M\vert_{V^{\pp}})\geq 3$, by Lemma \ref{1:noloop},  $p\gg_{d,k,s} 1$, then
   \begin{equation}\label{1:pqm}
   \begin{split}
  P(x_{1},\dots,x_{d-k},c_{d-k+1},\dots,c_{d})
    =M(x_{1},\dots,x_{d-k},c_{d-k+1},\dots,c_{d})Q(x_{1},\dots,x_{d-k})
   \end{split}
   \end{equation}
    for some polynomial  $Q\in \poly_{p}(\V\to\F_{p})$ with   $\deg(Q)\leq s-2$ for all $x_{1},\dots,x_{d-k}\in\F_{p}$.
    Note that for all $1\leq i\leq k$ and $x_{d-k+i}\in\F_{p}$, writing $x:=(x_{1},\dots,x_{d})\in\V$, we have that
    $$(h_{i}A)\cdot (x-c)=x_{d-k+i}-c_{d-k+i}.$$
    We may then use this fact to replace the terms $c_{d-k+1},\dots,c_{d}$ in (\ref{1:pqm}) by $x_{d-k+1},\dots,x_{d}$ and get to a expression of the form (\ref{1:pmm}).
    
    
    \
    
    Note that $V^{\pp}$ is of co-dimension $k$. So if $d\geq k+\dim(V\cap V^{\pp})+3$, or if $d\geq 2k+3$, then $\rank(M\vert_{V^{\pp}})\geq 3$ by Proposition \ref{1:iissoo} (ii) and (iii). 
   \end{proof}

\subsection{A special anti-derivative property} 

  Let $F,G\colon\V\to\F_{p}$ be polynomials with $F=MG$ for some quadratic form $M\colon\V\to\F_{p}$. Let $\partial_{i}$ denote the $i$-th directional formal derivative. Then it is not hard to see that $$\partial_{j}M\partial_{i}F-\partial_{i}M\partial_{j}F=M(\partial_{j}M\partial_{i}G-\partial_{i}M\partial_{j}G)$$
  is divisible  by $M$. In this section, we study the converse of this question: if $\partial_{j}M\partial_{i}F-\partial_{i}M\partial_{j}F$  is divisible by $M$ for all $1\leq i,j\leq d$, then is it true that $F$ is also divisible by $M$? The following is an answer to this question.

 	\begin{prop}\label{1:antideri}
 		Let $d\in\N_{+},s\in\N$ with $d\geq 3$, $p$ be a prime, and $M\colon\V\to\F_{p}$ be a quadratic form of rank at least 3.
 	    Let $Q\in\poly(\V\to\F_{p})$ be  a polynomial of degree at most $s$ with $Q(\bold{0})=0$ such that
 		$$\partial_{1}M(n)\partial_{i}Q(n)=\partial_{i}M(n)\partial_{1}Q(n)$$
 		for all $1\leq i\leq d$ and $n\in V(M)$. If $p\gg_{d,s} 1$, then
 	 $Q=MQ'$ for some $Q'\in\poly(\V\to\F_{p})$  of degree at most $s-2$. 
 	\end{prop}

 	\begin{proof}
For convenience write $n=(n_{1},\dots,n_{d})\in\V$.
 		By Lemma \ref{1:cov} and a change of variable, we may assume without loss of generality that 
 		$$M(n)=cn_{1}^{2}+n_{2}^{2}+\dots+n_{d'}^{2}+c'n_{d'+1}-\lambda$$
 		for some $3\leq d'\leq d$, and $\lambda,c,c'\in\F_{p}$ with $c\neq 0$.
 		Let $A$ be the (diagonal) matrix associated to $M$ with $a_{i}$ be the $(i,i)$-th entry of $A$. Then $a_{1}=c\neq 0$, $a_{2}=\dots=a_{d'}=1$,  and $a_{d'+1}=\dots=a_{d}=0$. 
 		By Proposition \ref{1:noloop},
 		there exist  $N_{i}\in \poly(\V\to\F_{p})$, $\deg(N_{i})\leq s-2$ for $1\leq i\leq d$ such that 
 		\begin{equation}\label{1:force302}
 		\partial_{1}M\partial_{i}Q-\partial_{i}M\partial_{1}Q=2MN_{i}.
 		\end{equation}		
 
 We briefly explain the idea of the first part of the proof.
 Observe that if $N_{i}\equiv 0$ for all $1\leq i\leq d$, then (\ref{1:force302}) indicates that $\partial_{i}Q/\partial_{i}M$ should be independent of $i$. So it is reasonable to expect $\partial_{i}Q=(\partial_{i}M)F$ for some polynomial $F$ for all $1\leq i\leq d$. Taking into consideration of the term  $MN_{i}$, one can expect that $\partial_{i}Q-(\partial_{i}M)F$ is divisible by $M$ for all $1\leq i\leq d$. To make this intuition precise, we have

 		\textbf{Claim 1.} There exist $F,R_{1},\dots,R_{d}\in \poly(\V\to\F_{p})$ with $\deg(F)\leq s-2, \deg(R_{i})\leq s-3$ such that 
 		\begin{equation}\label{1:force40}
 		\partial_{i}Q=(\partial_{i}M)F+MR_{i}
 		\end{equation}
 		for all $1\leq i\leq d$.
 		
 		For $2\leq i\leq d'$, we may write $N_{i}(n)$ as $$N_{i}(n)=-n_{i}N_{i,1}(n)+cn_{1}N_{i,2}(n)+n_{i}n_{1}N_{i,3}(n)$$
 		for some $N_{i,1}, N_{i,2}, N_{i,3}\in \poly(\V\to\F_{p})$ in a unique way, where $N_{i,1}(n)$ is independent of $n_{1}$ and $N_{i,2}(n)$ is independent of $n_{i}$. Moreover, $\deg(N_{i,1}),\deg(N_{i,2})\leq s-3$ and $\deg(N_{i,3})\leq s-4$. So  (\ref{1:force302}) implies that
 		\begin{equation}\label{1:force31}
 		cn_{1}(\partial_{i}Q(n)-M(n)N_{i,2}(n))-n_{i}(a_{i}\partial_{1}Q(n)-M(n)N_{i,1}(n))=n_{i}n_{1}M(n)N_{i,3}(n)
 		\end{equation}
		for all $1\leq i\leq d'$.

 		From (\ref{1:force31}) we see that there exists a polynomial $F_{i}$ of degree at most $s-2$ such that
 			\begin{equation}\label{1:force31a}
 		n_{1}F_{i}(n):=(2c)^{-1}(a_{i}\partial_{1}Q(n)-M(n)N_{i,1}(n)), 
 			\end{equation}
 			and that
 			\begin{equation}\label{1:force31b}
 			\partial_{i}Q(n)=2n_{i}F_{i}(n)+M(n)(N_{i,2}(n)+c^{-1}n_{i}N_{i,3}(n))
 			\end{equation}	
 			if we substitute (\ref{1:force31a}) back into (\ref{1:force31}).
 		We may write $\partial_{1}Q(n)=Q'(n)+n_{1}Q''(n)$ for some $Q',Q''\in \poly(\V\to\F_{p})$ with $\deg(Q')\leq s-1$ and $\deg(Q'')\leq s-2$  such that $Q'(n)$ is independent from $n_{1}$.  Since $N_{i,1}(n)$ is independent from $n_{1}$, comparing the terms in (\ref{1:force31a}) which are divisible by $n_{1}$, we have that 
 		\begin{equation}\label{1:force323}
 		a_{i}Q'(n)=(M(n)-cn^{2}_{1})N_{i,1}(n)
 		\end{equation}
 	and 
 		\begin{equation}\label{1:force33}
 		F_{i}(n)=(2c)^{-1}(a_{i}Q''(n)-cn_{1}N_{i,1}(n)).
 		\end{equation}

Since $a_{2}=\dots=a_{d'}=1$, it follows from (\ref{1:force323}) that 
 		all of $N_{2,1},\dots,N_{d',1}$ equal to a same $N_{1}$. So by (\ref{1:force33}), all of $F_{2},\dots,F_{d'}$ equals to a same $F$. So 
 		 (\ref{1:force31b}) implies that
 		 \begin{equation}\nonumber
 		 \partial_{i}Q(n)=2n_{i}F(n)+M(n)(N_{i,2}(n)+c^{-1}n_{i}N_{i,3}(n))=\partial_{i}M(n)F(n)+M(n)(N_{i,2}(n)+c^{-1}n_{i}N_{i,3}(n))
 		 \end{equation}
 		 for all $2\leq i\leq d'$. In other words, (\ref{1:force40}) holds for all $2\leq i\leq d'$.

 		 On the other hand, (\ref{1:force31a}) implies that
		  \begin{equation}\label{1:force0044}
 		 \partial_{1}Q(n)=2cn_{1}F(n)+M(n)N_{2,1}(n)=\partial_{1}M(n)F(n)+M(n)N_{2,1}(n).
 		 \end{equation}  
So (\ref{1:force40}) holds for   $i=1$.

	 	Now for $d'+2\leq i\leq d$, since $a_{i}=0$,
		 (\ref{1:force302}) implies that $cn_{1}\partial_{i}Q(n)=M(n)N_{i}(n)$. 
		 It is then not hard to see that $N_{i}(n)$ can be written as $cn_{1}R_{i}(n)$ for some polynomial $R_{i}$ of degree at most $s-3$. So $$\partial_{i}Q(n)=M(n)R_{i}(n)=\partial_{i}M(n)F(n)+M(n)R_{i}(n).$$		
So (\ref{1:force40}) holds for all  $d'+2\leq i\leq d$. 
	
		Finally, 
		 by (\ref{1:force302}) and (\ref{1:force0044}), we have that 
		 $$2cn_{1}(\partial_{d'+1}Q(n)-c'F(n))=M(n)(2N_{d'+1}(n)+c'N_{2,1}(n)).$$
		 It is then not hard to see that $2N_{d'+1}(n)+c'N_{2,1}(n)$ can be written as $2cn_{1}R_{d'+1}(n)$ for some polynomial $R_{d'+1}$ of degree at most $s-3$. So $$\partial_{d'+1}Q(n)=c'F(n)+M(n)R_{d'+1}(n)=\partial_{d'+1}M(n)F(n)+M(n)R_{d'+1}(n).$$	
So (\ref{1:force40}) holds for  $i=d'+1$.
 		This finishes the proof of Claim 1.

 		\

		 We now explain the idea of the second part of the proof. 
Our strategy is to show that one can inductively expand $Q$ as
 $Q=MW_{1}+M^{2}W_{2}+\dots$ for some polynomials $W_{1},W_{2},\dots$ in the sense that $Q-(MW_{1}+\dots+M^{\ell}W_{\ell})$ is a polynomial of degree at most $\deg(Q)$ which is divisible by $M^{\ell+1}$ for all $\ell\in\N_{+}$. Since this procedure must terminate after finitely many steps, we conclude that $Q=MW_{1}+\dots+M^{\ell}W_{\ell}$ for some $\ell\in\N_{+}$, and so $Q$ is divisible by $M$.

 		To be more precise, we have
 		
 		\textbf{Claim 2.} For all $1\leq \ell<p$, there exist  $H_{1},\dots,H_{d},W_{\ell}\in \poly(\V\to\F_{p})$ with $\deg(W_{\ell})\leq \deg(Q)-2$, $\deg(H_{i})\leq \deg(Q)-2\ell-1$ such that for all $1\leq i\leq d$, $$\partial_{i}(Q-MW_{\ell})=M^{\ell}H_{i}.$$
 		
 	Since Claim 1 holds, by (\ref{1:force40}), we have that 
 		\begin{equation}\label{1:force41}
 		\partial_{i}(Q-MF)=M(R_{i}-\partial_{i}F).
 		\end{equation}
 		Since $\deg(F)\leq \deg(Q)-2$,
 		Claim 2 holds for $\ell=1$. Now suppose that Claim 2 holds for $\ell-1$ for some $2\leq\ell<p$. By induction hypothesis, there exist  $H_{1},\dots,H_{d}, W_{\ell-1},W'_{\ell-1}\in \poly(\V\to\F_{p})$ with $\deg(W_{\ell-1})\leq \deg(Q)-2$, $\deg(H_{i})\leq \deg(Q)-2\ell+1$   for all $1\leq i\leq d$ such that writing $Q':=Q-MW_{\ell-1}$, we have that
 		$$\partial_{i}Q'=M^{\ell-1}H_{i}$$
for all $1\leq i\leq d$.
 		Since $\deg(\partial_{i}Q')\leq p-1$, for all $1\leq j\leq d$,
 		\begin{equation}\label{1:force42}
 		\partial_{j}\partial_{i}Q'=\partial_{j}(M^{\ell-1}H_{i})=M^{\ell-2}((\ell-1)(\partial_{j}M)H_{i}+M\partial_{j}H_{i}).
 		\end{equation}
 		By symmetry,
 		\begin{equation}\label{1:force43}
 		\partial_{i}\partial_{j}Q'=M^{\ell-2}((\ell-1)(\partial_{i}M)H_{j}+M\partial_{i}H_{j}).
 		\end{equation}
 		Combining (\ref{1:force42}) and (\ref{1:force43}), we have that for all $1\leq i,j\leq d$, there exists a polynomial $N_{i,j}\in \poly(\V\to\F_{p})$ such that
 		\begin{equation}\label{1:force44}
 		(\partial_{i}M)H_{j}-(\partial_{j}M)H_{i}=MN_{i,j}.
 		\end{equation}
 		By Proposition \ref{1:noloop}, $\deg(N_{i,j})\leq \deg(H_{i})-1\leq \deg(Q)-2\ell$. 
For all $1\leq i,j\leq d', i\neq j$, there exist unique polynomials $N_{1,i,j},N_{2,i,j},N_{3,i,j}\in \poly(\V\to\F_{p})$ with $\deg(N_{1,i,j}),\deg(N_{2,i,j})$ $\leq \deg(N_{i,j})-1\leq s-2\ell-1$ and $\deg(N_{3,i,j})\leq \deg(N_{i,j})-2\leq s-2\ell-2$ such that
\begin{equation}\label{1:force44r}
 		N_{i,j}=(\partial_{i}M)N_{1,i,j}+(\partial_{j}M)N_{2,i,j}+\frac{1}{2}(\partial_{i}M)(\partial_{j}M)N_{3,i,j},
 		\end{equation}
 		where
		 $N_{1,i,j}$ is independent of $n_{j}$ and  $N_{2,i,j}$ is independent of $n_{i}$. 
 		Since  (\ref{1:force44}) implies that 
 		\begin{equation}\label{1:force45}
 		a_{i}n_{i}H_{j}(n)-a_{j}n_{j}H_{i}(n)=M(n)N_{i,j}(n)/2
 		\end{equation} 
and thus  
\begin{equation}\label{1:force45d}
 		a_{j}n_{j}H_{i}(n)-a_{i}n_{i}H_{j}(n)=M(n)N_{j,i}(n)/2
 		\end{equation} 
by symmetry, it follows from (\ref{1:force45}) and (\ref{1:force45d}) that $N_{i,j}+N_{j,i}=0$. So
 		$$n_{i}(N_{1,i,j}(n)+N_{2,j,i}(n))+n_{j}(N_{1,j,i}(n)+N_{2,i,j}(n))+n_{i}n_{j}(N_{3,i,j}(n)+N_{3,j,i}(n))=0.$$
 		Since $N_{1,i,j}+N_{2,j,i}$ is independent of $n_j$ and $N_{1,j,i}(n)+N_{2,i,j}(n)$ is independent of $n_i$, we have that 
 		\begin{equation}\label{1:force44rr}
 		N_{1,i,j}+N_{2,j,i}=0, N_{3,i,j}+N_{3,j,i}=0.
 		\end{equation}
Combining  (\ref{1:force44r}), (\ref{1:force45}) and (\ref{1:force44rr}),
 		  we have that
 		\begin{equation}\label{1:force46}
 		a_{i}n_{i}H_{j}(n)-a_{j}n_{j}H_{i}(n)=M(n)(n_{i}N_{1,i,j}(n)-n_{j}N_{1,j,i}(n)+n_{i}n_{j}N_{3,i,j})
 		\end{equation} 
 		for all distinct $1\leq i,j\leq d'$. For all $1\leq k\leq d', k\neq i,j$, (\ref{1:force46}) implies that
 		\begin{equation}\label{1:force47}
 		a_{i}n_{k}n_{i}H_{j}(n)-a_{j}n_{j}n_{k}H_{i}(n)=M(n)(n_{k}n_{i}N_{1,i,j}(n)-n_{j}n_{k}N_{1,j,i}(n)+n_{i}n_{j}n_{k}N_{3,i,j}).
 		\end{equation} 
 		By symmetry, 
 		\begin{equation}\label{1:force48}
 		a_{j}n_{i}n_{j}H_{k}(n)-a_{k}n_{k}n_{i}H_{j}(n)=M(n)(n_{i}n_{j}N_{1,j,k}(n)-n_{k}n_{i}N_{1,k,j}(n)+n_{i}n_{j}n_{k}N_{3,j,k})
 		\end{equation} 
 		and
 		\begin{equation}\label{1:force49}
 		a_{k}n_{j}n_{k}H_{i}(n)-a_{i}n_{i}n_{j}H_{k}(n)=M(n)(n_{j}n_{k}N_{1,k,i}(n)-n_{i}n_{j}N_{1,i,k}(n)+n_{i}n_{j}n_{k}N_{3,k,i}).
 		\end{equation} 
 		Considering $(\ref{1:force47})\times a_{k}+(\ref{1:force48})\times a_{i}+(\ref{1:force49})\times a_{j}$, we have that
 		$$n_{k}n_{i}(a_{k}N_{1,i,j}-a_{i}N_{1,k,j})+n_{i}n_{j}(a_{i}N_{1,j,k}-a_{j}N_{1,i,k})+n_{j}n_{k}(a_{j}N_{1,k,i}-a_{k}N_{1,j,i})$$
		is divisible by $n_{i}n_{j}n_{k}$. 
 		In particular, $n_{k}n_{i}(a_{k}N_{1,i,j}(n)-a_{i}N_{1,k,j}(n))$ is divisible by $n_{j}$.
 		Since $N_{1,i,j}-N_{1,k,j}$ is independent from $n_{j}$, we must have that $a_{k}N_{1,i,j}=a_{i}N_{1,k,j}$ for all distinct  $1\leq i,j,k\leq d'$.
 		This implies that $N_{1,i,j}=a_{i}N_{j}$ for some $N_{j}\in \poly(\V\to\F_{p})$ of degree at most $\deg(Q)-2\ell-1$ for all $1\leq i\leq d', i\neq j$. 
 		 Substituting this back to (\ref{1:force46}), we have that for all $1\leq i,j\leq d',i\neq j$,
 		\begin{equation}\label{1:force46a}
 		a_{i}n_{i}(H_{j}(n)-M(n)N_{j}(n))-a_{j}n_{j}(H_{i}(n)-M(n)N_{i}(n))=n_{i}n_{j}M(n)N_{3,i,j}.
 		\end{equation} 
		Setting $j=1$ and $i=2$,  	(\ref{1:force46a}) implies that 
		\begin{equation}\label{1:force46ab}
 		H_{1}(n)-M(n)N_{1}(n)=2cn_{1}G(n)=\partial_{1}M(n)G(n) 
 		\end{equation} 
		  for some polynomial $G$ of degree at most $\deg(Q)-2\ell$.  
		
		By (\ref{1:force44}) and (\ref{1:force46ab}), for all $1\leq j\leq d$,
		\begin{equation}\nonumber
 		\begin{split}
 		 (\partial_{1}M)H_{j}=MN_{1,j}+(\partial_{j}M)H_{1}
		 =M(N_{1,j}+(\partial_{j}M)N_{1})
		 +(\partial_{j}M)(\partial_{1}M)G, 
 		\end{split}
 		\end{equation}
		which implies that 
		\begin{equation}\nonumber
 		\begin{split}
 		 2cn_{1}(H_{j}(n)-\partial_{j}M(n)G(n))=M(n)(N_{1,j}(n)+\partial_{j}M(n)N_{1}(n)).
 		\end{split}
 		\end{equation}
		It is not hard to see that $N_{1,j}(n)+\partial_{j}M(n)N_{1}(n)=2cn_{1}F_{j}(n)$ for some polynomial $F_{j}$ of degree at most $\deg(Q)-2\ell-1$. Therefore, 
		\begin{equation}\nonumber
 		\begin{split}
 		H_{j}-(\partial_{j}M)G=MF_{j}.
 		\end{split}
 		\end{equation}		
		So for all $1\leq i\leq d$,
		\begin{equation}\nonumber
 		\begin{split}
 		&\quad\partial_{i}(Q-M(W_{\ell-1}+\ell^{-1}M^{\ell-1}G))
		=M^{\ell-1}H_{i}-\partial_{i}(\ell^{-1}M^{\ell}G)
		\\&=M^{\ell-1}H_{i}-(\partial_{i}M)M^{\ell-1}G-\ell^{-1}M^{\ell}\partial_{i}G
		=M^{\ell}(F_{i}-\ell^{-1}\partial_{i}G).
 		\end{split}
 		\end{equation}
 		Since $\deg(W_{\ell-1}+\ell^{-1}M^{\ell-1}G)\leq \max\{\deg(W_{\ell-1}),\deg(M^{\ell-1}G)\}\leq \deg(Q)-2$,
 		the claim follows for $\ell$ by taking $W_{\ell}=W_{\ell-1}+\ell^{-1}M^{\ell-1}G$. This finishes the proof of Claim 2.
 		
 		\
 		
 		Taking $\ell=\lfloor \deg(Q)/2\rfloor+1$ in Claim 2, we have that exists a polynomial $W$ of degree at most $s-2$ such that $\partial_{i}(Q-MW)=0$ for all $1\leq i\leq d$.  
 		This implies that $Q=MW+a$ for some $a\in\F_{p}$. Since $Q(\bold{0})=0$, we have that $Q=MW$ and we are done.
 	\end{proof}

	  \subsection{Solution to a special polynomial equation}\label{1:s:43}

In this section, we study the solutions to polynomial equations of the form
\begin{equation}\label{1:t43t}
\Delta_{h_{s-1}}\dots\Delta_{h_{1}}P(n)+\Delta_{h_{s}}\dots\Delta_{h_{1}}Q(n)=0, \text{ for all/many } (n,h_{1},\dots,h_{s})\in\Gow_{s}(V(M))
\end{equation}
		for some quadratic form $M\colon\V\to\F_{p}$,  which will be used crucially later in the paper. 
We begin with a preparatory result:

  \begin{prop}\label{1:basicpp55}
 	  	Let $\d>0$, $d,d',s\in\N$ with $d\geq d'\geq 4$, $p\gg_{d,s} \d^{-O_{d,s}(1)}$ be a prime number, 
 	  	$P,Q\in \poly(\V\to\F_{p})$ be  polynomials of degrees at most $s$ and $s+1$ respectively with $Q(\bold{0})=0$, $c\in\F_{p}\backslash\{0\}$, and $H\subseteq \V$ with $\vert H\vert>\d p^{d}$.
 	   	Let $M\colon\V\to\F_{p}$ be a non-degenerate quadratic form given by $$M(n_{1},\dots,n_{d}):=cn^{2}_{1}+n^{2}_{2}+\dots+n_{d'}^{2}.$$
 	  	Suppose that for all $n\in\V$ and $h=(h_{1},\dots,h_{d})\in H$ with $M(n)=(nA)\cdot h=0$ (where $A$ is the matrix associated to $M$), we have that $P(n)+\sum_{i=1}^{d}h_{i}\partial_{i}Q(n)=0$. Then  there exist $P',Q'\in \poly(\V\to\F_{p})$ of degrees at most $s-2$ and $s-1$ respectively such that $P=MP'$ and $Q=MQ'$.
 	\end{prop}

	 Note that if $h_{1}\neq 0$, then $(nA)\cdot h=0$ if and only if $$n=(-(h_{2}z_{2}+\dots+h_{d'}z_{d'}),ch_{1}z_{2},\dots,ch_{1}z_{d})$$ for some $z_{2},\dots,z_{d}\in\F_{p}^{d-1}$. On the other hand, the set of $h\in\V$ with $h_{1}=0$ is of cardinality $p^{d-1}$. Combining this observation with
  Proposition   \ref{1:antideri}, we may reduce Proposition \ref{1:basicpp55} to the following proposition: 
 	  
 	  \begin{prop}\label{1:HN}
 	  	Let $\d>0$, $d,d',s\in\N$ with $d\geq d'\geq 4$, $p\gg_{d,s} \d^{-O_{d,s}(1)}$ be a prime number, 
 	  	$P,Q_{1},\dots, Q_{d}\in \poly(\V\to\F_{p})$ be of degrees at most $s$, $c\in\{1,\dots,p-1\}$ and $H\subseteq \V$ with $\vert H\vert>\d p^{d}$. 
 	  	Let $$M(n_{1},\dots,n_{d}):=cn^{2}_{1}+n^{2}_{2}+\dots+n_{d'}^{2}$$ and
 	  	$$L_{h}(z_{2},\dots,z_{d}):=(-(h_{2}z_{2}+\dots+h_{d'}z_{d'}),ch_{1}z_{2},\dots,ch_{1}z_{d})$$
		for $h=(h_{1},\dots,h_{d})\in\V$.
 	  	Suppose that for all $z\in\F_{p}^{d-1}$ and $h=(h_{1},\dots,h_{d})\in H$, $M(L_{h}(z))=0$ implies that $$P(L_{h}(z))+\sum_{i=1}^{d}h_{i}Q_{i}(L_{h}(z))=0.$$ Then for all $n\in\V$, $M(n)=0$ implies that $$\text{$P(n)=0$ and $\partial_{j}M(n)Q_{i}(n)=\partial_{i}M(n)Q_{j}(n)$ for all $1\leq i,j\leq d$.}$$  
 	  \end{prop}	
 	  \begin{proof}
 	  		We assume that all the constants in the proof depends implicitly on $d$ and $s$.
 	  		Let $A$ be the (diagonal) matrix associated to $M$ and $a_{i}$ be the $(i,i)$-th entry of $A$. Passing to a subset if necessary, we may assume without loss of generality that $h_{1}\neq 0$ for all $h\in H$.
 	  		The assumption is then equivalent of saying that for all $n\in\V$ and $h\in H$, if $(nA)\cdot h=M(n)=0$ then $P(n)+\sum_{i=1}^{d}h_{i}Q_{i}(n)=0$.
 	  	Since $d'\geq 3$, by Lemma \ref{1:counting}, the set $h\in\V$ with $M(h)=0$ is $O_{d,s}(p^{d-1})$. Passing to a subset if necessary, we may assume without loss of generality that $M(h)\neq 0$ for all $h\in H$.

			Let $Z$ be the set of $n\in\V$ with $nA=\bold{0}$.  Then $\vert Z\vert=p^{d-d'}\leq p^{d-4}$.
 	  			Denote $V_{h}:=\{n\in\V\colon (nA)\cdot h=0\}=(\sp_{\F_{p}}\{h\})^{\pp}$.
 	  		We say that $(n,h)\in\V\times\V$ is a \emph{good pair} if $(nA)\cdot h=M(n)=0$, $n\in\V\backslash Z$ 
			and $h\in H$.
 	  		 For each $h\in H$,  $V_{h}$ is a subspace of $\V$ of co-dimension 1 and $V_{h}\cap (V_{h})^{\perp_{M}}=\{\bold{0}\}$ (since $M(h)\neq 0$). By Proposition \ref{1:iissoo}, $\rank(M\vert_{V_{h}})\geq d'-1\geq 3$.
 	  		By Corollary \ref{1:counting01} and the fact that $\vert Z\vert\leq p^{d-4}$, there exist $p^{d-2}(1+O(p^{-1/2}))$ many $n\in\V$ such that $(n,h)$ is a good pair.  
 	  		 So there are in total at least $\d p^{2d-2}(1+O(p^{-1/2}))$ many good pairs.  On the other hand, for each   $n\in \V\backslash Z$ with $M(n)=0$, since $nA\neq \bold{0}$, there exist $p^{d-1}$ many $h\in\V$ such that $(nA)\cdot h=0$. 
	By Lemma \ref{1:bzt}, the number of $n\in \V\backslash Z$ with $M(n)=0$ is 		 $p^{d-1}(1+O(p^{-1/2}))$.
	 From this it is easy to see that there exists a subset $X\subseteq \V\backslash Z$ with $\vert X\vert>\frac{\d}{2}p^{d-1}(1+O(p^{-1/2}))$ such that for all $n\in X$, $M(n)=0$, and there exists a subset $H_{n}\subseteq H$ with $\vert H_{n}\vert>\frac{\d}{2}p^{d-1}(1+O(p^{-1/2}))$ such that $(n,h)$ is a good pair for all $h\in H_{n}$.
 	  		
 	  	For all $n\in X$, let $W_{n}$ be the subspace of $\F_{p}^{d+1}$ spanned by $(1,h), h\in H_{n}$. We claim that $W_{n}=\F_{p}\times V_{n}$. Clearly, $W_{n}\subseteq\F_{p}\times V_{n}$, so it suffices to show that $W_{n}$ is of dimension $d$. Suppose on the contrary that $W_{n}$ is of dimension at most $d-1$. Let $v_{i}=(v_{i,0},\dots,v_{i,d})\in\F_{p}^{d+1}, 1\leq i\leq L$ be a basis of $W_{n}$. If $(1,h)=a_{1}v_{1}+\dots+a_{L}v_{L}$, we must have that $a_{1}v_{i,0}+\dots+a_{L}v_{L,0}=1$. 
 	  		So there are at most $p^{L-1}\leq p^{d-2}$ many $(1,h)$ belonging to $W_{n}$, meaning that $\vert H_{n}\vert\leq p^{d-2}$, a contradiction since $\vert H_{n}\vert>\frac{\d}{2}p^{d-1}(1+O(p^{-1/2}))$ and $p\gg \d^{-O(1)}$. So $W_{n}=\F_{p}\times V_{n}$.
 	  		This proves the claim.
 	  		
 	  		Since $P(n)+\sum_{i=1}^{d}h_{i}Q_{i}(n)=0$ for all $n\in X$ and $h\in H_{n}$, we have that $$(P(n),Q_{1}(n),\dots,Q_{d}(n))\cdot (1,h)=0$$ for all $h\in H_{n}$. By linearity, 
 	  		$$(P(n),Q_{1}(n),\dots,Q_{d}(n))\cdot v=0$$ for all $v\in W_{n}$. Since by the claim $W_{n}=\F_{p}\times V_{n}$, we have that $P(n)=0$ and that
 	  		$$(nA)\cdot h=0 \Rightarrow (Q_{1}(n),\dots,Q_{d}(n))\cdot h=0$$
 	  		for all $h\in\V$ and $n\in X$. This implies that 
 	  		 $(Q_{1}(n),\dots,Q_{d}(n))=F(n)(nA)$ for some $F(n)\in\F_{p}$ for all $n\in X$.
 	  		 In other words, 
 	  		 $$\partial_{j}M(n)Q_{i}(n)=2a_{i}a_{j}n_{i}n_{j}F(n)=\partial_{i}M(n)Q_{j}(n)$$
 	  		 for all $1\leq i,j\leq d$ and $n\in X$.
 	  		
 	  		Finally,
 	  		since $\vert X\vert>\frac{\d}{2}p^{d-1}(1+O(p^{-1/2}))$, by Lemma \ref{1:bzt},  for all $n\in\V$, $M(n)=0$ implies that $P(n)=0$ and $\partial_{j}M(n)Q_{i}(n)=\partial_{i}M(n)Q_{j}(n)$ for all $1\leq i,j\leq d$.
 	  \end{proof}	
 	  
We are now able provide a solution to (\ref{1:t43t}):

 \begin{prop}\label{1:packforce0}
 	Let $d,s\in\N_{+}$, $k\in\N$, $\d>0$, $p\gg_{d,k} \d^{-O_{d,k}(1)}$ be a prime,  and $M\colon\V\to\F_{p}$ be a quadratic form. 
	Let $W$ be a subset of $\Gow_{s}(V(M))$ such that either
	\begin{itemize}
	    \item $W=\Gow_{s}(V(M))$ and $\rank(M)\geq s+3$; or
	    \item $\vert W\vert\geq\d\vert\Gow_{s}(V(M))\vert$ and $\rank(M)\geq s^{2}+s+3$.
	\end{itemize}	
	Let $P,Q\in\poly(\V\to \F_{p})$ with $\deg(P)\leq k-1$ and $\deg(Q)\leq k$ be such  that for all $(n,h_{1},\dots,h_{s})\in W$, we have that
		$$\Delta_{h_{s-1}}\dots\Delta_{h_{1}}P(n)+\Delta_{h_{s}}\dots\Delta_{h_{1}}Q(n)=0$$ (where $\Delta_{h_{s-1}}\dots\Delta_{h_{1}}P(n)$ is understood as $P(n)$ when $s=1$). Then $$P=MP_{1}+P_{2} \text{ and } Q=MQ_{1}+Q_{2}$$ for some $P_{1},P_{2},Q_{1},Q_{2}\in\poly(\V\to \F_{p})$ with $\deg(P_{1})\leq k-3$, $\deg(P_{2})\leq s-2$, $\deg(Q_{1})\leq k-2$, $\deg(Q_{2})\leq s-1$.	 
 \end{prop}
 
Therefore, Proposition \ref{1:packforce0} provides solutions to equation (\ref{1:t43t}) for two cases. The first case is when (\ref{1:t43t})  holds for all $(n,h_{1},\dots,h_{s})\in \Gow_{s}(V(M))$, the second case is when (\ref{1:t43t})  holds for a dense subset of $(n,h_{1},\dots,h_{s})\in \Gow_{s}(V(M))$ (in the later case we need to impose a stronger lower bound for the dimension $d$).
 \footnote{The reason we do not obtain the lower bound $s+3$ in the second case is because the dimension restriction in the Fubini's theorem (Theorem \ref{1:ct}) is not optimal. The lower bound $s^{2}+s+3$ could be improved if we obtain a finer version of Theorem \ref{1:ct} (see also Remark \ref{1:ipmct}). We do not pursue this improvement in this paper.}

 \begin{proof}[Proof of Proposition \ref{1:packforce0}]	 
 Throughout the proof we assume that $p\gg_{d,k} \d^{-O_{d,k}(1)}$.
  Let $A$ be the matrix associated with $M$ and denote $d'=\rank(M)$. By Lemmas \ref{1:cov} and \ref{1:changeh}, under a change of variable, it suffices to consider the case when $A$ is a diagonal matrix with diagonal $(c,1,\dots,1,0,\dots,0)$ for some $c\in\F_{p}\backslash\{0\}$, where there are $d-d'$ zero on the diagonal.
 	For $h\in\V$, let $U_{h}:=\{n\in\V\colon M(n+h/2)=M(n-h/2)\}$. Then we may write $U_{h}=V_{h}+u_{h}$ for some $u_{h}\in \V$, where 	$V_{h}:=\{n\in\V\colon (nA)\cdot h=0\}$.
	Let $Z$ be the set of $h\in\V$ such that $\dim(V_{h})\neq d-1$ or $\dim(V_{h}\cap V_{h}^{\pp})\neq d-d'$.
	
	\textbf{Claim 1.} We have that $\vert Z\vert\leq O_{d}(p^{d-1})$.

	Let $Z'$ be the set of $h\in\V$ such that either $(hA)\cdot h=0$ or at least one of the coordinate of $h$ is 0. 
	Since $\rank(A)=d'\geq 3$, by Lemma \ref{1:counting}, it is not hard to see that $\vert Z'\vert=O_{d}(p^{d-1})$. Now fix any $h=(h_{1},\dots,h_{d})\in\V\backslash Z'$. Then since $hA\neq\bold{0}$, we have that $\dim(V_{h})=d-1$. Note that
	$$\{-h_{i}e_{1}+ch_{1}e_{i}\colon 2\leq i\leq d'\}\cup\{e_{j}\colon d'+1\leq j\leq d\}$$
	is a basis of $V_{h}$, where $e_{i}$ is the $i$-th standard unit vector. From this it is not hard to compute that
	$$(h_{1},\dots,h_{d'},0,\dots,0), e_{d'+1},\dots,e_{d}$$
	is a basis of $(V_{h})^{\pp}$. Therefore, $\dim((V_{h})^{\pp})=d-d'+1$.
	Finally, it is not hard to compute by induction that 
		$$\det\begin{bmatrix}
	-h_{2} & ch_{1} & & & \\
	-h_{3} & & ch_{1} & & \\
	\dots & & & \dots & \\
	-h_{d'} & & & & ch_{1} \\
	h_{1} & \dots & \dots & \dots & h_{d'}
	\end{bmatrix}=(-1)^{d'+1}(ch_{1})^{d'-2}((hA)\cdot h)\neq 0.$$
	This implies that $V_{h}+(V_{h})^{\pp}=\V$. So
	$$\dim(V_{h}\cap V_{h}^{\pp})=\dim(V_{h})+\dim(V_{h}^{\pp})-\dim(V_{h}+(V_{h})^{\pp})=(d-1)+(d-d'+1)-d=d-d'.$$ 
	Therefore, $Z\subseteq Z'$ and thus $\vert Z\vert\leq O_{d}(p^{d-1})$.
	This proves Claim 1.
	
\
 
 By the construction of $Z$ and Proposition \ref{1:iissoo} (ii), for all $h\in\V\backslash Z$, we have that $\rank(M_{h})=d'-1$.
	 Let $L_{h}\colon \F_{p}^{d-1}\to V_{h}$ be any bijective linear transformation and denote $M_{h}(m):=M(L_{h}(m)+u_{h}-h/2)$. 
 By the definition of $U_{h}$, we have that $M(L_{h}(m)+u_{h}-h/2)=M(L_{h}(m)+u_{h}+h/2)$ for all $m\in\F_{p}^{d-1}$. From this it
   is not hard to check that 
	   \begin{equation}\label{1:iieef22}
   (m,m_{1},\dots,m_{s-1})\in \Gow_{s-1}(M_{h})\Leftrightarrow (L_{h}(m)+u_{h}-h/2,L_{h}(m_{1}),\dots,L_{h}(m_{s-1}),h)\in \Gow_{s}(M).
   \end{equation}

 For all $h_{s}\in\V$, let $W_{h_{s}}$ be the set of $(n,h_{1},\dots,h_{s-1})\in (\V)^{s}$ such that $(n,h_{1},\dots,h_{s})\in W$. For all $h\in\V\backslash Z$, let $W'_{h}$ be the set of $(m,m_{1},\dots,m_{s-1})$ such that $$(L_{h}(m)+u_{h}-h/2,L_{h}(m_{1}),\dots,L_{h}(m_{s-1}))\in W_{h}.$$
  Since $(n,h_{1},\dots,h_{s})\in \Gow_{s}(V(M))$ for all $(n,h_{1},\dots,h_{s-1})\in W_{h_{s}}$, 
 it follows from (\ref{1:iieef22}) that 
  \begin{equation}\label{1:iieefws}
   W'_{h}\subseteq \Gow_{s-1}(V(M_{h})) \text{ and } \vert W'_{h}\vert=\vert W_{h}\vert \text{ for all } h\in \V\backslash Z.
   \end{equation}

 If $\vert W\vert\geq\d\vert\Gow_{s}(V(M))\vert$ and $\rank(M)\geq s^{2}+s+3$, then by Proposition \ref{1:ctsp}, we have that
 $$\sum_{h_{s}\in\V}\vert W_{h_{s}}\vert=\vert W\vert>\d\vert\Gow_{s}(V(M))\vert=\d p^{(s+1)d-(\frac{s(s+1)}{2}+1)}(1+O_{s}(p^{-1/2})).$$
 Moreover, for all but at most $s(s+1)p^{d+s-\rank(M)}$ many $h_{s}\in\V$, we have that 
 $\vert W_{h_{s}}\vert\leq p^{ds-(\frac{s(s+1)}{2}+1)}$.
 Also the sum of $\vert W_{h_{s}}\vert$
  over those at most $s(s+1)p^{d+s-d'}$ exceptionally $h_{s}\in\V$ is at most  $s(s+1)p^{(s+1)d+s-d'}$. It then follows from the Pigeonhole Principle  that there exists $W_{\ast}\subseteq \V\backslash Z$ of cardinality at least $\d p^{d}/4$ such that 
 $\vert W_{h_{s}}\vert\geq \d p^{ds-(\frac{s(s+1)}{2}+1)}/4$ for all $h_{s}\in W_{\ast}$.
 Since $\rank(M_{h})\geq (s-1)^{2}+(s-1)+3$, by Proposition \ref{1:ctsp} and (\ref{1:iieefws}), we have that
  \begin{equation}\nonumber
   \vert W'_{h}\vert=\vert W_{h}\vert\geq\d p^{ds-(\frac{s(s+1)}{2}+1)}/4=\d p^{(d-1)s-(\frac{s(s-1)}{2}+1)}/4=O(\d) \vert \Gow_{s-1}(V(M_{h}))\vert
   \end{equation}
   for all $h\in W_{\ast}\backslash Z$.
   
   On the other hand, if $W=\Gow_{s}(V(M))$ and $\rank(M)\geq s+3$, then for all $h\in \V\backslash Z$,  it follows from (\ref{1:iieef22}) that 
   $ W'_{h}=\Gow_{s-1}(V(M_{h}))$. Therefore, in both cases,  
   there exists $W_{\ast}\subseteq \V$ of cardinality at least $\d p^{d}/4$ such that 
  \begin{equation}\label{1:iieef202}
   \vert W'_{h}\vert=\vert W_{h}\vert\geq O(\d) \vert \Gow_{s-1}(V(M_{h}))\vert
   \end{equation}
   for all $h\in W_{\ast}\backslash Z$.

So the assumption and (\ref{1:iieef22}) implies that
    $$\Delta_{L_{h}(m_{s-1})}\dots\Delta_{L_{h}(m_{1})}P(L_{h}(m)+u_{h}-h/2)+\Delta_{L_{h}(m_{s-1})}\dots\Delta_{L_{h}(m_{1})}\Delta_{h}Q(L_{h}(m)+u_{h}-h/2)=0$$
    for all $h\in W_{\ast}\backslash Z$ and $(m,m_{1},\dots,m_{s-1})\in W'_{h}$. 
    Writing $P_{h}:=P(L_{h}(\cdot)+u_{h}-h/2)$ and $Q_{h}:=Q(L_{h}(\cdot)+u_{h}+h/2)-Q(L_{h}(\cdot)+u_{h}-h/2)$, we have that 
     \begin{equation}\label{1:iieef20}
   \Delta_{m_{s-1}}\dots\Delta_{m_{1}}(P_{h}(m)+Q_{h}(m))=0
   \end{equation}
       for all $h\in W_{\ast}\backslash Z$ and $(m,m_{1},\dots,m_{s-1})\in W'_{h}$. We remark that (\ref{1:iieef20}) is also valid when $s=1$, in which case  (\ref{1:iieef20}) is understood as 
           \begin{equation}\label{1:iieef30}
  P_{h}(m)+Q_{h}(m)=0
   \end{equation}
for all $h\in W_{\ast}\backslash Z$ and $m\in W'_{h}$.

\textbf{Claim 2.} 
Let $t\in\N$.
		Suppose that for all $h\in W_{\ast}\backslash Z$, we may write
		\begin{equation}\label{1:iieef40}
  P_{h}+Q_{h}=M_{h}R_{h}+R'_{h}
   \end{equation}
    for some $R_{h},R'_{h}\in\poly(\V\to\F_{p})$ with $\deg(R_{h})\leq k-3$ and $\deg(R'_{h})\leq t-1$. Then 
    $$P=MP'+P'' \text{ and } Q=MQ'+Q''$$
    for some $P',P'',Q',Q''\in\poly(\V\to\F_{p})$ with $\deg(P')\leq k-3$, $\deg(P'')\leq t-1$, $\deg(Q')\leq k-2$ and $\deg(Q'')\leq t$. 
    
    If $\max\{\deg(P),\deg(Q)-1\}\leq t-1$, then Claim 2 follows by setting $P=P'',Q=Q''$ and $P'=Q'=0$.
    Assume now that Claim 2 holds when $\max\{\deg(P),\deg(Q)-1\}\leq k'-1$ for some $t\leq k'\leq k-1$ and we prove that Claim 2 holds when $\max\{\deg(P),\deg(Q)-1\}\leq k'$.
    Note that (\ref{1:iieef40}) implies that
   \begin{equation}\label{1:iieef000}
   \begin{split}
   &\quad P(L_{h}(m)+u_{h}-h/2)+Q(L_{h}(m)+u_{h}+h/2)-Q(L_{h}(m)+u_{h}-h/2)
   \\&=M(L_{h}(m)+u_{h}-h/2)R_{h}(m)+R'_{h}(m)
   \end{split}
   \end{equation}
   for all $h\in W_{\ast}\backslash Z$. Since $\deg(R'_{h})\leq k'-1$, it is not hard to see that $\deg(R_{h})\leq k'-2$.
   Let $\tilde{P},\tilde{Q},\tilde{M},\tilde{R}_{h}$ be the degree $k',k'+1,2,k'-2$  terms of $P,Q,M,R_{h}$ respectively. Write $h=(h_{1},\dots,h_{d})$.
   Comparing the degree $k'$ terms of (\ref{1:iieef000}) and notice that the degree $k'$ terms of $Q(L_{h}(m)+u_{h}+h/2)-Q(L_{h}(m)+u_{h}-h/2))$  is  $\sum_{i=1}^{d}h_{i}\partial_{i}\tilde{Q}(L_{h}(m))$, we have that
    $$\tilde{P}(L_{h}(m))+\sum_{i=1}^{d}h_{i}\partial_{i}\tilde{Q}(L_{h}(m))=\tilde{M}(L_{h}(m))\tilde{R}_{h}(m).$$ 
    
     Therefore for all $n\in\V$ and $h\in W_{\ast}\backslash Z$ with $(nA)\cdot h=M(n)=0$, there exists $m\in\F_{p}^{d-1}$ with $L_{h}(m)=n$ and thus 
     $$\tilde{P}(n)+\sum_{i=1}^{d}h_{i}\partial_{i}\tilde{Q}(n)=\tilde{M}(n)\tilde{R}_{h}(m)=0.$$ 
     
     Since $\vert Z\vert\leq O_{d}(p^{d-1})$ by Claim 1, we have that $\vert W_{\ast}\backslash Z\vert\geq  p^{d}/8$.
         Since $\rank(M)\geq s+3\geq 4$,
     by Proposition \ref{1:basicpp55}, we have that $\tilde{P}=\tilde{M}P'$ and $\tilde{Q}=\tilde{M}Q'$ for some $P',Q'\in\poly(\V\to\F_{p})$ with $\deg(P')\leq k'-2\leq k-3$ and $\deg(Q')\leq k'-1\leq k-2$. Since $M-\tilde{M}$ is of degree at most 1, we may write 
    $$P=MP'+P'', Q=MQ'+Q''$$
    for some $P'',Q''\in\poly(\V\to\F_{p})$ with $\deg(P'')\leq k'-1$ and $\deg(Q'')\leq k'$. 
   Writing $P''_{h}:=P''(L_{h}(\cdot)+u_{h}-h/2)$ and $Q''_{h}:=Q''(L_{h}(\cdot)+u_{h}+h/2)-Q''(L_{h}(\cdot)+u_{h}-h/2)$, 
   It follows from (\ref{1:iieef40}) and the identity
   $$M_{h}=M(L_{h}(\cdot)+u_{h}-h/2)=M(L_{h}(\cdot)+u_{h}+h/2)$$
   that 
	 for all $h\in W_{\ast}\backslash Z$, we may write
		\begin{equation}\label{1:iieef440}
  P''_{h}+Q''_{h}=M_{h}R''_{h}+R'''_{h}
   \end{equation}
   for some $R''_{h},R'''_{h}\in\poly(\V\to\F_{p})$ with $\deg(R''_{h})\leq k-3$ and $\deg(R'''_{h})\leq t-1$. Since $\max\{\deg(P''),\deg(Q'')-1\}\leq k'-1$, by induction hypothesis, (\ref{1:iieef440}) implies that 
    $$P''=MP'''+P'''', Q''=MQ'''+Q''''$$
    for some $P''',P'''',Q''',Q''''\in\poly(\V\to\F_{p})$ with $\deg(P''')\leq k-3$, $\deg(P'''')\leq t-1$, $\deg(Q''')\leq k-2$ and $\deg(Q'''')\leq t$. Then
    $$P=M(P'+P''')+P'''' \text{ and } Q=M(Q'+Q''')+Q''''.$$
   This completes the induction step and proves Claim 2.
    
     \
 
  	We now complete the proof of Proposition \ref{1:packforce0} by induction on $s$.
	 If $s=1$, then since $\rank(M_{h})\geq 3$, by (\ref{1:iieef30}) and Proposition \ref{1:noloop},
	  we may write
	    \begin{equation}\nonumber
  P_{h}+Q_{h}=M_{h}R_{h}
   \end{equation}
   for some $R_{h}\in\poly(\V\to\F_{p})$ with $\deg(R_{h})\leq k-3$ for all $h\in W_{\ast}\backslash Z$.  By Claim 2 (setting $t=0$), 
   $$P=MP_{1} \text{ and } Q=MQ_{1}+C$$ for some $C\in\F_{p}$ and $P_{1},Q_{1}\in\poly(\V\to \F_{p})$ with $\deg(P_{1})\leq k-3$, $\deg(Q_{1})\leq k-2$.

 	Suppose we have proved Proposition \ref{1:packforce0} for $s-1$ for some $s\geq 2$. We now prove it for $s$.  For all $h\in W_{\ast}\backslash Z$, 
     since $\rank(M_{h})=\rank(M)-1$ and $\deg(P_{h}+\Delta_{h}Q_{h})\leq k-1$, by applying the induction hypothesis to (\ref{1:iieef20}), we may write
    $$P_{h}+Q_{h}=M_{h}R_{h}+R'_{h}$$
    for some $R_{h},R'_{h}\in\poly(\V\to\F_{p})$ with $\deg(R_{h})\leq k-3$ and $\deg(R'_{h})\leq s-2$ for all $h\in W_{\ast}\backslash Z$. 
   By Claim 2 (setting $t=s+1$), 
   $$P=MP'+P'', Q=MQ'+Q''$$
    for some $P',P'',Q',Q''\in\poly(\V\to\F_{p})$ with $\deg(P')\leq k-3$, $\deg(P'')\leq s-2$, $\deg(Q')\leq k-2$ and $\deg(Q'')\leq s-1$.  This completes the induction step and we are done.
        \end{proof}

        \subsection{Intrinsic definitions for polynomials on $V(M)$}\label{1:s:idf}
 
 Let $M\colon\V\to\F_{p}$ be a quadratic form. 
 Recall that we define a function $f\colon V(M)\to \F_{p}$ to be a polynomial if $f$ is the restriction to $V(M)$ of some polynomial $f'\colon\V\to\F_{p}$ defined in $\V$. It is natural to ask if there is an intrinsic way to define polynomials on $V(M)$ (i.e. without looking at points outside $V(M)$). A natural way to do so is to use   anti-derivative properties of polynomials.
 Let $g\in\poly(\V\to\F_{p})$. It is not hard to see that $g$ is of degree at most $s-1$ if and only if $\Delta_{h_{s}}\dots\Delta_{h_{1}}g(n)=0$ for all $n,h_{1},\dots,h_{s}\in\V$. This observation provides us a promising alternative way to define polynomials and leads to the following question: 
 
 \begin{conj}[Intrinsic definition for polynomials]\label{1:att0}
 Let $g\colon\V\to\F_{p}$ be a function, $d,s\in\N_{+}$ and $p\gg_{d} 1$ be a prime. Let $M\colon\V\to\F_{p}$ be a quadratic form with $\rank(M)\gg_{s} 1$. Then the followings are equivalent:
	\begin{enumerate}[(i)]
 		\item $\Delta_{h_{s}}\dots\Delta_{h_{1}}g(n)=0$ for all $(n,h_{1},\dots,h_{s})\in \Gow_{s}(V(M))$;
 		\item  there exists a polynomial $g'\in\poly(\V\to\F_{p})$ of degree at most $s-1$ such that $g(n)=g'(n)$ for all $n\in V(M)$.
		\end{enumerate}	
 \end{conj}

 If Conjecture \ref{1:att0} holds, then we may use Part (i) of Conjecture \ref{1:att0} as an intrinsic definition for polynomials on $V(M)$.
 In this paper, we were unable to prove  Conjecture \ref{1:att0}. Instead, we prove the following special case of Conjecture \ref{1:att0}, which is good enough for the purpose of this paper:

 \begin{prop}\label{1:att1}
 	Let $d,s\in\N_{+}$, $p\gg_{d} 1$ be a prime, and $M\colon\V\to\F_{p}$ be a quadratic form associated with the matrix $A$ of rank at least $s+3$. Then for any $g\in\poly(\V\to \F_{p})$ with  $\deg(g)\leq s$, the following are equivalent:
 	\begin{enumerate}[(i)]
 		\item 
		for all $(n,h_{1},\dots,h_{s})\in \Gow_{s}(V(M))$,
		we have that 
		$\Delta_{h_{s}}\dots\Delta_{h_{1}}g(n)=0;$
 		\item  we have 
		$$g=Mg_{1}+g_{2}$$ for some $g_{1},g_{2}\in\poly(\V\to \F_{p})$ with  $\deg(g_{1})\leq s-2$, $\deg(g_{2})\leq s-1$;
		\item  we have 
		$$g(n)=((nA)\cdot n)g_{1}(n)+g_{2}(n)$$ for some $g_{1},g_{2}\in\poly(\V\to \F_{p})$ with  $\deg(g_{1})\leq s-2$, $\deg(g_{2})\leq s-1$.
 	\end{enumerate}	 
 \end{prop}
\begin{proof}
We first show that (ii) implies (i).
 	For all $(n,h_{1},\dots,h_{s})\in \Gow_{s}(V(M))$, note that 
 	$$\Delta_{h_{s}}\dots\Delta_{h_{1}}g(n)=\Delta_{h_{s}}\dots\Delta_{h_{1}}(M(n)g_{1}(n)).$$
 	Since $M(n+\e_{1}h_{1}+\dots+\e_{s}h_{s})=0$ for all $\e_{1},\dots,\e_{s}\in\{0,1\}$, we have $\Delta_{h_{s}}\dots\Delta_{h_{1}}(M(n)g_{1}(n))=0$ and we are done.
	
	 The part (i) $\Rightarrow$ (ii) follows from Proposition \ref{1:packforce0} (by setting $k=s$, $P\equiv 0$ and $Q=g$). Finally, (ii) $\Leftrightarrow$ (iii) since the degree of the polynomial $M(n)-(nA)\cdot n$ is at most 1.
\end{proof}

 As a corollary of Proposition \ref{1:att1}, we have:

 \begin{coro}\label{1:att4}
 	Let $d,k,s\in\N_{+}$, $p\gg_{d} 1$ be a prime, and $M\colon\V\to\F_{p}$ be a non-degenerate quadratic form associated with the matrix $A$.
 	Let $m_{1},\dots,m_{k}\in\V$ be linearly independent and $M$-non-isotropic vectors.
 	If $d\geq k+s+3$, then for any  polynomial $g\in\poly(\V\to \F_{p})$ of degree at most $s$, the following are equivalent:
 	\begin{enumerate}[(i)]
 		\item  for all $(n,h_{1},\dots,h_{s})\in \Gow_{s}(V(M)^{m_{1},\dots,m_{k}})$, we have that $\Delta_{h_{s}}\dots\Delta_{h_{1}}g(n)=0$;
 		\item  we have $$g(n)=M(n)g_{0}(n)+\sum_{i=1}^{k}(M(n+m_{i})-M(n))g_{i}(n)+g'(n)$$ for some homogeneous $g_{0},\dots,g_{k}\in\poly(\V\to \F_{p})$ with $\deg(g_{0})=s-2$ and $\deg(g_{1})=\dots=\deg(g_{k})=s-1$, and some $g'\in\poly(\V\to \F_{p})$ of degree at most $s-1$;\footnote{Recall that $g_{0},\dots,g_{k}$ are allowed to be  constant zero by Convention \ref{1:c00}.}
 		\item  we have $$g(n)=((nA)\cdot n)g_{0}(n)+\sum_{i=1}^{k}((nA)\cdot m_{i})g_{i}(n)+g'(n)$$ for some homogeneous $g_{0},\dots,g_{k}\in\poly(\V\to \F_{p})$ with $\deg(g_{0})=s-2$ and $\deg(g_{1})=\dots=\deg(g_{k})=s-1$, and some $g'\in\poly(\V\to \F_{p})$ of degree at most $s-1$.
 	\end{enumerate}	 
 \end{coro}
 \begin{proof}
 Clearly (ii) is equivalent to (iii). Note that for all $(n,h_{1},\dots,h_{s})\in \Gow_{s}(V(M)^{m_{1},\dots,m_{k}})$ and $\e_{1},\dots,\e_{s}\in\{0,1\}$, we have that $$M(n+\e_{1}h_{1}+\dots+\e_{s}h_{s})=M(n+\e_{1}h_{1}+\dots+\e_{s}h_{s}+m_{i})=0$$ for all $1\leq i\leq k$. So (ii) implies (i). So it suffices to show that (i) implies (ii).

By a change of variables $n\mapsto n+m$ for some $m\in\V$, it suffices to consider the special case when $M(n)=(nA)\cdot n-\lambda$ for some invertible symmetric matrix $A$ and some $\lambda\in\F_{p}$.
Let $V$ be the span of $m_{1},\dots,m_{k}$. Since $m_{1},\dots,m_{k}$ are linearly independent and $M$-non-isotropic, and since $A$ is invertible, we have that $V^{\pp}$ is a subspace of $\V$ of dimension $d-k$ with $V\cap V^{\pp}=\{\bold{0}\}$.
Let
  $$W:=\{n\in\V\colon 2(nA)\cdot m_{i}+(m_{i}A)\cdot m_{i}=0, 1\leq i\leq k\}.$$
   Since $m_{1},\dots,m_{k}$ are linearly independent and $A$ is invertible,  we may write $W=V^{\pp}+w$ for some $w\in\V$.

   Let $\phi\colon \F_{p}^{d-k}\to V^{\pp}$ be any bijective linear transformation. Denote $M'(n'):=M(\phi(n')+w)$ and $g'(n'):=g(\phi(n')+w)$ for all $n'\in \F_{p}^{d-k}$.  
   By Proposition \ref{1:iissoo} (ii), $M'$ is a non-degenerate quadratic form and thus $\rank(M')=d-k\geq s+3$.
   By Lemma \ref{1:changeh}, we have that $(n,h_{1},\dots,h_{s})\in \Gow_{s}(V(M)^{m_{1},\dots,m_{k}})$ if and only if there exists $(n',h'_{1},\dots,h'_{s})\in \Gow_{s}(V(M'))$  with $(n,h_{1},\dots,h_{s})=(\phi(n')+w,\phi(h'_{1}),\dots,\phi(h'_{s}))$.   
  So (i) implies that $\Delta_{h'_{s}}\dots\Delta_{h'_{1}}g'(n')=0$ for all $(n',h'_{1},\dots,h'_{s})\in \Gow_{s}(V(M'))$.
    By Proposition \ref{1:att1}, 
    \begin{equation}\label{1:attcnm}
    g'(n')=M'(n')g'_{1}(n')+g'_{2}(n')
    \end{equation}
      for some $g'_{1},g'_{2}\in\poly(\F_{p}^{d-k}\to \F_{p})$ with $\deg(g'_{1})=s-2$ and $\deg(g'_{2})\leq s-1$. We may further require  $g'_{1}$ to be homogeneous.
    
   By a change of variables $n\to nB$ for some invertible matrix $B$, we may assume without loss of generality that $m_{i}A=e_{d-k+i}$ for $1\leq i\leq k$. Then $V^{\pp}=\sp_{\F_{p}}\{e_{1},\dots,e_{d-k}\}$ and $W=V^{\pp}+w$ for $w=(0,\dots,0,-(m_{1}A)\cdot m_{1}/2,\dots,-(m_{1}A)\cdot m_{k}/2)$.
   Note that $n\in V(M)^{m_{1},\dots,m_{k}}$ if and only if $n=(n_{1},\dots,n_{d-k},-(m_{1}A)\cdot m_{1}/2,\dots,-(m_{1}A)\cdot m_{k}/2)\in V(M)$ for some $n_{1},\dots,n_{d-k}\in\F_{p}$. So it follows from (\ref{1:attcnm}) that
    \begin{equation}\label{1:attcnm2}
    \begin{split}
  &\quad  g(n_{1},\dots,n_{d-k},-(m_{1}A)\cdot m_{1}/2,\dots,-(m_{1}A)\cdot m_{k}/2)
  \\&=M(n_{1},\dots,n_{d-k},-(m_{1}A)\cdot m_{1}/2,\dots,-(m_{1}A)\cdot m_{k}/2)g'_{1}(n_{1},\dots,n_{d-k})+g'_{2}(n_{1},\dots,n_{d-k}).
    \end{split}
    \end{equation}

    Note that for all $1\leq i\leq k$ and $n=(n_{1},\dots,n_{d})\in\V$, we have that
    $$n_{d-k+i}+(m_{i}A)\cdot m_{i}/2=\frac{1}{2}(2(nA)\cdot m_{i}+(m_{i}A)\cdot m_{i})=\frac{1}{2}(M(n+m_{i})-M(n)).$$
    So for $F=g$ or $M$, the difference
    $$F(n_{1},\dots,n_{d})-F(n_{1},\dots,n_{d-k},-(m_{1}A)\cdot m_{1}/2,\dots,-(m_{1}A)\cdot m_{k}/2)$$
    can be written in the form
    $$\sum_{i=1}^{k}(M(n+m_{i})-M(n))Q_{i}(n)$$
    for some  $Q_{i}\in \poly_{p}(\V\to\F_{p})$ with   $\deg(Q)_{i}\leq \deg(F)-1$.
    Combining this with (\ref{1:attcnm2}), we deduce that $$g(n)=M(n)g_{0}(n)+\sum_{i=1}^{k}(M(n+m_{i})-M(n))g_{i}(n)+g'(n)$$ for some $g_{0},\dots,g_{k}\in\poly(\V\to \F_{p})$ with $\deg(g_{0})\leq s-2$ and $\deg(g_{1}),\dots,\deg(g_{k})\leq s-1$, and some $g'\in\poly(\V\to \F_{p})$ of degree at most $s-1$. Absorbing the lower degree terms of $g_{0},\dots,g_{k}$ by $g'$ if necessary, we may further require that $g_{0},\dots,g_{k}$ are homogeneous and that $\deg(g_{0})=s-2$, $\deg(g_{1})=\dots=\deg(g_{k})=s-1$. So (ii) holds and we are done.
 \end{proof}
  
  	\subsection{A parallel matrix proposition}
	
	Let $A$ and $B$ be $\F_{p}$-valued $d\times d$ matrices. If $B$ is a multiple of $A$, then clearly $nB=\bold{0}$ implies that $nA=\bold{0}$. It is natural to ask if the converse is true, namely if for many $n\in\V$, $nB=\bold{0}$ implies that $nA=\bold{0}$, then is it true that $B$ must be a multiple of $A$. The next proposition answers this question: 

  \begin{prop}\label{1:spe1}
   Let $d\in\N_{+}$, $\d>0$, $p\gg_{d}\d^{-O_{d}(1)}$ be a prime,  and $A$ be a  $\F_{p}$-valued $d\times d$ matrix of rank at least 3. Let $B$ be a $\F_{p}$-valued matrix $d\times d$ matrix and $v\in\V$. Let $W$ be a subset of $\V$ of cardinality at least $\d p^{d}$. Suppose that for all $w\in W$ and $n\in\V$ with $(nA)\cdot w=0$, we have that
    $$(nB+v)\cdot w=0.$$
     Then $v=\bold{0}$ and $B=cA$ for some $c\in\F_{p}$.
\end{prop}
\begin{proof}
Throughout the proof we assume that $p\gg_{d}\d^{-O_{d}(1)}$.
Since $\rank(A)\geq 1$, the set of $w\in\V$ with $wA=\bold{0}$ is of cardinality at most $p^{d-1}$. So   shrinking $\d$ to $\d/2$ if necessary, we may assume without loss of generality that $wA\neq \bold{0}$ for all $w\in W$.
Fix any $w\in W$. Since $wA\neq \bold{0}$, and since $(nA)\cdot w=0$ implies that $(nB+v)\cdot w=0$ for all $n\in\V$, it is not hard to check that we must have that $v\cdot w=0$ and $wB=c_{w}wA$ for some $c_{w}\in\F_{p}$. 

Since $W$ is of cardinality at least $\d p^{d}$, By Lemma \ref{1:iiddpp},
there exist linearly independent $w_{1},\dots,w_{d}\in W$. So $v\cdot w_{i}=0$ for $1\leq i\leq d$. This implies that $v=\bold{0}$.

For $z\in\V$, let $M_{z}(n):=(nA)\cdot (n+z)$. Then 
 		\begin{equation}\label{1:reduction44}
 		\sum_{z\in\V}\vert V(M_{z})\cap W\vert=\sum_{z\in\V}\vert\{w\in W\colon (wA)\cdot (w+z)=0\}\vert=\sum_{w\in W}\vert\{z\in \V\colon (wA)\cdot (w+z)=0\}\vert.
 		\end{equation}
 	Since $wA\neq \bold{0}$, we have $\vert\{z\in \V\colon (wA)\cdot (w+z)=0\}\vert=p^{d-1}$.   So (\ref{1:reduction44}) is at least
 	$\vert W\vert\cdot p^{d-1}\geq \d p^{2d-1}$. 
 	By the Pigeonhole Principle, there exists  $z\in \V$ such that $\vert V(M_{z})\cap W\vert\gg \d p^{d-1}$.
 	
 	Since $wB=c_{w}wA$ for some $c_{w}\in\F_{p}$ for all $w\in W$, we have that
 	$$(wB)\cdot (w+z)=c_{w}(wA)\cdot (w+z)=0$$
 	for all $w\in V(M_{z})\cap W$. Since $\vert V(M_{z})\cap W\vert\gg \d p^{d-1}$, by Lemma \ref{1:bzt} and Proposition \ref{1:noloop}, we may write
 	$$(wB)\cdot (w+z)=r_{z}M_{z}(w)=r_{z}(wA)\cdot (w+z)$$
 for some $r_{z}\in\F_{p}$ for all $w\in\V$. Viewing both sides as polynomials in the variable $w$ and comparing all of their degree 2 terms, we must have that $B=r_{z}A$.	We are done.
\end{proof}

\section{Algebraic properties for quadratic forms in $\Z/p$}\label{1:s5}

In this section, we extend the results in Section \ref{1:s4} to $\Q$-valued polynomials.

\subsection{Quadratic forms in $\Z/p$}\label{1:pdefn}

 By using the correspondence between polynomials in $\V$ and in $\Z^{d}$, we may lift all the definitions pertaining to quadratic forms in $\V$ to the $\Z^{d}$ setting in the natural way.

\begin{defn}[Quadratic forms in $\Z^{d}$]
	We say that a function $M\colon\Z^{d}\to\Z/p$ is a \emph{quadratic form} if 
	$$M(n)=\frac{1}{p}((nA)\cdot n+n\cdot u+v)$$
	for some $d\times d$ symmetric matrix $A$ in $\Z$, some $u\in \Z^{d}$ and $v\in \Z$.
We say that $A$ is the matrix \emph{associated to} $M$. 
\end{defn}

It is more rigorous to name $M$ as a \emph{$p$-quadratic form}. However, since the quantity $p$ is alway clear in this paper, we simply call $M$ to be a quadratic form for short.
Whenever we define a quadratic form in this paper, we will write either $M\colon\V\to\F_{p}$  or $M\colon\Z^{d}\to\Z/p$ to specify to domain and range of $M$, so that the term quadratic form will not cause confusions.

By Lemma \ref{1:lifting},
any quadratic form $\tilde{M}\colon\Z^{d}\to\Z/p$ associated with the matrix $\tilde{A}$ induces a quadratic form $M:=\iota\circ p\tilde{M}\circ\tau\colon\F_{p}^{d}\to\F_{p}$ associated with the matrix $\iota(\tilde{A})$. Conversely, any quadratic form $M\colon\F_{p}^{d}\to\F_{p}$ associated with the matrix $A$ admits a regular lifting $\tilde{M}\colon\Z^{d}\to\Z/p$, which is a quadratic form associated with the matrix $\tau(A)$.

For a quadratic form $\tilde{M}\colon\Z^{d}\to\Z/p$, we say that $\tilde{M}$ is \emph{pure/homogeneous/$p$-non-degenerate} if the  quadratic form $M:=\iota\circ p\tilde{M}\circ\tau$ induced by $\tilde{M}$ is pure/homogeneous/non-degenerate. 
The \emph{$p$-rank} of $\tilde{M}$, denoted by $\rank_{p}(\tilde{M})$, is defined to be the rank of $M$.

We say that $h_{1},\dots,h_{k}\in\Z^{d}$ are \emph{$p$-linearly independent} if for all $c_{1},\dots,c_{k}\in\Z/p$, $c_{1}h_{1}+\dots+c_{k}h_{k}\in\Z$
 implies that $c_{1},\dots,c_{k}\in\Z$, or equivalently, if $\iota(h_{1}),\dots,\iota(h_{k})$ are linearly independent.

We can also lift the definition of set of zeros and Gowers sets to the $\Z/p$ setting.
\begin{itemize}
\item  For a polynomial $P\in\poly(\Z^{k}\to\R)$, let $V_{p}(P)$ denote the set of $n\in \Z^{k}$ such that $P(n+pm)\in \Z$ for all $m\in\Z^{k}$.
\item For $h_{1},\dots,h_{t}\in\Z^{k}$, let $V_{p}(P)^{h_{1},\dots,h_{t}}$ denote the set of $n\in\Z^{k}$ such that $P(n+pm),P(n+h_{1}+pm),\dots,P(n+h_{t}+pm)\in\Z$ for all $m\in\Z^{k}$.
\item For $\Omega\subseteq\Z^{d}$ and $s\in\N$, let $\Gow_{p,s}(\Omega)$ denote the set of $(n,h_{1},\dots,h_{s})\in(\Z^{d})^{s+1}$ such that $n+\e_{1}h_{1}+\dots+\e_{s}h_{s}\in\Omega+p\Z^{d}$ for all $\e_{1},\dots,\e_{s}\in\{0,1\}$.   
We say that $\Gow_{p,s}(\Omega)$ is the \emph{$s$-th $p$-Gowers set} of $\Omega$.
\end{itemize}

   The following lemma is straightforward:  
  \begin{lem}\label{1:expp}
  For any $p$-periodic set $\Omega\subseteq \Z^{k}$ (recall the definition in Section \ref{1:ssdn}) and $K\in\N_{+}$, we have that
  $$\frac{1}{p^{k}}\vert \Omega \cap [p]^{k}\vert=\frac{1}{(pK)^{k}}\vert \Omega \cap [pK]^{k}\vert.$$
  \end{lem}

 \subsection{Lifting results from $\F_{p}$ to $\Z/p$}\label{1:s:lltt}
Using the connections between $\F_{p}$-valued polynomials and $\Z/p$-valued polynomials, one can easily extend many results in Section \ref{1:s4} to the $\Z/p$-setting. 
The following is an example:

\begin{coro}[Lifting of Proposition  \ref{1:noloop}]\label{1:noloop2}
	Let $d\in\N_{+}$, $s\in\N$,   $p\gg_{d,s} 1$ be a prime number, 
	$P\in \poly_{p}(\Z^{d}\to\Z/p\vert\Z)$   be a polynomial of degree at most $s$, and $M\colon\Z^{d}\to \Z/p$ be a quadratic form of $p$-rank at least 3. Then either $\vert V_{p}(M)\cap V_{p}(P)\cap[p]^{d}\vert\leq O_{d,s}(p^{d-2})$ or $V_{p}(M)\subseteq V_{p}(P)$.
	
	Moreover, if $V_{p}(M)\subseteq V_{p}(P)$, then $P=MP_{1}+P_{0}$ for some 
	$P_{0}\in\poly(\Z^{d}\to\Z)$ of degree at most $s$, and some integer coefficient polynomial $P_{1}\in\poly(\Z^{d}\to\Z)$ of degree at most $s-2$.
\end{coro}

\begin{proof}
Let $M'\colon\V\to\F_{p}$ be the quadratic form of rank at least 3 induced by $M$, $P'\colon\V\to\F_{p}$ be the polynomial of degree at most $s$ induced by $P$ (whose existence is given by Lemma \ref{1:lifting} (i)). 
By Proposition \ref{1:noloop}, either $\vert V(M')\cap V(P')\vert\leq O_{d,s}(p^{d-2})$ or $V(M')\subseteq V(P')$. By the definition of lifting, the form case implies that $\vert V_{p}(M)\cap V_{p}(P)\cap[p]^{d}\vert=\vert V(M')\cap V(P')\vert\leq O_{d,s}(p^{d-2})$, and the later case implies that $V_{p}(M)\subseteq V_{p}(P)$.

We now assume that $V_{p}(M)\subseteq V_{p}(P)$. By the definition of lifting, we have that $V(M')\subseteq V(P')$. By Proposition \ref{1:noloop}, $P'=M'Q'$ for some $Q'\in\poly(\V\to\F_{p})$ of degree at most $s-2$.
Let $Q\in\poly(\Z^{d}\to\Z/p)$ be a regular lifting of $Q'$ of degree at most $s-2$  (whose existence is given by Lemma \ref{1:lifting} (ii)). By Lemma \ref{1:lifting} (iv) and (v), $P-pMQ$ is a lifting of $P'-M'Q'\equiv 0$.
Then by Lemma \ref{1:lifting} (iii), we have that $P-pMQ$ is integer valued and is of degree at most $s$. 
By Lemma \ref{1:ivie}, we may write $Q=pQ_{1}+Q_{2}$ for some $\Z$-valued polynomial $Q_{1}$ and some integer coefficient polynomial $Q_{2}$ both having degrees at most $s-2$.
Then
$P=MQ_{2}+((P-pMQ)+pMQ_{1})$ with $(P-pMQ)+pMQ_{1}$ being an integer valued polynomial of degree at most $s$. We are done.
\end{proof}

The approach in the proof of Corollary \ref{1:noloop2} can be adapted to lift many other results from the $\F_{p}$-setting to the $\Z/p$-setting. 
%
Below are the liftings of some results from  Section \ref{1:s4} that will be used in the further.  
We leave the proofs to the interested readers.
		


	  \begin{coro}[Lifting of Corollary \ref{1:noloop3}]\label{1:noloop32}
 	Let  $d,k\in\N_{+}$, $k,s\in\N$, $p\gg_{d,k,s} 1$ be a prime number, 
 	$P\in \poly_{p}(\Z^{d}\to\Z/p\vert\Z)$ be a polynomial of degree at most $s$, $M\colon\Z^{d}\to\Z/p$ be a non-degenerate quadratic form induced by some $M'\colon\V\to\F_{p}$, $c\in\V$, and $V$ be a subspace of $\V$ of dimension $k$ with a basis $\iota(h_{1}),\dots,\iota(h_{k})$ for some $h_{1},\dots,h_{k}\in\Z^{d}$. 
	Denote  
	 $$L_{i}(n):=\frac{1}{p}(h_{i}\tau(A))\cdot (n-\tau(c))$$
		for $1\leq i\leq k$.
	 Suppose that $\rank(M'\vert_{V^{\pp}})\geq 3$. 
 	 Then either $\vert V_{p}(M)\cap \iota^{-1}(V^{\pp}+c)\cap V_{p}(P)\vert\leq O_{d,k,s}(p^{d-k-2})$ or $V_{p}(M)\cap \iota^{-1}(V^{\pp}+c)\subseteq V_{p}(P)$.
 	
 	Moreover, if $V_{p}(M)\cap \iota^{-1}(V^{\pp}+c)\subseteq V_{p}(P)$, then 
 	\begin{equation}\nonumber
 	P=Q+MP_{0}+\sum_{i=1}^{k}L_{i}P_{i}
 	\end{equation}
 	for some  integer coefficient  polynomials $P_{0},\dots,P_{k}\in \poly_{p}(\Z^{d}\to\Z)$ with $\deg(P_{0})\leq s-2$ and $\deg(P_{i})\leq s-1, 1\leq i\leq k$, and for some $Q\in  \poly_{p}(\Z^{d}\to\Z)$ with $\deg(Q)\leq s$.
 	
 	In particular, the conclusion of this corollary holds if $d\geq k+\dim(V\cap V^{\pp})+3$, or if $d\geq 2k+3$.
 \end{coro}   
  

\begin{coro}[Lifting of a special case of Proposition \ref{1:packforce0}]\label{1:att301}
Let $d,s\in\N_{+}$, $k\in\N$, $\d>0$, $p\gg_{d,k} \d^{-O_{d,k}(1)}$ be a prime, $M\colon\Z^{d}\to\Z/p$ be a quadratic form of $p$-rank at least $s^{2}+s+3$, and $W$ be a subset of $\Gow_{p,s}(V(M))\cap([p]^{d})^{s+1}$ of cardinality at least $\d \vert\Gow_{p,s}(V(M))\cap([p]^{d})^{s+1}\vert$. Let $P,Q\in\poly_{p}(\Z^{d}\to \Z/p\vert\Z)$ with $\deg(P)\leq k-1$ and $\deg(Q)\leq k$. Suppose that for all $(n,h_{1},\dots,h_{s})\in W$, we have that 		
		$$\Delta_{h_{s-1}}\dots\Delta_{h_{1}}P(n)+\Delta_{h_{s}}\dots\Delta_{h_{1}}Q(n)\in\Z,$$ (where $\Delta_{h_{s}}\dots\Delta_{h_{2}}P(n)$ is understood as $P(n)$ when $s=1$). Then $$P=MP_{1}+P_{2}+P_{3} \text{ and } Q=MQ_{1}+Q_{2}+Q_{3}$$ for some integer coefficient  polynomials $P_{1}$ and $Q_{1}$ of degrees at most $k-3$ and $k-2$ respectively, some integer valued polynomials $P_{2}$ and $Q_{2}$ of degrees at most $k-1$ and $k$ respectively, and some $\Z/p$-valued polynomials $P_{3}$ and $Q_{3}$ of degrees at most $s-2$ and $s-1$ respectively.
\end{coro}

\begin{coro}[Lifting of Proposition \ref{1:spe1}]\label{1:spe2}
   Let $d\in\N_{+}$,  $\d>0$, $p\gg_{d} \d^{-O_{d}(1)}$ be a prime, and $A$ be a  $\Z$-valued $d\times d$ matrix of $p$-rank at least 3. Let $B$ be a $\Z$-valued $d\times d$ matrix and $v\in\Z^{d}$. Let $W$ be a subset of $[p]^{d}$ of cardinality at least $\d p^{d}$. Suppose that for all $w\in W$ and $n\in\Z^{d}$ with $(nA)\cdot w\in p\Z$, we have that
    $$(nB+v)\cdot w\in p\Z.$$
     Then $v\in p\Z$ and $B=cA+pB_{0}$ for some $c\in\Z$ and some $d\times d$ integer valued matrix $B_{0}$.
\end{coro}

\subsection{The $p$-expansion trick and applications}\label{1:s:52}


        

We also need to extend all the results in Section \ref{1:s:lltt} from the $\Z/p$-setting  to the $\Q$-setting, and thus cover all rational polynomials. 

 	Let $M\colon\Z^{d}\to\Z/p$ be a quadratic form. Note that $\poly(V_{p}(M)\to \R\vert\Z)$ is the set of polynomials $f\in\poly(\Z^{d}\to \R)$ such that $M(n)\in\Z\Rightarrow f(n)\in\Z$. 
We first provide an explicit description for $\poly(V_{p}(M)\to \R\vert\Z)$ as well as some more general sets.

	\begin{prop}\label{1:basicpp12}
		Let  $d\in\N_{+}$, $k,s\in\N$, $p\gg_{d,k,s} 1$ be a prime number, 
 	 $M\colon\Z^{d}\to\Z/p$ be a non-degenerate quadratic form induced by some $M'\colon\V\to\F_{p}$ associated with the matrix $A$, $c\in\V$, and $V$ be a subspace of $\V$ of dimension $k$ with a basis $\iota(h_{1}),\dots,\iota(h_{k})$ for some $h_{1},\dots,h_{k}\in\Z^{d}$. Denote  
	 $$L_{j}(n):=\frac{1}{p}(h_{j}\tau(A))\cdot (n-\tau(c))$$
		for $1\leq j\leq k$.	Suppose that $\rank(M'\vert_{V^{\perp_{M'}}})\geq 3$. 
 		The followings are equivalent:
		\begin{enumerate}[(i)]
		\item $f$ belongs to $\poly(V_{p}(M)\cap \iota^{-1}(V^{\pp}+c)\to \R\vert\Z)$ and is of degree at most $s$;
		\item we have
		\begin{equation}\label{1:fvpp}
		Q_{0}f=\sum_{i=(i_{0},\dots,i_{k})\in\N^{k}\colon 2i_{0}+i_{1}+\dots+i_{k}\leq s}R_{i}M^{i_{0}}\prod_{j=1}^{k}L_{j}^{i_{j}}
		\end{equation}
	for some  $Q_{0}\in\N_{+}, Q_{0}\leq O_{d,s}(1)$ and some integer valued polynomials $R_{i}$ of degree at most $\deg(f)-(2i_{0}+i_{1}+\dots+i_{k})$, and that $p^{\deg(f)}f$ is an integer valued polynomial.  
		\end{enumerate}
	Moreover, if all the coefficients of $f$ are in $\Z/p^{r}$ for some $r\in\N$, then we may  write
	\begin{equation}\nonumber
		f=\sum_{i=(i_{0},\dots,i_{k})\in\N^{k}\colon 2i_{0}+i_{1}+\dots+i_{k}\leq s, i_{0}+i_{1}+\dots+i_{k}\leq \min\{r,s\}}R_{i}M^{i_{0}}\prod_{j=1}^{k}L_{j}^{i_{j}}
		\end{equation}
	for  some integer coefficient polynomials $R_{i}$ of degree at most $\deg(f)-(2i_{0}+i_{1}+\dots+i_{k})$.
			
		 In particular, the conclusion of this proposition holds if $d\geq k+\dim(V\cap V^{\pp})+3$, or if $d\geq 2k+3$.
	\end{prop}	
	
	The approach of the proof of Proposition \ref{1:basicpp12} is to decompose a $\Q$-valued polynomial as the sum of $\Z/p^{i}$-valued polynomials, and then repeatedly use results for $\Z/p$-valued polynomials to deduce similar results for the $\Q$-valued polynomial $f$.
		To be more precise, for any polynomial $f$ with rational coefficients, we may write $Qf=\sum_{i=0}^{k}\frac{f_{i}}{p^{i}}$ for some $Q\in\Z, p\nmid Q$ and some ``good" polynomials $f_{i}$. Here ``good" means that $f_{i}$ behaves very similar to a $\Z/p$-valued polynomial. One can then prove that periodicity conditions on $f$ descents to similar conditions on $f_{k}$, leading to the conclusion that $\frac{f_{k}}{p}$ is ``good". We may then absorb the term $\frac{f_{k}}{p^{k}}$ by $\frac{f_{k-1}}{p^{k-1}}$ and reduce $k$ by 1. Inductively, we have $k=0$, which implies that $Qf$ itself is ``good". For convenience we refer to this approach as the \emph{$p$-expansion trick}.
	
	\begin{proof}[Proof of Proposition \ref{1:basicpp12}] 	
	If  (ii) holds, then for all $n\in V_{p}(M)\cap \iota^{-1}(V^{\pp}+c)$, we have that $M(n), L_{1}(n),\dots,L_{k}(n)\in\Z$ and thus $f(n)\in(\Z/Q_{0})\cap (\Z/p^{\deg(f)})=\Z$. So
	$f$ belongs to $\poly(V_{p}(M)\cap \iota^{-1}(V^{\pp}+c)\to \R\vert\Z)$.

	 	We now assume that (i) holds. 	
		We a slight abuse of notation, the matrices associated with both $M$ and $M'$ will be denoted by $A$.
		Assume  that $$M(n)=\frac{1}{p}((n\tau(A))\cdot n+u\cdot n+v)$$
		for some 
		$u\in\Z^{d}$ and $v\in\Z$. 	
		Fix any $n\in V_{p}(M)\cap \iota^{-1}(V^{\pp}+c)$. Then $n+pm\in V_{p}(M)\cap \iota^{-1}(V^{\pp}+c)$ for all $m\in \Z^{d}$. So $f(n+pm)\in\Z$ for all $m\in \Z^{d}$. By the multivariate polynomial interpolation, it is not hard to see that $f$ takes values in $\frac{1}{Qp^{s'}}\Z$ for some $Q\in\mathbb{N}_{+}, Q\leq O_{d,s}(1)$.
		So there exists $Q_{0}\in\mathbb{N}_{+}, Q_{0}\leq O_{d,s}(1)$ and some minimal $0\leq s'\leq s$ such that  
		$$Q_{0}f=\sum_{i=0}^{s'}\frac{f_{i}}{p^{i}}$$
		for some integer valued polynomials $f_{i}\colon\Z^{d}\to\mathbb{Z}$ of degree at most $s$.

		 We say that $g$ is a \emph{good} polynomial of \emph{level $L$} if 
		$$g=\sum_{i=(i_{0},\dots,i_{k})\in\N^{k}\colon 2i_{0}+i_{1}+\dots+i_{k}\leq s, i_{0}+\dots+i_{k}\leq L}R_{i}M^{i_{0}}\prod_{j=1}^{k}L_{j}^{i_{j}}$$
	for some  integer valued polynomials $R_{i}$ of degree at most $\deg(f)-(2i_{0}+i_{1}+\dots+i_{k})$.
		Let $k\in\N$ be the smallest integer such that we can write $Q_{0}f$ as 
		$$Q_{0}f=\sum_{i=0}^{k}\frac{f_{i}}{p^{i}}$$
		for some good polynomial $f_{i}\colon\Z^{d}\to\mathbb{Z}$ of degree at most $s$ and level at most $s'-i$. 
		Obviously such $k$ exists and is at most $s'$. Our goal is to show that $k=0$.	
		
		Suppose that 
		$k>0$. Since $n\in V_{p}(M)\cap \iota^{-1}(V^{\pp}+c)\Rightarrow f(n)\in\mathbb{Z}$, we have that
		$$n\in V_{p}(M)\cap \iota^{-1}(V^{\pp}+c)\Rightarrow f_{k}(n)\in p\Z.$$ 
		Suppose that 
		$$f_{k}=\sum_{i=(i_{0},\dots,i_{k})\in\N^{k}\colon 2i_{0}+i_{1}+\dots+i_{k}\leq s, i_{0}+\dots+i_{k}\leq s'-k}R_{i}M^{i_{0}}\prod_{j=1}^{k}L_{j}^{i_{j}}$$
	for some  integer valued polynomials $R_{i}$ of degree at most $\deg(f)-(2i_{0}+i_{1}+\dots+i_{k})$.
			
		\textbf{Claim.}  
		For any $n\in V_{p}(M)\cap \iota^{-1}(V^{\pp}+c)$ with $2\iota(n)A+\iota(u),\iota(h_{1})A,\dots,\iota(h_{k})A$ being linearly independent,
		we have that $R_{i}(n)\in p\mathbb{Z}$ for all $i$.

		Fix $n\in V_{p}(M)\cap \iota^{-1}(V^{\pp}+c)$ with  $2\iota(n)A+\iota(u),\iota(h_{1})A,\dots,\iota(h_{k})A$  being linearly independent. Note that $n+pm\in V_{p}(M)\cap \iota^{-1}(V^{\pp}+c)$ for all $m\in\Z^{d}$. So we have that $f_{k}(n+pm)\in p\Z$ for all $m\in\Z^{d}$. Since
		\begin{equation}\nonumber
		\begin{split}
		M(n+pm)^{i_{0}}\equiv (M(n)+(2(nA)+u)\cdot m)^{i_{0}}\mod p\Z
		\end{split}
		\end{equation} 
		and
		\begin{equation}\nonumber
		\begin{split}
		L_{j}(n+pm)^{i_{j}}\equiv (L_{j}(n)+(h_{j}A)\cdot m)^{i_{j}}\mod p\Z
		\end{split}
		\end{equation} 
		for all $1\leq j\leq k$, the fact that $f_{k}(n+pm)\in p\Z$
		implies that
		\begin{equation}\label{1:vpp}
		\begin{split}
		\sum_{(i_{0},\dots,i_{k})\in\N^{k}\colon 2i_{0}+i_{1}+\dots+i_{k}\leq s, i_{0}+\dots+i_{k}\leq s'-k}R_{i}(n)(M(n)+(2(nA)+u)\cdot m)^{i_{0}}\prod_{j=1}^{k}(L_{j}(n)+(h_{j}A)\cdot m)^{i_{j}}\in p\Z.
		\end{split}
		\end{equation} 

		Since $2\iota(n)A+\iota(u),\iota(h_{1})A,\dots,\iota(h_{k})A$  are linearly independent, the map 
		$$m\mapsto (2(n\tau(A))+u)\cdot m,(h_{1}\tau(A))\cdot m,\dots,(h_{k}\tau(A))\cdot m) \mod p\Z^{k+1}$$
		is an injection from $\{0,\dots,p-1\}^{d}$ to $\{0,\dots,p-1\}^{k+1}$.		
	  It then follows from (\ref{1:vpp}) that 
	  \begin{equation}\label{1:vpp56}
		\begin{split}
		\sum_{i=(i_{0},\dots,i_{k})\in\N^{k}\colon 2i_{0}+i_{1}+\dots+i_{k}\leq s,i_{0}+\dots+i_{k}\leq s'-k}R_{i}(n)x_{0}^{i_{0}}\prod_{j=1}^{k}x_{j}^{i_{j}}\in p\Z.
		\end{split}
		\end{equation} 
		for all $x_{0},\dots,x_{k+1}\in\Z$. viewing (\ref{1:vpp56}) as a polynomial in $x_{0},\dots,x_{k}$, we have that 
		$QR_{i}(n)\in p\Z$ for all $i$ for some $Q\in\mathbb{N}_{+}, Q\leq O_{d,s}(1)$. Since $R_{i}(n)\in\Z$, we deduce that $R_{i}(n)\in (\frac{p}{Q}\Z)\cap \Z=p\Z$.  This proves the claim.

\

 Let $W$ be the set of $n\in V(M')\cap\iota^{-1}(V^{\pp}+c)$ such that  $2\iota(n)A+\iota(u),\iota(h_{1})A,\dots,\iota(h_{k})A$ are linearly dependent. Since $\iota(h_{1}),\dots,\iota(h_{k})$ are linearly independent and since $M'$ is non-degenerate,
  $2\iota(n)A+\iota(u),\iota(h_{1})A,\dots,\iota(h_{k})A$ are linearly dependent if and only if $\iota(n)A+\iota(u)A^{-1}/2$ belongs to the span of $\iota(h_{1}),\dots,\iota(h_{k})$. Recall that $V$ is the span of $\iota(h_{1}),\dots,\iota(h_{k})$. We see that there exist $w\in\V$ such that $W\subseteq (V\cap V^{\perp_{M'}})+w$. So by Proposition \ref{1:iissoo}, 
 $$\vert W\vert\leq p^{\dim(V^{\perp_{M'}})-\rank(M\vert_{V^{\perp_{M'}}})}\leq p^{d-k-3}.$$  
	 We may  invoke  Corollary \ref{1:noloop32} from the claim to conclude that 
	 \begin{equation}\nonumber
 	\frac{1}{p}R_{i}=R'_{i}+MR_{i,0}+\sum_{j=1}^{k}L_{j}R_{i,j}
 	\end{equation}
 	for some  integer valued  polynomials $R_{i,0},\dots,R_{i,k}\in \poly_{p}(\Z^{d}\to\Z)$ with $\deg(R_{i,0})\leq \deg(f)-2$ and $\deg(R_{i,j})\leq \deg(f)-1, 1\leq j\leq k$, and for some $R'_{i}\in  \poly_{p}(\Z^{d}\to\Z)$ with $\deg(R'_{i})\leq \deg(f)$.
	  This implies that $\frac{1}{p}f_{k}$
		is a good polynomial of degree at most $s$ and of level at most $s'-(k-1)$, which contradicts to the minimality of $k$. We have thus proved that $k=0$ and thus $Q_{0}f$ is a good polynomial of level at most $s'$.
		
		Therefore, we may write $Q_{0}f=\frac{1}{p^{\deg(f)}}g$ for some integer valued polynomial $g$.  To show that 
		(ii) holds, it suffices to show that $g(n)\in Q_{0}\Z$ for all $n\in\Z^{d}$.
		Fix any $n_{0}\in V_{p}(M)^{h_{1},\dots,h_{k}}$. For any $n\in\Z^{d}$, there exists $m\in\Z^{d}$ such that $n_{0}+pm\equiv n \mod (s!)^{d}Q_{0}\Z^{d}$. 
		Since $g$ is integer valued, we have that
		$g(n)\equiv g(n_{0}+pm) \mod Q_{0}\Z$. On the other hand, since $n_{0}+pm\in V_{p}(M)\cap \iota^{-1}(V^{\pp}+c)$, we have that $f(n_{0}+pm)=g(n_{0}+pm)/Q_{0}p^{\deg(f)}\in\Z$. So $g(n)\equiv g(n_{0}+pm)\equiv 0 \mod Q_{0}\Z$. 
		
	 	If in addition  all the coefficients of $f$ are in $\Z/p^{r}$ for some $r\in\N$, then we may set $s'=\min\{r,s\}$ in the previous discussion and deduce that $Q_{0}f$ is a good polynomial of level at most $\min\{r,s\}$. Enlarging $Q_{0}$ if necessary, we may write
		\begin{equation}\nonumber
		Q_{0}f=\sum_{i=(i_{0},\dots,i_{k})\in\N^{k}\colon 2i_{0}+i_{1}+\dots+i_{k}\leq s, i_{0}+i_{1}+\dots+i_{k}\leq \min\{r,s\}}R_{i}M^{i_{0}}\prod_{j=1}^{k}L_{j}^{i_{j}}
		\end{equation}
	for  some integer coefficient polynomials $R_{i}$ of degree at most $\deg(f)-(2i_{0}+i_{1}+\dots+i_{k})$.
	 Let  $Q^{\ast}$ be any integer with $Q^{\ast}Q_{0}\equiv 1 \mod p^{r}\Z$. Then $(Q^{\ast}Q_{0}-1)f$ is an integer coefficient polynomial, and thus the conclusion follows as  $f=Q^{\ast}Q_{0}f+(Q^{\ast}Q_{0}-1)f$ since the term $(Q^{\ast}Q_{0}-1)f$ can be absorbed by $Q^{\ast}R_{(0,\dots,0)}$.
%
%
	\end{proof}
	
	As a special case of Proposition \ref{1:basicpp12}, we have
	
				\begin{coro}\label{1:basicpp1}
		Let $d\in \N_{+},s\in\N$, $p$ be a prime, and $M\colon\Z^{d}\to\Z/p$ be a 	quadratic form of $p$-rank at least 3. If $p\gg_{d,s} 1$, then the followings are equivalent: 
		\begin{enumerate}[(i)]
		\item $f$ belongs to $\poly(V_{p}(M)\to \R\vert\Z)$ and is of degree at most $s$;
		\item we have $$Q_{0}f=\sum_{i=0}^{\lfloor s/2\rfloor}M^{i}R_{i}$$
	for some  $Q_{0}\in\N_{+}, Q_{0}\leq O_{d,s}(1)$ and some integer valued polynomials $R_{i}$ of degree at most $\deg(f)-2i$.
		\end{enumerate}
	\end{coro}

 Our next task is to use the $p$-expansion trick to generalize  Corollary \ref{1:att301}.

   \begin{prop}\label{1:att30}
  		Let $d,K,s\in\N_{+}$, $k\in\N$, $\d>0$, $p\gg_{d,k} \d^{-O_{d,k}(1)}$ be a prime, $M\colon\Z^{d}\to\Z/p$ be a quadratic form of $p$-rank at least $s^{2}+s+3$, and $W$ be a subset of $\Gow_{p,s}(V(M))\cap([p^{2}K]^{d})^{s+1}$ of cardinality at least $\d \vert\Gow_{p,s}(V(M))\cap([p^{2}K]^{d})^{s+1}\vert$. Let $P,Q\in\poly(\Z^{d}\to \Q)$ with $\deg(P)\leq k-1$ and $\deg(Q)\leq k$. Suppose that for all $(n,h_{1},\dots,h_{s})\in W+p^{2}K(\Z^{d})^{s+1}$, we have that 		
		$$\Delta_{h_{s-1}}\dots\Delta_{h_{1}}P(n)+\Delta_{h_{s}}\dots\Delta_{h_{1}}Q(n)\in\Z,$$ (where $\Delta_{h_{s}}\dots\Delta_{h_{2}}P(n)$ is understood as $P(n)$ when $s=1$). Then $$qP=P_{1}+P_{2} \text{ and } qQ=Q_{1}+Q_{2}$$ for some integer $q\in\N_{+}$ with $q\ll_{d,k} \d^{-O_{d,k}(1)}$, some $P_{1},Q_{1}\in\poly(V_{p}(M)\to \R\vert\Z)$ with $\deg(P_{1})\leq k-1$, $\deg(Q_{1})\leq k$, and some $P_{2},Q_{2}\in\poly(\Z^{d}\to \Q)$ with $\deg(P_{2})\leq s-2$, $\deg(Q_{2})\leq s-1$.
 \end{prop}
 \begin{rem}\label{1:rpprt}
 The proof of Proposition \ref{1:att30} involves with a very intricate application of the $p$-expansion trick. However, such a generalization Corollary \ref{1:att301} is necessary in the study of equidistribution properties for polynomial sequences (Theorem \ref{1:sLei}), where we need to deal with polynomial sequences which are apriori not necessarily periodic.  See the footnote after (\ref{1:force51}) for details.
 \end{rem}
  	\begin{proof}[Proof of Proposition \ref{1:att30}]
 	Throughout the proof we assume that $p\gg_{d,k} \d^{-O_{d,k}(1)}$.
 	For convenience denote $d':=\rank_{p}(M)\geq s^{2}+s+3$. We assume that $s\leq k+1$ since otherwise one can simply take $P=P_{2}, Q=Q_{2}, q=1$ and $P_{1}=Q_{1}=0$.
 	Since $P,Q$ take values in $\Q$ and are of degree at most $k$, by the multivariate polynomial interpolation,
 	there exists $q\in\mathbb{N}, p\nmid q$ and $t_{0}\in\N$  such that  
 	$$qP=\sum_{j=0}^{t_{0}}\frac{P'_{j}}{p^{j}} \text{ and } qQ=\sum_{j=0}^{t_{0}}\frac{Q'_{j}}{p^{j}}$$
 	for some integer coefficient polynomials $P'_{j},Q'_{j}\in\poly(\Z^{d}\to\Z)$ of with $\deg(P'_{j})\leq k-1$ and  $\deg(Q'_{j})\leq k$. 
 	We say that a polynomial $f\colon\Z^{d}\to\Z$ is  \emph{good} if 
 	$$f=\sum_{j=0}^{\lfloor \deg(f)/2\rfloor}M^{j}f_{j}$$
 	for some integer coefficient  polynomial $f_{j}\colon\Z^{d}\to\mathbb{Z}$ of degree at most  $\deg(f)-2j$.
 	Clearly $P'_{0},\dots,P'_{t_{0}},Q'_{0},\dots,Q'_{t_{0}}$ are good polynomials.
 	
 	Let $t\in\N$ be the smallest integer such that we can write  
 	$$q'P=\sum_{j=0}^{t}\frac{P''_{j}}{p^{j}}+P' \text{ and } q'Q=\sum_{j=0}^{t}\frac{Q''_{j}}{p^{j}}+Q'$$
 	for some $q'\in\N, p\nmid q'$, some good polynomials $P''_{j},Q''_{j}\in\poly(\Z^{d}\to\Z)$ of with $\deg(P''_{j})\leq k-1$,  $\deg(Q''_{j})\leq k$ and some $P',Q'\in\poly(\Z^{d}\to\Q)$ of with $\deg(P')\leq k-2$ and  $\deg(Q')\leq k-1$.  Obviously such $t$ exists and is at most $t_{0}$. Our first goal is to show that $t=0$.	
 	
 	Suppose on the contrary that $t>0$.
 	We consider the last terms and write
 	$$P''_{t}=\sum_{j=0}^{\lfloor (k-1)/2\rfloor}M^{j}P_{j} \text{ and } Q''_{t}=\sum_{j=0}^{\lfloor k/2\rfloor}M^{j}Q_{j}$$
 	for some integer coefficient  polynomials $P_{j},Q_{j}\in\poly(\Z^{d}\to\Z)$ with $\deg(P_{j})\leq k-1-2j$ and $\deg(Q_{j})\leq k-2j$.
	We say that a polynomial $g\in\poly(\Z^{d}\to\Z)$ is \emph{simple} if
	$$q'g=pg_{0}+pMg_{1}+g_{2}$$ for some $q'\in\N_{+}, q'\leq O_{k,s}(1)$, some $g_{0}\in\poly(\Z^{d}\to \Z)$ with $\deg(g_{0})\leq \deg(g)$, some 
 		 integer coefficient polynomials $g_{1}\in\poly(\Z^{d}\to \Z)$ with $\deg(g_{0})\leq \deg(g)-2$, and some $g_{2}\in\poly_{p}(\Z^{d}\to \Z/p\vert\Z)$ with $\deg(g_{0})\leq \deg(g)-1$. If we can show that all of $P_{0},\dots,P_{\lfloor (k-1)/2\rfloor},Q_{0},\dots,Q_{\lfloor k/2\rfloor}$ are simple, then enlarging the constant $q'$ if necessary, we may absorb the term $\frac{P''_{t}}{p^{t}}$ by $\frac{P''_{t-1}}{p^{t-1}}+P'$, and the term $\frac{Q''_{t}}{p^{t}}$ by $\frac{Q''_{t-1}}{p^{t-1}}+Q'$, which contradicts to the minimality of $t$.

		 By induction,
	it then suffices to show that if all of $P_{i+1},\dots,P_{\lfloor k/2\rfloor},Q_{i+1},\dots,Q_{\lfloor k/2\rfloor}$ are simple for some $0\leq i\leq \lfloor k/2\rfloor$, then $P_{i}$ and $Q_{i}$ are also simply (for the special case $i=\lfloor k/2\rfloor$, we show that $P_{i}$ and $Q_{i}$ are simply without any induction hypothesis). 	 
	
	Since $P_{i+1},\dots,P_{\lfloor k/2\rfloor},Q_{i+1},\dots,Q_{\lfloor k/2\rfloor}$ are simple, enlarging the constant $q'$ and
	 absorbing the term $\frac{1}{p^{t}}\sum_{j=i+1}^{\lfloor k/2\rfloor}M^{j}P_{j}$ by $\frac{P''_{t-1}}{p^{t-1}}+P'$, and the term $\frac{1}{p^{t}}\sum_{j=i+1}^{\lfloor k/2\rfloor}M^{j}Q_{j}$ by $\frac{Q''_{t-1}}{p^{t-1}}+Q'$ if necessary, we may assume without loss of generality that 
	$$P''_{t}=\sum_{j=0}^{i}M^{j}P_{j} \text{ and } Q''_{t}=\sum_{j=0}^{i}M^{j}Q_{j}.$$

	By the Pigeonhole Principle, there exists a subset $W_{0}$ of $\Gow_{p,s}(V(M))\cap ([p]^{d})^{s+1}$ of cardinality at least $\d\vert \Gow_{p,s}(V(M))\cap ([p]^{d})^{s+1}\vert/2$ and for each $(n,h_{1},\dots,h_{s})\in W_{0}$ a set $W_{0}(n,h_{1},\dots,h_{s})\subseteq ([pK]^{d})^{s+1}$ of cardinality at least $\d (pK)^{d(s+1)}/2$ such that $(n+pm_{0},h_{1}+pm_{1},\dots,h_{s}+pm_{s})\in W$ for all $(n,h_{1},\dots,h_{s})\in W_{0}$ and $(m_{0},\dots,m_{s})\in W_{0}(n,h_{1},\dots,h_{s})$.
	
 	Fix any $(n,h_{1},\dots,h_{s})\in W_{0}$ and $(m_{0},\dots,m_{s})\in W_{0}(n,h_{1},\dots,h_{s})$.
 	We have $$\Delta_{h_{s-1}+pm_{s-1}}\dots \Delta_{h_{1}+pm_{1}}P(n)+\Delta_{h_{s}+pm_{s}}\dots \Delta_{h_{1}+pm_{1}}Q(n)\in\Z$$ by assumption.  So 
	\begin{equation}\label{1:attee44}
 	\begin{split}
 	\Delta_{h_{s-1}+pm_{s-1}}\dots \Delta_{h_{1}+pm_{1}}P''_{t}(n)+\Delta_{h_{s}+pm_{s}}\dots \Delta_{h_{1}+pm_{1}}Q''_{t}(n)\in p\Z.
 	\end{split}
 	\end{equation}

Assume that 
$$M(n)=\frac{1}{p}((nA)\cdot n+u\cdot n+v)$$
for some $d\times d$ symmetric integer coefficient matrix $A$, some $u\in\Z^{d}$ and some $v\in\Z$. Then for all $m,n\in V_{p}(M)$, we have 
\begin{equation}\label{1:attee444}
 	\begin{split}
 	M(n+pm)\equiv M(n)+((2(nA)+u)\cdot m) \mod p\Z.
 	\end{split}
 	\end{equation}

	For any $r\in\N_{+}$ $\e=(\e_{1},\dots,\e_{r})\in\{0,1\}^{r}$ and $a=(a_{1},\dots,a_{r})\in(\Z^{d})^{r}$, denote 
	$$\e\cdot a:=\e_{1}a_{1}+\dots+\e_{r}a_{r}\in \Z^{d}.$$
	Since  $P_{j}$ and $Q_{j}$ have integer coefficients, it is not hard to see from (\ref{1:attee44}) and (\ref{1:attee444}) that any $(n,h_{1},\dots,h_{s})\in W_{0}$ and $(m_{0},\dots,m_{s})\in W_{0}(n,h_{1},\dots,h_{s})$, we have
 	\begin{equation}\label{1:attee4}
 	\begin{split}
 	&\frac{1}{p}\sum_{j=0}^{i}\sum_{\e_{\ast}\in\{0,1\}^{s-1}}(-1)^{\vert\e_{\ast}\vert}(M(n'+\e_{\ast}\cdot h_{\ast})+(2(n'+\e_{\ast}\cdot h_{\ast})A+u)\cdot (\e_{\ast}\cdot m_{\ast}))^{j}
	P_{j}(n'+\e_{\ast}\cdot h_{\ast})
	\\&\qquad+\frac{1}{p}\sum_{j=0}^{i}\sum_{\e\in\{0,1\}^{s}}(-1)^{\vert\e\vert}(M(n'+\e\cdot h)+(2(n'+\e\cdot h)A+u)\cdot (\e\cdot m))^{j}
	Q_{j}(n'+\e\cdot h)\in \Z,
 	\end{split}
 	\end{equation}
  where $n'=n+pm_{0}$, $h=(h_{1},\dots,h_{s})$, $h_{\ast}=(h_{1},\dots,h_{s-1})$,  $m=(m_{1},\dots,m_{s})$, and $m_{\ast}=(m_{1},\dots,m_{s-1})$, and the terms $\e_{\ast}, h_{\ast},m_{\ast}$ are regarded as non-existing if $s=1$. 
  Since $P_{i}$ and $Q_{i}$ have integer coefficients, the residue class mod $\Z$ of the left hand side of (\ref{1:attee4}) remains unchanged if we replace $m_{i}$ by $m_{i} \mod p\Z$ for $0\leq i\leq s$. 
  In order words, for fixed $(n,h_{1},\dots,h_{s})\in W_{0}$, the left hand side of (\ref{1:attee4}) is a $p$-periodic polynomial in the variables $(m_{0},\dots,m_{s})$.
 Since $W_{0}(n,h_{1},\dots,h_{s}) \mod p(\Z^{d})^{s+1}$ contains at least $\d p^{d(s+1)}/2$ residue classes mod $p\Z$, it follows from Lemma \ref{1:ns} that (\ref{1:attee4}) holds for  all $m_{0},\dots,m_{s}\in \Z^{d}$ and for all $(n,h_{1},\dots,h_{s})\in W_{0}$ (and thus for all $(n,h_{1},\dots,h_{s})\in W_{0}+p(\Z^{d})^{s+1}$).

	 Suppose that $2nA+u,h_{1}A,\dots,h_{s}A$ are $p$-linearly independent. It is not hard to see that for any   $1\leq j\leq i$ and $c_{j}\in\Z^{d}$, there exists $m_{j}\in\Z^{d}$ such that $$(2nA+u)\cdot m_{j}, (h_{j'}A)\cdot m_{j}, (h_{i+j}A)\cdot m_{j}-c_{j}\in p\Z$$ for all  $1\leq j'\leq s, j'\neq i+j$.
	 	For $i+1\leq j\leq s$, let $m_{j}=\bold{0}$.
		 
 	Then for all $0\leq j\leq i$, and $\e_{\ast}\in\{0,1\}^{s-1}$, we have that
	\begin{equation}\label{1:aatt11}
 	\begin{split}
 	&\quad 
	(M(n+\e_{\ast}\cdot h_{\ast})+(2(n+\e_{\ast}\cdot h_{\ast})A+u)\cdot (\e_{\ast}\cdot m_{\ast}))^{j}P_{j}(n+\e_{\ast}\cdot h_{\ast}) 
	\\&=
	(M(n+\e_{\ast}\cdot h_{\ast})+(2(n+\e'\cdot h'+\e''_{\ast}\cdot h''_{\ast})A+u)\cdot (\e'\cdot m'))^{j}P_{j}(n+\e_{\ast}\cdot h_{\ast}) 
 	\\&\equiv\Bigl(M(n+\e_{\ast}\cdot h_{\ast})+2\sum_{\ell=1}^{i}\e_{\ell}\e_{i+\ell}c_{\ell}\Bigr)^{j}P_{j}(n+\e_{\ast}\cdot h_{\ast}) \mod p\Z,
 	\end{split}
 	\end{equation}
	where $h':=(h_{1},\dots,h_{i})$, $h''_{\ast}:=(h_{i+1},\dots,h_{s-1})$, $\e':=(\e_{1},\dots,\e_{i})$ and $\e''_{\ast}:=(\e_{i+1},\dots,\e_{s-1})$.
	Similarly, for all  $0\leq j\leq i$, and $\e\in\{0,1\}^{s}$, we have that
	\begin{equation}\label{1:aatt12}
 	\begin{split}
 	&\quad 
	(M(n+\e\cdot h)+(2(n+\e\cdot h)A+u)\cdot (\e\cdot m))^{j}Q_{j}(n+\e\cdot h) 
 	\\&\equiv\Bigl(M(n+\e\cdot h)+2\sum_{\ell=1}^{i}\e_{\ell}\e_{i+\ell}c_{\ell}\Bigr)^{j}Q_{j}(n+\e\cdot h) \mod p\Z.
 	\end{split}
 	\end{equation}	
 	Combining (\ref{1:attee4}), (\ref{1:aatt11}) and (\ref{1:aatt12}), we have that
	\begin{equation}\label{1:aatt1}
 	\begin{split}
 	&\sum_{j=0}^{i}\sum_{\e_{\ast}\in\{0,1\}^{s-1}}(-1)^{\vert\e_{\ast}\vert} \Bigl(M(n+\e_{\ast}\cdot h_{\ast})+2\sum_{\ell=1}^{i}\e_{\ell}\e_{i+\ell}c_{\ell}\Bigr)^{j}P_{j}(n+\e_{\ast}\cdot h_{\ast})	
	\\&+\sum_{j=0}^{i}\sum_{\e\in\{0,1\}^{s}}(-1)^{\vert\e\vert} \Bigl(M(n+\e\cdot h)+2\sum_{\ell=1}^{i}\e_{\ell}\e_{i+\ell}c_{\ell}\Bigr)^{j}Q_{j}(n+\e\cdot h)\in p\Z	
	\end{split}
 	\end{equation}	
	for all $c_{1},\dots,c_{j}\in\Z^{d}$.
 	Viewing   (\ref{1:aatt1}) as a polynomial of $c_{1},\dots,c_{i}$ and considering the coefficient of the term $c_{1}\cdot\ldots\cdot c_{i}$,  by interpolation, we have that
 	\begin{equation}\nonumber
 	\begin{split}
	&	q''\sum_{\e'''_{\ast}\in\{0,1\}^{s-2i-1}}(-1)^{\vert\e'''_{\ast}\vert} P_{i}(n+h_{1}+\dots+h_{2i}+\e'''_{\ast}\cdot h'''_{\ast})
 	\\&+q''\sum_{\e'''\in\{0,1\}^{s-2i}}(-1)^{\vert\e'''\vert} Q_{i}(n+h_{1}+\dots+h_{2i}+\e'''\cdot h''')\in p\Z
 	\end{split}
 	\end{equation}	
	for some $q''\in\N_{+}, q''\leq O_{d,k}(1)$,
	where $h'''=(h_{2i+1},\dots,h_{s}),$ $h'''_{\ast}=(h_{2i+1},\dots,h_{s-1})$, $\e'''_{\ast}=(\e_{2i+1},\dots,\e_{s-1})$ and  $\e'''=(\e_{2i+1},\dots,\e_{s})$, and the sum $$\sum_{\e\in\{0,1\}^{r}}(-1)^{\vert\e\vert}a(n+\e\cdot h)$$ is considered as $a(n)$ if $r=0$, and as $0$ if $r<0$. 
	In conclusion, we have that
 	\begin{equation}\label{1:attee5}
 	\begin{split}	
 	q''(\Delta_{h_{s-1}}\dots\Delta_{h_{2i+1}}P_{i}(n')+\Delta_{h_{s}}\dots\Delta_{h_{2i+1}}Q_{i}(n')) \in p\Z
 	\end{split}
 	\end{equation}	
	with $n'=n+h_{1}+\dots+h_{2i}$
 	for all $(n,h_{1},\dots,h_{s})\in W_{0}+p(\Z^{d})^{s+1}$ such  that $2nA+u,h_{1}A,\dots,h_{s}A$ are $p$-linearly independent. 
	
	It is convenient to work in the $\V$-setting in the rest of the proof.
	Let $\tilde{M}\colon\V\to\F_{p}$ be the non-degenerate quadratic form induced by $M$, and let $W'$ denote the set of $(n',h_{2i+1},\dots,h_{s})\in (\V)^{s-2i+1}$ such that $n'=n+h_{1}+\dots+h_{2i}$ for some $(n,h_{1},\dots,h_{s})\in \iota(W_{0})$ such  that $2n\iota(A)+\iota(u),h_{1}\iota(A),\dots,h_{s}\iota(A)$ are linearly independent.
	Then it is not hard to see that (\ref{1:attee5}) holds for all $(n',h_{2i+1},\dots,h_{s})\in \iota^{-1}(W')$.  If we can show that $\vert W'\vert\geq \d \vert \Gow_{s-2i}(V(\tilde{M}))\vert/8$, then it follows from (\ref{1:attee5}) and Corollary \ref{1:att301} that $P_{i}$ and $Q_{i}$ are simply.

	We now estimate the cardinality of $W'$.
	Since $d'\geq s^{2}+s+3$, it follows from Proposition \ref{1:ctsp} that 
	\begin{equation}\label{1:temps0}
	\vert\Gow_{j}(V(\tilde{M}))\vert=p^{(j+1)d-(\frac{j(j+1)}{2}+1)}(1+O_{j}(p^{-1/2}))
	\end{equation}
	  for all $0\leq j\leq s$. 
	For convenience denote $\x:=(n',h_{2i+1},\dots,h_{s})$. We say that $\x$ is \emph{good} if $2n'\iota(A)+\iota(u),h_{2i+1}\iota(A),\dots,h_{s}\iota(A)$ are linearly independent. Let $W'_{\x}$ denote the set of $(h_{1},\dots,h_{2i})\in(\V)^{2i}$ such that $(n'-h_{1}-\dots-h_{2i},h_{1},\dots,h_{s})\in \iota(W)$. Then
	\begin{equation}\label{1:temps1}
	\sum_{\x\in \Gow_{s-2i}(V(\tilde{M}))}\vert W'_{\x}\vert=\vert W_{0}\vert\geq \d \vert\Gow_{s}(V(\tilde{M}))\vert/2=\frac{1}{2}\d p^{(s+1)d-(\frac{s(s+1)}{2}+1)}(1+O_{s}(p^{-1/2})).
	\end{equation}
	Therefore, informally speaking, by (\ref{1:temps0}) and  (\ref{1:temps1}), on average the cardinality of $W'_{\x}$ is at least $\d p^{2id-(2s-2i+1)i}/2$.
	
 Note that for good $\x$, $(h_{1},\dots,h_{2i})\in W'_{\x}$ only if $(n'-h_{1}-\dots-h_{2i},h_{1},\dots,h_{s})\in \Gow_{s}(V(\tilde{M}))$. Using Lemma \ref{1:changeh}, this is equivalent of saying that
 \begin{itemize}
    \item $\tilde{M}(n')-\tilde{M}(n'-h_{j})=(2n'\iota(A)+\iota(u))\cdot h_{j}=0$ for all $1\leq j\leq 2i$;
    \item $(h_{j}\iota(A))\cdot h_{j'}=0$ for all $1\leq j,j'\leq s$ with at least one of $j,j'$ belonging to $\{1,\dots,2i\}$.  
 \end{itemize}
 There are in total $(2s-2i+1)i\leq (s+1/2)^{2}/2$ independent equations above in total.\footnote{Using the terminology in Appendix \ref{1:s:AppB}, this means that $W'_{\x}$ is a subset of an $\tilde{M}$-set of total co-dimension $(2s-2i+1)i$.} Since $d\geq s^{2}+s+3$, by repeatedly using Corollary \ref{1:counting02} and using the fact that $\x$ is good, it is not difficult to show that 
 \begin{equation}\label{1:temps2}
	\vert W'_{\x}\vert\leq p^{2id-(2s-2i+1)i}(1+O_{s}(p^{-1/2})).
\end{equation} 
Here we will not present the proof for (\ref{1:temps2}), as this is a special case of Theorem \ref{1:ct}.

 Let $V:=\{n\iota(A)\colon n\in\V\}\subseteq \V$. Then $V$ is a subspace of $\V$ of dimension $d'$. Let $V'$ be the set of $m,m_{2i+1},\dots,m_{s}\in V$ such that $2m+\iota(u),m_{2i+1},\dots,m_{s}$ are linearly dependent. Similar to the proof of Lemma \ref{1:iiddpp} (i), it is not hard to see the $\vert V'\vert\leq sp^{(d+1)(s-2i)}$. On the other hand, for each $m\in V$, the number of $n\in\V$ such that $n\iota(A)=m$ is at most $p^{d-d'}$. Therefore, the number of $\x\in(\V)^{s-2i+1}$ which is not good is at most 
 $p^{d-d'}\vert V'\vert\leq sp^{(s-2i+1)d+(s-2i)-d'}.$ So
 \begin{equation}\label{1:temps3}
	\sum_{\x\in (\V)^{s-2i+1}, \text{ $\x$ is not good}}\vert W'_{\x}\vert\leq sp^{(s-2i+1)d+(s-2i)-d'}\cdot p^{2id}=sp^{(s+1)d+(s-2i)-d'}.
	\end{equation}
 Since $(s+1)d+(s-2i)-d'<(s+1)d-(\frac{s(s+1)}{2}+1)$, it follows from (\ref{1:temps0}), (\ref{1:temps1}), (\ref{1:temps2}), (\ref{1:temps3}) and the Pigeonhole Principle that 
 there exists a subset $W''\subseteq \Gow_{s-2i}(V(\tilde{M}))$ of cardinality at least $\d \vert \Gow_{s-2i}(V(\tilde{M}))\vert/20$ such that $\vert W'_{\x}\vert\geq \d p^{2id-(2s-2i+1)i}/20$ for all $\x\in W''$. In particular, $W'_{\x}$ is non-empty for all $\x\in W''$.
  Therefore, $W''$ is a subset of $W'$ and thus  
  $\vert W'\vert\geq \vert W''\vert\geq \d \vert \Gow_{s-2i}(V(\tilde{M}))\vert/20$. 
  
  \

  In conclusion, we have shown by induction that $t=0$ and therefore we may write 
  $$qP=P'_{0} \text{ and } qQ=Q'_{0}$$
  for some good polynomials $P'_{0},Q'_{0}$ of degrees at most $k-1$ and $k$ respectively.
Our final goal is to show that one can further require $q$ to be small.

    Let $q^{\ast}\in\N_{+}$ be such that $qq^{\ast}\equiv 1 \mod p^{k}\Z$. Then  
    $P=\frac{qq^{\ast}-1}{q}P'_{0}+q^{\ast}P'_{0}$.
    Denote $P_{0}:=(qq^{\ast}-1)P'_{0}$. Since $P'_{0}$ is good, we have that $P_{0}$ has coefficients in $\Z$ 
and that  $\Delta_{h_{s-1}}\dots\Delta_{h_{1}}(q^{\ast}P'_{0})(n)\in\Z$ for all $(n,h_{1},\dots,h_{s})\in \Gow_{p,s}(V(M))$.
Similarly, $Q_{0}:=(qq^{\ast}-1)Q'_{0}$ also has coefficients in $\Z$, and $\Delta_{h_{s}}\dots\Delta_{h_{1}}(q^{\ast}Q'_{0})(n)\in\Z$ for all $(n,h_{1},\dots,h_{s})\in \Gow_{p,s}(V(M))$.
      So we have that
      $$\Delta_{h_{s-1}}\dots\Delta_{h_{1}}P_{0}(n)+\Delta_{h_{s}}\dots\Delta_{h_{1}}Q_{0}(n)\in q\Z$$
for all $(n,h_{1},\dots,h_{s})\in W+([pK']^{d})^{s+1}$, where for convenience we denote $K':=pK$.

    By the Pigeonhole Principle, there exists some $(n,h_{1},\dots,h_{s})\in \Gow_{p,s}(V(M))\cap([p]^{d})^{s+1}$ and a subset $J\subseteq ([K']^{d})^{s+1}$ of cardinality at least $\d{ K'}^{d(s+1)}$ such that $(n+pm_{0},h_{1}+pm_{1},\dots,h_{s}+pm_{s})\in W$ for all $(m_{0},\dots,m_{s})\in J$. 
      Write 
    $$G(m_{0},m_{1},\dots,m_{s}):=\Delta_{h_{s-1}+pm_{s-1}}\dots\Delta_{h_{1}+pm_{1}}P_{0}(n+pm_{0})+\Delta_{h_{s}+pm_{s}}\dots\Delta_{h_{1}+pm_{1}}Q_{0}(n+pm_{0}).$$
    Then $G$ is a rational polynomial of degree at most $k$.  
    By assumption,
    $\frac{1}{q}G(\m)\in \Z$ for all $\m\in J+([K']^{d})^{s+1}$.  
    So by Proposition \ref{1:polyGT}, there exists $r\in\N_{+}$ with $r\ll_{d,k,s} \d^{-O_{d,k,s}(1)}$ such that all the coefficients of $\frac{1}{q}G$ are in $\Z/r$.
        Write
    \begin{equation}\nonumber
    \begin{split}
       &\quad G'(m_{0},m_{1},\dots,m_{s}):=\frac{1}{q}G\Bigl(\frac{1}{p}m_{0}-n,\frac{1}{p}m_{1}-h_{1},\dots,\frac{1}{p}m_{s}-h_{s}\Bigr)
       \\&=\Delta_{m_{s-1}}\dots\Delta_{m_{1}}(\frac{1}{q}P_{0})(m_{0})+\Delta_{m_{s}}\dots\Delta_{m_{1}}(\frac{1}{q}Q_{0})(m_{0}).
    \end{split}
    \end{equation}
    Then  the coefficients of $G'$ are in $\Z/rp^{k}$. On the other hand, since $P_{0}$ and $Q_{0}$ are integer coefficient polynomials, the coefficients of $G'$ are in $\Z/q$ and are thus in $(\Z/rp^{k})\cap(\Z/q)\subseteq \Z/r$. 
    
    Now assume that 
    $$\frac{1}{q}P_{0}(n):=\sum_{i\in\N^{d},\vert i\vert\leq k} a_{i}n^{i} \text{ and } \frac{1}{q}Q_{0}(n):=\sum_{i\in\N^{d},\vert i\vert\leq k} b_{i}n^{i}$$
    for some $a_{i},b_{i}\in\Q$.
    Fix any $i=(i_{1},\dots,i_{d})\in\N^{d}$ with $\vert i\vert\geq s-1$.
    Let $(m_{1},\dots,m_{\vert i\vert+1})\in (\Z^{d})^{\vert i\vert}$ be any vector such that $m_{s}=\bold{0}$
    and that
    the number of $m_{i'}, 1\leq i'\leq \vert i\vert+1, i'\neq s$ which equals to the $j$-th standard unit vector is  equal to $i_{j}$ for all $1\leq j\leq d$. Then it is not hard to see that 
    $$\Delta_{m_{\vert i\vert+1}}\dots\Delta_{m_{s+1}}G'(\bold{0},m_{1},\dots,m_{s})=i!a_{i}.$$
    Since the coefficients of $G'$ are in  $\Z/r$, we have that $a_{i}\in \Z/r'$ for all  $i\in\N^{d}$ with $\vert i\vert\geq s-1$, where $r':=(k!)^{d}r$. 
    
    Now fix any $i=(i_{1},\dots,i_{d})\in\N^{d}$ with $\vert i\vert\geq s$.
    Let $(m_{1},\dots,m_{\vert i\vert})\in (\Z^{d})^{\vert i\vert}$ be any vector such that      the number of $m_{i'}, 1\leq i'\leq \vert i\vert$ which equals to the $j$-th standard unit vector is  equal to $i_{j}$ for all $1\leq j\leq d$. Then it is not hard to see that 
    $$\Delta_{m_{\vert i\vert}}\dots\Delta_{m_{s+1}}G'(\bold{0},m_{1},\dots,m_{s})=i!b_{i}+(i-m_{s})!a_{i-m_{s}}.$$
   Since the coefficients of $G'$ are in  $\Z/r$ and since $a_{i-m_{s}}\in \Z/r'$, we have that  $b_{i}\in \Z/r'$ for all  $i\in\N^{d}$ with $\vert i\vert\geq s$. 
 So we may write $P_{0}$ as $\frac{1}{q}P_{0}=P_{1}+P_{2}$ for some polynomial $P_{1}$ with coefficients in $\Z/r'$ having at most the same degree as $P$ and some $P_{2}$ of degree at most $s-2$. So
 $$r'P=r'(q^{\ast}P'_{0}+P_{1})+r'P_{2}$$
         where $r'(q^{\ast}P'_{0}+P_{1})$ take integer values in $V_{p}(M)$. A similar decomposition apples to $Q$ and so we are done.   
    \end{proof}

  By a change of variable trick, we use Proposition \ref{1:att30} to deduce the follow result:
 
    \begin{prop}\label{1:att3001}
  	  		Let $d,s\in\N_{+}, k\in\N$, 
		$p\gg_{d} 1$ be a prime, $c\in\V$, 
		$V$ be a subspace of $\V$ of dimension $k$ with a basis $\iota(h_{1}),\dots,\iota(h_{k})$ for some $h_{1},\dots,h_{k}\in\Z^{d}$, 
		$M\colon\V\to\F_{p}$ be a non-degenerate quadratic form with $\rank(M\vert_{V^{\pp}})\geq s^{2}+s+3$, 
		and $g\in\poly(\Z^{d}\to \Q)$ with $\deg(g)\leq s$. Suppose that for all $(n_{0},n_{1},\dots,n_{s})\in \iota^{-1}(\Gow_{s}(V(M)\cap (V^{\pp}+c)))$, we have that 		
		$\Delta_{n_{s}}\dots\Delta_{n_{1}}g(n_{0})\in\Z.$ Then $$qg=g_{1}+g_{2}$$ for some integer $q\in\N_{+}$ with $q\ll_{d}  1$, some $g_{1}\in\poly(\iota^{-1}(V(M)\cap (V^{\pp}+c))\to \R\vert\Z)$ with $\deg(g_{1})\leq s$,  and some $g_{2}\in\poly(\Z^{d}\to \Q)$ with $\deg(g_{2})\leq s-1$.
 \end{prop}
  \begin{proof}
  If $k=0$, then the conclusion follows from Proposition \ref{1:att30}. So now we assume that $k\geq 1$.
     Since $\Delta_{h_{s}}\dots\Delta_{h_{1}}g(n)=0$ whenever $\deg(g)<s$. We may assume without loss of generality that $g$ is a homogeneous polynomial of degree $s$. 
     Let $A$ be the matrix associated to $M$, and $\tilde{M}$  be the regular liftings of $M$. 
     Let $M_{0}\colon\V\to\F_{p}$ be the quadratic form given by $M_{0}(n):=(nA)\cdot n$.
  %
      Since $M$ is non-degenerate, $V^{\pp}$ is of dimension $r:=d-k$.
     Let $v_{1},\dots,v_{r}$ be a basis of $V^{\pp}$ 
     and  $\phi\colon  \F_{p}^{r}\to V^{\pp}$ be the bijective linear transformation with $\phi(e_{i}):=v_{i}$ for $1\leq i\leq r$. Let $M'\colon \F_{p}^{r}\times \V\to \F_{p}$ be the quadratic form given by $M'(m,n):=M(\phi(m)+c)$ for all $m\in\F_{p}^{r}$ and $n\in\V$. Suppose that $\phi(m):=mB$ for some $r\times d$ matrix $B$. Then $BAB^{T}$ is the matrix associated to $M'$.
      Pick any $\tilde{v}_{1},\dots,\tilde{v}_{r}\in\Z^{d}$ such that $\iota(\tilde{v}_{i})=v_{i}$ for all $1\leq i\leq r$,  and let $\tilde{\phi}\colon\Z^{r}\to\Z^{d}$ be the linear map given by $\tilde{\phi}(e_{i}):=\tilde{v}_{i}$ for $1\leq i\leq r$.  It is clear that $\iota(\tilde{V})=V$ and $\iota\circ\tilde{\phi}=\phi\circ\iota$.

	
	
	Let $\tilde{c}:=\tau(c)$ and $g'\in\poly(\Z^{r+d}\to\Q)$ be the polynomial defined by $g'(m,n):=g(\tilde{\phi}(m)+pn+\tilde{c})$.
	It is not hard to check that for all $((n_{0}',n_{0}''),(n'_{1},n''_{1}),\dots,(n'_{s},n''_{s}))\in \iota^{-1}(\Gow_{s}(V(M')))$, we have that $(\tilde{\phi}(n_{0}')+pn_{0}''+\tilde{c},\tilde{\phi}(n'_{1})+pn''_{1},\dots,\tilde{\phi}(n'_{s})+pn''_{s})\in \iota^{-1}(\Gow_{s}(V(M)\cap (V+c)))$.
	By assumption, we have that 
	$$\Delta_{(n'_{s},n''_{s})}\dots\Delta_{(n'_{1},n''_{1})}g'(n_{0}',n_{0}'')=\Delta_{\tilde{\phi}(n'_{s})+pn''_{s}}\dots\Delta_{\tilde{\phi}(n'_{1})+pn''_{1}}g(\tilde{\phi}(n_{0}')+pn_{0}''+\tilde{c})\in\Z.$$
	
	Since $\rank(M')=\rank(M\vert_{V^{\pp}})\geq s^{2}+s+3$, 
	by Proposition \ref{1:att30},  we may write
	\begin{equation}\label{1:gqggq}
	qg'=g'_{1}+g'_{2}
	\end{equation}
	  for some integer $q\in\N_{+}$ with $q\ll_{d} 1$, some $g'_{1}\in\poly(\iota^{-1}(V(M'))\to \R\vert\Z)$ with $\deg(g'_{1})\leq s$,  and some $g'_{2}\in\poly(\Z^{d}\to \Q)$ with $\deg(g'_{2})\leq s-1$. 
	  Let  $\tilde{M}'$ be the regular liftings of  $M'$.
	      Since $g'_{1}\in\poly(\iota^{-1}(V(M'))\to \R\vert\Z)=\poly(V_{p}(\tilde{M}')\to \R\vert\Z)$, by Corollary \ref{1:basicpp1}, enlarging $q$ if necessary, we may assume that $g'_{1}$ is of the form
	$$g'_{1}=\sum_{i=0}^{\lfloor s/2\rfloor}R_{i}\tilde{M'}^{i}$$
	for some integer coefficient polynomial $R_{i}$ of degree at most $s-2i$.
	Since $g$ is a homogeneous polynomial of degree $s$, comparing the leading coefficients on both sides of (\ref{1:gqggq}), we have that 
 	\begin{equation}\label{1:gqggq2}
	qg(\tilde{\phi}(m)+pn)=\sum_{i=0}^{\lfloor s/2\rfloor}R'_{i}(m,n)(\frac{1}{p}(m\tau(BAB^{T}))\cdot m)^{i} 
		\end{equation}
	for some integer coefficient homogeneous polynomial $R_{i}$ of degree  $s-2i$. 
	Note that if $n_{\ast}\in \iota^{-1}(V(M_{0})\cap V^{\pp})$, then $\iota(n_{\ast})=\phi(m_{0})$ for some $m_{0}\in \F_{p}^{r}$. Then there exist $(m,n)\in\Z^{r+d}$ such that $n_{\ast}=\tilde{\phi}(m)+pn$ and that $\iota(m)=\phi(m_{0})$. Since $\phi(m_{0})\in V(M_{0})$, we have that $(\phi(m_{0})A)\cdot \phi(m_{0})=(m_{0}BAB^{T})\cdot m_{0}=0$. So  $\frac{1}{p}(m\tau(BAB^{T}))\cdot m\in\Z$. Since $R'_{i}(m,n)\in\Z$, it follows from (\ref{1:gqggq2}) that $qg(n_{\ast})=qg(\tilde{\phi}(m)+pn)\in\Z$.  
%
%
%
   %
   %
   
   In other words, we have that $qg\in\poly(\iota^{-1}(V(M_{0})\cap V^{\pp})\to\R\vert\Z)$.  By Proposition \ref{1:basicpp12}, enlarging $q$ again if necessary, we may write $qf$ as
   \begin{equation}\label{1:fvppr}
		qg=\sum_{i=(i_{0},\dots,i_{k})\in\N^{k}\colon 2i_{0}+i_{1}+\dots+i_{k}\leq s}R_{i}\tilde{M}_{0}^{i_{0}}\prod_{j=1}^{k}L_{0,j}^{i_{j}}
		\end{equation}
   for some integer valued polynomials $R_{i}$ of degree at most $s-(2i_{0}+i_{1}+\dots+i_{k})$, where $\tilde{M}_{0}$ is the regular lifting of $M_{0}$ and
   $$L_{0,j}(n):=\frac{1}{p}(h_{j}\tau(A))\cdot n.$$
   Denote
     $$L_{j}(n):=\frac{1}{p}(h_{j}\tau(A))\cdot (n-\tau(c)).$$
   It then follows from (\ref{1:fvppr}) that we may write $qg$ as
    \begin{equation}\nonumber
		qg:=g_{1}+g_{2}:=\sum_{i=(i_{0},\dots,i_{k})\in\N^{k}\colon 2i_{0}+i_{1}+\dots+i_{k}\leq s}R_{i}M^{i_{0}}\prod_{j=1}^{k}L_{j}^{i_{j}}+g_{2}
		\end{equation}
for some $g_{2}\in\poly(\Z^{d}\to\Q)$ of degree at most $s-1$. Again by Proposition \ref{1:basicpp12}, we have that $g_{1}\in\poly(\iota^{-1}(V(M)\cap (V^{\pp}+c))\to \R\vert\Z)$.
     We are done.
  \end{proof}

As a consequence of Proposition \ref{1:att3001}, we have the following result, which will be used in \cite{SunC,SunD}:
 
 \begin{prop}\label{1:att3}
  		Let $d,s\in\N_{+}, k\in\N$, 
		$p\gg_{d} 1$ be a prime, $c\in\V$, 
		$V$ be a subspace of $\V$ of dimension $k$ with a basis $\iota(h_{1}),\dots,\iota(h_{k})$ for some $h_{1},\dots,h_{k}\in\Z^{d}$, 
		$M\colon\V\to\F_{p}$ be a non-degenerate quadratic form with $\rank(M\vert_{V^{\pp}})\geq s^{2}+s+3$ associated with a matrix $A$. Let $M_{0}\colon\V\to\F_{p}$ be the quadratic form given by $M_{0}(n):=(nA)\cdot n$.
	Then for any  $g\in\poly(\Z^{d}\to \Q)$  of degree at most $s$, the followings are equivalent:
 	\begin{enumerate}[(i)]
 		\item $\Delta_{n_{s}}\dots\Delta_{n_{1}}g(n_{0})\in \frac{1}{q}\Z$ for some $q\in\N,p\nmid q$ for all $(n_{0},n_{1},\dots,n_{s})\in \Gow_{p,s}(\iota^{-1}(V(M)\cap (V^{\pp}+c)))$;
  		\item we have  $$g=\frac{1}{q}g_{1}+g_{2}$$ for some $q\in\N,p\nmid q$,
 		some $g_{1}\in \poly(\iota^{-1}(V(M)\cap (V^{\pp}+c))\to \R\vert\Z)$ of degree at most $s$ and some $g_{2}\in\poly(\Z^{d}\to \Q)$ of degree at most $s-1$;
 		\item we have $$g=\frac{1}{q}g_{1}+g_{2}$$ for some $q\in\N,p\nmid q$,
 		some $g_{1}\in \poly(\iota^{-1}(V(M_{0})\cap V^{\pp})\to \R\vert\Z)$ of degree at most $s$, and some $g_{2}\in\poly(\Z^{d}\to \Q)$ of degree at most $s-1$.
 	\end{enumerate}	 
	 Moreover, if $g$ is a homogeneous polynomial of degree $s$, then we may require $g_{2}=0$ in (iii).
 %
  \end{prop}
 \begin{proof}
 	Clearly (ii) implies (i). Similar to the argument at the end of Proposition \ref{1:att3001}, it follows from Proposition \ref{1:basicpp12} that
	(ii) is equivalent to (iii).  The part (i) $\Rightarrow$ (ii) is a consequence of Proposition \ref{1:att3001}. 
	
	If in addition $g$ is a homogeneous polynomial of degree $s$, then in condition (iii), we may express  $g_{1}$ using Proposition \ref{1:basicpp12} and then compare the leading terms on both sides of condition (iii) to conclude that one can set $g_{2}=0$.
 \end{proof}

\section{Proof of the main equidistribution theorem}\label{1:s:pp9}
  
  \subsection{Some preliminary reductions}\label{1:s:qse}

  We prove Theorem \ref{1:sLei} in this section.
  We start with a technical lemma on the connections between subspaces of $\V$ and those of $\Z^{d}$.  Let $V+c$ be an affine subspace of $\F_{p}^{d}$ of dimension $d'$ and $L\colon\F_{p}^{d'}\to V$ be a bijective linear transformation.  Let $\tilde{V}$ be the $\Z$-span of $\tau(L(e_{1})),\dots,\tau(L(e_{d'}))$ and $\tilde{L}\colon \Z^{d'}\to \tilde{V}$ be the bijective linear transformation given by
  $$\tilde{L}(n_{1},\dots,n_{d'}):=n_{1}\tau(L(e_{1}))+\dots+n_{d'}\tau(L(e_{d'})).$$
  In other words, $\tilde{L}/p$ is the regular lifting of $L$.
Then it is clear that $\iota\circ\tilde{L}=L\circ\iota$ and $\tilde{L}\circ\tau\equiv \tau\circ L \mod p\Z^{d}$.

   Let 
    $B$ be a subset of $V+c$. 
  It is clear that there is a bijective between $L^{-1}(B-c)$ and $B$ given by $z\mapsto L(z)+c$. However, the map $z\mapsto \tilde{L}(z)+\tau(c)$ is not necessarily a bijection between $\iota^{-1}(L^{-1}(B-c))$ and $\iota^{-1}(B)$. In order to connect the sets $\iota^{-1}(L^{-1}(B-c))$ and $\iota^{-1}(B)$ by $\tilde{L}$, we define a map
    $\Phi_{L,c}\colon\Z^{d'}\times\Z^{d}\to\Z^{d}$ as
  $$\Phi_{L,c}(z,w):=\tau(c)+\tilde{L}(z)+pw.$$
  
     \begin{lem}\label{1:298ruf}
       For all $K,K'\in\N_{+}$, the map $\Phi_{L,c} \mod pK\Z^{d}$ is a ${K'}^{d}$-to-1 map between $(\iota^{-1}(L^{-1}(B-c))\cap[pK']^{d'})\times [K]^{d}$ and $\iota^{-1}(B)\cap [pK]^{d}$.
  %
   \end{lem}
   \begin{proof}
   For convenience denote $\Phi:=\Phi_{L,c}$.
        Let $(z,w)\in \iota^{-1}(L^{-1}(B-c))\times\Z^{d}$.  
        Then 
        $$\iota\circ\Phi(z,w)=\iota(\tau(c)+\tilde{L}(z)+pw)=c+\iota\circ\tilde{L}(z)=c+L\circ\iota(z)\in B.$$
        So   the range of $\Phi \mod pK\Z^{d}$ is contained in $\iota^{-1}(B)\cap [pK]^{d}$.
            
        Now for all $n\in \iota^{-1}(B)$, We may pick some $z'\in L^{-1}(B-c)$ such that $\iota(n)=L(z')+c$. Let $z=\tau(z')$. Then $z$ belongs to $\iota^{-1}(L^{-1}(B-c))\cap[p]^{d'}$ and 
        $$\tilde{L}(z)+\tau(c)=\tilde{L}\circ\tau(z')+\tau(c)\equiv \tau\circ L(z')+\tau(c)=\tau\circ\iota(n)\equiv n \mod p\Z^{d}.$$
        So there exists $w\in \Z^{d}$ such that $n=\tilde{L}(z)+\tau(c)+pw.$ Therefore, $\Phi(z, w \mod K\Z^{d})\equiv n \mod pk\Z^{d}$.               
        So the map $\Phi \mod pK\Z^{d}$ is surjective.
        
        Finally, for $(z,w), (z',w')\in (\iota^{-1}(L^{-1}(B-c))\cap[pK']^{d'})\times [K]^{d}$, we have that $\Phi(z,w)\equiv\Phi(z',w') \mod pK\Z^{d}$ if and only if
        $\tilde{L}(z-z')+p(w-w')\in pK\Z^{d}$. If this is the case, then $\tilde{L}(z-z')\in p\Z^{d}$ and thus $\bold{0}=\iota\circ\tilde{L}(z-z')=L\circ \iota(z-z')$. Since $L$ is injective, this implies that $\iota(z-z')=\bold{0}$ and thus $z\equiv z' \mod p\Z^{d'}$. Now for any such $z-z'$ and any fixed $w$, there is a unique $w'\in [K]^{d}$ such that $\frac{1}{p}\tilde{L}(z-z')+(w-w')\in K\Z^{d}$. Since for each $z\in\Z^{d'}$, the number of $z'\in [pK']^{d}$ for which $z\equiv z' \mod p\Z^{d'}$ is equal to ${K'}^{d}$. So $\Phi \mod pK\Z^{d}$ is a ${K'}^{d}$-to-1 map.
  %
 %
%
   %
   \end{proof}


  Before proving  Theorem \ref{1:sLei},
 we start with some simplifications.
  Suppose first that the conclusion holds for $V+c=\V$. For the general case, let $L\colon\F_{p}^{d-r}\to V$ be a bijective linear transformation. Let  $M'\colon\F_{p}^{d-r}\to\F_{p}$ be the quadratic map given by $M'(z):=M(L(z)+c)$. 
  Then it is clear that $L^{-1}((V(M)\cap(V+c))-c)=V(M')$. 
  
  Let $g\in\poly(\Z^{d}\to G_{\N})$ be rational. By Corollary \ref{1:r222p}, we may assume that $g$ is $pK$-periodic for some $K\in\N_{+}$.  Then for any function $F\colon G/\Gamma\to\C$, it follows from Lemma \ref{1:298ruf} (setting $B=V(M)\cap (V+c)$ and $K'=K$) that
  \begin{equation}\nonumber
      \begin{split}
      &\quad\E_{n\in \iota^{-1}(V(M)\cap(V+c))}F(g(n)\Gamma)=\E_{n\in \iota^{-1}(V(M)\cap(V+c))\cap [pK]^{d}}F(g(n)\Gamma)
      \\&=\E_{z\in  \iota^{-1}(V(M'))\cap [pK]^{d-r}}\E_{w\in [K]^{d}}F(g(\tau(c)+\tilde{L}(z)+pw)\Gamma)
      \\&=\E_{(z,w)\in  \iota^{-1}(V(M''))\cap [pK]^{2d-r}}F(g(\tau(c)+\tilde{L}(z)+pw)\Gamma)
      \end{split}
  \end{equation}
   where $M''\colon\F_{p}^{2d-r}\to\F_{p}$ is the quadratic form given by
   $M''(z,w):=M'(z)$, and $\tilde{L}/p$ is the regular lifting of $L$.
   It is clear that all of $M,M'$ and $M''$ have the same rank, and that the map $(z,w)\mapsto g(\tau(c)+\tilde{L}(z)+pw)$ is rational. So by assumption, if $(g(n)\Gamma)_{n\in\iota^{-1}(V(M)\cap (V+c))}$ is not $\d$-equidistributed on $G/\Gamma$, then  there exists a nontrivial type-I horizontal character $\eta$ with $0<\Vert\eta\Vert\ll_{d,m} \d^{-O_{d,m}(1)}$ such that $\eta\circ g(\tau(c)+\tilde{L}(z)+pw) \mod \Z$ is a constant for all $z\in \iota^{-1}(V(M'))$ and $w\in\Z^{d}$. By Lemma \ref{1:298ruf}, we have that  $\eta\circ g(n) \mod \Z$ is a constant for all $n\in \iota^{-1}(V(M)\cap (V+c))$.
     So in the rest of Section \ref{1:s:pp9}, we may make the following assumption.

  \textbf{Assumption 1.} $V+c=\V, r=0$ and $\rank(M)\geq s+13$.

  Denote $d'=\rank(M)$.
  Suppose that we have already proved Theorem \ref{1:sLei}  for all  quadratic forms of the form 
  \begin{equation}\label{1:red1}
  M_{c,c',\l}(n_{1},\dots,n_{d}):=cn_{1}^{2}+n^{2}_{2}+\dots+n_{d'}^{2}+c'n_{d'+1}-\l
  \end{equation}	
  for some $c,c',\l\in\F_{p}, c\neq 0$. 
  Let $M$ be an arbitrary quadratic form of rank $d'$. By Lemma \ref{1:cov}, there exist  $c,c',\l\in\F_{p}$ and $v\in\V$ with $c\neq 0$ and a $d\times d$ invertible matrix $R$ in $\F_{p}$ such that $M(n)=M_{c,c',\l}(nR+v)$. 
  
  For convenience denote $M':=M_{c,c',\lambda}$. Then $R(V(M)+vR^{-1})=V(M')$.
  Let $g\in\poly(\Z^{d}\to G_{\N})$ be rational. By Corollary \ref{1:r222p}, we may assume that $g$ is $pK$-periodic for some $K\in\N_{+}$.  Then for any function $F\colon G/\Gamma\to\C$, it follows from Lemma \ref{1:298ruf} (setting $B=V(M)$ and $K'=K$) that
  \begin{equation}\nonumber
      \begin{split}
      &\quad\E_{n\in \iota^{-1}(V(M))}F(g(n)\Gamma)=\E_{n\in \iota^{-1}(V(M))\cap [pK]^{d}}F(g(n)\Gamma)
      \\&=\E_{z\in  \iota^{-1}(V(M'))\cap [pK]^{d}}\E_{w\in [K]^{d}}F(g(\tau(-vR^{-1})+\tilde{R^{-1}}(z)+pw)\Gamma)
      \\&=\E_{(z,w)\in  \iota^{-1}(V(M''))\cap [pK]^{2d}}F(g(\tau(-vR^{-1})+\tilde{R^{-1}}(z)+pw)\Gamma)
      \end{split}
  \end{equation}
   where $M''\colon\F_{p}^{2d}\to\F_{p}$ is the quadratic form given by
   $M''(z,w):=M'(z)$, and $\tilde{R^{-1}}/p$ is the regular lifting of $R^{-1}$.
   It is clear that all of $M,M'$ and $M''$ have the same rank, and that $M''$ is also of the form (\ref{1:red1}) (for the same $d'$ but with dimension increased from $d$ to $2d$). Moreover, the map $(z,w)\mapsto g(\tau(-vR^{-1})+\tilde{R}^{-1}(z)+pw)$ is rational. So by assumption, if $(g(n)\Gamma)_{n\in\iota^{-1}(V(M))}$ is not $\d$-equidistributed on $G/\Gamma$, then  there exists a nontrivial type-I horizontal character $\eta$ with $0<\Vert\eta\Vert\ll_{d,m} \d^{-O_{d,m}(1)}$  such that $\eta\circ g(\tau(-vR^{-1})+\tilde{R^{-1}}(z)+pw)\mod \Z$ is a constant for all $z\in \iota^{-1}(V(M'))$ and $w\in\Z^{d}$. By Lemma \ref{1:298ruf}, we have that  $\eta\circ g(n) \mod \Z$ is a constant for all $n\in \iota^{-1}(V(M))$.
     So in the rest of Section \ref{1:s:pp9}, we may make the following assumption.

  
  \textbf{Assumption 2.} $M=M_{c,c',\l}$ is given by (\ref{1:red1}).

   Let $\psi\colon G\to\R^{m}$ denote the Mal'cev coordinate map associated to $\mathcal{X}$.
  Now suppose that the concludes holds under the additional assumption that $g(\bold{0})=id_{G}$, the \emph{identity element} of $G$.
   For the general case, write $g(\bold{0})=\{g(\bold{0})\}[g(\bold{0})]$.  
  Let $g'(n):=\{g(\bold{0})\}^{-1}g(n)g(\bold{0})^{-1}\{g(\bold{0})\}$. Then $g'$ is rational by the Baker-Campbell-Hausdorff formula   and $g(n)\Gamma=\{g(\bold{0})\}g'(n)\Gamma$ for all $n\in\Z^{d}$.
   Since $\Vert F(\{g(\bold{0})\}\cdot)\Vert_{\Lip}\leq O(1)\Vert F\Vert_{\Lip}$ by Lemma A.5 of \cite{GT12b}, if $(g(n)\Gamma)_{n\in \iota^{-1}(V(M))}$ is not $\d$-equidistributed on $G/\Gamma$, then $(g'(n)\Gamma)_{n\in \iota^{-1}(V(M))}$ is not $O(\d)$-equidistributed on $G/\Gamma$. Since $g'(\bold{0})=id_{G}$, by assumption, there exists a nontrivial type-I horizontal character $\eta$ with $0<\Vert\eta\Vert\ll_{d,m} \d^{-O_{d,m}(1)}$   such that  for all $n\in \iota^{-1}(V(M))$,
  $\eta\circ g'(n) \mod \Z$ is a constant, and thus $\eta\circ g(n) \mod \Z$ is a constant.
   So in the rest of Section \ref{1:s:pp9}, we may make the following assumption.
  
  \textbf{Assumption 3.} $g(\bold{0})=id_{G}$. As a result, we may further assume that $G=G_{0}=G_{1}$.

   Now suppose that the concludes holds  under the additional assumption that   that $\psi(g(e_{i}))$ belongs $[0,1)^{m}$ for all $1\leq i\leq d$, where $e_{i}$ is the $i$-th standard unit vector in $\Z^{d}$.
   For the general case, 
    Write $g(e_{i})=\{g(e_{i})\}[g(e_{i})]$ 
  and let $g'(n):=g(n)\prod_{i=1}^{d}[g(e_{i})]^{-n_{i}}$.
  Then $\psi(g'(e_{i}))\in[0,1)^{m}$
 and $g'$ is rational by the Baker-Campbell-Hausdorff formula.
 Moreover, for any type-I horizontal character $\eta$, since $\eta$ vanishes on $\Gamma$, $\eta\circ g'$ differs from $\eta\circ g$ by a constant. This enables us to prove the general case of Theorem \ref{1:sLei}.
    So in the rest of Section \ref{1:s:pp9}, we may make the following assumption.
    
    \textbf{Assumption 4.}   $\psi(g(e_{i}))\in[0,1)^{m}$ for all $1\leq i\leq d$.

 It is helpful to introduce some further convention of notations.
 Throughout this section, we assume that $p\gg_{d,m}\d^{-O_{d,m}(1)}$.
   Denote
  $$M_{c}(n_{1},\dots,n_{d})=cn_{1}^{2}+n^{2}_{2}+\dots+n_{d'}^{2},$$
  which is the degree two terms of $M$.
  and let $\tilde{M}\colon\Z^{d}\to\Z/p$ be a regular lifting of $M$, and 
  $\tilde{M}_{c}\colon\Z^{d}\to\Z/p$ be a homogeneous regular lifting of $M_{c}$ (whose existences are guaranteed by Lemma \ref{1:lifting}).
 Then $M(n)=\iota\circ p\tilde{M}\circ\tau(n)$ and $M_{c}(n)=\iota\circ p\tilde{M}_{c}\circ\tau(n)$ for all $n\in\V$. 
 It is clear that $\iota^{-1}(V(M))=V_{p}(\tilde{M})$.
 
  Let $A$ be the $d\times d$ $\F_{p}$-valued matrix associated to $M$ and $M_{c}$.   Then
  $\tau(A)$ is the $d\times d$ integer valued matrix associated to $\tilde{M}$ and $\tilde{M}_{c}$. With a slight abuse of notations, we identify the matrix $A$ both as the $\F_{p}$-valued matrix and as the $[p]$-valued matrix $\tau(A)$.

  Let $(G/\Gamma)_{\N}$ be a nilmanifold with filtration $(G_{i})_{i\in\N}$. Denote $m:=\dim(G)$, $m_{i}:=\dim(G_{i})$, $m_{ab}:=\dim(G)-\dim([G,G])$ (the dimension of the Type-1 horizontal torus of $(G/\Gamma)_{\N}$), $m_{\lin}:=m-m_{2}$ (the \emph{linear degree} of $(G/\Gamma)_{\N}$), and $m_{\ast}:=m_{ab}-m_{\lin}$ (the \emph{nonlinear degree} of $(G/\Gamma)_{\N}$).

 \subsection{The case $s=1$}

 We first prove Theorem  \ref{1:sLei}  for the case $s=1$ (the case $s=0$ is trivial). We make use of the following corollary of Lemma \ref{1:counting}.

 \begin{lem}\label{1:cout1}
Let $d,K\in\N_{+}$ and $p$ be a prime number. Let  $\tilde{M}\colon\Z\to\Z/p$ be a  quadratic form of $p$-rank $r$. Suppose that $r\geq 3$.  
	Then  for all $\xi=(\xi_{1},\dots,\xi_{d})\in\Z^{d}\backslash pK\Z^{d}$,
	we have that 
	\begin{equation}\nonumber
	\begin{split}
	\E_{n\in V_{p}(\tilde{M})\cap [pK]^{d}}\exp\Bigl(\frac{\xi}{pK}\cdot n\Bigr)=O(p^{-\frac{r-2}{2}}).
	\end{split}
	\end{equation}
\end{lem}

\begin{proof}
Note that
	\begin{equation}\nonumber
	\begin{split}
	&\quad \E_{n\in V_{p}(\tilde{M})\cap [pK]^{d}}\exp\Bigl(\frac{\xi}{pK}\cdot n\Bigr)
	=\E_{n\in V_{p}(\tilde{M})\cap [p]^{d}}\E_{m\in [K]^{d}}\exp\Bigl(\frac{\xi}{pK}\cdot (n+pm)\Bigr)
	\\&=\E_{n\in V_{p}(\tilde{M})\cap [p]^{d}}\E_{m\in [K]^{d}}\exp\Bigl(\frac{\xi}{pK} \cdot n\Bigr)\E_{m\in [K]^{d}}\exp\Bigl(\frac{\xi}{K}\cdot m\Bigr).
	\end{split}
	\end{equation}	
Since $\xi\in\Z^{d}$, it is not hard to see that $\E_{m\in [K]^{d}}\exp\Bigl(\frac{\xi}{K}\cdot m\Bigr)=0$ unless $\xi=K\xi'$ for some $\xi'\in\Z^{d}$, in which case 
$$\E_{n\in V_{p}(\tilde{M})\cap [pK]^{d}}\exp\Bigl(\frac{\xi}{pK}\cdot n\Bigr)=\E_{n\in V_{p}(\tilde{M})\cap [p]^{d}}\exp\Bigl(\frac{\xi'}{p}\cdot n\Bigr)$$
and thus the conclusion follows from Lemma \ref{1:counting}.
\end{proof}

Since $s=1$, we have
  $m_{\ast}=m$ and $G/\Gamma=\R^{m}/\Z^{m}=\T^{m}$. 
 We may write 
 $$g(n)=(b_{1}\cdot n,\dots,b_{m}\cdot n)+(v_{1},\dots,v_{m})$$   for some $K\in \N_{+}$, $b_{i}\in\Z^{d}/pK$ and $v_{i}\in\R$. Clearly $g$ is $pK$-periodic.  
 If $(g(n)\Gamma)_{n\in V_{p}(\tilde{M})\cap[pK]^{d}}$ is not $\d$-equidistributed, then by Lemma \ref{1:3.7}, then there exists a vertical frequency $\xi=(\xi_{1},\dots,\xi_{m})\in\Z^{m}$ with $\vert\xi\vert\ll_{m} \d^{-O_{m}(1)}$ such that 
 \begin{equation}\label{1:e4.01}
 \begin{split}
 \Bigl\vert\E_{n\in V_{p}(\tilde{M})\cap[pK]^{d}}\exp(\xi\cdot g(n))-\int_{\T^{m}} \exp(\xi\cdot x)dm_{\T^{m}}(x)\Bigr\vert\gg_{m} \d^{O_{m}(1)}.
 \end{split}
 \end{equation}  
 where $m_{\T^{m}}$ is the Haar measure on $\T^{m}$.  
 So we have that $\xi\neq \bold{0}$. Thus $$\vert\E_{n\in V_{p}(\tilde{M})\cap[pK]^{d}}\exp(\xi\cdot g(n))\vert\gg_{m} \d^{O_{m}(1)}.$$ Note that $\xi\cdot g(n)=(\xi B)\cdot n+\xi\cdot v$, where $B$ is the $\Z/pK$-valued $m\times d$ matrix whose $i$-th row is $b_{i}$. 
 If $\xi B\notin \Z^{d}$, then since $\rank_{p}(M)\geq 3$, Lemma \ref{1:cout1}  implies that  $$\vert\E_{n\in V_{p}(\tilde{M})\cap[pK]^{d}}\exp((\xi B)\cdot n)\vert=O(p^{-\frac{d-2}{2}}),$$ which is impossible if $p\gg_{m} \d^{-O_{m}(1)}$. So we must have that $\xi B\in\Z^{d}$, which implies that $\xi\circ g \mod \Z$ is a constant. This finishes the proof.

 \subsection{Construction of a nilsequence in a joining system}\label{1:s5.4}
 
 We now assume that $s\geq 2$ (and thus $d'\geq s+13$), and that Theorem  \ref{1:sLei}  holds for smaller $s$ (and any $m_{\ast}$), and for the same value of $s$ and a smaller value of $m_{\ast}$.  
 By Corollary \ref{1:r222p}, we may assume that $g$ is $pK$-periodic for some $K\in\N_{+}$. 
 By Lemma \ref{1:3.7}, if $(g(n)\Gamma)_{n\in V_{p}(\tilde{M})\cap[pK]^{d}}$ is not $\d$-equidistributed on $G/\Gamma$, then there exist a vertical frequency $\xi\in\Z^{m_{s}}$ with $\Vert\xi\Vert\ll_{m} \d^{-O_{m}(1)}$ and a function $F\colon G/\Gamma\to\C$ with $\Vert F\Vert_{\Lip}\leq 1$ and vertical frequency $\xi$ such that 
 \begin{equation}\label{1:e4.1}
 \begin{split}
 \Bigl\vert\E_{n\in V_{p}(\tilde{M})\cap[pK]^{d}}F(g(n)\Gamma)-\int_{G/\Gamma} F\,dm_{G/\Gamma}\Bigr\vert\gg_{m} \d^{O_{m}(1)},
 \end{split}
 \end{equation}  
 where  $m_{G/\Gamma}$ is the Haar measure of $G/\Gamma$.
 If $\xi=\bold{0}$, then we may invoke the method on pages 504--505 of \cite{GT12b} to reduce the problem to the $s-1$ degree case and then use the induction hypothesis to finish the proof. So from now on we assume that $\xi\neq\bold{0}$. 
 
 Since $F$ is of vertical frequency $\xi$, (\ref{1:e4.1}) implies that 
 \begin{equation}\label{1:e4.2}
 \begin{split}
 \vert\E_{n\in V_{p}(\tilde{M})\cap[pK]^{d}}F(g(n)\Gamma)\vert\gg_{m} \d^{O_{m}(1)}.
 \end{split}
 \end{equation}  
 By taking squares on both sides of (\ref{1:e4.2}), we have that 
 \begin{equation}\label{1:hhppo}
  \begin{split}
&\quad 
\frac{1}{\vert V_{p}(\tilde{M})\cap[pK]^{d}\vert^{2}}\sum_{h\in V_{p}(\tilde{M})\cap[pK]^{d}}\sum_{n\in V_{p}(\tilde{M})^{h}\cap[pK]^{d}}F(g(n+h)\Gamma)\overline{F}(g(n)\Gamma)
\\&=\E_{(n,h)\in \Gow_{p,1}(V_{p}(\tilde{M}))\cap[pK]^{2d}}F(g(n+h)\Gamma)\overline{F}(g(n)\Gamma)
\\&=\E_{n,m\in V_{p}(\tilde{M})\cap[pK]^{d}}F(g(m)\Gamma)\overline{F}(g(n)\Gamma)
\gg_{m} \d^{O_{m}(1)},
 \end{split}
 \end{equation}  
 where we used Lemma \ref{1:expp} and the fact that $\vert \Gow_{p,1}(V_{p}(\tilde{M}))\cap[p]^{d}\vert=\vert V_{p}(\tilde{M})\cap[p]^{d}\vert^{2}=p^{2d-2}(1+O(p^{-1/2}))$ by Lemma \ref{1:counting} and the fact that $\Gow_{p,1}(V_{p}(\tilde{M}))$ is a $p$-periodic set. 
 	 Since $d\geq 5$, it follows from Proposition \ref{1:ctsp} that
 $$\E_{h\in [pK]^{d}}\vert\E_{n\in V_{p}(\tilde{M})^{h}\cap [pK]^{d}}F(g(n+h)\Gamma)\overline{F}(g(n)\Gamma)\vert\gg_{m} \d^{O_{m}(1)}.$$
  By the Pigeonhole Principle, there exists
 a subset $J_{1}$ of $[pK]^{d}$ of cardinality $\gg_{m}\d^{O_{m}(1)}(pK)^{d}$ such that 
 \begin{equation}\label{1:e4.3}
 \begin{split}
 \vert\E_{n\in V_{p}(\tilde{M})^{h}\cap [pK]^{d}}F(g(n+h)\Gamma)\overline{F(g(n)\Gamma)}\vert\gg_{m} \d^{\O_{m}(1)}
 \end{split}
 \end{equation} 
 for all $h\in J_{1}$.
 Recall that $A$ is the $d\times d$ $\Z$-valued matrix associated to $M$.  Since $d'\geq 3$, by Lemmas \ref{1:counting} and  \ref{1:expp}, passing to  a subset if necessary, we may assume that  
  \begin{equation}\label{1:resv20}
 (hA)\cdot h, h_{1},\dots,h_{d}\notin p\Z \text{ for all } h=(h_{1},\dots,h_{d})\in J_{1}.
 \end{equation}

  Recall that $M\colon\V\to\F_{p}$ is the quadratic form induced by $\tilde{M}$ and $V_{p}(\tilde{M})^{h}=\iota^{-1}(V(M)^{\iota(h)})$.  For $h\in J_{1}$, note that 
  $V(M)^{\iota(h)}=V(M)\cap U_{h}$, where
   $U_{h}:=\{n\in\F_{p}^{d}\colon M(n+\iota(h))=M(n)\}$. Then we may write $U_{h}=V_{h}+u_{h}$ for some $u_{h}\in \F_{p}^{d}$, where $V_{h}:=\{n\in\F_{p}^{d}\colon (nA)\cdot \iota(h)=0\}$. 
   It follows from (\ref{1:e4.3}) that $V_{h}$ is a subspace of $\V$ of co-dimension 1.
   Let $L_{h}\colon \F_{p}^{d-1}\to V_{h}$ be any bijective linear transformation, and let $M_{h}\colon \F_{p}^{d-1}\to\F_{p}$ be the quadratic form given by $M_{h}(z):=M(L_{h}(z)+u_{h})$ and let $\tilde{M}_{h}\colon\Z^{d-1}\to\Z/p$ be the quadratic form induced by $M_{h}$. 
     Since $(hA)\cdot h\notin p\Z$, we have that  
 \begin{equation}\label{1:hfeowpwef0}
 \rank_{p}(\tilde{M}_{h})=\rank(M_{h})=\rank(M\vert_{\sp_{\F_{p}}\{\iota(h)\}^{\perp_{M}}})=\rank(M)-1\geq (s-1)+13 \text{ for all } h\in J_{1}.
 \end{equation}

    Let $\tilde{L}_{h}/p$ be the regular lifting of $L_{h}$.  Then
      \begin{equation}\label{1:trueequivu2}
(hA)\cdot \tilde{L}_{h}(z)\in p\Z \text{ for all } z\in \Z^{d-1},
\end{equation}
     Since $V(M)^{\iota(h)}=V(M)\cap (V_{h}+u_{h})$, it is not hard to see that
   $L_{h}^{-1}(V(M)^{\iota(h)}-u_{h})=V(M_{h})$ and $V_{p}(\tilde{M}_{h})=\iota^{-1}(V(M_{h}))$.
    Since $g$ is $pK$-periodic, by (\ref{1:e4.3}) and Lemma \ref{1:298ruf} (setting $B=V(M)^{\iota(h)}$ and $K'=K$), the average
        \begin{equation}\label{1:e4.3mmn}
 \begin{split}
 \vert\E_{(z,m)\in (V_{p}(\tilde{M}_{h})\cap [pK]^{d-1})\times [pK]^{d}}F(g(pm+\tilde{L}_{h}(z)+\tau(u_{h})+h)\Gamma)\overline{F(g(pm+\tilde{L}_{h}(z)+\tau(u_{h}))\Gamma)}\vert
 \end{split}
 \end{equation} 
 is at least $\gg_{m} \d^{\O_{m}(1)}$
 for all $h\in J_{1}$. 

  We decompose $g$ as $g=g_{\nlin}g_{\lin}$, where 
 $$g_{\lin}(n_{1},\dots,n_{d}):=g(e_{1})^{n_{1}}\dots g(e_{d})^{n_{d}}$$
 is the \emph{linear component} of $g$ and  
 $$g_{\nlin}(n):=g(n)g_{\lin}(n)^{-1}$$
 is the \emph{nonlinear component} of $g$, where $e_{i}$ is the $i$-th standard unit vector in $\Z^{d}$. 
  Since $g\in\poly(\Z^{d}\to G_{\N})$, we have that $g_{\lin},g_{\nlin}\in\poly(\Z^{d}\to G_{\N})$ by Corollary B.4 of \cite{GTZ12}. 

 Note that $g_{\nlin}(\bold{0})=g_{\nlin}(e_{1})=\dots=g_{\nlin}(e_{d})=id_{G}$, and $g_{\nlin}$ takes values in $G_{2}$ by the Baker-Campbell-Hausdorff formula. For all $h\in\Z^{d}$, let $g_{h}\colon\Z^{d}\to G^{2}$ be the sequence
 \begin{equation}\label{1:e4.5}
 \begin{split}
g_{h}(n):=
 (\{{g}_{\lin}(h)\}^{-1}{g}(n+h)[{g}_{\lin}(h)]^{-1},{g}(n)),
 \end{split}
 \end{equation}
 and $F_{h}\colon G^{2}/\Gamma^{2}\to\C$ be the function
$$ F_{h}(x,y):=F(\{{g}_{\lin}(h)\}x)\overline{F}(y).
$$
 Then (\ref{1:e4.3mmn}) implies that 
 \begin{equation}\label{1:e4.311}
 \begin{split}
  \vert\E_{(z,m)\in (V_{p}(\tilde{M}_{h})\cap [pK]^{d-1})\times [pK]^{d}}F_{h}({g}_{h}(pm+\tilde{L}_{h}(z)+\tau(u_{h}))\Gamma^{2})\vert\gg_{m} \d^{\O_{m}(1)}
 \end{split}
 \end{equation}
 for all $h\in J_{1}$.
  
 For groups $G$ and $H$ with $H<G$, denote 
 $$G\times_{H}G:=\{(g,h)\in G\times G\colon g^{-1}h\in H\}.$$
 Note that $G\times_{H}G$ is a group as long as $[G,H]\subseteq H$.
 It is not hard to check that
  ${g}_{h}$  takes values in $G^{\square}:=G\times_{G_{2}}G$ for all $h\in\Z^{d}$ (see page 35 of \cite{GT12b} for details).
  So (\ref{1:e4.311}) implies that 
 \begin{equation}\label{1:e4.6}
 \begin{split}
 \vert\E_{(z,m)\in (V_{p}(\tilde{M}_{h})\cap [pK]^{d-1})\times [pK]^{d}}F^{\square}_{h}(g^{\square}_{h}(pm+\tilde{L}_{h}(z)+\tau(u_{h}))\Gamma^{\square})\vert\gg_{m} \d^{\O_{m}(1)}
 \end{split}
 \end{equation} 
 for all $h\in J_{1}$,
 where $F^{\square}_{h}$, $g^{\square}_{h}$ and $\Gamma^{\square}$ are the restrictions of $F_{h},$ ${g}_{h}$ and $\Gamma^{2}$ to $G^{\square}$ respectively. By repeating the argument on page 35 of \cite{GT12b}, $F^{\square}_{h}$ is invariant under $G^{\triangle}_{s}:=\{(g_{s},g_{s})\colon g_{s}\in G_{s}\}$. Thus $F^{\square}_{h}$ induces a function $\overline{F_{h}^{\square}}$ on $\overline{G^{\square}}:=G^{\square}/G^{\triangle}_{s}$. By (\ref{1:e4.6}), we have that 
 \begin{equation}\label{1:e4.7}
 \begin{split}
 \vert\E_{(z,m)\in (V_{p}(\tilde{M}_{h})\cap [pK]^{d-1})\times [pK]^{d}}\overline{F^{\square}_{h}}(\overline{g^{\square}_{h}}(pm+\tilde{L}_{h}(z)+\tau(u_{h}))\overline{\Gamma^{\square}})\vert\gg_{m} \d^{\O_{m}(1)},
 \end{split}
 \end{equation} 
 where $\overline{\Gamma^{\square}}:=\Gamma^{\square}/(\Gamma^{\square}\cap G_{k}^{\triangle})$ and $\overline{g^{\square}_{h}}$ is the projection of $g^{\square}_{h}$ to $\overline{G^{\square}}$. 
 
 By Proposition 4.2 of \cite{GT14} and Lemma 7.4 of \cite{GT12b},
 we have that
 \begin{itemize}
 	\item the groups $G^{\square}$ and $\overline{G^{\square}}$ are connected, simply-connected nilpotent Lie  groups of degree $s-1$ with respect to the $\N$-filtrations given by $(G^{\square})_{i}:=G_{i}\times_{G_{i+1}}G_{i}$ and $(\overline{G^{\square}})_{i}:=(G_{i}\times_{G_{i+1}}G_{i})/G^{\triangle}_{k}$, respectively;
 	\item there exist  $O_{m,s}(\d^{-O_{m,s}(1)})$-rational Mal'cev basis $\X^{\square}$ for $G^{\square}/\Gamma^{\square}$, and $\overline{\X^{\square}}$ for $\overline{G^{\square}}/\overline{\Gamma^{\square}}$, adapted to the filtrations $(G^{\square})_{\N}$ and  $(\overline{G^{\square}})_{\N}$, respectively;
 	\item the Mal'cev coordinate map $\psi_{\mathcal{X}^{\square}}(x,x')$ of   $\X^{\square}$ is a polynomial of degree $O_{m,s}(1)$ with rational coefficients of complexity $O_{m,s}(\d^{-O_{m,s}(1)})$ in the coordinates $\psi(x)$ and $\psi(x')$;
 	\item under the metrics induced by $\X^{\square}$ and $\overline{\X^{\square}}$, we have that $\Vert F^{\square}_{h}\Vert_{\Lip}, \Vert\overline{F^{\square}_{h}}\Vert_{\Lip}\ll_{m,s}(\d^{-O_{m,s}(1)})$;
	\item $g^{\square}_{h}$ belongs to $\poly(\Z^{d}\to (G^{\square})_{\N})$ and $\overline{g^{\square}_{h}}$ belongs to $\poly(\Z^{d}\to (\overline{G^{\square}})_{\N})$.
 \end{itemize}
 Since  $g$ is rational and $pK$-periodic, it is not hard to see that
 \begin{itemize}
	\item $g^{\square}_{h}$  and $\overline{g^{\square}_{h}}$ are rational and $pK$-periodic.
	 \end{itemize}
	 

Let $\tilde{M}'_{h}\colon \Z^{2d-1}\to \Z/p$ be the quadratic form given by $\tilde{M}'_{h}(z,m):=\tilde{M}_{h}(z)$. It is clear from (\ref{1:hfeowpwef0}) that $\rank_{p}(\tilde{M}'_{h})=\rank_{p}(\tilde{M}_{h})\geq (s-1)+13$ for all $h\in J_{1}$. So
  by (\ref{1:e4.7}),  and  the  induction hypothesis, for all $h\in J_{1}$, there exist a type-I horizontal character $\overline{\eta}_{h}\colon\overline{G^{\square}}\to\R$ with $0<\Vert\overline{\eta}_{h}\Vert\ll_{d,m}\d^{-O_{d,m}(1)}$ and a constant $a_{h}\in\R$ such that  for all $m\in\Z^{d}$ and $z\in V_{p}(\tilde{M}_{h})$, 
 $$\overline{\eta}_{h}\circ \overline{g^{\square}_{h}}(pm+\tilde{L}_{h}(z)+\tau(u_{h}))\equiv a_{h} \mod \Z.$$

 By the Pigeonhole Principle, there exists a subset $J_{2}$ of $J_{1}$ with $\vert J_{2}\vert\gg_{d,m}\d^{O_{d,m}(1)}(pK)^{d}$  such that $\overline{\eta}_{h}$ equals to a same $\overline{\eta}$ for all $h\in J_{2}$. Writing $\eta_{0}\colon G^{\square}\to\R$ to be the type-I horizontal character $\eta_{0}(x)=\overline{\eta}(\overline{x})$ where $\overline{x}$ is the projection of $x$, we have that $\eta_{0}$ is a type-I horizontal character with $0<\Vert \eta_{0}\Vert\ll_{d,m}\d^{-O_{d,m}(1)}$ and that 
 \begin{equation}\nonumber
 \begin{split}
 \eta_{0}\circ g^{\square}_{h}(pm+\tilde{L}_{h}(z)+\tau(u_{h}))\equiv a_{h} \mod \Z \text{ for all $h\in J_{2}, m\in\Z^{d}, z\in V_{p}(\tilde{M}_{h})$.}
 \end{split}
 \end{equation} 
   Since $g^{\square}_{h}$ is $pK$-periodic, it follows from Lemma \ref{1:298ruf} that 
    \begin{equation}\label{1:e4.80m2}
 \begin{split}
 \eta_{0}\circ g^{\square}_{h}(n)\equiv a_{h} \mod \Z \text{ for all $h\in J_{2}, n\in V_{p}(\tilde{M})^{h}$.}
 \end{split}
 \end{equation} 
 Finally, since $g$ is $pK$-periodic, $(g^{\square}_{h+pKm})^{-1}g^{\square}_{h}$ take values in $\Gamma^{2}$ for all $h\in J_{2}$ and $m\in\Z^{d}$. So $ \eta_{0}\circ g^{\square}_{h}= \eta_{0}\circ g^{\square}_{h+pKm} \mod\Z$. Therefore, we may upgrade (\ref{1:e4.80m2}) to conclude that
    \begin{equation}\label{1:e4.8}
 \begin{split}
 \eta_{0}\circ g^{\square}_{h}(n)\equiv a_{h} \mod \Z \text{  for all  $h\in J_{2}+pK\Z^{d}, n\in V_{p}(\tilde{M})^{h}$}
 \end{split}
 \end{equation}
for some constant $a_{h}$.

 Denote $\eta_{1}\colon G\to\R, \eta_{1}(g):=\eta_{0}(g,g)$ and $\eta_{2}\colon G_{2}\to\R, \eta_{2}(g)=\eta_{0}(g,id_{G})$, we have that $\eta_{0}(g',g)=\eta_{1}(g)+\eta_{2}(g'g^{-1})$ for all $(g',g)\in G^{\square}$.
 Similar to Lemma 7.5 of \cite{GT12b}, $\eta_{1}$ is a type-I horizontal character on $G$, and  $\eta_{2}$ is a type-I  horizontal character on $G_{2}$ which annihilates $[G,G_{2}]$. Moreover, $\Vert\eta_{1}\Vert,\Vert\eta_{2}\Vert\ll_{d,m}\d^{-O_{d,m}(1)}$.

 Similar to the approach on pages 12 and 13 of \cite{GT14} (along with some notation changes), 
 one can compute that (we leave the details to the interested readers)  
 \begin{equation}\label{1:force1}
 \begin{split}
  Q(n+h)-Q(n)+P(n)+\sigma(h)\cdot n-a_{h}\equiv 0 \mod \Z
 \end{split}
 \end{equation}
 for all $h\in J_{2}+pK\Z^{d}$ and $n\in V_{p}(\tilde{M})^{h}$,
 where $$P(n):=\eta_{1}({g}(n)), Q(n):=\eta_{2}({g}_{\nlin}(n)),
 \sigma(h):=\omega_{h}+h\Lambda,$$
 $$\omega_{h}:=(\eta_{2}([{g}(e_{1}),\{{g}_{\lin}(h)\}]),\dots,\eta_{2}([{g}(e_{d}),\{{g}_{\lin}(h)\}]))\in\R^{d}$$ 
 and $\Lambda=(\l_{i,j})_{1\leq i,j\leq d}$ is the $d\times d$ matrix in $\R$ with $\l_{i,j}=\eta_{2}([{g}(e_{i}),{g}(e_{j})])$ if $i<j$ and  $\l_{i,j}=0$ otherwise. Since $g$ is rational, it is clear that all the entries of $\Lambda$ are in $\Q$.

  \subsection{Solving polynomial equations}
  
 Our next goal is to solve the polynomial equation (\ref{1:force1}). An equation of similar form was obtained in (4.10) of \cite{GT14}. However, equation (\ref{1:force1}) differs from (4.10) of \cite{GT14} in the sense that (\ref{1:force1})  only holds for $n\in V_{p}(\tilde{M})^{h}$ instead of for all $n\in\Z^{d}$. This makes equation (\ref{1:force1}) significantly more difficult than  (4.10) of \cite{GT14}. In fact, the bulk of Sections \ref{1:s4} and \ref{1:s5} are devoted to the preparation for solving equation (\ref{1:force1}).


Since $g$ is rational and is of step at most $s$, we have that $P$ and $Q$ are of degree at most $s$ and take values in $\Q$.
  Replacing $n$ in (\ref{1:force1}) by $n, n+h_{2}, n+h_{3},n+h_{2}+h_{3}$ respectively, replacing $h$ in (\ref{1:force1}) by $h_{1}$, and taking the second order differences, we have that
\begin{equation}\label{1:force001}
 \begin{split}
  \Delta_{h_{3}}\Delta_{h_{2}}\Delta_{h_{1}}Q(n)+\Delta_{h_{2}}\Delta_{h_{1}}P(n)\in\Z
 \end{split}
 \end{equation}
 for all $h_{3}\in J_{2}+pK\Z^{d}$ and $(n,h_{1},h_{2})\in \Gow_{p,2}(V_{p}(\tilde{M})^{h_{3}})$. 
 By Proposition \ref{1:ctsp} (and a change of variable), it is not hard to compute that the number of $(n,h_{1},h_{2},h_{3})\in \Gow_{p,3}(V_{p}(\tilde{M}))\cap ([pK]^{d})^{4}$ with $h_{3}\in J_{2}$ is $\gg_{d,m}\d^{-O_{d,m}(1)}\vert \Gow_{p,3}(V_{p}(\tilde{M}))\cap ([pK]^{d})^{4}\vert$.
 Since $d'\geq s+13\geq 3^{2}+3+3$, it follows from (\ref{1:force001}) and Proposition \ref{1:att30}\footnote{If $K$ is not divisible by $p$, then we may replace the set $J_{2}$ by $(J_{2}+pK\Z^{d})\cap [p^{2}K]^{d}$ to apply Proposition \ref{1:att30}.} that 
 \begin{equation}\label{1:forcemid}
 	Q=\frac{1}{q}Q_{1}+Q_{0} \text{ and } P=\frac{1}{q}P_{1}+P_{0}
 	\end{equation}
  for some $q\in\N_{+}$ with $q\ll_{d}\d^{-O_{d}(1)}$, some polynomials $Q_{1},P_{1}\in\poly(V_{p}(\tilde{M})\to \R\vert\Z)$ with $\deg(Q_{1})\leq \deg(Q), \deg(P_{1})\leq \deg(P)$, 
 	 and some $Q_{0},P_{0}\colon\Z^{d}\to\R$ such that
 	\begin{equation}\label{1:force51}
 	Q_{0}(n)=(n\Omega)\cdot n+v\cdot n+v_{0} \text{ and } P_{0}(n)=u\cdot n+u_{0}
 	\end{equation}
 	for some $v,u\in\R^{d}$, $v_{0},u_{0}\in\R$ and symmetric $d\times d$ $\Q$-valued matrix $\Omega$.\footnote{If $P$ and $Q$ were $\Z/p$-valued polynomials, then we could apply Corollary \ref{1:att301} instead of Proposition \ref{1:att30} and simplify the proof. However, this is not the case in general. This is the reason why we need to improve Corollary \ref{1:att301} to the more general version   Proposition \ref{1:att30} in Section \ref{1:s:52} (see also Remark \ref{1:rpprt}).}

	Our next step is to derive better descriptions for the lower degree polynomials $Q_{0}$ and $P_{0}$ by substituting (\ref{1:forcemid}) and (\ref{1:force51}) back to (\ref{1:force1}).
 	For all $h\in J_{2}+pK\Z^{d}$ and $n\in V_{p}(\tilde{M})^{h}$, it follows from (\ref{1:forcemid})
 	  that $$Q(n)-Q_{0}(n)\equiv Q(n+h)-Q_{0}(n+h)\equiv P(n)- P_{0}(n)\equiv 0 \mod \frac{1}{q}\Z.$$ 
 So	(\ref{1:force1}) implies that 
 		\begin{equation}\label{1:reductiona}
 		\begin{split}
 	t_{h}\cdot n\equiv b_{h} \mod \Z \text{ for all $h\in J_{2}+pK\Z^{d}$ and $n\in V_{p}(\tilde{M})^{h}$}
 		\end{split}
 		\end{equation}
 	 for some $b_{h}\in\R$ where
	 \begin{equation}\label{1:tthh}
 		\begin{split}
 	t_{h}:=q(2h\Omega+u+\sigma(h)).
	 		\end{split}
 		\end{equation}
 	Note that the conditions $n\in V_{p}(\tilde{M})^{h}$  remain unchanged if we replace $n$ by $n+pn'$ for any $n'\in\Z^{d}$. So the condition $t_{h}\cdot n\equiv b_{h} \mod \Z$ is also unaffected if we change $n$ to $n+pn'$. Therefore, we must have that 
	$t_{h}\in\Z^{d}/p, b_{h}\in\Z/p$	for all $h\in J_{2}+pK\Z^{d}$.

 Fix $h\in J_{2}+pK\Z^{d}$.
We temporary convert the language to the $\V$-setting.
Recall that $M_{h}\colon\F_{p}^{d}\to\F_{p}$ is the quadratic form induced by $\tilde{M}_{h}$, and 
  by  (\ref{1:hfeowpwef0}) $\rank(M_{h})\geq 3$. 

\textbf{Claim.} There exist $v_{0},\dots,v_{d-1}\in V(M_{h})$ such that $v_{1}-v_{0},\dots,v_{d-1}-v_{0}$ are linearly independent.

Fix any $v_{0}\in V(M_{h})$ (whose existence is given by Lemma \ref{1:counting}). Suppose that for some $0\leq k\leq d-2$, we have chosen $v_{1},\dots,v_{k}\in V(M_{h})$ such that $v_{1}-v_{0},\dots,v_{k}-v_{0}$ are linearly independent. If $k\leq d-3$, then
  the number of $v\in\V$ with $v_{1}-v_{0},\dots,v_{k}-v_{0}, v-v_{0}$ being linearly dependent is at most $p^{k}\leq p^{d-3}$ while $\vert V(M_{h})\vert=p^{d-2}(1+O(p^{-1/2}))$. If $p\gg_{d} 1$, then there exists $v_{k+1}\in V(M_{h})$ such that $v_{1}-v_{0},\dots,v_{k+1}-v_{0}$ are linearly independent. If $k=d-2$, then let $Q(v)$ denote the determinant of the $(d-1)\times (d-1)$ matrix with $v_{1}-v_{0},\dots,v_{d-2}-v_{0}, v-v_{0}$ being the row vectors. Then $Q$ is a non-constant degree 1 polynomial in $v$. So  $Q$ can not be written in the form $Q=Q'M_{h}$ for some $Q'\in\poly(\V\to\F_{p})$ with $\deg(Q')\leq \deg(Q)$. By Proposition \ref{1:noloop}, there exists $v_{d-1}\in V(M_{h})\backslash V(Q)$. So $v_{1}-v_{0},\dots,v_{d-1}-v_{0}$ are linearly dependent. This completes the proof of the claim.
  
  \

  Translating the claim back into the $\Z^{d}$-setting, we may find $w_{0},\dots,w_{d-1}\in V_{p}(\tilde{M}_{h})$ such that $w_{1}-w_{0},\dots,w_{d-1}-w_{0}$ are $p$-linearly independent. Then for all $1\leq i\leq d-1$, it follows from (\ref{1:reductiona}) that
  $$t_{h}\cdot (\tilde{L}_{h}(w_{i}-w_{0}))=t_{h}\cdot ((\tilde{L}_{h}(w_{i})+u_{h})-(\tilde{L}_{h}(w_{0})+u_{h}))\equiv b_{h}-b_{h}=0 \mod\Z.$$
  
  Since $w_{1}-w_{0},\dots,w_{d-1}-w_{0}$ are $p$-linearly independent, by linearity, there exists $q_{h}\in\N, p\nmid q_{h}$ such that 
   $q_{h}t_{h}\cdot \tilde{L}_{h}(e_{i})\in \Z$ for all $1\leq i\leq d-1$. 
   Since $t_{h}\in\Z^{d}/p$, we have
   \begin{equation}\label{1:sshan}
   t_{h}\cdot \tilde{L}_{h}(e_{i})\in \frac{1}{q_{h}}\Z\cap \frac{1}{p}\Z=\Z.
   \end{equation}

   Recall that $h_{1}\notin p\Z$ since $h\in J_{2}+pK\Z^{d}$. Since $t_{h}\in \Z^{d}/p$,
   we may write $t_{h}=c_{h}(hA)+(0,t'_{h})$ for some $c_{h}\in\Z/p$ and $t'_{h}\in\Z^{d-1}/p$. 
   Since $(hA)\cdot \tilde{L}_{h}(e_{i})\in p\Z$ for all $1\leq i\leq d-1$ by (\ref{1:trueequivu2}), it follows from (\ref{1:sshan}) that $t'_{h}B\in \Z^{d-1}$, where $B$ is the $(d-1)\times (d-1)$ matrix with whose column vectors are transposes of $\tilde{L}_{h}(e_{1}),\dots,\tilde{L}_{h}(e_{d-1})$ with the first coefficient removed. 
  
   If $\tilde{L}_{h}(e_{1}),\dots, \tilde{L}_{h}(e_{d-1}),(1,0,\dots,0)$ are $p$-linearly dependent, then the fact that $(hA)\cdot \tilde{L}_{h}(e_{i})\in p\Z, 1\leq i\leq d-1$ implies that $(hA)\cdot \tilde{L}_{h}(1,0,\dots,0)=\tau(c)h_{1}\in p\Z$, a contradiction. Therefore $\tilde{L}_{h}(e_{1}),\dots,\tilde{L}_{h}(e_{d-1}),(1,0,\dots,0)$ are $p$-linearly independent and thus $\det(B)\notin p\Z$.
     So $t'_{h}\in\frac{1}{q'_{h}}\Z^{d-1}$ for some $q'_{h}\in\N, p\nmid q'_{h}$. Since $t'_{h}\in\frac{1}{p}\Z^{d-1}$, we have that $t'_{h}\in\Z^{d-1}$.
   Therefore, we have that 
	 \begin{equation}\label{1:force511}
 		\begin{split}
 	t_{h}-c_{h}(hA)\in\Z^{d} \text{ for some $c_{h}\in\Z/p$	 for all $h\in J_{2}+pK\Z^{d}$.}
	\end{split}
 	\end{equation}

 Next we need to take a closer look at (\ref{1:force511}).
 Since the map $(b,c)\to\eta_{2}([b,c])$ is bilinear, the map $c\to\eta_{2}([g(e_{i}),c])$ is a homomorphism. So there exists $\xi=(\xi_{1},\dots,\xi_{d})\in (\R^{m})^{d}$ such that 
 $$\eta_{2}([g(e_{i}),x])=\xi_{i}\cdot\psi(x) \mod \Z$$
 for all $1\leq i\leq d$ and $x\in G$. 
  Similar to the discussion after (7.15) of \cite{GT12b}, all but the first $m_{\lin}$ coefficients of $\xi_{i}$ are non-zero, and $\vert\xi_{i}\vert\ll_{d,m}\d^{-O_{d,m}(1)}$.
  Denote $\gamma_{i}:=\psi(g(e_{i}))$ and $\gamma:=(\gamma_{1},\dots,\gamma_{d})\in (\R^{m})^{d}$. 
  Since $g$ is rational, it is clear that $\xi,\gamma\in(\Q^{m})^{d}$.

   We need the following variation of Lemma 4.3 of \cite{GT14} to continue the proof of Theorem \ref{1:sLei}.

\begin{lem}\label{1:gt43}
    Let $d,m,N\in\mathbb{N}_{+}$, $p$ be a prime, $\d>0,$ $\beta,\alpha_{1},\dots,\alpha_{d}\in\R$ and $\xi,\gamma_{0},\gamma_{1},\dots,\gamma_{d}\in\Q^{m}$. Let $V$ be a subset of $\{-N,\dots,N\}^{d}$ of cardinality at least $\d (2N+1)^{d}$ 
    such that
    $$\beta+\sum_{i=1}^{d}\alpha_{i}h_{i}+\xi\cdot\Bigl\{\gamma_{0}+\sum_{i=1}^{d}\gamma_{i}h_{i}\Bigr\}\in \Z$$
    for all $h=(h_{1},\dots,h_{d})\in V$. 
    If  $N\gg_{\d,d,m.p}1$,\footnote{We caution the readers that $N$ is dependent on   $p$. However, this will not cause any trouble for the application of this lemma. }
    then at least one of the following holds:
    \begin{enumerate}[(i)]
    \item there exists $r\in\Z$ with $0<r\ll_{d,m} \d^{-O_{d,m}(1)}$ such that $r\xi\in \Z^{m}$;
    \item there exists $\theta\in\Z^{m}$ with $0<\vert \theta\vert\ll_{d,m} \d^{-O_{d,m}(1)}$ such that $\theta\cdot \gamma_{i}\in \Z$ for all $1\leq i\leq d$.
    \end{enumerate}
\end{lem}
\begin{proof}
Since 
\begin{equation}\label{1:havqoeir}
\Bigl\{\gamma_{0}\Bigr\}+\Bigl\{\sum_{i=1}^{d}\gamma_{i}h_{i}\Bigr\}-\Bigl\{\gamma_{0}+\sum_{i=1}^{d}\gamma_{i}h_{i}\Bigr\}
\end{equation}
 takes at most $2^{m}$ different values, passing to a subset of $V$ if necessary, we may assume without loss of generality that (\ref{1:havqoeir}) take a same value for all $h\in V$. Absorbing constant terms by $\beta$ if necessary,  we may assume without loss of generality that $\gamma_{0}=\bold{0}$.

Since $$a\{b\}=\{a\}\{b\}+\lfloor a\rfloor b-\lfloor a\rfloor\lfloor b\rfloor$$ for all $a,b\in\R$, we have that 
$$\beta+\sum_{i=1}^{d}\alpha_{i}h_{i}+\xi\cdot\Bigl\{\sum_{i=1}^{d}\gamma_{i}h_{i}\Bigr\}\equiv \beta+\sum_{i=1}^{d}(\alpha_{i}+\lfloor\xi\rfloor\cdot \gamma_{i})h_{i}+\{\xi\}\cdot\Bigl\{\sum_{i=1}^{d}\gamma_{i}h_{i}\Bigr\}\mod \Z.$$
So it suffices to prove Lemma \ref{1:gt43} for the case $\xi\in[0,1)^{m}$.
By Lemma 4.3 of \cite{GT14},  either there exists $r\in\Z, 0<r\ll_{d,m}\d^{-O_{d,m}(1)}$ such that $\Vert r\xi_{i} \Vert_{\R/\Z}\ll_{d,m} \d^{-O_{d,m}(1)}/N$ for all $1\leq i\leq d$ (where $\xi=(\xi_{1},\dots,\xi_{m})$), or there exists $\theta\in\Z^{m}$ with $0<\vert \theta\vert\ll_{d,m} \d^{-O_{d,m}(1)}$ such that $\Vert \theta\cdot\gamma_{i}\Vert_{\R/\Z}\ll_{d,m} \d^{-O_{d,m}(1)}/N$ for all $1\leq i\leq d$. 

Since $\xi,\gamma_{1},\dots,\gamma_{d}\in\Q^{m}$, by choosing $N$ to be sufficiently large depending on $\d,d,m$ and $p$,
 we have that 
 	 either  $r\xi_{i}\in\Z$ for all $1\leq i\leq d$, or $\theta\cdot\gamma_{i}\in\Z$ for all $1\leq i\leq d$.
\end{proof}

  Denote $\gamma(h):=\sum_{j=1}^{d}\gamma_{j}h_{j}\in\Q^{m}$.
  Then the conclusion (ii) in Lemma \ref{1:gt43} can be reinterpreted as $\theta\cdot\gamma(h)\in \Z$ for all $h\in \Z^{d}$.
 Let $\{\gamma(h)\}\in[0,1)^{m}$ 
  be the coordinate-wise fraction part of   $\gamma(h)$.
 Then
 $$\sigma(h)=h\Lambda+\xi\ast \{\gamma(h)\},$$
 where $\ast$ is the coordinate-wise product of vectors in $\R^{m}$.
 So for all $h\in J_{2}+pK\Z^{d}$, it follow from (\ref{1:tthh}) and (\ref{1:force511}) that 
 \begin{equation}\label{1:reduction1}
 q(hB+u+\xi\ast \{\gamma(h)\})-c_{h}(hA)\in\Z^{d} \text{ for some } c_{h}\in\Z/p,
 \end{equation}
where
$B:=2\Omega+\Lambda$. 

 \begin{prop}\label{1:p222}
 	We have that either there exists $r\in\N_{+}$ with $r\ll_{d,m} \d^{-O_{d,m}(1)}$ such that $r\xi\in(\Z^{m})^{d}$, or there  exists $\theta\in\Z^{m_{\lin}}\times\{0\}^{m-m_{\lin}}$ with $0<\vert \theta\vert\ll_{d,m}\d^{-O_{d,m}(1)}$ such that $\theta\cdot\gamma_{i}\in\Z$ for all $1\leq i\leq d$.
 \end{prop}	
 \begin{proof}
  The idea of the proof is to  apply Lemma \ref{1:gt43}. Let $N\gg_{\d,d,m,p} 1$ to be chosen later.
 	For $w\in\F_{p}^{d}$, let   $V_{w}:=\{n\in \F_{p}^{d}\colon (nA)\cdot w=0\}$ and $Z_{w}=J_{2}\cap\iota^{-1}(V_{w})$.   Since 
	$$\sum_{w\in \V, wA\neq \bold{0}}\vert Z_{w}\vert\geq p^{d-1}\vert J_{2}\vert\gg_{d,m} \d^{O_{d,m}(1)}p^{d-1}(pK)^{d},$$ $\vert Z_{w}\vert\leq (pK)^{d}$,
	and $$\vert\{w\in \V, wA\neq \bold{0}\}\vert\geq p^{d}-p^{d-d'}\geq p^{d}-p^{d-2},$$
	by the Pigeonhole Principle,
 there exists a subset $W$ of $\{w\in \V, wA\neq \bold{0}\}$ of cardinality $\gg_{d,m}\d^{O_{d,m}(1)} p^{d}$ such that $\vert Z_{w}\vert\gg_{d,m} \d^{O_{d,m}(1)} p^{d-1}K^{d}$ for all $w\in W$. 
  
	For $w\in W$, $w'=(w'_{1},\dots,w'_{d})\in\iota^{-1}(w)$ and $h\in Z_{w}+pK\Z^{d}$, note that $(hA)\cdot w'\equiv \tau((\iota(h)A)\cdot w)=0 \mod p\Z$. 	Taking the dot product of both sides of (\ref{1:reduction1}) with $w'$,  we have that 
 	\begin{equation}\label{1:target}
 	hB\cdot qw'+u\cdot qw'+\Bigl(\sum_{i=1}^{d}qw'_{i}\xi_{i}\Bigr)\cdot \{\gamma(h)\}\in\Z^{d} \text{ for all $w\in W, w'\in\iota^{-1}(w)$ and $h\in Z_{w}+pK\Z^{d}$}.
 	\end{equation} 	
%
	%
	
	 	Since $wA\neq \bold{0}$,  $V_{w}$ is a  subspace of $\V$ of co-dimension 1. Let $L_{w}\colon\F_{p}^{d-1}\to V_{w}$ be any bijective linear transformation and let $\tilde{L}_{w}/p$ be the regular lifting of $L_{w}$.  Since all of $B,\xi_{i},\gamma$ are rational, there exists $K'\in\N_{+}$ such that for any $h,h'\in\Z^{d}$, the fractional part of the left hand side of (\ref{1:target}) remains unchanged if we replace $h$ by $h+pK'h'$.
	%
	%
		By Lemma \ref{1:298ruf}, the map 
		$\Phi\colon \Z^{d-1}\times\Z^{d}\to\Z^{d}$ given by 
		$$\Phi(z,m):=\tilde{L}_{w}(z)+pm$$
		mod $pK\Z^{d}$ is a   ${K'}^{d-1}$-to-1 map between $[pK']^{d-1}\times[K]^{d}$ and $\iota^{-1}(V_{w})\cap[pK]^{d}$.
		Since $\vert Z_{w}\vert=\vert \iota^{-1}(V_{w})\cap J_{2}\vert\gg_{d,m} \d^{O_{d,m}(1)} p^{d-1}K^{d}$, we have that $\vert\Phi^{-1}(Z_{w})\cap ([pK']^{d-1}\times[K]^{d})\vert\gg_{d,m} \d^{O_{d,m}(1)} (pK')^{d-1}K^{d}$.
		So
		there exists some $m_{w}\in [K]^{d}$ such that the set  
		$$I_{w}:=\{z\in [pK']^{d-1}\colon \Phi(z,m_{w})  \in Z_{w}+pK\Z^{d}\}$$ is of cardinality $\gg_{d,m} \d^{O_{d,m}(1)} (pK')^{d-1}$.
	So it follows from (\ref{1:target}) that
	 	\begin{equation}\label{1:target2}
 	\tilde{L}_{w}(z)B\cdot qw'+u+(pm_{w}B)\cdot qw'+\Bigl(\sum_{i=1}^{d}qw'_{i}\xi_{i}\Bigr)\cdot \{\gamma(\tilde{L}_{w}(z)+pm_{w})\}\in\Z^{d}
 	\end{equation} 	
	for all $z\in I_{w}$.	By the choice of $K'$, we have that (\ref{1:target2}) also holds for all $z\in I_{w}+pK'\Z^{d-1}$.  	
Note that	
	\begin{equation}\nonumber
	\begin{split}
	 \liminf_{N\to\infty}\frac{\vert (I_{w}+pK'\Z^{d-1})\cap \{-N,\dots,N\}^{d-1}\vert}{(2N+1)^{d-1}}
=\frac{\vert I_{w}\vert}{(pK')^{d-1}}\gg_{d,m}\d^{O_{d,m}(1)}.
	\end{split}
	\end{equation}

    Recall that the last $m-m_{\lin}$ coefficients of $\xi_{i}$ are zero and so the last $m-m_{\lin}$ coefficients of $\xi_{i}$ and $\gamma_{i}$  does not affect the  expression in (\ref{1:target2}).
   If $N\gg_{\d,d,m,p} 1$, then we may apply Lemma \ref{1:gt43} to (\ref{1:target2}) to conclude that   either there exists $r_{w}\in\Z, 0<r_{w}\ll_{d,m}\d^{-O_{d,m}(1)}$ such that $r_{w}q\sum_{i=1}^{d}w'_{i}\xi_{i}\in\Z^{m}$ for all $w'\in\iota^{-1}(w)$, or there exists $\theta_{w}\in\Z^{m_{\lin}}\times\{0\}^{m-m_{\lin}}$, $0<\vert \theta_{w}\vert\ll_{d,m}\d^{-O_{d,m}(1)}$ such that $\theta_{w}\cdot(\gamma\circ \tilde{L}_{w}(z))\in\Z$ for all $z\in \Z^{d-1}$ (or equivalently, $\theta_{w}\cdot\gamma(h)\in\Z$ for all $h\in V_{w}$). 
   
 	By the Pigeonhole Principle,   there exists a subset $W'$ of $W$ of cardinality $\gg_{d,m}\d^{O_{d,m}(1)} p^{d}$ such that either there exists $r\in\Z$ with $0<r\ll_{d,m}\d^{-O_{d,m}(1)}$ such that 
	$rq\sum_{i=1}^{d}w'_{i}\xi_{i}\in\Z^{m}$ for all $w'\in \iota^{-1}(W')$, or there exists $\theta\in\Z^{m}$ with $0<\vert \theta\vert\ll_{d,m}\d^{-O_{d,m}(1)}$ such that $\theta\cdot\gamma(h) \in \Z$ for all $h\in \tilde{V}_{w}:=\tilde{L}_{w}(\Z^{d-1})=\sp_{\Z}\{\tau(L_{w}(e_{1})),\dots,\tau(L_{w}(e_{d-1}))\}$.   
 	
 	 If it is the former case, then since $\iota^{-1}(W')=\tau(W')+p\Z^{d}$, we must have that $rq\xi_{i}\in\Z^{m}/p$ for all $1\leq i\leq d$. Since $p\gg_{d,m}\d^{-O_{d,m}(1)}$,  there exist $w_{1},\dots,w_{d}\in W'$ such that the determinant $Q$ of the matrix $(w_{i,j})_{1\leq i,j\leq d}$ is not divisible by $p$, where $w_{j}:=(w_{j,1},\dots,w_{j,d})$. Then the fact that $rq\sum_{i=1}^{d}\tau(w_{j,i})\xi_{i}\in\Z^{m}$ for $1\leq j\leq d$ implies that $rq\xi_{i}\in \Z^{m}/Q$. 
	 So  $rq\xi_{i}\in \frac{1}{Q}\Z^{m}\cap \frac{1}{p}\Z^{m}=\Z^{m}$ and thus $rq\xi\in (\Z^{m})^{d}$ (recall that $q\ll_{d}\d^{-O_{d}(1)}$).

 	 If it is the later case, then it is clear that there exist two $w,w'\in W'$ such that $\tilde{V}_{w}\neq \tilde{V}_{w'}$ and that $\theta\cdot\gamma(h)\in\Z$ for all $h\in \tilde{V}_{w}\cup \tilde{V}_{w}$. By linearity,   $\theta\cdot\gamma(h)\in\Z$ for all $h\in \tilde{V}_{w}+\tilde{V}_{w'}$. Since $\tilde{V}_{w}$ and $\tilde{V}_{w'}$ are distinct subspaces of $\Z^{d}$ of co-dimension 1 and $d\geq 2$, we have that $\tilde{V}_{w}+\tilde{V}_{w'}=\Z^{d}$. Equivalently, this means that $\theta\cdot\gamma_{i}\in\Z$ for all $1\leq i\leq d$.    
	 \end{proof}

 \subsection{Completion of the proof}

 We will use Proposition \ref{1:p222} to complete the proof of Theorem \ref{1:sLei}.
 Suppose first that
 there  exists $\theta\in\Z^{m_{\lin}}\times\{0\}^{m-m_{\lin}}$ with $0<\vert \theta\vert\ll_{d,m}\d^{-O_{d,m}(1)}$ such that $\theta\cdot\gamma_{i}\in\Z$ for all $1\leq i\leq d$. Consider the map $\eta\colon G\to\R/\Z$ defined by
 $$\eta(x):=\theta\cdot \psi(x).$$
 Then $\eta$ is a type-I horizontal character and is of complexity at most $O_{d,m}(\d^{-O_{d,m}(1)})$. Moreover, for all $n_{1},\dots,n_{d}\in\Z$, we have that 
 $$\eta(g(n_{1},\dots,n_{d}))=\eta(g(e_{1})^{n_{1}}\dots g(e_{d})^{n_{d}})=\sum_{i=1}^{d}n_{i}\theta\cdot\gamma_{i}\equiv 0 \mod \Z.$$
 This 
 completes the proof of Theorem \ref{1:sLei}.

 We now consider the case   $r\xi\in(\Z^{m})^{d}$ for some $r\in\N_{+}$ with $r\ll_{d,m} \d^{-O_{d,m}(1)}$. For $1\leq j\leq m$, let $\alpha_{j}\colon G\to\R$ be given by $\alpha_{j}(x):=\eta_{2}([x,\exp(X_{j})])$.
 Note that $\alpha_{j}$ is a type-I horizontal character,  annihilates $G_{2}$, and $\Vert\alpha_{j}\Vert\ll_{d,m}\d^{-O_{d,m}(1)}$. If $\alpha_{j}$ is nontrivial for some $1\leq j\leq m$, then for all $n_{1},\dots,n_{d}\in\Z$, we have that
 $$r\alpha_{j}\circ g(n_{1},\dots,n_{d})=r\alpha_{j}(g(e_{1})^{n_{1}}\dots g(e_{d})^{n_{d}})=\sum_{i=1}^{d}n_{i}r\xi_{i}\cdot\psi(X_{j})\equiv 0 \mod \Z,$$
 and we again proves Theorem \ref{1:sLei}.

 Now suppose that all of  $\alpha_{j}$ are trivial. In this case $\eta_{2}$ annihilates $[G,G]$. 
 So we have that $\Lambda$ is the zero matrix and that $\xi$ is the zero vector. Therefore $\sigma(h)=0$. 
  It follows from (\ref{1:force511}) that
 \begin{equation}\label{1:reduction2}
 t_{h}-c_{h}(hA)=q(2h\Omega+u)-c_{h}(hA)\in\Z^{d}
 \end{equation}
 for some $c_{h}\in\Z/p$ for all $h\in J_{2}+pK\Z^{d}$.
 This implies that $t_{h}=q(2h\Omega+u)\in\Z^{d}/p$ for all $h\in J_{2}+pK\Z^{d}$, which further implies that $q\Omega$ and $qu$ takes $\Z/p$ coefficients.
 So for all $h\in J_{2}+pK\Z^{d}$ and $n\in \Z^{d}$ with $(nA)\cdot w\in p\Z^{d}$, (\ref{1:reduction2}) implies that 
 $q(2h\Omega+u)\cdot n\in\Z.$
Since $J_{2}$ contains  $\gg_{d,m}\d^{O_{d,m}(1)}p^{d}$ residue classes mod $p\Z$, it follows from
  Corollary \ref{1:spe2} that $qu\in\Z$ and $2q\Omega=rA+B'$ for   or some $r\in\Z/p$ and matrix $B'$ with integer entries.   It then follows from (\ref{1:forcemid}) and (\ref{1:force51}) that (replacing $q$ by $2q$ if necessary)
  \begin{equation}\label{1:reductionssa}
\eta_{2}(g_{\nlin}(n))=Q(n)=\frac{1}{q}Q_{2}(n)+T(n)
 \end{equation}
 for some 
  $Q_{2}\in\poly(V_{p}(\tilde{M})\to\R\vert\Z)$ and some polynomial $T\colon\Z^{d}\to\R$ of degree at most 1.

 If $\eta_{2}$ is trivial, then $Q\equiv 0$    and   $\Omega$ can be taken to be the zero matrix.  It follows from (\ref{1:reduction2}) that 
 $$qu-c_{h}(hA)\in\Z^{d}$$
 for some $c_{h}\in\Z/p$ for all $h\in J_{2}$. 
 Note that if $qu\notin\Z^{d}$, then the number of such $h$ is at most $K^{d}\cdot p^{d+1-\rank_{p}(A)}\leq K^{d}p^{d-1}<\vert J_{2}\vert$, a contradiction.
 So we must have that $qu\in\Z^{d}$. So it follows from (\ref{1:forcemid}) that if $M(n)\in\Z$, then
  $$q\eta_{1}(g(n))=qP(n)=P_{1}(n)+(qu)\cdot n+qu_{0}\in \Z+qu_{0}.$$
 Since $\eta_{0}(g',g)=\eta_{1}(g)+\eta_{2}(g'g^{-1})$ is nontrivial and $\eta_{2}$ is trivial, we have that  $\eta_{1}$
 is nontrivial. Since $q\ll_{d}\d^{-O_{d}(1)}$,
  this completes the proof of Theorem \ref{1:sLei} by setting $\eta=q\eta_{1}$.

  So from now on we assume that $\eta_{2}$ is not trivial.
 If $m_{\ast}=0$, then $\eta_{2}$ vanishes on $G_{2}$ and so it is trivial, a contradiction.  
 We now assume that $m_{\ast}\geq 1$.   Since $\eta_{2}$   annihilates $[G,G]$, we may extend $\eta_{2}$ to a genial type-I horizontal character $\tilde{\eta}_{2}$ on $G$ in an arbitrary way such that $\tilde{\eta}_{2}\vert_{G_{2}}=\eta_{2}$, and that the complexity of $\tilde{\eta}_{2}$ is comparable to $\eta_{2}$. With a slight abuse of notation, we denote $\tilde{\eta}_{2}$ by $\eta_{2}$ as well and treat $\eta_{2}$ as a genial type-I horizontal character on $G$. 
 Suppose that $\eta_{2}\colon G\to\R$ is given by $\eta_{2}(x)=\theta\cdot \psi(x)$ for some $\theta=(\theta_{1},\dots,\theta_{m})\in\Z^{m}$ with $\vert \theta\vert\ll_{d,m}\d^{-\O_{d,m}(1)}$.
 Since $\eta_{2}$ annihilates $[G,G]$, we have that $\theta_{i}=0$ for all $i>m_{ab}$.
  
  Denote $G'_{0}=G'_{1}=G$ and $G'_{i}=G_{i}\cap\ker(\eta_{2})$ for $i\geq 2$.  
  By Lemma 7.8 of \cite{GT12b}, $G_{\N}'=(G'_{i})_{i\in\N}$ is a filtration of $G$ with degree at most $s$ and nonlinearity degree $m'_{\ast}\leq m_{\ast}-1$. Each $G'_{i}$ is closed, connected and $O_{d,m}(\d^{-\O_{d,m}(1)})$-rational with respect to the Mal'cev basis $\X$ on $G/\Gamma$ adapted to $G_{\N}$. We need a special factorization proposition:
  
  \begin{prop}\label{1:specialfactorization}
  Fix any $n_{\ast}\in V_{p}(\tilde{M})$.
  	There exists $r\in\N_{+}$ with $r\ll_{d,m}\d^{-O_{d,m}(1)}$ such that we may write $g=g'\gamma$ for some rational $g'\in\poly(\Z^{d}\to G'_{\N})$ and some   $\gamma\in \poly_{\approx r, n_{\ast}}(V_{p}(\tilde{M})\to G_{\N})$.
  \end{prop}

   We first complete the proof of Theorem \ref{1:sLei} assuming Proposition \ref{1:specialfactorization}. 
   Let the notations be the same as Proposition \ref{1:specialfactorization}. By Proposition \ref{1:oehf}, there exists $r'\in\N_{+}$ with $r'\ll_{d,m}\d^{-O_{d,m}(1)}$ such that $\gamma\in \poly_{r'}(V_{p}(\tilde{M})\to G_{\N})$.
   %
   %
   %
    For all $m\in [r']^{d}$,  there exists $\tilde{m}\in\Z^{d}$ with $\tilde{m}-m\in r'\Z^{d}$ and $\tilde{m}-n_{\ast}\in p\Z^{d}$. Since  $n_{\ast}\in V_{p}(\tilde{M})$, we have $\tilde{m}\in V_{p}(\tilde{M})$. 
    Since  $\gamma\in \poly_{r'}(V_{p}(\tilde{M})\to G_{\N})$, we have that
    \begin{equation}\label{1:fkowp}
    \text{$\gamma(\tilde{m})\Gamma=\gamma(r'n+m)\Gamma$ for all $n\in\Z^{d}$ with $r'n+m\in V_{p}(\tilde{M})$}.
    \end{equation}
    
	 Denote
	 $$g'_{m}(n):=\{\gamma(\tilde{m})\}^{-1}g'(r'n+m)\{\gamma(\tilde{m})\}.$$
   Since $g'$ and $\gamma$ are rational, by Corollary \ref{1:r222p}, there exists $K'\in\N_{+}$ which is divisible by $pK$ such that $g'$ and $g'_{m}, m\in [r']^{d}$ are all $K'$-periodic.
   %
    For  $m\in\Z^{d}$, let $M_{m}\colon\Z^{d}\to\Z/p$ be the quadratic map given by $\tilde{M}_{m}(n):=\tilde{M}(r'n+m)$. Then since $r\ll_{d,m}\d^{-O_{d,m}(1)}$, $\tilde{M}_{m}$ and $\tilde{M}$ have the same $p$-rank. Note that for all $n\in V_{p}(\tilde{M}_{m})$, we have that $r'n+m\in V_{p}(\tilde{M})$.
   Then it follows from (\ref{1:fkowp}) and  (\ref{1:e4.2}) that
 \begin{equation}\nonumber
 \begin{split}
&\quad
 \d^{O_{d,m}(1)}\ll_{d,m}
 \vert\E_{n\in V_{p}(\tilde{M})\cap[pK]^{d}}F(g(n)\Gamma)\vert
 =\vert\E_{n\in V_{p}(\tilde{M})\cap[r'K']^{d}}F(g(n)\Gamma)\vert
 \\&=\vert\E_{m\in [r']^{d}}\E_{n\in V_{p}(\tilde{M}_{m})\cap [K']^{d}}F(g(r'n+m)\Gamma)\vert+O(p^{-1/2})
  \\&=\vert\E_{m\in [r']^{d}}\E_{n\in V_{p}(\tilde{M}_{m})\cap [K']^{d}}F(g'(r'n+m)\gamma(r'n+m)\Gamma)\vert+O(p^{-1/2})
    \\&=\vert\E_{m\in [r']^{d}}\E_{n\in V_{p}(\tilde{M}_{m})\cap [K']^{d}}F(g'(r'n+m)\gamma(\tilde{m})\Gamma)\vert+O(p^{-1/2})
 \\&=\vert\E_{m\in [r']^{d}}\E_{n\in V_{p}(\tilde{M}_{m})\cap [K']^{d}}F(\{\gamma(\tilde{m})\}g'_{m}(n)\Gamma)\vert+O(p^{-1/2}).
 \end{split}
 \end{equation}
    So by the Pigeonhole Principle, there exists some $m\in [r']^{d}$ such that 
    $$\vert\E_{n\in V_{p}(\tilde{M}_{m})\cap [K']^{d}}F'_{m}(g'_{m}(n)\Gamma)\vert\gg_{d,m}\d^{O_{d,m}(1)}.$$
   %
 where 
 $F'_{m}:=F(\{\gamma(\tilde{m})\}\cdot)$. By the construction of $K'$, $g'_{m}$ is a rational and $K'$-periodic polynomial  in $\poly(\Z^{d}\to G'_{\N})$
 and $\Vert F'\Vert_{\Lip}=O(1)$.  We may now invoke the induction hypothesis to deduce that there exists a nontrivial type-I horizontal character $\eta$ of complexity at most $O_{d,m}(\d^{-O_{d,m}(1)})$ such that 
 $\eta\circ g'_{m}\mod \Z$ is a constant on $V_{p}(\tilde{M}_{m})$. 
 Since 
 $$\eta\circ g'_{m}(n)=\eta\circ g'(r'n+m)=\eta\circ g(r'n+m)-\eta\circ \gamma(r'n+m)\equiv \eta\circ g(r'n+m)-\eta\circ \gamma(m)\mod \Z$$
 for all $n\in V_{p}(\tilde{M}_{m})$, we have that   $\eta\circ g\mod \Z$ is a constant on $V_{p}(\tilde{M})$. This finishes the proof of Theorem \ref{1:sLei}.
 
 So it remains to prove Proposition \ref{1:specialfactorization}.
  
 \begin{proof}[Proof of Proposition \ref{1:specialfactorization}]
 	Let $0\leq s'\leq s$ be the largest integer such that at least one of the entries $\theta_{m-m_{s'}+1},\dots,\theta_{m-m_{s'+1}}$ is nonzero (where we denote $m_{0}:=0$).
 Since 
 $\eta_{2}$ is nontrivial, such an $s'$ always exists and $s'\geq 2$. Denote  $m':=m-m_{s'+1}$.
 Then we may write $\theta=(\theta',\bold{0})$ for some $\theta'\in \Z^{m'}$. 
  
  Assume that 
  $$\psi(g(n))=\sum_{i\in\N^{d},0\leq\vert i\vert\leq s}w_{i}\binom{n}{i},$$
  for some $w_{i}\in\{0\}^{m-m_{\vert i\vert}}\times\R^{m_{\vert i\vert}}$. 
 Since $Q(\bold{0})=Q(e_{i})=0$, we may write
 $$\psi(g_{\nlin}(n))=\sum_{m\in\N^{d},2\leq \vert i\vert\leq s}t_{i}\binom{n}{i}$$
 for some $t_{i}\in\{0\}^{m-m_{\vert i\vert}}\times\R^{m_{\vert i\vert}}$. Write $w_{i}=(w'_{i},w''_{i}), t_{i}=(t'_{i},t''_{i})$ for some $t'_{i},w'_{i}\in\R^{m'}$ and $t''_{i},w''_{i}\in\R^{m-m'}$. 
 	Since the first $m-m_{\vert i\vert}$ entries of $t_{i}$ are zero, and the last $m_{s'+1}$ entries of $\theta$ are zero, we have that $\theta\cdot t_{i}=0$ if $\vert i\vert\geq s'+1$. 
 	By (\ref{1:reductionssa}),
 $$\frac{1}{q}Q_{2}(n)+T(n)=\eta_{2}(g_{\nlin}(n))=\sum_{i\in\N^{d},2\leq \vert i\vert\leq s}(\theta\cdot t_{i})\binom{n}{i}=\sum_{i\in\N^{d},2\leq \vert i\vert\leq s'}(\theta'\cdot t'_{i})\binom{n}{i}.$$
 In particular, we have that $\deg(Q_{2})\leq s'$.

 Since $\theta$ is of complexity  $\ll_{d,m}\d^{-\O_{d,m}(1)}$, there exists $q'\in\N_{+}, q'\ll_{d,m}\d^{-\O_{d,m}(1)}$ and $u=(u_{1},\dots,u_{m})\in\Z^{m}$ whose only nonzero entries are $u_{m-m_{s'}+1},\dots,u_{m-m_{s'+1}}$ such that 
 $\theta\cdot u=q'.$ Denote $$\gamma(n):=\prod_{i=m-m_{s'}+1}^{m-m_{s'+1}}\exp\Bigl(\frac{1}{qq'}u_{i}(Q_{2}(n)-Q_{2}(n_{\ast}))X_{i}\Bigr).$$
 Since $\deg(Q_{2})\leq s'$, it follows from Corollary B.4 of \cite{GTZ12} that $\gamma\in\poly(\Z^{d}\to (G_{2})_{\N})$.


 We now show that there exists $r\in\N_{+}$ with $r\ll_{d,m}\d^{-\O_{d,m}(1)}$ such that $\gamma$ belongs to $\poly_{\approx r, n_{\ast}}(V_{p}(\tilde{M})\to (G_{2})_{\N}\vert(\Gamma\cap G_{2})).$
  In fact, it suffices to show that $Q_{2}(n_{\ast}+rn)-Q_{2}(n_{\ast})\in qq'\Z$ for all $n\in\Z^{d}$ with $n_{\ast}+rn\in V_{p}(\tilde{M})$. By interpolation, all the coefficients of the polynomial $Q_{2}$  belongs to $\Z/p^{s'}q''$ for some $s'=O_{d,s}(1)$ and $q''\in\N_{+}$ with $q''\leq O_{d,s}(1)$.  So for all $n\in\Z^{d}$, we have that $Q_{2}(n_{\ast}+rn)-Q_{2}(n_{\ast})\in qq'\Z/p^{s'}$, where $r:=qq'q''$.  On the other hand, if $n_{\ast}+rn, n_{\ast}\in V_{p}(\tilde{M})$, then $Q_{2}(n_{\ast}+rn)-Q_{2}(n_{\ast})\in \Z$. 
 Since $qq'$ is not divisible by $p$, this implies that $Q_{2}(n_{\ast}+rn)-Q_{2}(n_{\ast})\in qq'\Z$. So
 $\gamma\in\poly_{\approx r, n_{\ast}}(V_{p}(\tilde{M})\to (G_{2})_{\N}\vert(\Gamma\cap G_{2})).$


Writing $h:=g\gamma^{-1}$, we have that $h\in\poly(\Z^{d}\to (G_{2})_{\N}\vert(\Gamma\cap G_{2}))$.
Note that for all $n,h_{1},h_{2}\in \Z^{d}$,
 \begin{equation}\label{1:ddd1}
 \begin{split}
 &\quad \Delta_{h_{2}}\Delta_{h_{1}}(\eta_{2}\circ g(n))
 =\eta_{2}\circ(\Delta_{h_{2}}\Delta_{h_{1}}g)(n)=\eta_{2}\circ(\Delta_{h_{2}}\Delta_{h_{1}}g_{\nlin})(n)
 \\&=\Delta_{h_{2}}\Delta_{h_{1}}(\eta_{2}\circ g_{\nlin})(n)
 =\Delta_{h_{2}}\Delta_{h_{1}}(\frac{1}{q}Q_{2}(n)+T(n))=\Delta_{h_{2}}\Delta_{h_{1}}(\frac{1}{q}Q_{2}(n)).
 \end{split}
 \end{equation}
 So (\ref{1:ddd1}) implies that 
 \begin{equation}\label{1:ddd3}
 \begin{split}
 &\quad \Delta_{h_{2}}\Delta_{h_{1}}(\eta_{2}\circ h(n))
=\Delta_{h_{2}}\Delta_{h_{1}}(\eta_{2}\circ g(n))-\Delta_{h_{2}}\Delta_{h_{1}}(\eta_{2}\circ \gamma(n))
=0.
 \end{split}
 \end{equation}

 Set 
  $g':=g\gamma^{-1}=g_{\nlin}g_{\lin}\gamma^{-1}$. 
  It is clear that $g'\in\poly(\Z^{d}\to G_{\N})$. 
 Moreover, by (\ref{1:ddd3}), 
 $$\Delta_{h_{2}}\Delta_{h_{1}}(\eta_{2}\circ g')(n)=\eta_{2}(\Delta_{h_{2}}\Delta_{h_{1}}g'(n))=0$$
 for all $n,h_{1},h_{2}\in\Z^{d}$.
  So it is not hard to see that $g'$ belongs to $\poly(\Z^{d}\to G'_{\N})$.
  Finally, since $g'=g\gamma^{-1}$, it follows from the Baker-Campbell-Hausdorff formula    that $g'$ is rational. 
 \end{proof}

     We end this section with an immediate consequence of Theorem \ref{1:sLei}, which will be used in \cite{SunC}.
     
  \begin{coro} 
  	Let $0<\d<1/2, d,k,m\in\N_{+},s,r\in\N$ with $d\geq r$, and $p\gg_{d,m} \d^{-O_{d,m}(1)}$ be a prime. Let $M\colon\V\to\F_{p}$ be a quadratic form  and $V+c$ be an affine subspace of $\V$ of co-dimension $r$. Suppose that $\rank(M\vert_{V+c})\geq s+13$.  Let $G/\Gamma$ be an $s$-step $\N$-filtered nilmanifold of dimension $m$, equipped with an $\frac{1}{\d}$-rational Mal'cev basis $\mathcal{X}$, and that  $g\in \poly(\Z^{d}\to G_{\N})$ be a rational polynomial sequence. Let  $\gamma\in G$ be an element of complexity at most 1, and let $g'\in \poly(\Z^{d}\to G'_{\N})$
	be the map given by
	 $g'(n):=\gamma^{-1}g(n)\gamma$ for all $n\in\Z^{d}$, where $G'_{i}:=\gamma^{-1} G_{i}\gamma$ for all $i\in\N$. Denote $\Gamma':=\gamma^{-1}\Gamma\gamma$.
	 If  $(g(n)\Gamma)_{n\in \iota^{-1}(V(M)\cap(V+c))}$ is not $\d$-equidistributed on $G/\Gamma$, then $(g'(n)\Gamma')_{n\in \iota^{-1}(V(M)\cap(V+c))}$ is not $C^{-1}\d^{C}$-equidistributed on $G'/\Gamma'$ for some $C=C(d,m)\geq 1$.
	   \end{coro}
\begin{proof}
   By Theorem \ref{1:sLei}, there exists a nontrivial type-I horizontal character $\eta$ with $0<\Vert\eta\Vert\ll_{d,m} \d^{-O_{d,m}(1)}$  such that $\eta\circ g'=\eta\circ g \mod \Z$  is a constant on $\iota^{-1}(V(M)\cap(V+c))$.
By Lemma A.13 of \cite{GT12b}, $G'/\Gamma'$ is an $s$-step $\N$-filtered nilmanifold of dimension $m$, equipped with an $\frac{1}{\d^{O(1)}}$-rational Mal'cev basis. Since $\vert\E_{n\in \iota^{-1}(V(M)\cap(V+c))}\exp(\eta(g'(n)))\vert=1$, this implies that $(g'(n)\Gamma')_{n\in \iota^{-1}(V(M)\cap(V+c))}$ is not $C^{-1}\d^{C}$-equidistributed on $G'/\Gamma'$ for some $C=C(d,m)\geq 1$.
\end{proof}

\section{Sets and equations generated by quadratic forms}\label{1:s:mset}

\subsection{$M$-sets}

Our next goal is to extend Theorem \ref{1:sLei} from $V(M)$ to the intersection of the zeros of many quadratic polynomials. When working with a subset of $(\V)^{k}$ or $(\Z^{d})^{k}$,  it is convenient to group the variables $(n_{1},\dots,n_{dk})$ into $k$ groups $$(n_{1},\dots,n_{d}),(n_{d+1},\dots,n_{2d}),\dots,(n_{d(k-1)+1},\dots,n_{k})$$ and treat each tuple $(n_{d(j-1)+1},\dots,n_{j})$ as a single multi-dimensional variable. To this end, we introduce some definitions reflecting such a point of view.

We say that a linear transformation  $L\colon(\V)^{k}\to (\V)^{k'}$ is \emph{$d$-integral} if there exist $a_{i,j}\in\F_{p}$ for $1\leq i\leq k$ and $1\leq j\leq k'$ such that 
$$L(n_{1},\dots,n_{k})=\Bigl(\sum_{i=1}^{k}a_{i,1}n_{i},\dots,\sum_{i=1}^{k}a_{i,k'}n_{i}\Bigr)$$
for all $n_{1},\dots,n_{k}\in \V$. 
Let $L\colon(\V)^{k}\to \V$ be a $d$-integral linear transformation given by
$L(n_{1},\dots,n_{k})=\sum_{i=1}^{k}a_{i}n_{i}$ for some  $a_{i}\in\F_{p},1\leq i\leq k$. We say that $L$ is the $d$-integral linear transformation \emph{induced} by $(a_{1},\dots,a_{k})\in\F_{p}^{k}$.

Similarly, we say that a linear transformation  $L\colon(\Z^{d})^{k}\to (\frac{1}{p}\Z^{d})^{k'}$ is \emph{$d$-integral} if there exist $a_{i,j}\in\Z/p$ for $1\leq i\leq k$ and $1\leq j\leq k'$ such that 
$$L(n_{1},\dots,n_{k})=\Bigl(\sum_{i=1}^{k}a_{i,1}n_{i},\dots,\sum_{i=1}^{k}a_{i,k'}n_{i}\Bigr)$$
for all $n_{1},\dots,n_{k}\in \Z^{d}$. 

There is a natural correspondence between $d$-integral linear transformations in $\Z/p$ and those in $\F_{p}$. 
The proof of the following lemma is similar to the proof of Lemma \ref{1:lifting} and we omit the details:

\begin{lem}\label{1:lifting2}
Every $d$-integral linear transformation  $\tilde{L}\colon(\Z^{d})^{k}\to (\frac{1}{p}\Z^{d})^{k'}$ induces a $d$-integral linear transformation  $L\colon(\V)^{k}\to (\V)^{k'}$. Conversely, every $d$-integral linear transformation  $L\colon(\V)^{k}\to (\V)^{k'}$ admits a regular lifting $\tilde{L}\colon(\Z^{d})^{k}\to (\frac{1}{p}\Z^{d})^{k'}$ which is a $d$-integral linear transformation.

Moreover, we have that $p\tilde{L}\circ \tau\equiv \tau\circ L \mod p(\Z^{d})^{k'}$.
\end{lem}

Let $d\in\N_{+}$, $p$  be a prime, $M\colon\V\to\F_{p}$ be a quadratic form with $A$ being the associated matrix. We need to work with a large family of sets generated by $M$.
 We say that a function $F\colon(\V)^{k}\to\F_{p}$ is an \emph{$(M,k)$-integral quadratic function} 
if 
\begin{equation}\label{1:thisisf2}
F(n_{1},\dots,n_{k})=\sum_{1\leq i\leq j\leq k}b_{i,j}(n_{i}A)\cdot n_{j}+\sum_{1\leq i\leq k} v_{i}\cdot n_{i}+u
\end{equation}
for some $b_{i,j}, u\in \F_{p}$ and $v_{i}\in\F_{p}^{d}$.
We say that an $(M,k)$-integral quadratic function $F$ is \emph{pure}
if $F$ can be written in the form of (\ref{1:thisisf2}) with $v_{1}=\dots=v_{k}=\bold{0}$. 
We say that  an $(M,k)$-integral quadratic function $F\colon(\V)^{k}\to\F_{p}$ is \emph{nice}
if 
\begin{equation}\nonumber
F(n_{1},\dots,n_{k})=\sum_{1\leq i\leq k'}b_{i}(n_{k'}A)\cdot n_{i}+u
\end{equation}
for some $0\leq k'\leq k$, $b_{i}, u\in \F_{p}$.
 
For $F$ given in (\ref{1:thisisf2}), denote 
$$v_{M}(F):=(b_{k,k},b_{k,k-1},\dots,b_{k,1},v_{k},b_{k-1,k-1},\dots,b_{k-1,1},v_{k-1},\dots,b_{1,1},v_{1},u)\in\F_{p}^{\binom{k+1}{2}+kd+1},$$
and
$$v'_{M}(F):=(b_{k,k},b_{k,k-1},\dots,b_{k,1},v_{k},b_{k-1,k-1},\dots,b_{k-1,1},v_{k-1},\dots,b_{1,1},v_{1})\in\F_{p}^{\binom{k+1}{2}+kd}.$$
Informally, we say that $b_{i,j}$ is the $n_{i}n_{j}$-coefficient, $v_{i}$ is the $n_{i}$-coefficient, and $u$ is the constant term coefficient for these vectors.

\begin{defn}[Different types of $(M,k)$-families]
An \emph{$(M,k)$-family} is a collections of $(M,k)$-integral quadratic functions.
Let $\mathcal{J}=\{F_{1},\dots,F_{r}\}$ be an $(M,k)$-family.
\begin{itemize}
    \item We say that $\mathcal{J}$ is \emph{pure} if all of $F_{1},\dots,F_{r}$ are pure.
    \item We say that $\mathcal{J}$ is \emph{consistent} if $(0,\dots,0,1)$ does not belong to the span of $v_{M}(F_{1}),$ $\dots,$ $v_{M}(F_{r})$, or equivalently, there is no linear combination of $F_{1},\dots,F_{r}$ which is a constant nonzero function, or equivalently, for all $c_{1},\dots,c_{r}\in\F_{p}$, we have
    $$c_{1}v'_{M}(F_{1})+\dots+c_{r}v'_{M}(F_{r})=\bold{0}\Rightarrow c_{1}v_{M}(F_{1})+\dots+c_{r}v_{M}(F_{r})=\bold{0}.$$
    \item We say that $\mathcal{J}$ is \emph{independent} if $v'_{M}(F_{1}),\dots,v'_{M}(F_{r})$ are linearly independent, or equivalently, there is no nontrivial linear combination of $F_{1},\dots,F_{r}$ which is a constant function, or equivalently, for all $c_{1},\dots,c_{r}\in\F_{p}$, we have
    $$c_{1}v'_{M}(F_{1})+\dots+c_{r}v'_{M}(F_{r})=\bold{0}\Rightarrow c_{1}=\dots=c_{r}=0.$$
\item We say that $\mathcal{J}$ is \emph{nice} if there exist some bijective $d$-integral  linear transformation $L\colon(\V)^{k}\to(\V)^{k}$ and some $v\in(\V)^{k}$ such that $F_{i}(L(\cdot)+v)$ is nice for all $1\leq i\leq r$.
\end{itemize}

The dimension of the span of $v'_{M}(F_{1}),\dots,v'_{M}(F_{r})$ is called the \emph{dimension} of an $(M,k)$-family $\{F_{1},\dots,F_{r}\}$.

When there is no confusion, we call an $(M,k)$-family to be an \emph{$M$-family} for short.
\end{defn}

We say that a subset $\Omega$ of $(\V)^{k}$ is an  \emph{$M$-set} if there exists an $(M,k)$-family $\{F_{i}\colon (\V)^{k}\to\V\colon 1\leq i\leq r\}$
 such that $\Omega=\cap_{i=1}^{r}V(F_{i})$. 
 We call either $\{F_{1},\dots,F_{r}\}$ or the ordered set $(F_{1},\dots,F_{r})$ an \emph{$M$-representation} of $\Omega$. 
 
  Let $\P\in\CP$.
We say that   $\Omega$  is $\P$  if one can choose the  $M$-family $\{F_{1},\dots,F_{r}\}$ to be $\P$.
 We say that the $M$-representation $(F_{1},\dots,F_{r})$ is $\P$ if $\{F_{i}\colon 1\leq i\leq r\}$ is $\P$.
  We say that $r$ is the \emph{dimension} of the $M$-representation $(F_{1},\dots,F_{r})$.
  The \emph{total co-dimension} of a consistent $M$-set $\Omega$, denoted by $r_{M}(\Omega)$, is the minimum of the dimension of the independent $M$-representations of $\Omega$.\footnote{We will show in Proposition \ref{1:yy33} that $r_{M}(\Omega)$ is independent of the choice of the independent $M$-representation if $d$ and $p$ are sufficiently large.}

  	We say that an independent $M$-representation $(F_{1},\dots,F_{r})$ of $\Omega$ is \emph{standard} if
	the matrix 
		$\begin{bmatrix}
		v_{M}(F_{1})\\
		\dots \\
		v_{M}(F_{r})
		\end{bmatrix}$ is in the reduced row echelon form (or equivalently, the matrix 
		$\begin{bmatrix}
		v'_{M}(F_{1})\\
		\dots \\
		v'_{M}(F_{r})
		\end{bmatrix}$ is in the reduced row echelon form).
	If $(F_{1},\dots,F_{r})$ is a standard $M$-representation of $\Omega$, then we may relabeling $(F_{1},\dots,F_{r})$ as $$(F_{k,1},\dots,F_{k,r_{k}},F_{k-1,1},\dots,F_{k-1,r_{k-1}},\dots,F_{1,1},\dots,F_{1,r_{1}})$$
  for some $r_{1},\dots,r_{k}\in \N$ such that $F_{i,j}$ is non-constant with respect to $n_{i}$ and is independent of $n_{i+1},\dots,n_{k}$.
	We also call $(F_{i,j}\colon 1\leq i\leq k, 1\leq j\leq r_{i})$	a \emph{standard $M$-representation} of $\Omega$. The vector $(r_{1},\dots,r_{k})$ is called the \emph{dimension vector} of this representation.  
	
	\begin{conv}\label{1:fpism}
	We allow $(M,k)$-families, $M$-representations and   dimension vectors to be empty. In particular, $(\V)^{k}$ is considered as a nice and consistent $M$-set with total co-dimension zero. 
	\end{conv}

\begin{ex}
Let $d,s\in\N_{+}$, $p$ be a prime, $M\colon\V\to\F_{p}$ be a quadratic form associated with the matrix $A$,  and $h_{1},\dots,h_{s}\in\V$. Then the set $V(M)^{h_{1},\dots,h_{s}}$ defined in Definition \ref{1:vh} is an $M$-set with a dimension vector $(s+1)$. Moreover, it is consistent if $h_{1}A,\dots,h_{s}A$ are linearly independent. On the other hand, $V(M)^{h_{1},\dots,h_{s}}$ is not a nice $M$-set. 
\end{ex}

\begin{ex}\label{1:rreepp}
	Let $M\colon\V\to\F_{p}$ be a quadratic form associated with the matrix $A$ and $s\in\N$. 
	It follows from Lemma \ref{1:changeh} that the Gowers set $\Gow_{s}(V(M))$ is a consistent $M$-set with a dimension vector $(1,1,\dots,s)$, whose total co-dimension is $(s^{2}+s+2)/2$.
	On the other hand, let $L(n,h_{1},\dots,h_{s}):=(h_{1},\dots,h_{s},n)$. Then it follows from Lemma \ref{1:changeh} that $L(\Gow_{s}(V(M)))$ is a consistent $M$-set with a dimension vector $(0,\dots,s-1,s+1)$, whose total co-dimension is also $(s^{2}+s+2)/2$. We will show later in Example \ref{1:mainex} that these $M$-sets are nice when $M$ is non-degenerate.
\end{ex}

  	We refer the readers to Appendix \ref{1:s:AppB} for some basic properties as well as more examples of $M$-sets, and    to Appendix \ref{1:s:dec} for the Fubini's theorem on $M$-sets (which will be used extensively in \cite{SunC,SunD}).

\subsection{Irreducible properties for sets generated by quadratic forms}\label{1:s:9}

In Section \ref{1:ss91}, we will boost Theorem \ref{1:sLei} to cover partially $p$-periodic polynomials on a large class of $M$-sets. Before generalizing  Theorem \ref{1:sLei}, we need to address a technical issue on the irreducible property for  $M$-sets: if $P$ is a polynomial whose set of solutions has a ``large" intersection  with some $M$-set $\Omega$, then does it imply that all the elements in $\Omega$ are solutions to $P$? 
To answer this question, we start with the following definition:

\begin{defn}[$\d$-irreducible sets]
	Let $k\in\N_{+},s\in\N, \d>0$ and $p$ be a prime. We say that a subset $\Omega$ of $\F_{p}^{k}$ is \emph{$\d$-irreducible up to degree $s$} if for all $P\in\poly(\F_{p}^{k}\to\F_{p})$ of degree at most $s$, either $\vert V(P)\cap \Omega\vert\leq \d\vert\Omega\vert$ or $\Omega\subseteq V(P)$.

	We say that a $p$-periodic subset   $\Omega$ of $\Z^{k}$  is \emph{weakly $(\d,p)$-irreducible up to degree $s$} if for all $P\in\poly(\Z^{k}\to\Z/p)$ of degree at most $s$, 
	either $\vert V_{p}(P)\cap \Omega\cap [p]^{d}\vert\leq\d\vert\Omega\cap [p]^{d}\vert$ or 
	$\Omega\subseteq V_{p}(P).$
	We say that $\Omega$   is \emph{strongly $(\d,p)$-irreducible up to degree $s$} if for all $P\in\poly(\Z^{k}\to\Z/p^{s})$ of degree at most $s$, either
	$\vert V_{p}(P)\cap \Omega\cap [p]^{d}\vert\leq \d\vert\Omega\cap [p]^{d}\vert$ or $\Omega\subseteq V_{p}(P).$
\end{defn}	

\begin{rem}
	If $\Omega\subseteq \F_{p}^{k}$ is a $1/2$-irreducible algebraic set, then it is not hard to see that for any $P,Q\in\poly(\F_{p}^{k}\to\F_{p})$, if $\Omega\subseteq V(P)\cup V(Q)$, then either $\Omega\subseteq V(P)$ or $\Omega\subseteq V(Q)$. In the language of algebraic geometry, this means that $\Omega$ is an \emph{irreducible algebraic set} (or an \emph{algebraic variety}).
\end{rem}

The following basic proposition explains the connections between different notions of irreducibility.

\begin{prop}[Connections between notions  of irreduciblities]\label{1:f2z}
	Let $k\in \N_{+},s\in\N, \d>0$ and $p$ be a prime with $p>s$.
	\begin{enumerate}[(i)]
		\item If $\Omega\subseteq\F_{p}^{k}$ is  $\d$-irreducible up to degree $s$, then $\tau(\Omega)\subseteq [p]^{k}$ is $(\d,p)$-irreducible up to degree $s$.
		\item Conversely, if a $p$-periodic subset $\Omega\subseteq \Z^{k}$ is  $(\d,p)$-irreducible up to degree $s$, then $\iota(\Omega)\subseteq \F_{p}^{k}$ is $\d$-irreducible up to degree $s$.
		\item If a $p$-periodic subset  $\Omega\subseteq \Z^{k}$ is strongly $(\d,p)$-irreducible up to degree $s$, then it is also weakly $(\d,p)$-irreducible up to degree $s$.
	\end{enumerate}	  	
\end{prop}	
\begin{proof}
	Part (i) follows from the fact that for every $P\in\poly(\Z^{k}\to\Z/p)$ of degree at most $s$, we have $\iota\circ pP\circ\tau\in \poly(\F_{p}^{k}\to\F_{p})$ and 
	$V_{p}(P)\cap\tau(\Omega)=\tau(V(\iota\circ pP\circ\tau)\cap \Omega)$. 
	Part (ii) follows from the fact that for every $P\in \poly(\F_{p}^{k}\to\F_{p})$, there exists a regular lifting $\tilde{P}\in \poly(\Z^{k}\to\Z/p)$ of $P$ with $\tilde{P}=\frac{1}{p}\tau(P\circ \iota) \mod \Z$ and $V(P)\cap \iota(\Omega)=\iota(V_{p}(\tilde{P})\cap \Omega)$. 
	Part (iii) follows from the inclusion $\poly(\Z^{k}\to\Z/p)\subseteq \poly(\Z^{k}\to\Z/p^{s})$.
\end{proof}	

 It is not hard to see that the full set $\V$ is $\d$-irreducible. To be more precise, we have
 
\begin{prop}[$\V$ is irreducible]\label{1:iri00}
Let $d\in\N_{+},s\in\N,\d>0$  and $p$ be a prime. If $p\gg_{d,s} \d^{-O_{d,s}(1)}$, then 
\begin{enumerate}[(i)]
	\item $\V$ is $\d$-irreducible up to degree $s$;
	\item  $\Z^{d}$ is weakly $(\d,p)$-irreducible up to degree $s$;
	\item   $\Z^{d}$ is strongly $(\d,p)$-irreducible up to degree $s$.
\end{enumerate}		
\end{prop}
\begin{proof}
Part (i) follows from Lemma \ref{1:ns}. Part (ii) follows from Part (i) and Proposition \ref{1:f2z}.
	We now prove Part (iii). The idea is to use  Part (ii) and  the $p$-expansion trick.
	Let	$g\in\poly(\Z^{d}\to\Z/p^{s})$ be of degree at most $s$ with $\vert V_{p}(g)\cap [p]^{d}\vert>\d p^{d}$.
	By the multivariate polynomial interpolation,  there exists $Q\in\mathbb{N}, p\nmid Q$  such that  
	$Qg=\sum_{i=0}^{s}\frac{g_{i}}{p^{i}}$
	for some integer valued polynomials $g_{i}\colon\Z^{d}\to\mathbb{Z}$ of degree at most $s$.
	Let $k\in\N$ be the smallest integer such that there exists $Q\in\mathbb{N}, p\nmid Q$  with
	$Qg=\sum_{i=0}^{k}\frac{g_{i}}{p^{i}}$
	for some  integer valued polynomials $g_{i}\colon\Z^{d}\to\mathbb{Z}$ . Obviously such $k$ exists and is at most $s$. Our goal is to show that $k=0$.	
	
	Suppose that 
	$k>0$.
	For all $n\in  V_{p}(g)$ and $m\in\Z^{d}$, by definition $g(n+pm)\in\Z$. Therefore, we must have that $g_{k}(n+pm)\in p\Z$.   
	So $\vert V_{p}(\frac{1}{p}g_{k})\cap [p]^{d}\vert\geq \vert V_{p}(g)\cap [p]^{d}\vert>\d p^{d}$. By Part (ii), we have that $V_{p}(\frac{1}{p}g_{k})\supseteq [p]^{d}$ and thus $\frac{1}{p}g_{k}(n)\in\Z$ for all $n\in\Z^{d}$.
Then there  exists $Q'\in\mathbb{N}, p\nmid Q'$  such that
	$\frac{Q'}{p}g_{k}$ and thus $Q'(g_{k-1}+\frac{1}{p}g_{k})$ is an integer valued polynomial, a contradiction to the minimality of $k$.
	
	We have thus proved that $k=0$ and 	thus $Qg$ itself is integer valued for some $Q\in\mathbb{N}, p\nmid Q$. Then for all $n\in [p]^{d}$, we have that $g(n)\in \Z/Q$. Since $g$ takes values in $\Z/p^{s}$, we have that $g(n)\in (\Z/Q)\cap (\Z/p^{s})=\Z$. So $[p]^{d}$ is strongly $(\d,p)$-irreducible up to degree $s$.
\end{proof}

As a consequence of Lemma \ref{1:bzt}, we have: 

\begin{prop}[$V(M)$ is irreducible]\label{1:iri0}
	 	Let $d\in\N_{+},s\in\N,\d>0$  and $p$ be a prime. If $p\gg_{d,s} \d^{-O_{d,s}(1)}$, then 
	 	\begin{enumerate}[(i)]
	  		\item for every quadratic form $M\colon \V\to\F_{p}$ of rank at least 3, the set $V(M)$ is $\d$-irreducible up to degree $s$;
	  		\item for every  quadratic form $M\colon \Z^{d}\to\Z/p$ of $p$-rank at least 3, the set $V_{p}(M)$ is weakly $(\d,p)$-irreducible up to degree $s$;
	 		\item for every  quadratic form $M\colon \Z^{d}\to\Z/p$ of $p$-rank at least 3, the set $V_{p}(M)$ is strongly $(\d,p)$-irreducible up to degree $s$.
	  	\end{enumerate}	
\end{prop}
\begin{proof}
	Part (i) follows from Lemma \ref{1:bzt}. Part (ii) follows from Part (i) and Proposition \ref{1:f2z}.
	We now prove Part (iii) using  Part (ii) and  the $p$-expansion trick. The proof is very similar to that of Proposition \ref{1:basicpp12}.
    Let $g\in\poly(\Z^{d}\to\Z/p^{s})$ be of degree at most $s$ with $\vert V_{p}(g)\cap V_{p}(M)\cap[p]^{d}\vert>\d\vert V_{p}(M)\cap[p]^{d}\vert$.
	By the multivariate polynomial interpolation,  there exists $Q\in\mathbb{N}, p\nmid Q$  such that  
			$$Qg=\sum_{i=0}^{s}\frac{g_{i}}{p^{i}}$$
			for some integer valued polynomials $g_{i}\colon\Z^{d}\to\mathbb{Z}$ of degree at most $s$.
			We say that a polynomial $f$ is \emph{good} if  
			$$f=\sum_{i=0}^{\lfloor s/2\rfloor}M^{i}R_{i}$$
			for some integer valued polynomial $R_{i}\colon\Z^{d}\to\mathbb{Z}$ of degree at most $s-2i$.	
			Let $k\in\N$ be the smallest integer such that there exists $Q\in\mathbb{N}, p\nmid Q$  with
			$$Qg=\sum_{i=0}^{k}\frac{g_{i}}{p^{i}}$$
			for some good polynomials $g_{i}\colon\Z^{d}\to\mathbb{Z}$ . Obviously such $k$ exists and is at most $s$. Our goal is to show that $k=0$.	
			
			Suppose that 
			$k>0$.
			For all $n\in  V_{p}(g)\cap V_{p}(M)$ and $m\in\Z^{d}$, by definition $g(n+pm)\in\Z$. Therefore, we must have that $g_{k}(n+pm)\in p\Z$. 
			Suppose that 
			$$g_{k}=\sum_{i=0}^{\lfloor s/2\rfloor}M^{i}R_{i}$$
			for some integer valued polynomials $R_{i}\colon\Z^{d}\to\mathbb{Z}$ of degree at most $s-2i$. 
			Assume that $$M(n)=\frac{1}{p}((nA)\cdot n+u\cdot n+v)$$
			for some symmetric integer valued $d\times d$ matrix $A$, some $u\in\Z^{d}$ and some $v\in\Z$.
			Then for all $n\in  V_{p}(g)\cap V_{p}(M)$ and $m\in\Z^{d}$, writing $x:=M(n)+(2(nA)+u)\cdot m$, we have
			\begin{equation}\label{1:771}
				\begin{split}
				&\quad 0\equiv g_{k}(n+pm)\equiv \sum_{i=0}^{\lfloor s/2\rfloor}M(n+pm)^{i}R_{i}(n)
				\\&\equiv \sum_{i=0}^{\lfloor s/2\rfloor}(M(n)+(2(nA)+u)\cdot m)^{i}R_{i}(n)= \sum_{i=0}^{\lfloor s/2\rfloor}x^{i}R_{i}(n) \mod p\Z.
				\end{split}
			\end{equation} 
			Note that if $2(nA)+u\notin p\Z^{d}$, then for any $x\in [p]$, there exists $m\in\Z^{d}$ such that $M(n)+(2(nA)+u)\cdot m\equiv x \mod p\Z$. So viewing the right hand side of (\ref{1:771}) as a polynomial of $x$, we have that  
			$R_{i}(n)\in p\mathbb{Z}/Q'$ for some $Q'\in\N, p\nmid Q'$ for all $0\leq i\leq \lfloor s/2\rfloor$ and $n\in  V_{p}(g)\cap V_{p}(M)$ with $2(nA)+u\notin p\Z^{d}$.

			Since $\rank_{p}(M)\geq 3$, it is not hard to see that the number of $n\in[p]^{d}$ with 
			$2(nA)+u\in p\Z^{d}$ is at most $p^{d-3}$.  
			So $$\vert V_{p}(Q'R_{i})\cap V_{p}(M)\cap[p]^{d}\vert\geq \vert V_{p}(g)\cap V_{p}(M)\cap[p]^{d}\vert-p^{d-3}>\d\vert V_{p}(M)\cap[p]^{d}\vert/2.$$ By Part (ii), we have that $V_{p}(M)\subseteq V_{p}(Q'R_{i})$ for all $0\leq i\leq \lfloor s/2\rfloor$.
			Let $Q''\in\N, Q'\vert Q'', p\nmid Q''$ be such that $Q''R_{i}$ is an integer coefficient polynomial for all $0\leq i\leq \lfloor s/2\rfloor$.
			We may  invoke  Corollary \ref{1:noloop2}  to conclude that 
			$$\frac{Q''}{p}R_{i}=MR'_{i}+R''_{i}$$
			for some integer valued polynomials $R_{i},R''_{i}$ of degrees at most $s-2i-2$ and $s-2i$ respectively. Then
			$\frac{Q''}{p}g_{k}$ and thus $Q''(g_{k-1}+\frac{1}{p}g_{k})$ is good, a contradiction to the minimality of $k$.
			
			We have thus proved that $k=0$ and 	thus $Qg$ itself is good for some $Q\in\N, p\nmid Q$. Then for all $n\in V_{p}(M)$, we have that $g(n)\in \Z/Q$. Since $g$ takes values in $\Z/p^{s}$, we have that $g(n)\in (\Z/Q)\cap (\Z/p^{s})=\Z$ and thus $V_{p}(M)\subseteq V_{p}(g)$. So  $V_{p}(M)$ is strongly $(\d,p)$-irreducible up to degree $s$.
\end{proof}

We provide some basic propositions for the irreducible properties. Firstly, it is straightforward that the irreducible properties are preserved under translations.

\begin{prop}[Irreducible sets are translation invariant]\label{1:iri4}
	Let $d\in\N_{+},s\in\N, \d>0$ and $p>s$ be a prime. 
	\begin{enumerate}[(i)]
		\item Let $\Omega\subseteq \F_{p}^{d}$ be $\d$-irreducible up to degree $s$ and $v\in\V$. Then  $\Omega+v$ is also $\d$-irreducible up to degree $s$.
		\item Let $\Omega\subseteq \Z^{d}$ be a $p$-periodic set which is weakly $(\d,p)$-irreducible up to degree $s$ and $v\in\Z^{d}$. Then  $\Omega+v$ is also  weakly $(\d,p)$-irreducible up to degree $s$.
		\item Let $\Omega\subseteq \Z^{d}$ be a $p$-periodic set which is strongly $(\d,p)$-irreducible up to degree $s$ and $v\in\Z^{d}$. Then  $\Omega+v$ is also strongly $(\d,p)$-irreducible up to degree $s$.
	\end{enumerate}	
\end{prop}
\begin{proof}
	Let $P\in\poly(\V\to\F_{p})$ be a polynomial of degree at most $s$ such that $\vert V(P)\cap (\Omega+v)\vert\geq \d \vert \Omega+v\vert=\d \vert \Omega\vert$. Then $\vert V(P(\cdot-v))\cap \Omega\vert\geq  \d \vert \Omega\vert$. Since $\Omega$ is $\d$-irreducible up to degree $s$, we have that $V(P(\cdot-v))\subseteq \Omega$ and thus $V(P)\subseteq \Omega+v$. This proves Part (i).
	
	The proofs of Parts (ii) and (iii) are similar to Part (i) and we leave them to the interested readers.
\end{proof}

Moreover, the irreducible properties are preserved under products.
\begin{prop}[Products of irreducible sets are irreducible]\label{1:iri3}
	Let $d_{1},d_{2}\in\N_{+},s\in\N, \d>0$ and $p>s$ be a prime. Let $\Omega_{1}\subseteq \F_{p}^{d_{1}}$, $\Omega_{2}\subseteq \F_{p}^{d_{2}}$ and $\Omega=\Omega_{1}\times\Omega_{2}$.
	\begin{enumerate}[(i)]
		\item If $\Omega_{1}\subseteq \F_{p}^{d_{1}}$ and $\Omega_{2}\subseteq \F_{p}^{d_{2}}$ are  $\d$-irreducible up to degree $s$, then $\Omega_{1}\times \Omega_{2}\subseteq \F_{p}^{d_{1}+d_{2}}$ is $2\d$-irreducible up to degree $s$. 
		\item If $\Omega_{1}\subseteq \Z^{d_{1}}$ and $\Omega_{2}\subseteq \Z^{d_{2}}$ are $p$-periodic sets which are weakly $(\d,p)$-irreducible up to degree $s$, then $\Omega_{1}\times \Omega_{2}\subseteq \Z^{d_{1}+d_{2}}$ is weakly $(2\d,p)$-irreducible up to degree $s$. 
		\item If $\Omega_{1}\subseteq \Z^{d_{1}}$ and $\Omega_{2}\subseteq \Z^{d_{2}}$ are $p$-periodic sets which  are  strongly $(\d,p)$-irreducible up to degree $s$, then $\Omega_{1}\times \Omega_{2}\subseteq \Z^{d_{1}+d_{2}}$ is strongly $(2\d,p)$-irreducible up to degree $s$. 
	\end{enumerate}
\end{prop}
\begin{proof}
	Let  $P\in\poly(\F_{p}^{d_{1}+d_{2}}\to\F_{p})$ be a polynomial of degree at most $s$ such that $\vert V(P)\cap (\Omega_{1}\times \Omega_{2})\vert\geq 2\d \vert \Omega_{1}\vert\cdot\vert \Omega_{2}\vert$. By the Pigeonhole Principle, there exists a subset $U$ of $\Omega_{2}$ with $\vert U\vert\geq \d \vert \Omega_{2}\vert$ such that for all $y\in U$, $\vert V(P(\cdot,y))\cap \Omega_{1}\vert=\vert V(P)\cap (\Omega_{1}\times \{y\})\vert\geq\d \vert \Omega_{1}\vert$. Since $\Omega_{1}$ is $\d$-irreducible up to degree $s$, we have that $\Omega_{1}\subseteq V(P(\cdot,y))$ for all $y\in U$. In other words, for all $x\in \Omega_{1}$, $\vert V(P(x,\cdot))\cap \Omega_{2}\vert\geq \vert U\vert\geq \d \vert \Omega_{2}\vert$. Since $\Omega_{2}$ is $\d$-irreducible up to degree $s$, we have that $\Omega_{2}\subseteq V(P(x,\cdot))$ for all $x\in \Omega_{1}$. In other words, $\Omega_{1}\times \Omega_{2}\subseteq V(P)$. This proves Part (i).
	
	The proofs of Parts (ii) and (iii) are similar to Part (i) and we leave them to the interested readers.
\end{proof}

Finally, the irreducible properties are preserved under isomorphisms.
 
\begin{prop}[Irreducible sets are invariant under isomorphisms]\label{1:iso}
	Let $k\in\N_{+},s\in\N,\d>0$ and $p>s$ be a prime.
	Let $\Omega\subseteq \F_{p}^{k}$ and  $L\colon \F_{p}^{k}\to\F_{p}^{k}$ be a bijective linear transformation.
	\begin{enumerate}[(i)]
	\item If $\Omega$  is $\d$-irreducible up to degree $s$, then so is $L(\Omega)$.
	\item If $\iota^{-1}(\Omega)$  is weakly $(\d,p)$-irreducible up to degree $s$, then so is $\iota^{-1}(L(\Omega))$.
	\item If $\iota^{-1}(\Omega)$  is strongly $(\d,p)$-irreducible up to degree $s$, then so is $\iota^{-1}(L(\Omega))$.
	\end{enumerate}	
\end{prop}
\begin{proof}		
	We first prove Part (i).
	Let $P\in\poly(\F_{p}^{k}\to\F_{p})$ be a polynomial of degree at most $s$ with $$\vert V(P\circ L)\cap \Omega\vert=\vert V(P)\cap L(\Omega)\vert\geq \d\vert L(\Omega)\vert=\d\vert \Omega\vert.$$ Clearly $P\circ L\colon\F_{p}^{k}\to\F_{p}$ is a polynomial of degree at most $s$. Since $\Omega$  is $\d$-irreducible up to degree $s$, we have that $\Omega\subseteq V(P\circ L)$ and so $L(\Omega)\subseteq V(P)$. So  $L(\Omega)$  is $\d$-irreducible up to degree $s$.  
	
	We now prove Part (ii).
	Let  $P\in\poly(\Z^{k}\to\Z/p)$ be a polynomial of degree at most $s$ with $\vert V_{p}(P)\cap \iota^{-1}\circ L(\Omega)\cap [p]^{k}\vert\geq \d\vert \iota^{-1}\circ L(\Omega)\cap [p]^{k}\vert=\d\vert \Omega\vert$. 
	By Lemma \ref{1:lifting2}, there exists a  linear transformation $\tilde{L}\colon\Z^{k}\to\Z^{k}$ such that $\tilde{L}\circ \tau\equiv\tau\circ L \mod p\Z^{k}$.
	 Since $L$ is bijective, $\tilde{L}$ maps different residue classes mod $p\Z^{k}$ to different residue classes mod $p\Z^{k}$.
	 So
	$$\vert V_{p}(P)\cap \iota^{-1}\circ L(\Omega)\cap[p]^{k}\vert=\vert V_{p}(P)\cap \tilde{L}\circ \iota^{-1}(\Omega)\cap [p]^{k}\vert=\vert V_{p}(P\circ\tilde{L})\cap\iota^{-1}(\Omega)\cap [p]^{k}\vert.$$
	It is clear that $P\circ \tilde{L}$ is a polynomial belonging to $\poly(\Z^{k}\to\Z/p)$ of degree at most $s$.  Since $\iota^{-1}(\Omega)$  is weakly $(\d,p)$-irreducible up to degree $s$, we have that $\iota^{-1}(\Omega)\subseteq V_{p}(P\circ\tilde{L})$ and thus $\iota^{-1}\circ L(\Omega)\subseteq V_{p}(P)$ since  $P\circ \tilde{L}\in\poly(\Z^{k}\to\Z/p)$.  So  $\iota^{-1}\circ L(\Omega)$  is weakly $(\d,p)$-irreducible up to degree $s$.
	
	Note that if $P\in\poly(\Z^{k}\to\Z/p^{s})$ is a polynomial of degree at most $s$, then so is $P\circ \tilde{L}$.
	The rest proof of Part (iii) is almost identical to that of Part (ii) and we omit the details.
\end{proof}

We are now ready to state the main irreducible theorem for $M$-sets:

\begin{thm}[Nice and consistent $M$-sets are irreducible]\label{1:irrr}
	Let $d,K\in\N_{+},s\in\N$, $p$ be a prime, and $\d>0$. Let $M\colon\V\to \F_{p}$ be a  non-degenerate quadratic form 	and  $\Omega\subseteq (\V)^{K}$ be a  nice and consistent $M$-set. 	If 	$d\geq \max\{2r_{M}(\Omega)+1,4K-1\}$  and $p\gg_{d,s} \d^{-O_{d,s}(1)}$, then 
	\begin{enumerate}[(i)]
		\item 	$\Omega$ is $\d$-irreducible up to degree $s$;  		
		\item $\iota^{-1}(\Omega)$ is weakly $(\d,p)$-irreducible up to degree $s$;
		\item $\iota^{-1}(\Omega)$ is strongly $(\d,p)$-irreducible up to degree $s$.
	\end{enumerate}
\end{thm}

The proof of Theorem \ref{1:irrr} is rather complicated. We postpone the proof of Theorem \ref{1:irrr} to Appendix \ref{1:s:AppC} to prevent the readers from being distracted by technical details.

\begin{rem}
The method we use to prove Theorem \ref{1:irrr} can be adapted to show that all consistent $M$-sets are irreducible in all the three senses. However, we do not need such a general result in this paper.
\end{rem}

\section{Leibman dichotomy}\label{1:ss91}

In this section, we extend Theorem \ref{1:sLei} to all nice and consistent $M$-sets. For convenience, we introduce the following definition:

\begin{defn}[$(\d,K)$-Leibman dichotomy]
Let $d\in\N_{+},s\in\N$, $C,\d,K>0$,   $p$ be a prime, and $\Omega$ be a non-empty $p$-periodic subset of $\Z^{d}$. We say that $\Omega$ admits a \emph{rational (resp. partially periodic) $(\d,K,p)$-Leibman dichotomy up to step $s$ and complexity $C$}  
if for any  $\N$-filtered nilmanifold  $G/\Gamma$  of degree at most $s$ and complexity at most $C$,  and any $g\in \poly(\Z^{d}\to G_{\N})$ which is rational (resp. belongs to $\poly_{p}(\Omega\to G_{\N}\vert\Gamma)$), either $(g(n)\Gamma)_{n\in \Omega}$ is  $\d$-equidistributed on $G/\Gamma$,  
or  there exists a nontrivial type-I horizontal character $\eta$ with $0<\Vert\eta\Vert\leq K$  such that $\eta\circ g \mod\Z$ is a constant on $\Omega$. 

We say that a non-empty  subset $\Omega$  of $\V$ admits a \emph{rational (resp. partially periodic) $(\d,K)$-Leibman dichotomy up to step $s$ and complexity $C$}  if $\iota^{-1}(\Omega)$ admits a rational (resp. partially periodic) $(\d,K,p)$-Leibman dichotomy up to step $s$ and complexity $C$.
\end{defn}	

We remark that for a non-empty subset  $\Omega\subseteq\V$,  an equivalent way of saying that  $\Omega$ admits a  partially periodic $(\d,K)$-Leibman dichotomy up to step $s$ and complexity $C$ is that for any any  $\N$-filtered nilmanifold  $G/\Gamma$  of degree at most $s$ and complexity at most $C$,  and any $g\in \poly(\Omega\to G_{\N})$, either $(g(n)\Gamma)_{n\in \Omega}$ is  $\d$-equidistributed on $G/\Gamma$ (meaning that $(g\circ\tau(n)\Gamma)_{n\in \iota^{-1}(\Omega)}$ is  $\d$-equidistributed on $G/\Gamma$),  
or  there exists a nontrivial type-I horizontal character $\eta$ with $0<\Vert\eta\Vert\leq K$  such that $\eta\circ g \mod\Z$ is a constant on $\Omega$.

We may deduce the following  Leibman dichotomy from the main result of  \cite{GT12b}:   

\begin{thm}[$\V$ admits a Leibman dichotomy]\label{1:Lei0}
	Let $d\in\N_{+},s\in\N$, $C>0$ and $p$ be a prime. There exists   $K=O_{C,d,s}(1)$ such that for all $0<\d<1/2$,  $\V$ admits a rational $(\d,K\d^{-K})$-Leibman dichotomy up to step $s$ and complexity $C$.	
\end{thm}
\begin{proof}
Let $\d>0$, $k\in\N_{+}$ and $g\in \poly(\Z^{d}\to G_{\N})$ be 
rational. By Corollary \ref{1:r222p}, $g$ is $Q$-periodic for some $Q\in\N_{+}$. 
Suppose that the sequence $(g(n)\Gamma)_{n\in \Z^{d}}$
 is not $\d$-equidistributed on $G/\Gamma$. Since $g$ is $Q$-periodic, we have that  $(g(n)\Gamma)_{n\in [QN]^{d}}$ is not $\d$-equidistributed on $G/\Gamma$ for all $N\in\N_{+}$. By Theorem 8.6 of \cite{GT12b}, there exists a nontrivial type-I horizontal character $\eta$ with $0<\Vert\eta\Vert\leq O_{C,d,s}(\d^{-O_{C,d,s}(1)})$  such that $\Vert\eta\circ g\Vert_{C^{\infty}[QN]^{d}}\ll_{C,d,s} \d^{-O_{C,d,s}(1)}$ (see Definition 8.2 of \cite{GT12b} for the definition of the smoothness norm $\Vert\cdot\Vert_{C^{\infty}[N]^{d}}$).   Letting $N\to\infty$, we have that  $\Vert\eta\circ g\Vert_{C^{\infty}[QN]^{d}}=0$ if $N$ is sufficiently large and thus  $\eta\circ g\mod \Z$ is a constant on $[QN]^{d}$. Since    $g$ is $Q$-periodic, this implies that $\eta\circ g\mod \Z$ is a constant on $\Z^{d}$.
 \end{proof}

We may also rephrase  Theorem \ref{1:sLei} using the Leibman dichotomy as below:

 \begin{thm}[$V(M)$ admits a Leibman dichotomy]\label{1:rLei}
 	Let $d\in\N_{+},r,s\in\N$, $C>0$ and $p$ be a prime. There exists   $K=O_{C,d}(1)$ such that for any affine subspace $V+c$  of $\V$  of co-dimension $r$, any  quadratic form $M\colon\V\to\F_{p}$ with $\rank(M\vert_{V+c})\geq s+13$,  and any $0<\d<1/2$, 
 if $p\gg_{C,d} \d^{-O_{C,d}(1)}$, 
 	then $V(M)\cap (V+c)$ admits a rational $(\d,K\d^{-K})$-Leibman dichotomy up to step $s$ and complexity $C$.\footnote{It follows from Corollary \ref{1:counting01} that $V(M)\cap (V+c)$ is non-empty.}	
 \end{thm}

 \begin{rem}
 By Lemma \ref{1:oehfr}, it is not hard to see that if $\Omega$ admits a rational  $(\d,K,p)$-(or $(\d,K)$-) Leibman dichotomy up to step $s$ and complexity $C$, then  $\Omega$  also admits a  partially periodic $(\d,K,p)$-(or $(\d,K)$-) Leibman dichotomy up to step $s$ and complexity $O(C)$. So Theorems \ref{1:Lei0} and \ref{1:rLei} also hold for partially periodic Leibman dichotomies.
 \end{rem}

 We now turn to sets which admits partially periodic Leibman dichotomies.
 The following lemma is a straightforward consequence of the definition of the  Leibman dichotomy and Proposition \ref{1:BB}. Its proof is left to the interested readers.

 \begin{lem}\label{1:rrgg}
 	Let $d\in\N_{+},s\in\N$, $\d,K>0$, $p$ be a prime, $\Omega\subseteq \V$ be non-empty, $v\in\F_{p}^{d}$ and $L\colon\F_{p}^{d}\to\F_{p}^{d}$ be a bijective linear transformation. If $\Omega$  admits a partially periodic $(\d,K)$-Leibman dichotomy up to step $s$ and complexity $C$,
 	 then so does $L(\Omega)+v$.
 \end{lem}

We are now ready to state the following partial generalization of  Theorem \ref{1:sLei} for nice and consistent  $M$-sets:

\begin{thm}[Nice and consistent  $M$-sets admit  Leibman dichotomies]\label{1:veryr}
	Let $d,k\in\N_{+},s,r\in\N$ with $d\geq \max\{4r+1,4k+3,2k+s+11\}$, $C>0$ and $p$ be a prime. 
	There exists  $K:=O_{C,d}(1)$ such that for any  non-degenerate quadratic form $M\colon\V\to \F_{p}$, any nice and consistent $M$-set $\Omega\subseteq (\V)^{k}$ of total co-dimension $r$, and any $0<\d<1/2$,  if $p\gg_{C,d} \d^{-O_{C,d}(1)}$, 
	then  $\Omega$ admits a partially periodic $(\d,K\d^{-K})$-Leibman dichotomy up to step $s$ and complexity $C$.\footnote{It follows from Theorem \ref{1:ct} that $\Omega$ is non-empty.}	
\end{thm}

It is an interesting question to ask if Theorem \ref{1:veryr} holds for rational Leibman dichotomies. However, we do not need such a strong statement in later parts of the series.

The proof of Theorem \ref{1:veryr} is rather complicated. We postpone the proof of Theorem \ref{1:veryr} to Appendix \ref{1:s:AppD} to prevent the readers from being distracted by technical details.



\section{Factorization theorems}\label{1:ss92}

With all the preparations in the previous sections,
we are now ready to derive the factorization theorem for partially $p$-periodic polynomial sequences $g$ on a set $\Omega$ admitting a Leibman dichotomy. We show that $g$ can be written as the product of a small constant term $\e$, a  polynomial sequence $g'$ which is sufficiently well equidistributed on a subnilmanifold, and a polynomial sequence  $\gamma$ which has a small period on the set $\Omega$. 

We provide two sets of factorization results. The first set is provided in Section \ref{1:secf1} and is designed for \cite{SunD}, which deals with partially periodic polynomial sequences for a large family of nice and consistent $M$-sets $\Omega$. The second set is provided in Section \ref{1:secf2} and is designed for \cite{SunC}, which deals with rational polynomial sequences but for a more restricted family of sets $\Omega$.

\subsection{Factorization theorems of partially periodic sequences}\label{1:secf1}

To begin with, we prove a weak factorization theorem for partially periodic polynomial sequences, where $g'$ is only required to take values on a  subnilmanifold.
 For a nilmanifold $G/\Gamma$ with filtration $(G_{i})_{i\in \N}$, the \emph{total dimension} of $G/\Gamma$ is the quantity $\sum_{i\in \N}\dim(G_{i})$.
 
\begin{prop}[Weak factorization property I]\label{1:facf2}
	Let $d,s\in\N_{+}$, $C, \d, K>0$  and $p\gg_{C,d,K,s} 1$ be a prime. 	Let $\Omega\subseteq \V$ be a non-empty set admitting a partially periodic $(\d,K)$-Leibman dichotomy up to degree $s$ and complexity $C$.  Let $G/\Gamma$ be an $s$-step $\N$-filtered nilmanifold of complexity at most $C$, equipped with a $C$-rational Mal'cev basis $\mathcal{X}$, and let 	$g\in \poly_{p}(\iota^{-1}(\Omega)\to G_{\N}\vert\Gamma)$. If  $(g(n)\Gamma)_{n\in \iota^{-1}(\Omega)}$ is not $\d$-equidistributed on $G/\Gamma$, 
	then we may factorize $g$ as $$g(n)=\e g'(n)\gamma(n)  \text{ for all }  n\in \Z^{d},$$  where $\e\in G$ of complexity at most 1,
	$g'\in \poly_{p}(\iota^{-1}(\Omega)\to G'_{\N}\vert\Gamma')$, $g'(\bold{0})=id_{G}$ for some sub-nilmanifold $G'/\Gamma'$ of $G/\Gamma$ whose Mal'cev basis is $O_{C,d}(K)$-rational relative to $\mathcal{X}$ and whose total dimension is smaller than that of $G/\Gamma$, and   $\gamma\in\poly(\iota^{-1}(\Omega)\to G_{\N}\vert\Gamma)$ with $\gamma(\bold{0})=id_{G}$.
\end{prop}

\begin{proof}
The proof uses ideas from Proposition 9.2 of Green and Tao's work  \cite{GT12b} and the work of Candela and Sisask \cite{CS14}. In our case, the constructions of the sequences $g'$ and $\gamma$ are more intricate, as we need them to satisfy certain partial periodic conditions.

Throughout the proof we assume that $p\gg_{C,d,K,s} 1$.
Denote $m_{i}:=\dim(G_{i})$ for $i\in\N$ and $m:=\dim(G)$. 
Let $\mathcal{X}=\{X_{1},\dots,X_{m}\}$ be the Mal'cev basis of $G/\Gamma$.
Write $g:=\{g(\bold{0})\}g_{\ast}[g(\bold{0})]$. Then $g_{\ast}\in \poly_{p}(\iota^{-1}(\Omega)\to G_{\N}\vert\Gamma)$. So replacing $g$ by $g_{\ast}$ if necessary,  
	we may assume without loss of generality that $g(\bold{0})=id_{G}$.

	If $(g(n)\Gamma)_{n\in \iota^{-1}(\Omega)}$ is not $\d$-equidistributed on $G/\Gamma$, then since $\Omega$ admits a partially periodic $(\d,K)$-Leibman dichotomy,  there exists a nontrivial  type-I horizontal character $\eta$ with $0<\Vert\eta\Vert\leq K$  such that $\eta\circ g \mod \Z$  is a constant on $\iota^{-1}(\Omega)$. 
		We may assume that $g=\tilde{g}\circ\tau$, for some $\tilde{g}\in\poly_{p}(\iota^{-1}(\Omega)\to G_{\N}\vert\Gamma)$. Since $g(\bold{0})=id_{G}$, 
		we may use the type-I Taylor expansion to write $g$ as 
	$$g(n)=\prod_{i\in\N^{d},0<\vert i\vert\leq s}g_{i}^{\binom{n}{i}}  \text{ for all }  n\in\Z^{d}$$
	for some $g_{i}\in G_{\vert i\vert}$. 		
	%
	Let $G'$ be the connected component of $\ker(\eta)$. Then $G'$ is a subgroup of $G$ which is $O_{C,d}(K)$-rational relative to $\mathcal{X}$. Write
	$$\psi(g(n))=\sum_{i\in\N^{d},0<\vert i\vert\leq s}t_{i}\binom{n}{i}$$
	for some $t_{i}\in\{0\}^{m-m_{\vert i\vert}}\times\R^{m_{\vert i\vert}}$.  
	The type-I horizontal character $\eta$ is given in coordinates by
	$$\eta\circ g(n)=\sum_{i\in\N^{d},0<\vert i\vert\leq s}k\cdot t_{i}\binom{n}{i}$$
	for some $k=(k_{1},\dots,k_{m})\in\Z^{m}$ of complexity at most $O_{C,d}(K)$.
	Let $0\leq s'\leq s$ be the largest integer such that at least one of the entries $k_{m-m_{s'}+1},\dots,k_{m-m_{s'+1}}$ is nonzero (where we denote $m_{0}:=m$). Since 
	$\eta$ is nontrivial, such an $s'$ always exists.
	Denote 
	$m':=m-m_{s'+1}$.
	Then we may write $k=(k',\bold{0})$ for some $k'\in \Z^{m'}$. Write $t_{i}=(t'_{i},t''_{i})$ for some $t'_{i}\in\R^{m'}$ and $t''_{i}\in \R^{m-m'}$. 
	Since the first $m-m_{\vert i\vert}$ entries of $t_{i}$ are zero, and the last $m_{s'+1}$ entries of $k$ are zero, we have that $k\cdot t_{i}=0$ if $\vert i\vert\geq s'+1$. 
	So
	$$\eta\circ g(n)=\sum_{i\in\N^{d},0<\vert i\vert\leq s}k'\cdot t'_{i}\binom{n}{i}=\sum_{i\in\N^{d},0<\vert i\vert\leq s'}k'\cdot t'_{i}\binom{n}{i}.$$
	
	 Since $k$ is of complexity at most $O_{C,d}(K)$, there exists $Q\in\N_{+}, Q\leq O_{C,\d,d,K}(1)$ and $u=(u_{1},\dots,u_{m})\in\Z^{m}$ whose only nonzero entries are $u_{m-m_{s'}+1},\dots,u_{m-m_{s'+1}}$ such that 
	$k\cdot u=Q.$ Denote $$\gamma_{0}(n):=\prod_{i=m-m_{s'}+1}^{m-m_{s'+1}}\exp(Q^{\ast}u_{i}\eta(g(n))X_{i}),$$
	where $Q^{\ast}$ is any integer such that $Q^{\ast}Q\equiv 1 \mod p^{s}\Z$. Clearly $\gamma_{0}(\bold{0})=id_{G}$. 
	Since $g\in\poly_{p}(\iota^{-1}(\Omega)\to G_{\N}\vert\Gamma)$, we have that $\eta\circ g\in\poly_{p}(\iota^{-1}(\Omega)\to \R\vert\Z)$ by Lemma \ref{1:projectionpp}. This implies that
	$\gamma_{0}\in\poly(\iota^{-1}(\Omega)\to G_{\N}\vert\Gamma)$.

	Since $\Omega$ is non-empty, by Lemma \ref{1:goodcoordinates}, we have that $p^{s}\eta\circ g(n)\in\Z$ for all $n\in\Z^{d}$. 
	On the other hand, writing $h(n):=g(n)\gamma_{0}(n)^{-1}$ we have that $h\in\poly(\Z^{d}\to G_{\N})$ by Corollary B.4 of \cite{GTZ12} and that
	$$\eta\circ h(n)=\eta\circ g(n)-Q^{\ast}(k\cdot u)\eta\circ g(n)=(1-Q^{\ast}Q)\eta\circ g(n),$$
	which is an integer valued polynomial.	
		Since $\gamma_{0}\in\poly(\iota^{-1}(\Omega)\to G_{\N}\vert\Gamma)$ and $g\in\poly_{p}(\iota^{-1}(\Omega)\to G_{\N}\vert\Gamma)$, it follows from Lemma \ref{1:grg} that $h\in\poly_{p}(\iota^{-1}(\Omega)\to G_{\N}\vert\Gamma)$. By Proposition \ref{1:normalizegamma},
	there exists $\gamma'\in\poly(\Z^{d}\to G_{\N}\vert\Gamma)$   with $\gamma'(\bold{0})=id_{G}$ such that $\eta\circ h=\eta\circ \gamma'$.
	
	Setting $\gamma:=\gamma_{0}\gamma'$ and $g':=g\gamma^{-1}$, it is not hard to see that $\gamma\in\poly(\iota^{-1}(\Omega)\to G_{\N}\vert\Gamma)$ and $g'\in\poly(\Z^{d}\to G_{\N})$. 
	Moreover, since $g(\bold{0})=\gamma_{0}(\bold{0})=\gamma'(\bold{0})=id_{G}$, we have that $g'(\bold{0})=\gamma(\bold{0})=id_{G}$.
	Note that
	$$\eta\circ g'(n)=\eta\circ g(n)-\eta\circ \gamma_{0}(n)-\eta\circ \gamma'(n)=\eta\circ h(n)-\eta\circ \gamma'(n)=0.$$
	So $g'\in \poly(\Z^{d}\to G'_{\N})$.   
	Finally, for all $n\in \iota^{-1}(\Omega)$ and $m\in\Z^{d}$, we have that
	$$g'(n+pm)^{-1}g'(n)=\gamma(n+pm)g(n+pm)^{-1}g(n)\gamma(n)^{-1}.$$
	Since all of $\gamma(n+pm)$, $g(n+pm)^{-1}g(n)$ and $\gamma(n)^{-1}$ belong to $\Gamma$, we have that $g'(n+pm)^{-1}g'(n)\in\Gamma$. Therefore, we have that $g'\in\poly_{p}(\iota^{-1}(\Omega)\to G'_{\N}\vert\Gamma')$. 
	We are done since $g=g'\gamma$.
\end{proof}	
 
 By repeatedly using Proposition \ref{1:facf2}, we get the following strong factorization property:

  \begin{thm}[Strong factorization property I]\label{1:facf3}
  		Let $d,k\in\N_{+},r,s\in\N$, $C>0$, $\mathcal{F}\colon\R_{+}\to\R_{+}$ be a growth function, $p\gg_{C,d,\mathcal{F},k,r,s} 1$ be a prime, $M\colon\V\to \F_{p}$ be a quadratic form, and $\Omega\subseteq (\V)^{k}$ satisfying one of the following assumptions:
  		\begin{enumerate}[(i)] 
  			\item $r=0$ and $\Omega=(\V)^{k}$;
  			\item $k=1,r=0$ and $\Omega=V(M)\cap (V+c)$ for some affine subspace $V+c$ of $\V$ with $\rank(M\vert_{V+c})\geq s+13$;
  			\item $M$ is non-degenerate, $d\geq \max\{4r+1,4k+3,2k+s+11\}$ and $\Omega$ is a  nice and consistent $M$-set of total co-dimension $r$.
  		\end{enumerate}
  		Let $G/\Gamma$ be an $s$-step $\N$-filtered nilmanifold of  complexity at most $C$, and let $g$ be a polynomial sequence in $\poly_{p}(\iota^{-1}(\Omega)\to G_{\N}\vert\Gamma)$.	
  	There 		
  		exist some $C\leq C'\leq O_{C,d,\mathcal{F},k,r,s}(1)$, 
  		 a proper subgroup $G'$ of $G$ which is $C'$-rational relative to $\mathcal{X}$, and
  	 a factorization $$g(n)=\e g'(n)\gamma(n)  \text{ for all }  n\in (\Z^{d})^{k}$$  such that $\e\in G$ is of complexity $O_{C'}(1)$,
  		$g'\in \poly_{p}(\iota^{-1}(\Omega)\to G'_{\N}\vert\Gamma')$, $g'(\bold{0})=id_{G}$ and $(g'(n)\Gamma)_{n\in\iota^{-1}(\Omega)}$ is $\mathcal{F}(C')^{-1}$-equidistributed on $G'/\Gamma'$, where $\Gamma':=G'\cap \Gamma$, and that $\gamma\in\poly(\iota^{-1}(\Omega)\to G_{\N}\vert\Gamma)$, $\gamma(\bold{0})=id_{G}$.   		
  \end{thm}  
  \begin{proof}
  	The method is similar to the proof of Theorem 1.19 of \cite{GT12b}. Throughout the proof, we assume that $p\gg_{C,d,\mathcal{F},k,r,s} 1$. 	  Enlarging $C$ if necessary, we may assume without loss of generality that $\mathcal{F}(C)>2$.
	By Theorem \ref{1:Lei0}, \ref{1:rLei} or \ref{1:veryr}, $\Omega$ is a non-empty set which admits a partially periodic $(\mathcal{F}(C)^{-1},O_{C,d,\mathcal{F},k,r,s}(1))$-Leibman dichotomy up to step $s$ and complexity $C$. 
  	So by Proposition \ref{1:facf2}, either $(g(n)\Gamma)_{n\in \iota^{-1}(\Omega)}$ is   $\mathcal{F}(C)^{-1}$-equidistributed, or we may write $$g(n)=\e_{1}g_{1}(n)\gamma_{1}(n)  \text{ for all }  n\in (\Z^{d})^{k}$$  such that $\e_{1}\in G$ is  of complexity at most 1,
  	$g_{1}\in\poly_{p}(\iota^{-1}(\Omega)\to (G_{1})_{\N}\vert\Gamma_{1})$, $g_{1}(\bold{0})=id_{G}$ for some proper subgroup $G_{1}$ of $G$ which is $O_{C,d,\mathcal{F},k,r,s}(1)$-rational relative to $\mathcal{X}$ and whose total dimension is less than that of $G$, where  $\Gamma_{1}:=G_{1}\cap\Gamma$, and that $\gamma_{1}\in\poly(\iota^{-1}(\Omega)\to G_{\N}\vert\Gamma)$, $\gamma_{1}(\bold{0})=id_{G}$.
  	By Lemma A.10 of \cite{GT12b},  $G_{1}/(G_{1}\cap \Gamma)$ admits a $C_{1}$-rational Mal'cev basis for some $C_{1}=O_{C,d,\mathcal{F},k,r,s}(1)$.

  		By Theorem \ref{1:Lei0}, \ref{1:rLei} or \ref{1:veryr}, $\Omega$ admits a partially periodic $(\mathcal{F}(C_{1})^{-1},O_{C,d,\mathcal{F},k,r,s}(1))$-Leibman dichotomy up to step $s$ and complexity $C_{1}$. So
  	by Proposition \ref{1:facf2}, either 
  	$(g_{1}(n)\Gamma)_{n\in \iota^{-1}(\Omega)}$ is  $\mathcal{F}(C_{1})^{-1}$-equidistributed on $G_{1}/\Gamma_{1}$, or we may write $$g_{1}(n)=\e_{2}g_{2}(n)\gamma_{2}(n)  \text{ for all }  n\in (\Z^{d})^{k}$$  such that $\e_{2}\in G_{1}$ is of complexity at most 1,
  	$g_{2}\in\poly_{p}(\iota^{-1}(\Omega)\to (G_{2})_{\N}\vert\Gamma_{2})$, $g_{2}(\bold{0})=id_{G}$ for some proper subgroup $G_{2}$ of $G_{1}$ which is $O_{C,d,\mathcal{F},k,r,s}(1)$-rational relative to $\mathcal{X}$ and whose total dimension is less than that of $G_{1}$, where $\Gamma_{2}:=G_{2}\cap\Gamma$, and that $\gamma_{2}\in\poly(\iota^{-1}(\Omega)\to G_{\N}\vert\Gamma)$, $\gamma_{2}(\bold{0})=id_{G}$. 
	Since  $\e_{2}\in G_{1}$ is of complexity at most 1, we have that $\e_{2}$ is of complexity $O_{C_{1}}(1)$ as an element of $G$.	
	By Lemma A.10 of \cite{GT12b},  $G_{2}/(G_{2}\cap \Gamma)$ admits a $C_{2}$-rational Mal'cev basis for some $C_{2}=O_{C,d,\mathcal{F},k,r,s}(1)$. We may also require that $C_{2}\geq C_{1}$.

  	We may continue this process to get a sequence $C\leq C_{1}\leq C_{2}\leq \dots$. Since this iteration will terminate after at most $O_{C}(1)$ steps, we can find $1\leq t\leq O_{C}(1)$, some $C\leq C_{t}\leq O_{C,d,\mathcal{F},k,r,s}(1)$, and a factorization
  	$$g(n)=\e_{1}\dots\e_{t}g_{t}(n)\gamma_{1}(n)\dots\gamma_{t}(n)  \text{ for all }  n\in (\Z^{d})^{k}$$
  	$\e_{1},\dots,\e_{t}\in G$ are of complexities at most $O_{C_{t}}(1)$,
  	$g_{t}\in\poly_{p}(\iota^{-1}(\Omega)\to (G_{t})_{\N}\vert\Gamma_{t})$, $g_{t}(\bold{0})=id_{G}$ for some proper subgroup $G_{t}$ of $G$ which is $C_{t}$-rational relative to $\mathcal{X}$ and $(g_{t}(n)\Gamma)_{n\in\iota^{-1}(\Omega)}$ is $\mathcal{F}(C_{t})^{-1}$-equidistributed on $G_{t}/\Gamma_{t}$, where $\Gamma_{t}:=G_{t}\cap\Gamma$, and that $\gamma_{i}\in\poly(\iota^{-1}(\Omega)\to G_{\N}\vert\Gamma)$, $\gamma_{i}(\bold{0})=id_{G}$ for $1\leq i\leq t$. We are done by setting $C=C_{t}$, $\e=\e_{1}\dots\e_{t}$, $g'=g_{t}$ and $\gamma=\gamma_{1}\dots\gamma_{t}$ (since it is not hard to see from Lemma A.3 of \cite{GT12b} that $\e$ is of complexity $O_{t,C_{t}}(1)=O_{C_{t}}(1)$). 
  \end{proof}

 \begin{rem}\label{1:wpsp}
   It is tempting to conjecture that in Theorems \ref{1:facf3}, if $g\in\poly_{p}(\Z^{d}\to G_{\N}\vert\Gamma)$ (i.e. if $g$ is $p$-periodic), then one can further require $g'\in \poly_{p}(\Z^{d}\to G'_{\N}\vert\Gamma')$. However, we were unable to prove it. The difficulty is that even if we start with some $p$-periodic $g$, the polynomial sequence $g'$ obtained in Proposition \ref{1:facf2} is not necessarily $p$-periodic (we can only make $g'$ to be partially $p$-periodic on $\iota^{-1}(\Omega)$).    This is a major reason why we need to prove most of the results in this paper for the partially $p$-periodic polynomials.  
\end{rem} 

\subsection{Factorization theorems of rational sequences}\label{1:secf2}

 

  In this section, we present factorization theorems for rationa polynomial sequences. Again we start with a weak factorization property:

\begin{prop}[Weak factorization property II]\label{1:facf2r}
	Let $d,s\in\N_{+}$, $C, \d, K>0$  and $p\gg_{C,d,K,s} 1$ be a prime. Let $\Omega\subseteq \V$ be a non-empty set such that $k\Omega+u$ admits a rational $(\d,K)$-Leibman dichotomy up to degree $s$ and complexity $C$ for all $k\in\F_{p}\backslash\{0\}$ and $u\in\V$, and let $n_{\ast}\in \iota^{-1}(\Omega)$.  Let $G/\Gamma$ be an $s$-step $\N$-filtered nilmanifold of complexity at most $C$, equipped with a $C$-rational Mal'cev basis $\mathcal{X}$, and let  $g\in \poly(\Z^{d}\to G_{\N})$ be  rational. If  $(g(n)\Gamma)_{n\in \iota^{-1}(\Omega)}$ is not $\d$-totally equidistributed on $G/\Gamma$, 
	then we may factorize $g$ as $$g(n)=\e g'(n)\gamma(n)  \text{ for all }  n\in \Z^{d},$$  where $\e\in G$ is of complexity at most 1,
	$g'\in \poly(\Z^{d}\to G'_{\N})$ is  rational 
	for some sub-nilmanifold $G'/\Gamma'$ of $G/\Gamma$ whose Mal'cev basis is $O_{C,d}(K)$-rational relative to $\mathcal{X}$ and whose total dimension is smaller than that of $G/\Gamma$, and   $\gamma\in\poly_{\approx r,n_{\ast}}(\iota^{-1}(\Omega)\to G_{\N})$ 
	for some $r\in\N_{+}$ with $r\ll_{d,m}\d^{-O_{d,m}(1)}$. 
\end{prop}

\begin{proof} 
Throughout the proof we assume that $p\gg_{C,d,K,s} 1$.
Denote $m_{i}:=\dim(G_{i})$ for $i\in\N$ and $m:=\dim(G)$. 
Let $\mathcal{X}=\{X_{1},\dots,X_{m}\}$ be the Mal'cev basis of $G/\Gamma$.
Write $g:=\{g(n_{\ast})\}g_{\ast}[g(n_{\ast})]$. Then $g_{\ast}$ belongs to $\poly(\Z^{d}\to G_{\N})$ and is rational by the Baker-Campbell-Hausdorff formula. So replacing $g$ by $g_{\ast}$ if necessary,  
	we may assume without loss of generality that $g(n_{\ast})=id_{G}$.
 
	If $(g(n)\Gamma)_{n\in \iota(\Omega)^{-1}}$ is not $\d$-totally equidistributed on $G/\Gamma$, then there exists $r_{0}\in\N_{+}$ with $r_{0}<\d^{-1}$ and some $u\in [r_{0}]^{d}$ such that   $(g(n)\Gamma)_{n\in\iota^{-1}(\Omega)\cap (r_{0}\Z^{d}+u)}$ is not $\d$-equidistributed on $G/\Gamma$,  or equivalently $(g(r_{0}n+u)\Gamma)_{n\in\iota^{-1}(r_{0}^{-1}(\Omega-u))}$ is not $\d$-equidistributed on $G/\Gamma$, 
	where we slightly abuse the notation to treat $r_{0}$ and $u$ as elements in $\F_{p}$ and $\V$ in the expression $r_{0}^{-1}(\Omega-u)$.	
	Since $r_{0}^{-1}(\Omega-u)$ admits a rational $(\d,K)$-Leibman dichotomy  and since $g(n_{\ast})=id_{G}$,  there exists a nontrivial  type-I horizontal character $\eta$ with $0<\Vert\eta\Vert\leq K$ and some $a\in\Q$  such that $\eta\circ g(r_{0}n+u)+a\in\Z$  for all $n\in\iota^{-1}(r_{0}^{-1}(\Omega-u))$. 
	
Now for any $x\in\iota^{-1}(\Omega)$, there exist $v_{x},u_{x}\in\Z^{d}$ such that $x=u+r_{0}v_{x}+pu_{x}$. Then for any $y\in \Z^{d}$, since $x-u\equiv r_{0}(v_{x}+py) \mod p\Z^{d}$, we have that $v_{x}+py\in \iota^{-1}(r_{0}^{-1}(\Omega-u))$ and thus by assumption  $\eta\circ g(u+r_{0}(v_{x}+py))+a\in\Z$.  By interpolation, all the coefficients of the polynomial $\eta\circ g(u+r_{0}(v_{x}+p\cdot))+a$ belong to $\Z/(s!)^{d}$. Since $x=u+r_{0}(v_{x}+p\cdot(\frac{u_{x}}{r_{0}}))$, we have that $\eta\circ g(x)+a$ belongs to belong to $\Z/(r_{0})^{s}(s!)^{d}$ for all $x\in\iota^{-1}(\Omega)$. Since $g(n_{\ast})=id_{G}$, replacing $\eta$ by $(r_{0})^{s}(s!)^{d}\eta$ if necessary, we may assume without loss of generality that $\eta\circ g(n)\in\Z$ for all $n\in\iota^{-1}(\Omega)$.
	 	Let $G'$ be the connected component of $\ker(\eta)$. Then $G'$ is a subgroup of $G$ which is $O_{C,d}(K)$-rational relative to $\mathcal{X}$. 
	
	We may use the type-I Taylor expansion to write $g$ as
	$$g(n)=\prod_{i\in\N^{d},0<\vert i\vert\leq s}g_{i}^{\binom{n}{i}}  \text{ for all }  n\in\Z^{d}$$
	for some $g_{i}\in G_{\vert i\vert}$. 		
	 Write
	$$\psi(g(n))=\sum_{i\in\N^{d},0<\vert i\vert\leq s}t_{i}\binom{n}{i}$$
	for some $t_{i}\in\{0\}^{m-m_{\vert i\vert}}\times\R^{m_{\vert i\vert}}$.  
	The type-I horizontal character $\eta$ is given in coordinates by
	$$\eta\circ g(n)=\sum_{i\in\N^{d},0<\vert i\vert\leq s}\theta\cdot t_{i}\binom{n}{i}$$
	for some $\theta=(\theta_{1},\dots,\theta_{m})\in\Z^{m}$ of complexity at most $O_{C,d}(K)$.
	Let $0\leq s'\leq s$ be the largest integer such that at least one of the entries $\theta_{m-m_{s'}+1},\dots,\theta_{m-m_{s'+1}}$ is nonzero (where we denote $m_{0}:=m$). Since 
	$\eta$ is nontrivial, such an $s'$ always exists.
	Denote 
	$m':=m-m_{s'+1}$.
	Then we may write $\theta=(\theta',\bold{0})$ for some $\theta'\in \Z^{m'}$. Write $t_{i}=(t'_{i},t''_{i})$ for some $t'_{i}\in\R^{m'}$ and $t''_{i}\in \R^{m-m'}$. 
	Since the first $m-m_{\vert i\vert}$ entries of $t_{i}$ are zero, and the last $m_{s'+1}$ entries of $\theta$ are zero, we have that $\theta\cdot t_{i}=0$ if $\vert i\vert\geq s'+1$. 
	So
	$$\eta\circ g(n)=\sum_{i\in\N^{d},0<\vert i\vert\leq s}\theta'\cdot t'_{i}\binom{n}{i}=\sum_{i\in\N^{d},0<\vert i\vert\leq s'}\theta'\cdot t'_{i}\binom{n}{i}.$$
	
	 Since $\theta$ is of complexity at most $O_{C,d}(K)$, there exists $Q\in\N_{+}, Q\leq O_{C,\d,d,K}(1)$ and $u=(u_{1},\dots,u_{m})\in\Z^{m}$ whose only nonzero entries are $u_{m-m_{s'}+1},\dots,u_{m-m_{s'+1}}$ such that 
	$\theta\cdot u=Q.$ Denote $$\gamma(n):=\prod_{i=m-m_{s'}+1}^{m-m_{s'+1}}\exp(\frac{1}{Q}u_{i}\eta(g(n))X_{i}),$$
 Then $\gamma$ belongs to $\poly(\Z^{d}\to G_{\N})$. Since $g$ is  rational, so is $\gamma$ by the Baker-Campbell-Hausdorff formula.

 We now show that there exists $r\in\N_{+}$ with $r\ll_{d,m}\d^{-\O_{d,m}(1)}$ such that 
 $\gamma$ belongs to $\poly_{\approx r,n_{\ast}}(\iota^{-1}(\Omega)\to G_{\N})$.
 In fact, it suffices to show that $\eta\circ g(n_{\ast}+rn)\in Q\Z$ for all $n\in\Z^{d}$ with $n_{\ast}+rn\in \iota^{-1}(\Omega)$.
 For all $n\in \Z^{d}$, since $n_{\ast}+pn\in \iota^{-1}(\Omega)$, we have that 
 $\eta\circ g(n_{\ast}+pn)\in\Z$. By interpolation, all the coefficients of the polynomial $\eta\circ g$  belongs to $\Z/p^{s'}Q'$ for some $s'=O_{d,s}(1)$ and $Q'\in\N_{+}$ with $Q'\leq O_{d,s}(1)$. So for all $n\in\Z^{d}$, we have that $\eta\circ g(n_{\ast}+QQ'n)\in Q\Z/p^{s'}$. On the other hand, if $n_{\ast}+QQ'n\in \iota^{-1}(\Omega)$, then $\eta\circ g(n_{\ast}+QQ'n)\in\Z$, so we have that $\eta\circ g(n_{\ast}+QQ'n)\in(Q\Z/p^{s'})\cap \Z=Q\Z$. So 
 $\gamma\in\poly_{\approx QQ',n_{\ast}}(\iota^{-1}(\Omega)\to G_{\N})$.

	On the other hand, writing $g':=g\gamma^{-1}$ we have that $g'\in\poly(\Z^{d}\to G_{\N})$ by Corollary B.4 of \cite{GTZ12} and that $g'$ is rational by the Baker-Campbell-Hausdorff formula.  Since
	$$\eta\circ g'(n)=\eta\circ g(n)-\frac{1}{Q}(\theta\cdot u)\eta\circ g(n)=0,$$
	we have that $g'\in \poly(\Z^{d}\to G'_{\N})$.   We are done since $g=g'\gamma$.
%
%
%
\end{proof}

We may now repeatedly using Proposition \ref{1:facf2r} to deduce the following strong factorization property:

  \begin{thm}[Strong factorization property II]\label{1:facf3r}
  		Let $d\in\N_{+},r,s\in\N$, $C>0$, $\mathcal{F}\colon\R_{+}\to\R_{+}$ be a growth function, $p\gg_{C,d,\mathcal{F},s} 1$ be a prime, $M\colon\V\to \F_{p}$ be a quadratic form, and $\Omega\subseteq \V$ satisfying one of the following assumptions:
  		\begin{enumerate}[(i)] 
  			\item $\Omega=\V$;
  			\item $\Omega=V(M)\cap (V+c)$ for some affine subspace $V+c$ of $\V$ with $\rank(M\vert_{V+c})\geq s+13$.
  		\end{enumerate}
  		Let $n_{\ast}\in\iota^{-1}(\Omega)$, $G/\Gamma$ be an $s$-step $\N$-filtered nilmanifold of  complexity at most $C$, and let $g\in \poly(\Z^{d}\to G_{\N})$ be rational.	
  	There exist some $C\leq C'\leq O_{C,d,\mathcal{F},s}(1)$, 
  		 a proper subgroup $G'$ of $G$ which is $C'$-rational relative to $\mathcal{X}$, and
  	 a factorization $$g(n)=\e g'(n)\gamma(n)  \text{ for all }  n\in \Z^{d}$$  such that $\e\in G$ is of complexity $O_{C'}(1)$,
  		$g'\in \poly(\Z^{d}\to G'_{\N})$ is  rational  
		and $(g'(n)\Gamma)_{n\in\iota^{-1}(\Omega)}$  is $\mathcal{F}(C')^{-1}$-totally equidistributed on $G'/\Gamma'$, where $\Gamma':=G'\cap \Gamma$, and that $\gamma$ belongs to both $\poly_{\approx r,n_{\ast}}(\iota^{-1}(\Omega)\to G_{\N})$ and $\poly_{r}(\iota^{-1}(\Omega)\to G_{\N})$.
	%
	%
	for some $r\in\N_{+}$ with $r\leq O_{C,C',d,s}(1)$. 
  \end{thm}  
  
  \begin{proof}[Sketch of the proof]
   By Theorem \ref{1:Lei0} or \ref{1:rLei} combined with a change of variable, we have that  $k\Omega+u$ is a non-empty set which admits a rational $(\mathcal{F}(C)^{-1},O_{C,d,\mathcal{F},s}(1))$-Leibman dichotomy up to step $s$ and complexity $C$ for all $k\in\F_{p}\backslash\{0\}$ and $u\in\V$. By the Baker-Campbell-Hausdorff formula, it is not hard to see that the products of two polynomial sequences which are partially   $Q$-rational on  $\iota^{-1}(\Omega)$ defined over a nilmanifold of complexity $C$ is  partially $Q'$-rational on  $\iota^{-1}(\Omega)$ for some $Q'\leq O_{C,Q}(1)$.
  we may   obtain the conclusion of Theorem \ref{1:facf3r} by repeatedly using Proposition \ref{1:facf2r} (and then apply  Proposition \ref{1:oehf} at the last step). We leave the details to the interested readers.
  \end{proof}

\section{Open questions}\label{1:s:oq}

We collect some open questions in this section. 
Recall that we need to require $d\geq s+13$ in Theorem \ref{1:sLei00}. It is natural to ask:

\begin{ques}
    What is the optimal lower bound for the dimension $d$ in Theorem \ref{1:sLei00}?
\end{ques}

In order to improve the lower bound of $d$ in Theorem \ref{1:sLei00}, it is natural to attempt to sharpen the tools we used in proving  Theorem \ref{1:sLei00}. Therefore, one can ask:

\begin{ques}
    What are the optimal lower bounds for the dimension $d$ in Proposition \ref{1:packforce0}, Theorems  \ref{1:irrr}, \ref{1:veryr}, and \ref{1:ct}?
\end{ques}

 Let $M\colon\V\to\F_{p}$ be a quadratic form.
   It is an interesting question to ask what is the connection between $p$-periodic polynomial sequences on $\V$ and  partially $p$-periodic polynomial sequences on $V(M)$. In particular, is it true that a partially $p$-periodic polynomial sequence on $V(M)$ merely the restriction to $V(M)$ of a $p$-periodic polynomial sequences on $\V$?

    \begin{ques}
    	Let $d,s\in\N_{+}$, $p$ be a prime, $M\colon\V\to \F_{p}$ be a     	quadratic form of rank at least 3, and $G_{\N}$ be a nilpotent group of step $s$. 
    		Let $p\gg_{d,s} 1$. 
		 For all $g\in\poly_{p}(V(M)\to G_{\N}\vert \Gamma)$, is there any $g'\in\poly_{p}(\V\to G_{\N}\vert \Gamma)$ such that $g(n)\Gamma=g'(n)\Gamma$ for all $n\in V(M)$?
    \end{ques}	 

Another interesting question is to ask if there is an intrinsic definition for partially $p$-periodic polynomials on $V(M)$. This is Conjecture \ref{1:att0} which we restate below:

 \begin{ques}
 Let $g\colon\V\to\F_{p}$ be a function, $d,s\in\N_{+}$ and $p\gg_{d} 1$ be a prime. Let $M\colon\V\to\F_{p}$ be a quadratic form with $\rank(M)\gg_{s} 1$. Are the followings equivalent?
	\begin{enumerate}[(i)]
 		\item $\Delta_{h_{s}}\dots\Delta_{h_{1}}g(n)=0$ for all $(n,h_{1},\dots,h_{s})\in \Gow_{s}(V(M))$.
 		\item  There exists a polynomial $g'\in\poly(\V\to\F_{p})$ of degree at most $s-1$ such that $g(n)=g'(n)$ for all $n\in V(M)$.
		\end{enumerate}	
   \end{ques}
    
    As is mentioned in Section \ref{1:s:idf}, we were only able to answer a special case of this question.
   
   Another interesting question in algebra is that whether all the nice and consistent $M$-sets satisfy Hilbert Nullstellensatz. This is Conjecture \ref{1:ccmmj} which we restate below:

   \begin{ques}\label{1:cl1ss}
	Let $d,k\in\N_{+}$, $s\in\N$, $p$ be a prime, $M\colon\V\to\F_{p}$ be a non-degenerate quadratic form, and $\Omega\subseteq(\V)^{k}$ be a consistent  $M$-set given by $\Omega=V(\mathcal{J})$ for some consistent $(M,k)$-family $\mathcal{J}=\{f_{1},\dots,f_{r}\}$ for some $r\in\N_{+}$.  If $d\gg_{k,r,s} 1$ and $p\gg_{d} 1$, then for any  polynomial $P\in\poly((\V)^{k}\to\F_{p})$ of degree at most $s$ with $\Omega\subseteq V(P)$, can we write
	\begin{equation}\nonumber
	P
	=\sum_{j=1}^{r}f_{j}Q_{j}
	\end{equation}
	for some polynomial $Q_{j}$ of degree $O_{k,r,s}(1)$? 
   \end{ques}
 
   In Appendix \ref{1:s:AppC},
   We prove a weaker version of Conjecture \ref{1:ccmmj} in Proposition \ref{1:cl1s}. Also it is  proved in Proposition \ref{1:noloop}  that  Conjecture \ref{1:ccmmj} holds for $\Omega=V(M)$. But the general case remains unanswered.

   Finally, it is an important question to ask whether the factorization theorems holds in the setting of $p$-periodic polynomial sequences.
   
   \begin{ques}\label{1:extq2}
     In Proposition \ref{1:facf2} and Theorem \ref{1:facf3}, if we assume in addition that $g$ is $p$-periodic, then can we also require $g'$ to be $p$-periodic in the factorization of $g$?
   \end{ques}
 
   If the answer to Question \ref{1:extq2} is affirmative, then many arguments in this paper as well as in \cite{SunC,SunD} can be greatly simplified by restricting ourselves to  $p$-periodic polynomial sequences.

\appendix

\section{Properties for liftings of polynomials}\label{1:s:AppA}

  In this appendix, we provide some basic properties for liftings of polynomials.

    \begin{lem}[Basic properties for lifting]\label{1:lifting}
    	Let $d,d',d''\in\N_{+}$ and $p>\max\{d,d',d''\}$ be a prime.  
	    	\begin{enumerate}[(i)]
    		\item Every $f\in \poly(\Z^{d}\to (\Z/p)^{d'})$ of degree less than $p$	 induces $\iota\circ pf\circ\tau$, which is a polynomial of degree at most $\deg(f)$. Moreover, if $f$ is homogeneous, then so is $\iota\circ pf\circ\tau$; 
    		\item Every $F\in \poly(\V\to\F^{d'}_{p})$ of degree less than $p$ admits a regular lifting. Moreover, if $F$ is homogeneous, we may further require the lifting to be homogeneous; 
    		\item If $f\in \poly(\Z^{d}\to (\Z/p)^{d'})$  is  a  lifting of some $F\in \poly(\V\to\F_{p}^{d'})$ with both $f$ and $F$ having degrees at most $p-1$, then for any $f'\in \poly(\Z^{d}\to (\Z/p)^{d'})$ of degree less than $p$, $f'$ is a  lifting of $F$ if and only if $f-f'$ is integer valued;
    		\item If $f,f'\in \poly(\Z^{d}\to (\Z/p)^{d'})$ are liftings of $F, F'\in \poly(\V\to\F_{p}^{d'})$ respectively with all of $f,f',F$ and $F'$ having degrees at most $(p-1)/2$, then $f+f'$ is a lifting of $F+F'$ and  $pff'$ is a lifting of $FF'$;
		\item If $f\in \poly(\Z^{d}\to (\Z/p)^{d'})$  is  a  lifting of some $F\in \poly(\V\to\F_{p}^{d'})$ with both $f$ and $F$ having degree at most $\sqrt{p}$, and $f\in \poly(\Z^{d'}\to (\Z/p)^{d''})$  is  a  lifting of some $F\in \poly(\F_{p}^{d'}\to\F_{p}^{d''})$ with both $f$ and $F$ having degrees less than $\sqrt{p}$, then $f'\circ pf$ is  a  lifting of  $F'\circ F$;
		\item If $f\in \poly(\Z^{d}\to (\Z/p)^{d'})$  is  a  lifting of some $F\in \poly(\V\to\F_{p}^{d'})$ with both $f$ and $F$ having degrees less than $p$, then $\tau\circ F\equiv pf\circ \tau \mod \Z^{d'}$.
    	\end{enumerate}
    \end{lem}	
    \begin{proof}
    Part (i). 
	Since $f$ takes values in $(\Z/p)^{d'}$, by the multivariate polynomial interpolation, 
	there exists $Q\in\N, p\nmid Q$ such that 
    	\begin{equation}\nonumber
    	f(n_{1},\dots,n_{d}):=\sum_{0\leq a_{1},\dots,a_{d}\leq p-1, a_{i}\in\N}\frac{C_{a_{1},\dots,a_{d}}}{Qp^{s}}n^{a_{1}}_{1}\dots n^{a_{d-1}}_{d}
    	\end{equation}
    	for some $C_{a_{1},\dots,a_{d}}\in\Z^{d'}$ and $s\in\N$. Using the property that $f(n+pm)-f(n)\in\Z^{d'}$ for all $m,n\in\Z^{d}$, it is not hard to see that we may further write $f$ as
    	\begin{equation}\nonumber
    	f(n_{1},\dots,n_{d}):=\sum_{0\leq a_{1},\dots,a_{d}\leq p-1, a_{i}\in\N}\frac{C'_{a_{1},\dots,a_{d}}}{Qp}n^{a_{1}}_{1}\dots n^{a_{d-1}}_{d}
    	\end{equation}
    	for some $C'_{a_{1},\dots,a_{d}}\in\Z^{d'}$ and $Q\in\N, p\nmid Q$.
    	Let $C''_{a_{1},\dots,a_{d}}\in\{0,\dots,p-1\}^{d'}$ be such that $C''_{a_{1},\dots,a_{d}}Q\equiv C'_{a_{1},\dots,a_{d}} \mod p\Z^{d'}$. Denote $f'\colon\Z^{d}\to(\Z/p)^{d'}$,
    	\begin{equation}\nonumber
    		f'(n_{1},\dots,n_{d}):=\sum_{0\leq a_{1},\dots,a_{d}\leq p-1, a_{i}\in\N}\frac{C''_{a_{1},\dots,a_{d}}}{p}n^{a_{1}}_{1}\dots n^{a_{d-1}}_{d}.
    	\end{equation}
    	and $F\colon\V\to\F_{p}^{d'}$,
    	\begin{equation}\nonumber
    	F(n_{1},\dots,n_{d}):=\sum_{0\leq a_{1},\dots,a_{d}\leq p-1, a_{i}\in\N}\iota(C''_{a_{1},\dots,a_{d}})n^{a_{1}}_{1}\dots n^{a_{d-1}}_{d}.
    	\end{equation}
    	Since $f(n)-f'(n)\in (\Z^{d'}/Q)\cap (\Z^{d'}/p^{d'})=\Z^{d'}$ for all $n\in\Z^{d}$, we have that
     $\iota\circ pf\circ\tau=\iota\circ pf'\circ\tau=F$ is a well defined polynomial in $\poly(\V\to\F_{p}^{d'})$. It is also clear that the degree of $F$ is at most $\deg(f)$, and that $F$ is homogeneous if $f$ is homogeneous. 
    	
    	Part (ii). Assume that 
    	\begin{equation}\nonumber
    	F(n_{1},\dots,n_{d})=\sum_{0\leq a_{1},\dots,a_{d}\leq p-1, a_{i}\in\N}C_{a_{1},\dots,a_{d}}n^{a_{1}}_{1}\dots n^{a_{d-1}}_{d}
    	\end{equation}
    	for some $C_{a_{1},\dots,a_{d}}\in \F_{p}^{d'}$. Let
    	\begin{equation}\nonumber
    	f(n_{1},\dots,n_{d}):=\sum_{0\leq a_{1},\dots,a_{d}\leq p-1, a_{i}\in\N}C'_{a_{1},\dots,a_{d}}n^{a_{1}}_{1}\dots n^{a_{d-1}}_{d},
    	\end{equation}
    	where $C'_{a_{1},\dots,a_{d}}$ is any element in $\{0,\frac{1}{p},\dots,\frac{p-1}{p}\}^{d'}$ such that $\iota(pC'_{a_{1},\dots,a_{d}})=C_{a_{1},\dots,a_{d}}$. Then it is easy to check that $f$ is a regular lifting of $F$. Moreover, $f$ is homogeneous if $F$ is homogeneous.
    	
    	Part (iii). 
    	Suppose first that $f'$ is a  lifting of $F$. Note that both $f$ and $f'$ are $p$-periodic by Lemma \ref{1:pp2}.
    	Then $\iota\circ pf\circ \tau=F=\iota\circ pf'\circ \tau$. So $pf(\tau(n))-pf'(\tau(n))\in p\Z^{d'}$. This implies that $f(m)-f'(m)\in\Z^{d'}$ for all $m\in \tau(\F_{p})=\{0,\dots,p-1\}^{d}$. Since $f-f'$ is $p$-periodic, we have that $f(m)-f'(m)\in\Z^{d'}$ for all $m\in\Z^{d}$.
    	
    	Conversely, if $f(m)-f'(m)\in\Z^{d'}$ for all $m\in\Z^{d}$, then $pf(\tau(n))-pf'(\tau(n))\in p\Z^{d'}$ and thus $\iota\circ pf'\circ\tau(n)=\iota\circ pf\circ \tau(n)=F(n)$.
    	
    	Part (iv). It is clear that both $f+f'$ and $pff'$ have  degrees less than $p$.
	The conclusion follows from the fact that 
	$$\iota\circ pf\circ \tau+\iota\circ pf\circ \tau=\iota\circ p(f+f')\circ \tau \text{ and } (\iota\circ pf\circ \tau)\cdot(\iota\circ pf\circ \tau)=\iota\circ p^{2}ff'\circ \tau.$$
	
	Part (v). It is clear that $f'\circ pf\in \poly(\Z^{d}\to (\Z/p)^{d''})$ and $F'\circ F\in \poly(\V\to\F_{p}^{d''})$, both of which are of degree less than $p$. 	Since $\tau \circ\iota\equiv id \mod p\Z$, we have that $pf'\equiv pf'\circ \tau \circ\iota \mod p\Z^{d'}$ and thus
	$$\iota\circ pf'\circ pf\circ\tau=\iota\circ (pf'\circ\tau \circ\iota)\circ pf\circ\tau=(\iota\circ pf'\circ\tau) \circ(\iota\circ pf\circ\tau)=F'\circ F.$$
So 	$f'\circ pf$ is  a  lifting of  $F'\circ F$.

Part (vi). Since $\tau \circ\iota\equiv id \mod p$, we have that $$\tau\circ F=\tau\circ(\iota\circ pf\circ\tau)=(\tau\circ\iota)\circ pf\circ\tau\equiv pf\circ\tau \mod p\Z^{d'}.$$
    \end{proof}

  \section{Properties for $M$-sets}\label{1:s:AppB}
  
  In this appendix, we provide some elementary properties for $M$-sets. We also present some important examples of $M$-sets which will be used in \cite{SunB,SunC,SunD}.
  
  We start with the following lemma on the existence of $M$-representations.
  
  		\begin{lem}\label{1:rep0}
		     Let $d\in\N_{+}, r\in\N$, $p$ be a prime, and $M\colon\V\to\F_{p}$ be a quadratic form. For any consistent $(M,k)$-family $\mathcal{J}=\{F_{1},\dots,F_{r}\}$, the set $\Omega:=V(\mathcal{J})$ admits a standard $M$-representation $(F'_{1},\dots,F'_{r'})$ for some $r'\in\N$.
    
    Moreover, if $\mathcal{J}$ is independent, then we may further require $r'=r$; if $\mathcal{J}$ is nice/pure, then we may further require $F'_{1},\dots,F'_{r'}$ to be nice/pure. 
		\end{lem}
		\begin{proof}
	 We may assume without loss of generality
	  that $v_{M}(F_{1}),\dots,v_{M}(F_{r'})$ is a basis of $\sp_{\F_{p}}\{v_{M}(F_{1}),\dots,v_{M}(F_{r})\}$ for some $1\leq r'\leq r$.
	Fix $r'+1\leq i\leq r$. There exist $t_{1},\dots,t_{r'}\in\F_{p}$ such that
	\begin{equation}\nonumber
	\begin{split}
	t_{1}v_{M}(F_{1})+\dots+t_{r'}v_{M}(F_{r'})-v_{M}(F_{i})=\bold{0}.
	\end{split}
	\end{equation}
	Therefore,
	\begin{equation}\nonumber
	\begin{split}
	t_{1}F_{1}+\dots+t_{r'}F_{r'}-F_{i}\equiv 0.
	\end{split}
	\end{equation}
	So $V(F_{i})\subseteq \cap_{j=1}^{r'}V(F_{j})$	and thus 	
	$\Omega=\cap_{j=1}^{r'}V(F_{j}).$	
	 $\{F_{i}\colon (\V)^{k}\to\V\colon 1\leq i\leq r\}$ is consistent, so is $\{F_{i}\colon (\V)^{k}\to\V\colon 1\leq i\leq r'\}$. Also since $v_{M}(F_{1}),\dots,v_{M}(F_{r'})$ are linearly independent, 
	the matrix 
		$\begin{bmatrix}
		v_{M}(F_{1})\\
		\dots \\
		v_{M}(F_{r'})
		\end{bmatrix}$ is of rank $r'$ and does not contain a row of the form $(0,\dots,0,a),a\in\F_{p}$. We remark that one can choose $r'=r$ if $\mathcal{J}$ is independent.
	
	Let $X=(x_{i,j})_{1\leq i,j\leq r'}$ be an $r'\times r'$ invertible matrix and denote 
	$F'_{i}:=\sum_{j=1}^{r'}x_{i,j}F_{j}$ for $1\leq i\leq r'$. Then clearly $\sp_{\F_{p}}\{v_{M}(F'_{i}), 1\leq i\leq r'\}=\sp_{\F_{p}}\{v_{M}(F_{i}), 1\leq i\leq r\}$  and 	$\Omega=\cap_{j=1}^{r'}V(F'_{j}).$ Since the matrix 
		$\begin{bmatrix}
		v_{M}(F_{1})\\
		\dots \\
		v_{M}(F_{r'})
		\end{bmatrix}$ is of  rank $r'$ and  does not contain a row of the form $(0,\dots,0,a),a\in\F_{p}$,	we may choose $X$ in a way such that the matrix $\begin{bmatrix}
	v_{M}(F'_{1})\\
	\dots \\
	v_{M}(F'_{r'})
	\end{bmatrix}$ is in the reduced row echelon form  which does not contain a row of the form $(0,\dots,0,a),a\in\F_{p}$. So $v'_{M}(F'_{1}),\dots,v'_{M}(F'_{r'})$ are linearly independent and thus $(F'_{1},\dots,F'_{r'})$ is a standard $M$-representation of $\Omega$.
	
	Moreover, if $\mathcal{J}$ is nice/pure, then all of $F_{1},\dots,F_{r'}$ are nice/pure. By the knowledge of linear algebra, it is not hard to see that we may choose $X$ in a way such that $F'_{1},\dots,F'_{r'}$ are also nice/pure. This completes the proof.
			\end{proof}
		
		We next provide some basic properties for different notions of $M$-sets.
		
 \begin{coro}\label{1:rep}
     Let $d\in\N_{+}$, $p$ be a prime, and $M\colon\V\to\F_{p}$ be a quadratic form.
     \begin{enumerate}[(i)]
         \item Any independent $M$-family, $M$-set, or $M$-representation is also a consistent $M$-family, $M$-set, or $M$-representation respectively.
         \item Non-consistent $M$-sets are empty.\footnote{Conversely, by Theorem \ref{1:ct} to be proved later in this paper, consistent $M$-sets are non-empty provided that $d$ and $p$ are sufficiently large.}
         \item Any consistent $M$-set $\Omega$ admits a standard $M$-representation (so consistent $M$-sets are the same as independent $M$-sets).
     \end{enumerate}
 \end{coro}

\begin{proof} 
    Part (i) follows from the definitions. For Part (ii), if $\Omega=\cap_{i=1}^{r}V(F_{i})$ for some non-consistent $(M,k)$-family $\{F_{1},\dots,F_{r}\}$, then 1 is a linear combination of $F_{1},\dots,F_{r}$. So $\Omega\subseteq V(1)=\emptyset$. Part (iii) follows from Lemma \ref{1:rep0}.
    \end{proof}

The next proposition provides us some simple ways to construct new $M$-sets from an existing one.	
	
	\begin{prop}\label{1:yy3}
	Let $d,k,r\in\N_{+}$, $p$ be a prime, $M\colon\V\to\F_{p}$ be a quadratic form and $\mathcal{J}=\{F_{1},\dots,F_{r}\}, F_{i}\colon(\V)^{k}\to \F_{p}, 1\leq i\leq r$ be an $(M,k)$-family.
	\begin{enumerate}[(i)]
	\item For any subset $I\subseteq \{1,\dots,r\}$, $\{F_{i}\colon i\in I\}$ is an    $(M,k)$-family.
	\item Let $I\subseteq \{1,\dots,r\}$ be a subset such that all of $F_{i},i\in I$ are independent of the last $d$-variables. Then writing $G_{i}(n_{1},\dots,n_{k-1}):=F_{i}(n_{1},\dots,n_{k-1})$, we have that $\{G_{i}\colon i\in I\}$ is an $(M,k-1)$-family.  
	\item For any bijective $d$-integral linear transformation  $L\colon (\V)^{k}\to(\V)^{k}$ and any $v\in(\V)^{k}$, $\{F_{1}(L(\cdot)+v),\dots,F_{r}(L(\cdot)+v)\}, F_{i}\colon(\V)^{k}\to \F_{p}, 1\leq i\leq r$ is an $(M,k)$-family.
		\item Let $1\leq k'\leq k$ and $1\leq r'\leq r$ be such that $F_{1},\dots,F_{r'}$ are independent of $n_{k-k'+1},\dots,n_{r}$ and that every non-trival linear combination of $F_{r'+1},\dots,F_{r}$ dependents nontrivially on some of $n_{k-k'+1},\dots,n_{r}$. 
		Define $G_{i},H_{i}\colon (\V)^{k+k'}\to \F_{p}, 1\leq i\leq r$ by 
		$$\text{$G_{i}(n_{1},\dots,n_{k+k'}):=F_{i}(n_{1},\dots,n_{k})$ and $H_{i}(n_{1},\dots,n_{k+k'}):=F_{i}(n_{1},\dots,n_{k-k'},n_{k+1},\dots,n_{k+k'})$}.$$ Then $\{F_{1},\dots,F_{r'},G_{r'+1},\dots,G_{r},H_{r'+1},\dots,H_{r}\}$\footnote{Here we regard $F_{1},\dots,F_{r'}$ as functions of $n_{1},\dots,n_{k+k'}$, which acutally depends only on $n_{1},\dots,n_{k-k'}$.} is an $(M,k+k')$-family. 
	\end{enumerate}
Moreover, if $\mathcal{J}$ is nice, then so are the sets mentioned in Parts (i)-(iii);
if $\mathcal{J}$ is 
consistent/independent, then so are the sets mentioned in Parts (i)-(iv);  
if $\mathcal{J}$ is pure, then so are the sets mentioned in Parts (i), (ii), (iv), and so is the set mentioned in Part (iii) when $v=\bold{0}$.	
\end{prop}	
\begin{proof}
We only need to show that if $\mathcal{J}$ is 
consistent/independent, then so are the sets mentioned in Parts (i)-(iv), since the other parts are easy to verify.
For any $(M,k)$-integral quadratic function $G$, let $G''$ denote the degree 2 terms of $G$, and $G'$ denote the degree 1 and 2 terms of $G$.
 Then an $(M,k)$-family $\{G_{1},\dots,G_{t}\}$ is
 \begin{itemize}
 \item consistent if and only if for all $c_{1},\dots,c_{t}\in\F_{p}$, $c_{1}G'_{1}+\dots+c_{t}G'_{t}\equiv 0$ implies that $c_{1}G_{1}+\dots+c_{t}G_{t}\equiv 0$;
  \item independent if and only if for all $c_{1},\dots,c_{t}\in\F_{p}$, $c_{1}G'_{1}+\dots+c_{t}G'_{t}\equiv 0$ implies that $c_{1}=\dots=c_{t}=0$.
 \end{itemize}

Parts (i) and (ii) are straightforward. We now prove Part (iii).
Suppose first that 
$\sum_{i=1}^{r}c_{i}(F_{i}(L(\cdot)+v))'\equiv 0$ for some $c_{1},\dots,c_{r}\in\F_{p}$. Since $(F_{i}(L(\cdot)+v))'-F'_{i}(L(\cdot)+v)$ is a constant, so is $\sum_{i=1}^{r}c_{i}F'_{i}(L(\cdot)+v)$.
Since $L$ is bijective, we have that $\sum_{i=1}^{r}c_{i}F'_{i}$ and thus $\sum_{i=1}^{r}c_{i}F_{i}$ is a constant. If $\mathcal{J}$ is consistent, then this implies that $\sum_{i=1}^{r}c_{i}F_{i}\equiv 0$ and thus $\sum_{i=1}^{r}c_{i}(F_{i}(L(\cdot)+v))\equiv 0$. If $\mathcal{J}$ is independent, then this implies that $c_{1}=\dots=c_{t}=0$. 
This proves Part (iii).

For Part (iv), suppose first that 
\begin{equation}\label{1:cpi1}
\sum_{i=1}^{r'}c_{i}F'_{i}(n_{1},\dots,n_{k-k'})+\sum_{i=r'+1}^{r}(c'_{i}G'_{i}+c''_{i}H'_{i})(n_{1},\dots,n_{k+k'})\equiv 0
\end{equation}
  for some $c_{i}, c'_{i},c'_{i}\in\F_{p}$. Setting $n_{k-k'+j}=n_{k+j}$ for all $1\leq j\leq k'$, we have that 
 \begin{equation}\label{1:cpi2}
\sum_{i=1}^{r'}c_{i}F'_{i}(n_{1},\dots,n_{k-k'})+\sum_{i=r'+1}^{r}(c'_{i}+c''_{i})F'_{i}(n_{1},\dots,n_{k})\equiv 0.
\end{equation} 

If $\mathcal{J}$ is consistent, then (\ref{1:cpi2}) implies that $\sum_{i=1}^{r'}c_{i}(F_{i}-F'_{i})+\sum_{i=r'+1}^{r}(c'_{i}+c''_{i})(F_{i}-F'_{i})\equiv 0$.
 Since $F_{i}-F'_{i}=G_{i}-G'_{i}=H_{i}-H'_{i}$, we have that $$\sum_{i=1}^{r'}c_{i}(F_{i}-F'_{i})(n_{1},\dots,n_{k-k'})+\sum_{i=r'+1}^{r}(c'_{i}(G_{i}-G'_{i})+c''_{i}(H_{i}-H'_{i}))(n_{1},\dots,n_{k+k'})\equiv 0.$$ It then follows from (\ref{1:cpi1}) that 
$$\sum_{i=1}^{r'}c_{i}F_{i}(n_{1},\dots,n_{k-k'})+\sum_{i=r'+1}^{r}(c'_{i}G_{i}+c''_{i}H_{i})(n_{1},\dots,n_{k+k'})\equiv 0.$$
 So $\{F_{1},\dots,F_{r'},G_{r'+1},\dots,G_{r},H_{r'+1},\dots,H_{r}\}$ is consistent.
 
 If $\mathcal{J}$ is independent, then (\ref{1:cpi2}) implies that $c_{1}=\dots=c_{r'}=c'_{r'+1}+c''_{r'+1}=\dots=c'_{r}+c''_{r}=0$.
 Substituting this back to (\ref{1:cpi1}), we have that 
 \begin{equation}\nonumber
\sum_{i=r'+1}^{r}c'_{i}F'_{i}(n_{1},\dots,n_{k})\equiv -\sum_{i=r'+1}^{r}c'_{i}F'_{i}(n_{1},\dots,n_{k-k'},n_{k+1}\dots,n_{k+k'}).
\end{equation}
This implies that $\sum_{i=r'+1}^{r}c'_{i}F'_{i}(n_{1},\dots,n_{k})$ is independent of $n_{k-k'+1},\dots,n_{r}$.
If $c'_{r'+1},\dots,c'_{r}$ are not all zero, by assumption, $\sum_{i=r'+1}^{r}c'_{i}F_{i}(n_{1},\dots,n_{k})$ and thus $\sum_{i=r'+1}^{r}c'_{i}F'_{i}(n_{1},\dots,n_{k})$ dependents nontrivially on  $n_{k-k'+1},\dots,n_{r}$, a contradiction. 
 So  $c'_{r'+1}=\dots=c'_{r}=0$ and thus $c''_{r'+1}=\dots=c''_{r}=0$. So 
$\{F_{1},\dots,F_{r'},G_{r'+1},\dots,G_{r},H_{r'+1},\dots,H_{r}\}$ is independent.
\end{proof}

Finally, we provide some  examples of nice $M$-sets which will be used in \cite{SunB,SunC,SunD}.

\begin{ex}\label{1:mainex}
Let $d\in\N_{+}$, $s\in\N$, $p$ be a prime, and $M\colon\V\to\F_{p}$ be a non-degenerate quadratic form 
associated with the matrix $A$.
 Clearly, we may write 
\begin{equation}\label{1:mmaa}
V(M)=V(M')+b
\end{equation}
for some pure non-degenerate quadratic form $M'\colon \V\to\F_{p}$ of the form $M'(n)=(nA)\cdot n+c$ for some $c\in\F_{p}$, for and some $b\in\V$.
	\begin{itemize} 
		\item The set $\V$ is a nice and consistent $M$-set of total co-dimension 0 by Convention \ref{1:fpism}.
		\item The set $V(M)\subseteq \V$ is a nice and consistent $M$-set of total co-dimension 1, since $M'$ is pure.
		\item The Gowers set $\Gow_{s}(V(M))$ is a nice and consistent $M$-set. 
		To see this, note that $\Gow_{s}(V(M))=\Gow_{s}(V(M'))+(b,\bold{0},\dots,\bold{0})$. 
		Let $L\colon(\V)^{s+1}\to(\V)^{s+1}$ be the bijective $d$-integral linear transformation given by $L(n,h_{1},\dots,h_{s}):=(h_{1},\dots,h_{s},n)$.
		 By Lemma \ref{1:changeh}, $(h_{1},\dots,h_{s},n)\in L(\Gow_{s}(V(M')))$ if and only if $M'(n), M'(n+h_{i})-M'(n), (h_{i}A)\cdot h_{j}=0$ for all $1\leq i,j\leq s, i\neq j$. Since all of these functions forms a nice and independent $M$-family, we have that $L(\Gow_{s}(V(M')))$ is a nice and consistent $M$-set of total co-dimension $(s^{2}+s+2)/2$, and thus so is $\Gow_{s}(V(M))$ by Proposition \ref{1:yy3} (iii). 
	\item The set 
		$$\Omega_{1}:=\{(x_{1},\dots,x_{5})\in(\V)^{5}\colon x_{1}-x_{i},x_{i}\in V(M),2\leq i\leq 5\}$$
		  is a nice and consistent $M$-set.   
		  To see this, note that $\Omega_{1}=L(\Omega_{1}')+(2b,b,b,b,b)$, where $L\colon(\V)^{5}\to(\V)^{5}$ is given by $L(x_{1},\dots,x_{5})=(x_{5},x_{1},\dots,x_{4})$ and $\Omega'_{1}$ is the set of $(y_{1},\dots,y_{5})\in(\V)^{5}$ such that $M'(y_{i})$ and $M'(y_{i}-y_{5})$ are zero for all $1\leq i\leq 4$. Since all of these functions forms a nice and independent $M$-family, we have that $\Omega'_{1}$ is a nice and consistent $M$-set of total co-dimension 8, and thus so is $\Omega_{1}$ by Proposition \ref{1:yy3} (iii). 
\item 
		Let $s\geq 2$ and
		denote $\w:=(w_{1},\dots,w_{5})$, $\h':=(h_{1},\dots,h_{s-1})$ and  $\h:=(h_{1},\dots,h_{s+1})$ with $h_{i},w_{i}\in\V$.
		Let	$\Omega_{2}$ be the set of $(\w,\h)\in(\V)^{s+6}$ such that 
		 $(w_{1},\h),(w_{2},\h)\in \Gow_{s+1}(V(M))$, $(w_{3},\h',h_{s}),$ $(w_{4},\h',h_{s+1})\in \Gow_{s}(V(M))$, $(w_{5},\h')\in \Gow_{s-1}(V(M))$. 
		Then $\Omega_{2}=\Omega'_{2}+(b,b,b,b,b,\bold{0},\dots,\bold{0})$, where $\Omega_{2}'$ is the set of $(\w,\h)\in(\V)^{s+6}$ such that 
		 $(w_{1},\h),(w_{2},\h)\in \Gow_{s+1}(V(M'))$, $(w_{3},\h',h_{s}), (w_{4},\h',h_{s+1})\in \Gow_{s}(V(M'))$, $(w_{5},\h')\in \Gow_{s-1}(V(M'))$. 
		 Let $L\colon(\V)^{s+6}\to(\V)^{s+6}$ be the bijective $d$-integral linear transformation given by $L(\w,\h):=(\h,\w)$. 
		By Lemma \ref{1:changeh}, $(\h,\w)\in L(\Omega_{2}')$ if and only if $M(w_{i}), M(w_{i}+h_{j_{i}})-M(w_{i}), (h_{j}A)\cdot h_{j'}$ equal to zero for all $1\leq i\leq 5$, $1\leq j,j'\leq s+1, j\neq j'$, $1\leq j_{1},j_{2}\leq s+1$, $1\leq j_{3}\leq s$, $j_{4}\in\{1,\dots,s-1\}\cup\{s\}$ and $1\leq j_{5}\leq s-1$.
	Since all of these functions forms a nice and independent $M$-family, we have that $L(\Omega'_{2})$ is  a nice and consistent $M$-set of total co-dimension $(s^{2}+11s+12)/2$, and thus so is $\Omega_{2}$ by Proposition \ref{1:yy3} (iii).  
		
		\item 
		Let $M$ be of the form $M(n)=(nA)\cdot n$, $k\in\N_{+}, v_{0},\dots,v_{k}\in\V, r\in\F_{p}$ and $\Omega_{3}\subseteq (\V)^{k+1}$ be the set of $(x_{0},\dots,x_{k+1})\in (\V)^{k}$	such that $M(x_{i}-x_{j})=M(v_{i}-v_{j})$ and $M(x_{i})=r$ for all $0\leq i\leq j\leq k, i\neq j$. 
		Note that $(x_{0},\dots,x_{k})\in \Omega_{3}$ if and only if 
		$(x_{i}A)\cdot x_{i}-r$ and $(x_{i}A)\cdot x_{j}-r+\frac{1}{2}M(v_{i}-v_{j})$  are zero for all $0\leq i,j\leq k, i\neq j$. 
		Since all of these functions forms a nice and independent $M$-family, we have that $\Omega_{3}$ is a nice and consistent $M$-set of total co-dimension $(k+2)(k+1)/2$.
		\item Fix some $\Omega_{3}$ defined above. Let $\Omega_{4}$ denote the set of $(x,\x,\x')\in (\V)^{2k+1}$ such that $(x,\x),(x,\x')\in \Omega_{4}$, where $x\in\V$ and $\x,\x'\in(\V)^{k}$. 
		Note that $(x_{0},x_{1},\dots,x_{k},x'_{1},\dots,$ $x'_{k})$ belongs to $\Omega_{4}$ if and only if 
		$(x_{0}A)\cdot x_{0}-r$, $(x_{i}A)\cdot x_{i}-r$, $(x'_{i}A)\cdot x'_{i}-r$,  $(x_{0}A)\cdot x_{j}-r+\frac{1}{2}M(v_{0}-v_{j})$, $(x_{0}A)\cdot x'_{j}-r+\frac{1}{2}M(v_{0}-v_{j})$, $(x_{i}A)\cdot x_{j}-r+\frac{1}{2}M(v_{i}-v_{j})$, $(x'_{i}A)\cdot x'_{j}-r+\frac{1}{2}M(v_{i}-v_{j})$  are zero for all $1\leq i,j\leq k, i\neq j$. 
		Since all of these functions forms a nice and independent $M$-family, we have that $\Omega_{4}$ is a nice and consistent $M$-set of total co-dimension $k^{2}+3k+1$.
	\end{itemize}
\end{ex}

\section{Fubini's Theorem for $M$-sets}\label{1:s:dec}

In this section, we prove a Fubini-type theorem for consistent $M$-sets. To this end, we need to define the projections of $M$-sets. We start with defining the projections of ideals of the polynomial ring.
Let $n_{i}=(n_{i,1},\dots,n_{i,d})\in\V$ denote a $d$-dimensional variable for $1\leq i\leq k$. For convenience we denote $\F^{d}_{p}[n_{1},\dots,n_{k}]$ to be the polynomial ring $\F_{p}[n_{1,1},\dots,n_{1,d},\dots,n_{k,1},\dots,n_{k,d}]$.
Let $J,J',J''$ be  finitely generated ideals of the polynomial ring $\F^{d}_{p}[n_{1},\dots,n_{k}]$ and $I\subseteq \{n_{1},\dots,n_{k}\}$. 
Suppose that $J=J'+J''$, $J'\cap J''=\{0\}$, all the polynomials in $J'$ are independent of  $\{n_{1},\dots,n_{k}\}\backslash I$, and all the non-zero polynomials in $J''$ depend nontrivially on $\{n_{1},\dots,n_{k}\}\backslash I$. Then we say that $J'$ is an  \emph{$I$-projection} of $J$ and that $(J',J'')$ is an \emph{$I$-decomposition} of $J$.

\begin{prop}\label{1:ip}
Let $d,k\in\N_{+}$, $p\gg_{d,k} 1$ be a prime, and  $M\colon\V\to\F_{p}$ be a  quadratic form. Let $\mathcal{J}=\{F_{i}\in \F^{d}_{p}[n_{1},\dots,n_{k}]\colon 1\leq i\leq r\}$ be a consistent $(M,k)$-family and $I\subseteq \{n_{1},\dots,n_{k}\}$. Let $J$ be the ideal generated by $\mathcal{J}$. Then the $I$-projection of  $J$ exists and is unique.
\end{prop}
\begin{proof}
Note that for any $r\times r$ invertible matrix $(t_{i,j})_{1\leq i,j\leq r}$ in $\F_{p}$, $J$ is also the ideal generated by $\{F'_{i}:=t_{i,1}F_{1}+\dots+t_{i,r}F_{r} \colon 1\leq i\leq r\}$, which is also a consistent $(M,k)$-family. By the knowledge of linear algebra and the expression of $(M,k)$-quadratic functions, we may choose $(t_{i,j})_{1\leq i,j\leq r}$ in a way such that for some $0\leq r'\leq r$, $F'_{1},\dots,F'_{r'}$ are  independent of  $ \{n_{1},\dots,n_{k}\}\backslash I$, and $F'_{r'+1},\dots,F'_{r}$ depend nontrivially on $\{n_{1},\dots,n_{k}\}\backslash I$. Since $\{F'_{i}\colon 1\leq i\leq r\}$ is consistent, setting $J'$ to be the ideal generated by $F'_{1},\dots,F'_{r'}$ and $J''$ to be the ideal generated by $F'_{r'+1},\dots,F'_{k}$, we have that $J'$ is an $I$-projection  of $J$. 
This proves the existence of the $I$-projection (note that $J'$ and $J''$ dose not contain 1 since $\mathcal{J}$ is consistent).

Now suppose that $J$ has two $I$-projections $J'$ and $\tilde{J}'$. Then exist  some ideal  $J''$ of $\F_{p}[n_{1},\dots,n_{k}]$ such that $J=J'+J''$, $J'\cap J''=\{0\}$,   and all the non-zero polynomials in $J''$ depend nontrivially on $\{n_{1},\dots,n_{k}\}\backslash I$. 

Pick any $f\in \tilde{J}'$. We may write $f=f'+f''$ for some $f'\in J'$ and $f''\in J''$. Then $f''=f-f'$ is  independent of  $ \{n_{1},\dots,n_{k}\}\backslash I$. By the construction of $J''$, we have that $f''=0$ and thus $f=f'$. This means that $\tilde{J}'\subseteq J'$. By symmetry $J'\subseteq \tilde{J}'$ and so $\tilde{J}'=J'$. This proves the uniqueness of the $I$-projection.
\end{proof}

Let $\mathcal{J}$, $\mathcal{J}'$ and $\mathcal{J}''$ be  finite subsets of $\F^{d}_{p}[n_{1},\dots,n_{k}]$, and $I\subseteq \{n_{1},\dots,n_{k}\}$. 
Let $J,J'$, and $J''$ be the ideals generated by $\mathcal{J}$, $\mathcal{J}'$ and $\mathcal{J}''$ respectively. If $J'$ is an $I$-projection of $J$ with $(J',J'')$ being an $I$-decomposition of $J$, then we say that $\mathcal{J}'$ is an \emph{$I$-projection} of $\mathcal{J}$ and that $(\mathcal{J}',\mathcal{J}'')$ is an \emph{$I$-decomposition} of $\mathcal{J}$.
Note that the $I$-projection $\mathcal{J}'$ of $\mathcal{J}$ is not unique. However, by Proposition \ref{1:ip} the ideal generated by $\mathcal{J}'$ and the set of zeros $V(\mathcal{J}')$ are unique.

If $I$ is a subset of $\{1,\dots,k\}$, for convenience we say that $J'$ is an \emph{$I$-projection} of $J$ if $J'$ is an $\{n_{i}\colon i\in I\}$-projection of $J$. Similarly, we say that $(J',J'')$ is an \emph{$I$-decomposition} of $J$ if $(J',J'')$ is an $\{n_{i}\colon i\in I\}$-decomposition of $J$. Here $J,J',J''$ are either ideals or finite subsets of the polynomial ring.

We are now ready to define the projections of $M$-sets.
Let $\mathcal{J}\subseteq \F_{p}[n_{1},\dots,n_{k}]$ be a consistent $(M,k)$-family and 
let $\Omega=V(\mathcal{J})\subseteq(\V)^{k}$.
Let $I\cup J$ be a partition of $\{1,\dots,k\}$ (where $I$ and $J$ are non-empty), and $(\mathcal{J}_{I},\mathcal{J}'_{I})$ be an $I$-decomposition of $\mathcal{J}$.
Let $\Omega_{I}$ denote the set of $(n_{i})_{i\in I}\in(\V)^{\vert I\vert}$ such that $f(n_{1},\dots,n_{k})=0$ for all $f\in \mathcal{J}_{I}$ and $(n_{j})_{j\in J}\in(\V)^{\vert J\vert}$. 
Note that all $f\in \mathcal{J}_{I}$ are independent of $(n_{j})_{j\in J}$, and that $\Omega_{I}$ is independent of the choice of the $I$-decomposition.
We say that $\Omega_{I}$ is an \emph{$I$-projection} of $\Omega$.

 For $(n_{i})_{i\in I}\in (\V)^{\vert I\vert}$, let $\Omega_{I}((n_{i})_{i\in I})$ be the set of $(n_{j})_{j\in J}\in(\V)^{\vert J\vert}$ such that $f(n_{1},\dots,n_{k})=0$ for all $f\in \mathcal{J}'_{I}$. By construction, for any $(n_{i})_{i\in I}\in \Omega_{I}$, we have that $(n_{j})_{j\in J}\in\Omega_{I}((n_{i})_{i\in I})$ if and only if $(n_{1},\dots,n_{k})\in\Omega$. So for all  $(n_{i})_{i\in I}\in \Omega_{I}$, $\Omega_{I}((n_{i})_{i\in I})$ is independent of the choice of the $I$-decomposition.

  \begin{ex}
      Let $M\colon\V\to\F_{p}$ be a quadratic form associated to the matrix $A$. Consider the set $\Gow_{s}(V(M))$. By Lemma \ref{1:changeh}, we may write $\Gow_{s}(V(M))=V(\mathcal{J})\subseteq \F_{p}[n,h_{1},\dots,h_{s}]$ where 
      $$\mathcal{J}:=\{M(n),M(n+h_{i}), (h_{i}A)\cdot h_{j}\colon 1\leq i, j\leq s, i\neq j\}$$
      is a consistent $(M,s+1)$-family of total dimension $\frac{s(s+1)}{2}+1$.
      Let
      $$\mathcal{J}':=\{M(n),M(n+h_{i}), (h_{i}A)\cdot h_{j}\colon 1\leq i, j\leq t, i\neq j\}$$
       and $\mathcal{J}'':=\mathcal{J}\backslash \mathcal{J}'$. 
        It is not hard to compute that $(\mathcal{J}',\mathcal{J}'')$ is an  $\{n,h_{1},\dots,h_{t}\}$-decomposition of $\mathcal{J}$.  
       So $\Gow_{t}(V(M))$ is an $\{n,h_{1},\dots,h_{t}\}$-projection of $\Gow_{s}(V(M))$.
  \end{ex}

   We are now ready to state Fubini's theorem for $M$-sets:
 
\begin{thm}[Fubini's Theorem]\label{1:ct}
	Let $d,k\in\N_{+}$, $r_{1},\dots,r_{k}\in\N$ and $p$ be a prime.  Set $r:=r_{1}+\dots+r_{k}$.  
	Let  $M\colon\V\to\F_{p}$ be a  quadratic form with $\rank(M)\geq 2r+1$, and $\Omega\subseteq (\V)^{k}$ be a consistent $M$-set admitting a  standard $M$-representation of $\Omega$ with dimension vector $(r_{1},\dots,r_{k})$. 
		\begin{enumerate}[(i)]
		\item We have $\vert\Omega\vert=p^{dk-r}(1+O_{k,r}(p^{-1/2}))$;
		\item If $I=\{1,\dots,k'\}$ for some $1\leq k'\leq k-1$, then 
		$\Omega_{I}$ is a consistent $M$-set admitting a  standard  $M$-representation with dimension vector $(r_{1},\dots,r_{k'})$. Moreover,
		for all but at most $(k-k')r'_{k'}p^{dk'+r'_{k'}-1-\rank(M)}$ many $(n_{i})_{i\in I}\in(\V)^{k'}$, we have that $\Omega_{I}((n_{i})_{i\in I})$ has a standard $M$-representation with dimension vector $(r_{k'+1},\dots,r_{k})$  and that $\vert\Omega_{I}((n_{i})_{i\in I})\vert=p^{d(k-k')-(r_{k'+1}+\dots+r_{k})}(1+O_{k,r}(p^{-1/2}))$, where $r'_{k'}:=\max_{k'+1\leq i\leq k}r_{i}$. 
		\item 
		If $k\geq 2$, and $I\cup J$ is a partition of $\{1,\dots,k\}$ (where  $I$ and $J$ are non-empty), then for any function $f\colon \Omega\to\C$ with norm bounded by 1,   we have that
		\begin{equation}\label{1:ffbb}
		\E_{(n_{1},\dots,n_{k})\in\Omega}f(n_{1},\dots,n_{k})=\E_{(n_{i})_{i\in I}\in\Omega_{I}}\E_{(n_{j})_{j\in J}\in\Omega_{I}((n_{i})_{i\in I})}f(n_{1},\dots,n_{k})+O_{k,r}(p^{-1/2}).
		\end{equation}
	\end{enumerate} 
\end{thm}

The major difficulty in showing (\ref{1:ffbb}) is that the cardinality of $\Omega_{I}((n_{i})_{i\in I})$ is dependent on the choice of $(n_{i})_{i\in I}$.
However, under a careful computation, one can show that for almost all choices of $(n_{i})_{i\in I}$, the cardinalities of $\Omega_{I}((n_{i})_{i\in I})$ are almost the same. Therefore, it is still possible to write the average of $f$  as two iterated averages, at the cost of an error term $O_{k,r}(p^{-1/2})$.

\begin{proof}[Proof of Theorem \ref{1:ct}]
We first assume that  $I=\{1,\dots,k'\}$ and $J=\{k'+1,\dots,k\}$ for some $1\leq k'\leq k$.
We prove Theorem \ref{1:ct}  by induction with respect to $k$. We first show that Theorem \ref{1:ct} holds if $k=1$. Let $(F_{1},\dots,F_{r})$ be a  standard  $M$-representation of $\Omega$ (which exists by Corollary \ref{1:rep} (iii)). Since $k=1$, either $F_{1}(n)=(nA)\cdot n+v_{1}\cdot n+u_{1}$ and $F_{i}(n)=v_{i}\cdot n+u_{i}, 2\leq i\leq r$ for some linearly independent $v_{2},\dots,v_{r}$, or $F_{i}(n)=v_{i}\cdot n+u_{i}, 1\leq i\leq r$ for some linearly independent $v_{1},\dots,v_{r}$. Since $\{F_{1},\dots,F_{r}\}$ is consistent, in the former case $\Omega$ is the intersection of $V(F_{1})$ with an affine subspace of $\V$ of co-dimension $r-1$, and in the later case $\Omega$ is an affine subspace of $\V$ of co-dimension $r$. Since $\rank(M)-2(r-1)\geq 3$, by Lemma \ref{1:ns} and Corollary \ref{1:counting01},  we have that $\vert\Omega\vert=p^{d-r}(1+O(p^{-1/2}))$ in both cases.  Since $k=1$, there is nothing to prove in Parts (ii) and (iii).

	We now assume that the conclusion holds for 	$k-1$ for some $k\geq 2$ and we prove it for $k$.
	 We may then write $\Omega=V(\mathcal{J})$ for some $\mathcal{J}=(F_{i,j}\colon(\V)^{k}\to\F_{p}\colon 1\leq i\leq k, 1\leq j\leq r_{i})$, where $F_{i,j}$ depends nontrivially on $n_{i}$ and is independent of $n_{i'}$ for all $i'>i$.
	For convenience denote $\Omega_{k'}:=\Omega_{\{1,\dots,k'\}}$ and $\Omega_{k'}(n_{1},\dots,n_{k'}):=\Omega_{\{1,\dots,k'\}}(n_{1},\dots,n_{k'})$. 
	By Proposition \ref{1:yy3} (ii), it is not hard to see that $\Omega_{k'}=\cap_{1\leq i\leq k'}\cap_{1\leq j\leq r_{i}}V(F_{i,j})$ (where $F_{i,j}$ are viewed as a polynomial in $(n_{1},\dots,n_{k'})$ since it is independent of $n_{i+1},\dots,n_{k}$) and thus $\Omega_{k'}$ is a consistent $M$-set admitting a consistent $M$-representation with dimension vector $(r_{1},\dots,r_{k'})$. So by induction hypothesis, 
 \begin{equation}\label{1:coot1}
 \vert\Omega_{k'}\vert=p^{dk'-(r_{1}+\dots+r_{k'})}(1+O_{k,r}(p^{-1/2})).
 \end{equation}	
	For convenience also denote
	$\m=(n_{1},\dots,n_{k'})$ and $\n=(n_{k'+1},\dots,n_{k})$.
	Note that for all $\m\in\Omega_{k'}$, we have that $\n\in\Omega_{k'}(\m)$ if and only if $F_{i,j}(\m,\n)=0$ for all $i>k'$ and $1\leq j\leq r_{i}$.
	
	For $1\leq i\leq k$,
	since the matrix $\begin{bmatrix}
		v_{M}(F_{i,1})\\
		\dots \\
		v_{M}(F_{i,r_{i}})
		\end{bmatrix}$ is in the reduced row echelon form and does not contain a row of the form $(0,\dots,0,a),a\in\F_{p}$, we have that 
	$$F_{i,j}(n_{1},\dots,n_{i})=(n_{i}A)\cdot (u_{i,j,1}n_{1}+\dots+u_{i,j,i}n_{i})+v_{i,j}\cdot n_{i}+f_{i,j}(n_{1},\dots,n_{i-1})$$
	for all $1\leq j\leq r_{i}$
	for some $u_{i,j}=(u_{i,j,i},\dots,u_{i,j,1})\in(\V)^{i}$, $v_{i,j}\in\V$, and some  $(M,i-1)$-integral quadratic function $f_{i,j}$ such that the $r_{i}\times (i+1)d$ matrix $B_{i}:=\begin{bmatrix}
		u_{i,1}, v_{i,1}\\
		\dots \\
		u_{i,r_{i}}, v_{i,r_{i}}
		\end{bmatrix}$ is in the reduced row echelon form which does not contain zero rows. 
		 
Denote 	$u'_{i,j}:=(u_{i,j,i},\dots,u_{i,j,k'+1})$	and $v'_{i,j}(n_{1},\dots,n_{k'}):=v_{i,j}+(u_{i,j,1}n_{1}+\dots+u_{i,j,k'}n_{k'})A\in\V$	.
		Note that if we view $n_{1},\dots,n_{k'}$ as constants and consider $F_{i,j}$ as a function in the varaibles $n_{k'+1},\dots,n_{k}$, then for $k'+1\leq i\leq k$ and $1\leq j\leq r_{i}$,  	
$$F_{i,j}(n_{1},\dots,n_{i})=(n_{i}A)\cdot (u_{i,j,k'+1}n_{k'+1}+\dots+u_{i,j,i}n_{i})+v'_{i,j}(n_{1},\dots,n_{k'})\cdot n_{i}+f_{i,j}(n_{1},\dots,n_{i-1})$$
and thus 
$$v_{M}(F_{i,j}(\m,\cdot))=(0,\dots,0,u'_{i,j},v'_{i,j}(\m),\dots),$$
where $u'_{i,j}$ starts from the $n_{i}^{2}$ coordinate  and $v'_{i,j}(\m)$	terminates right before the $n_{i-1}^{2}$ coordinate (and we do not care about the remaining coordinates).	
		
Let $Z_{k'}$ denote the set of $\m\in(\V)^{k'}$ such that the $r_{i}\times (i-k'+1)d$ matrix $B_{i}(\m):=\begin{bmatrix}
		u'_{i,1}, v'_{i,1}(\m)\\
		\dots \\
		u'_{i,r_{i}}, v'_{i,r_{i}}(\m)
		\end{bmatrix}$
is not of  rank $r_{i}$ for some $k'+1\leq i\leq k$.
Note that the matrix with 
$$v_{M}(F_{k,1}),\dots,v_{M}(F_{k,r_{k}}),\dots,v_{M}(F_{1,1}),\dots,v_{M}(F_{1,r_{1}})$$
as row vectors is a block upper triangular matrix in the row echelon form of the form
$$\begin{bmatrix}
		B_{k} & \dots 	 & \dots & \dots	 & \dots	\\
		          & B_{k-1} & \dots & \dots  & \dots       \\
		          &              & \ddots & \dots & \dots \\
		          &              &           & B_{1} & \dots
		\end{bmatrix}$$
	which does not contain a row of the form $(0,\dots,0,a),a\in\F_{p}$ (note that there is only 1 column to the right of $B_{1}$, which correspondent to the constant coefficients).	
For all $\m\in\Omega_{k'}\backslash Z_{k'}$, $k'+1\leq i\leq k$ and $1\leq j\leq r_{i}$,
the matrix with 
\begin{equation}\nonumber
\begin{split}
v_{M}(F_{k,1}(\m,\cdot)),\dots,v_{M}(F_{k,r_{k}}(\m,\cdot)),\dots,v_{M}(F_{k'+1,1}(\m,\cdot)),\dots,v_{M}(F_{k'+1,r_{k'+1}}(\m,\cdot))
\end{split}
\end{equation}
as row vectors is a block upper triangular matrix of the form
$$\begin{bmatrix}
		B_{k}(\m) & \dots 	 & \dots & \dots	& \dots			\\
		          & B_{k-1}(\m) & \dots & \dots  & \dots            \\
		          &              & \ddots & \dots & \dots\\
		          &              &           & B_{k'+1}(\m) & \dots
		\end{bmatrix},$$
		where there is only 1 column to the right of $B_{k'+1}(\m)$.		
Since  $B_{i}(\m)$ is of rank $r_{i}$,
 this matrix does not contain a row of the form $(0,\dots,0,a),a\in\F_{p}$.  We may then use elementary row operations to transform this matrix into the reduced  echelon form which can be written as
$$\begin{bmatrix}
		B'_{k}(\m) & \dots 	 & \dots & \dots	& \dots			\\
		          & B'_{k-1}(\m) & \dots & \dots    & \dots          \\
		          &              & \ddots & \dots & \dots \\
		          &              &           & B'_{k'+1}(\m) & \dots
		\end{bmatrix}$$
for some 	$B'_{i}(\m)$, which are in the reduced echelon form and of the same size as $B_{i}(\m)$. Again this matrix does not contain a row of the form $(0,\dots,0,a),a\in\F_{p}$.
This implies that $\Omega_{k'}(\m)$ has a  standard $M$-representation with dimension vector $(r_{k'+1},\dots,r_{k})$.
 So by induction hypothesis,
 \begin{equation}\label{1:coot2}
 \vert\Omega_{k'}(\m)\vert=p^{d(k-k')-(r_{k'+1}+\dots+r_{k})}(1+O_{k,r}(p^{-1/2})) \text{ for all } \m\in\Omega_{k'}\backslash Z_{k'}.
 \end{equation}

	We now compute the cardinality of $Z_{k'}$. 
	Fix $k'+1\leq i\leq k$. If $B_{i}(\m)$ is not of rank $r_{i}$, then there exist $c_{1},\dots,c_{r_{i}}$ with at least one of them equal to 1 such that 
	$\sum_{j=1}^{r_{i}}c_{j}u'_{i,j}=\bold{0}$ and 
	$\sum_{j=1}^{r_{i}}c_{j}(v_{i,j}+(u_{i,j,1}n_{1}+\dots+u_{i,j,k'}n_{k'})A)=\bold{0}.$
	If $\sum_{j=1}^{r_{i}}c_{j}u_{i,j,t}=0$ for all $1\leq t\leq k'$, then we have that $\sum_{j=1}^{r_{i}}c_{j}v_{i,j}=\bold{0}$ and this contradicts to the fact that $B_{i}$ is of full rank. So at least one of $\sum_{j=1}^{r_{i}}c_{j}u_{i,j,t}, 1\leq t\leq k'$ is nonzero. Therefore, for all $m\in\V$, then number of $n_{1},\dots,n_{k'}\in\V$ with $m=\sum_{j=1}^{r_{i}}c_{j}(u_{i,j,1}n_{1}+\dots+u_{i,j,k'}n_{k'})$ is $p^{d(k'-1)}$.

	On the other hand,
	the number of $c_{j}\in\F_{p}, 1\leq j\leq r_{i}$ with at least one of them equal to 1 is at most $r_{i}p^{r_{i}-1}$; the number of $m$ with $\sum_{j=1}^{r_{i}}c_{j}v_{i,j}+mA=\bold{0}$ is at most $p^{d-\rank(M)}$.  So 
	 \begin{equation}\label{1:coot3}
 \vert Z_{k'}\vert\leq \sum_{i=k'+1}^{k}r_{i}p^{r_{i}-1}\cdot p^{d(k'-1)}\cdot p^{d-\rank(M)}\leq 
   (k-k')r'_{k'}p^{dk'+r'_{k'}-1-\rank(M)}.
 \end{equation}
  	So Part (ii) is proved by combining (\ref{1:coot1}), (\ref{1:coot2}) and (\ref{1:coot3}).

	Since 
	$\rank(M)\geq r_{1}+\dots+r_{k}$ by (\ref{1:coot1}), (\ref{1:coot2}), and (\ref{1:coot3}), we have
	\begin{equation}\nonumber
	\begin{split}
	&\quad\sum_{\m\in\Omega_{k'}\backslash Z_{k'}}\vert\Omega_{k'}(\m)\vert
	\\&=\vert \Omega_{k'}\backslash Z_{k'}\vert \cdot p^{d(k-k')-(r_{k'+1}+\dots+r_{k})}(1+O_{k,r}(p^{-1/2}))=p^{dk-(r_{1}+\dots+r_{k})}(1+O_{k,r}(p^{-1/2}))
	\end{split}
	\end{equation}
	and
\begin{equation}\label{1:keyerr}
	\begin{split}
	\sum_{\m\in Z_{k'}}\vert\Omega_{k'}(\m)\vert\leq \vert Z_{k'}\vert\cdot p^{d(k-k')}\leq (k-k')r'_{k'}p^{dk+r'_{k'}-1-\rank(M)}.
	\end{split}
	\end{equation}
    Since $\rank(M)\geq 2r+1$, we have that 
    $$\vert\Omega\vert=\sum_{\m\in\Omega_{k'}}\vert\Omega_{k'}(\m)\vert=p^{dk-(r_{1}+\dots+r_{k})}(1+O_{k,r}(p^{-1/2})).$$
    This proves Part (i).

	Finally, 	we have that
	\begin{equation}\label{1:6699}
	\begin{split}
	&\quad \vert\E_{(\m,\n)\in\Omega}f(\m,\n)-\E_{\n\in\Omega_{k'}}\E_{\n\in\Omega_{k'}(\m)}f(\m,\n)\vert
	\\&=\Bigl\vert\frac{1}{\vert\Omega_{k'}\vert}\sum_{\m\in\Omega_{k'}}\sum_{\n\in\Omega_{k'}(\m)}f(\m,\n)(\frac{\vert\Omega_{k'}\vert}{\vert\Omega\vert}-\frac{1}{\vert \Omega_{k'}(\m)\vert})\Bigr\vert
	\\&\leq  \frac{1}{\vert\Omega_{k'}\vert}\sum_{\m\in\Omega_{k'}}\sum_{\n\in\Omega_{k'}(\m)}\Bigl\vert\frac{\vert\Omega_{k'}\vert}{\vert\Omega\vert}-\frac{1}{\vert \Omega_{k'}(\m)\vert}\Bigr\vert
	=\frac{1}{\vert\Omega_{k'}\vert}\sum_{\m\in\Omega_{k'}}\Bigl\vert\frac{\vert\Omega_{k'}\vert\cdot \vert \Omega_{k'}(\m)\vert}{\vert\Omega\vert}-1\Bigr\vert
	\\&=\frac{1}{\vert\Omega_{k'}\vert}\sum_{\m\in\Omega_{k'}\backslash Z_{k'}}\Bigl\vert\frac{\vert\Omega_{k'}\vert\cdot \vert \Omega_{k'}(\m)\vert}{\vert\Omega\vert}-1\Bigr\vert+\frac{1}{\vert\Omega_{k'}\vert}\sum_{\m\in\Omega_{k'}\cap Z_{k'}}\Bigl\vert\frac{\vert\Omega_{k'}\vert\cdot \vert \Omega_{k'}(\m)\vert}{\vert\Omega\vert}-1\Bigr\vert.
	\end{split}
	\end{equation}
	
	Note that for $\m\in\Omega_{k'}\backslash Z_{k'}$, 
	\begin{equation}\nonumber
	\begin{split}
	&\quad\frac{\vert\Omega_{k'}\vert\cdot \vert \Omega_{k'}(\m)\vert}{\vert\Omega\vert}
	\\&=\frac{p^{dk'-(r_{1}+\dots+r_{k'})}(1+O_{k,r}(p^{-1/2}))\cdot p^{d(k-k')-(r_{k'+1}+\dots+r_{k})}(1+O_{k,r}(p^{-1/2}))}{p^{dk-(r_{1}+\dots+r_{k})}(1+O_{k,r}(p^{-1/2}))}=1+O_{k,r}(p^{-1/2}).
	\end{split}
	\end{equation}
	For all $\m\in\Omega_{k'}\cap Z_{k'}$,
	$$\frac{\vert\Omega_{k'}\vert\cdot \vert \Omega_{k'}(\m)\vert}{\vert\Omega\vert}\leq \frac{p^{d(k-k')}\vert\Omega_{k'}\vert}{\vert\Omega\vert}=p^{r_{k'+1}+\dots+r_{k}}(1+O_{k,r}(p^{-1/2})).$$ 
	Moreover,
	$$\frac{\vert Z_{k'}\vert}{\vert\Omega_{k'}\vert}=O_{k,r}(p^{(r_{1}+\dots+r_{k'}+r'_{k'})-\rank(M)-1}), \frac{\vert \Omega_{k'}\backslash Z_{k'}\vert}{\vert\Omega_{k'}\vert}=1+O_{k,r}(p^{-1/2}).$$
	So the right hand side of (\ref{1:6699}) is at most 
	$$O_{k,r}(p^{-1/2})(1+O_{k,r}(p^{-1/2}))+O_{k,r}(p^{(r+r'_{k'})-\rank(M)-1})=1+O_{k,r}(p^{-1/2})$$
	since $\rank(M)\geq 2r+1$. This proves Part (iii).
	
	\
 
 	We now consider the general case.
	Let  $\mathcal{J}=(F_{1},\dots,F_{r})$ be a standard $M$-representation of $\Omega$ with dimension vector $(r_{1},\dots,r_{k})$. Let $\sigma\colon\{1,\dots,k\}\to\{1,\dots,k\}$ be a permutation with $\sigma(\{1,\dots,k'\})=I$ for some $1\leq k'\leq k$, and let $L(n_{1},\dots,n_{k}):=(n_{\sigma(1)},\dots,n_{\sigma(k)})$ be the $d$-integral linear transformation which moves the variables $n_{i},i\in I$ to the front. By Proposition  \ref{1:yy3} (iii), $\{F_{1}\circ L^{-1},\dots,F_{r}\circ L^{-1}\}$ is an independent $(M,k)$-family and $L(\Omega)=\cap_{i=1}^{r}V(F_{1}\circ L^{-1})$. 
By Corollary \ref{1:rep} (iii), $L(\Omega)$ admits a standard $M$-representation with some  dimension vector $(r'_{1},\dots,r'_{k})$ such that $r'_{1}+\dots+r'_{k}=r$. We may now apply Theorem \ref{1:ct} for the case $I=\{1,\dots,k'\}$ and $J=\{k'+1,\dots,k\}$ to the set $L(\Omega)$ and complete the proof.
\end{proof}

\begin{rem}\label{1:ipmct}
	It is worth noting that Proposition \ref{1:ct} has many room for improvement. For example:
	\begin{itemize}
		\item the lower bound of $\rank(M)$ could be made into a function of the dimension vector $(r_{1},\dots,r_{k})$ instead of a function of $r$, which would give us a better lower bound.
		\item by a more careful computation, the error term $O_{k,r}(p^{-1/2})$ in (\ref{1:ffbb}) can be improved to $O_{k}(p^{-A_{d,k,r}})$ for some $A_{d,k,r}>0$ with $\lim_{d\to\infty} A_{d,k,r}=\infty$;
		\item the estimate for the sum $\sum_{(n_{1},\dots,n_{k'})\in Z_{k'}}\vert\Omega_{k'}(n_{1},\dots,n_{k'})\vert$ in (\ref{1:keyerr}) is very coarse. In fact, for most $(n_{1},\dots,n_{k'})\in Z_{k'}$, the cardinality of $\Omega_{k'}(n_{1},\dots,n_{k'})$ is much smaller than $p^{d(k-k')}$. If we  estimate this sum more carefully, it is promising to provide a much better lower bound for $\rank(M)$. 
	\end{itemize}	
 	However, we will not pursue these improvements in this paper.
\end{rem}	

\begin{rem}
An annoying fact is that
each time we apply Theorem \ref{1:ct} for a set $\Omega$, we need to first compute the $I$-projection $\Omega_{I}$, and then the set $\Omega_{I}((n_{i})_{i\in I})$. This computation is not difficult by using the algorithm in the proof of Proposition \ref{1:ip}, but is rather redundant. For conciseness, in this paper as well as in   \cite{SunC,SunD}, we only provide  essential details on computations of the sets $\Omega_{I}$ and $\Omega_{I}((n_{i})_{i\in I})$, instead of providing complete computations.
\end{rem}
 
\begin{rem}\label{1:cctt}
By Proposition  \ref{1:yy3} (i), if $\Omega\subseteq (\V)^{k}$ is a consistent $M$-set, then for all $I\subseteq \{1,\dots,k\}$, $\Omega_{I}$ is also a consistent $M$-set. So one can repeatedly use Theorem \ref{1:ct} to write $\E_{(n_{1},\dots,n_{k})\in\Omega}f(n_{1},\dots,n_{k})$ as a multiple average in any way we see fit. 
For example, by Lemma \ref{1:changeh}, if $M$ is a quadratic form associated with the matrix $A$, then we may write 
	\begin{equation}\nonumber
	\begin{split}
	&\quad\E_{(n,h_{1},\dots,h_{4})\in \Gow_{4}(V(M))}f(n,h_{1},\dots,h_{4})
	\\&=\E_{(h_{4},h_{2})\in (\V)^{2}\colon (h_{2}A)\cdot h_{4}=0}\E_{n\in V(M)^{h_{2},h_{4}}}\E_{(h_{1},h_{3})\in W(n,h_{2},h_{4})}f(n,h_{1},\dots,h_{4})+O(p^{-1/2}),
	\end{split}
	\end{equation}
where $W(n,h_{2},h_{4})$ is the set of $(h_{1},h_{3})\in (\V)^{2}$ such that 
$M(n+\e_{1}h_{1}+\dots+\e_{4}h_{4})=0$ for all $\e_{1},\dots,\e_{4}\in\{0,1\}$ with $(\e_{1},\e_{3})\neq (0,0)$. It is possible to further simplify the description of the set $W(n,h_{2},h_{4})$ and compute its the dimension vectors. However, in many applications we do not have to specify all the sets appearing in the averages explicitly.
\end{rem}

As an application of Theorem \ref{1:ct}, we have the following estimate for the cardinality of Gowers sets.

\begin{coro}[Cardinality of Gowers sets]\label{1:countingh}
	Let $d\in \N_{+}, r,s\in\N$ and $p$ be a prime. Let $M\colon\V\to\F_{p}$ be a quadratic form and $V+c$ be an affine subspace of $\V$ of co-dimension $r$. If either $\rank(M\vert_{V+c})$ or $\rank(M)-2r$ is at least $s^{2}+s+3$, then
	$$\vert \Gow_{s}(V(M)\cap (V+c))\vert=p^{(s+1)(d-r)-(\frac{s(s+1)}{2}+1)}(1+O_{s}(p^{-1/2})).$$
	\end{coro}	
\begin{proof}
	By Proposition \ref{1:iissoo}, it suffices to consider the case when  $\rank(M\vert_{V+c})\geq s^{2}+s+3$. 
	Let $\phi\colon \F_{p}^{d-r}\to V$ be any bijective linear transformation, and let $M'(m):=M(\phi(m)+c)$. By Lemma \ref{1:changeh},
	 $\vert \Gow_{s}(V(M)\cap (V+c))\vert=\vert \Gow_{s}(V(M'))\vert$. By Example \ref{1:rreepp}, $\Gow_{s}(V(M'))$ admits a standard $M$-representation with dimension vector 
	 $(1,1,2,\dots,s)$.
	 We may then apply Theorem \ref{1:ct} to conclude that $\vert \Gow_{s}(V(M'))\vert=p^{(s+1)(d-r)-(\frac{s(s+1)}{2}+1)}(1+O_{s}(p^{-1/2})).$
\end{proof}

Another application of Theorem \ref{1:ct} is to prove Proposition \ref{1:ctsp}. We leave the details to the interested readers.

We conclude this section with some properties on the dimension vectors of $M$-sets.

\begin{prop}\label{1:yy33}
Let $d,k,r\in\N_{+}$, $p$ be a prime, $M\colon\V\to\F_{p}$ be a quadratic form and $\Omega\subseteq (\V)^{k}$ be a consistent $M$-set. Suppose that $\Omega=\cap_{i=1}^{r}V(F_{i})$ for some consistent $(M,k)$-family $\{F_{1},\dots,F_{r}\}$. 
Let $1\leq k'\leq k$.
Suppose that $\rank(M)\geq 2r+1$ and $p\gg_{k,r} 1$.
\begin{enumerate}[(i)]
\item 
The dimension of all independent $M$-representations of $\Omega$ equals to $r_{M}(\Omega)$. 
\item We have that $r_{M}(\Omega)\leq r$, and that $r_{M}(\Omega)=r$ if the $(M,k)$-family $\{F_{1},\dots,F_{r}\}$ is independent.
\item For all $I\subseteq\{1,\dots,r\}$, the set $\Omega':=\cap_{i\in I}V(F_{i})$ is a consistent  $M$-set such that $r_{M}(\Omega')\leq r_{M}(\Omega)$. Moreover, if  the $(M,k)$-family $\{F_{1},\dots,F_{r}\}$ is independent, then $r_{M}(\Omega')=\vert I\vert$.
\item For any bijective $d$-integral linear transformation $L\colon(\V)^{k}\to(\V)^{k}$ and $v\in(\V)^{k}$, we have that $L(\Omega)+v$ is a consistent $M$-set and that $r_{M}(L(\Omega)+v)=r_{M}(\Omega)$.   
\item Assume that $\Omega$ admits a standard $M$-representation with dimension vector $(r_{1},\dots,$ $r_{k})$, Then for all $1\leq k'\leq k$, the set $$\{(n_{1},\dots,n_{k+k'})\in(\V)^{k+k'}\colon (n_{1},\dots,n_{k}), (n_{1},\dots,n_{k-k'},n_{k+1},\dots,n_{k+k'})\in\Omega\}$$ admits a standard $M$-representation with dimension vector $(r_{1},\dots,r_{k'},r_{k+1},\dots,r_{k'})$. 
\item Assume that $\Omega$ admits a standard $M$-representation with dimension vector $(r_{1},\dots,$ $r_{k})$.
For $I=\{1,\dots,k'\}$ for some $1\leq k'\leq k$, the $I$-projection $\Omega_{I}$ of $\Omega$ 
admits a standard $M$-representation with dimension vector $(r_{1},\dots,r_{k'})$. 
\item If $\Omega$ is a nice and consistent $M$-set or a  pure and consistent $M$-set, then $r_{M}(\Omega)\leq \binom{k+1}{2}$.
\end{enumerate}
\end{prop}
\begin{proof}
Part (i) is a consequence of Theorem \ref{1:ct} (i).
Part (ii) can be proved using the method similar to the proof of Corollary \ref{1:rep}. Parts (iii) and (iv) follow from Proposition  \ref{1:yy3} (i) and (iii). 
To prove Part (v), let $(F_{i,j}\colon 1\leq i\leq k, 1\leq j\leq r_{i})$ be a standard $M$-representation of $\Omega$. Then by the property of standard $M$-representations,
all of $F_{i,j}\colon k-k'+1\leq i\leq k, 1\leq j\leq r_{i}$ are independent of $n_{1},\dots,n_{k'}$, and 
 no nontrivial linear combination of $F_{i,j}\colon k-k'+1\leq i\leq k, 1\leq j\leq r_{i}$ is a constant. The conclusion now follows from Proposition  \ref{1:yy3} (iv). 
 
 Part (vi) is a consequence of Theorem \ref{1:ct} (i) and (ii).
 Part (vii) follows from Lemma \ref{1:rep0} and the fact that there are in total $\binom{k+1}{2}$ quadratic terms of the form $n_{i}n_{j}, 1\leq i,j\leq k$.
\end{proof}

\section{Irreducibility for Nice $M$-sets}\label{1:s:AppC}

We present the proof of Theorem \ref{1:irrr} in this appendix. We restate this theorem below for convenience.

\begin{thm}[Nice and consistent $M$-sets are irreducible]\label{1:irrrg}
	Let $d,K\in\N_{+},s\in\N$, $p$ be a prime, and $\d>0$. Let $M\colon\V\to \F_{p}$ be a  non-degenerate quadratic form 	and  $\Omega\subseteq (\V)^{K}$ be a  nice and consistent $M$-set. 	If 	$d\geq \max\{2r_{M}(\Omega)+1,4K-1\}$  and $p\gg_{d,s} \d^{-O_{d,s}(1)}$, then 
	\begin{enumerate}[(i)]
		\item 	$\Omega$ is $\d$-irreducible up to degree $s$;  		
		\item $\iota^{-1}(\Omega)$ is weakly $(\d,p)$-irreducible up to degree $s$;
		\item $\iota^{-1}(\Omega)$ is strongly $(\d,p)$-irreducible up to degree $s$.
	\end{enumerate}
\end{thm}

\subsection{Irreducible property for a sample $M$-set}
The proof of Theorem \ref{1:irrrg} is very technical. So 
we begin with the following special case of Theorem \ref{1:irrrg} to explain the main ideas:
 
\begin{prop}\label{1:irrr0}
	Let $d\in\N_{+},s\in\N$, $p$ be a prime, and $\d>0$. Let $M\colon\V\to \F_{p}$ be a  non-degenerate quadratic form given by $M(n):=n\cdot n$. If  $d\gg_{s} 1$ and $p\gg_{d}  \d^{-O_{d}(1)}$, then 
	the set	$\Gow_{3}(V(M))$ is $\d$-irreducible up to degree $s$.
\end{prop}

We will only provide a sketch of the proof of Proposition \ref{1:irrr0} to explain the main idea, as the rigorous proof  of Theorem \ref{1:irrrg} will be provided in later sections.
Note that $(n,h_{1},h_{2},h_{3})\in \Gow_{3}(V(M))$ if and only if $(n,h_{1},h_{2})\in \Gow_{2}(V(M))$ and
	 $M(n+h_{3})=h_{1}\cdot h_{3}=h_{2}\cdot h_{3}=0.$
	 Suppose we have shown that $\Gow_{2}(V(M))$ is $\d$-irreducible for all $\d>0$ if $p\gg_{d} \d^{-O_{d}(1)}$ and we wish to show that the same holds from $\Gow_{3}(V(M))$.
	Let $P\colon(\V)^{4}\to\F_{p}$ be a polynomial of degree $s$ such that $\vert V(P)\cap \Gow_{3}(V(M))\vert>\d\vert \Gow_{3}(V(M))\vert$.

	\textbf{Step 1: a parametrization for the last variable.} We first denote $x:=(n,h_{1},h_{2})\in \Gow_{2}(V(M))$ and express $h_{3}$ as a function of $x$.
	For convenience denote $h_{i}=(h_{i,1},h_{i,2},h'_{i})$ for some $h_{i,1},h_{i,2}\in\F_{p}$ and $h'_{i}\in\F_{p}^{d-2}$ for $i=1,2,3$.
	 If  $\Delta_{1,2}(h_{1},h_{2}):=\begin{vmatrix} 
	 h_{1,1} & h_{1,2} \\
	 h_{2,1} & h_{2,2} 
	 \end{vmatrix}\neq 0$, then let $\phi_{x}\colon \F_{p}^{d-2}\to\V$ denote the linear transformation given by
	 $$\phi_{x}(h'_{3})=\Bigl(\frac{\begin{vmatrix} -h'_{1}\cdot h'_{3} & h_{1,2} \\
	-h'_{2}\cdot h'_{3} & h_{2,2}
	 \end{vmatrix}}{\Delta_{1,2}(h_{1},h_{2})},\frac{\begin{vmatrix} h_{1,1} & -h'_{1}\cdot h'_{3} \\
	h_{2,1} & -h'_{2}\cdot h'_{3} 
	 \end{vmatrix}}{\Delta_{1,2}(h_{1},h_{2})},h'_{3}\Bigr).$$
	It is not hard to see that $h_{1}\cdot h_{3}=h_{2}\cdot h_{3}=0$ if and only if $h_{3}=\phi_{x}(w)$ for some $w\in \F_{p}^{d-2}$. Then for all $x\in \Gow_{2}(V(M))$ with $\Delta_{1,2}(h_{1},h_{2})\neq 0$, $(x,\phi_{x}(w))\in \Gow_{3}(V(M))\cap V(P)$ if and only if $P(x,\phi_{x}(w))=M(n+\phi_{x}(w))=0$. 
	
	Note that $M(n+\phi_{x}(\cdot))$ is a non-degenerate quadratic form for most of $x\in \Gow_{2}(V(M))$. We may write 
	$$M(n+\phi_{x}(w_{3},\dots,w_{d}))=\sum_{3\leq i\leq j\leq d} \Delta_{1,2}(h_{1},h_{2})^{-2}G_{1,2,i,j}(x)w_{i}w_{j}$$
	for some polynomials $G_{1,2,i,j}$ of degree $O(1)$.
	
	\textbf{Step 2: a factorization for the target polynomial.} Next we treat $P(x,\cdot)$ as a polynomial of $h_{3}$ to derive some properties for $P$ using the results from Section \ref{1:s4}.
	Since $\vert \Gow_{2}(V(M))\vert=O(p^{2d-4})$ and $\vert \Gow_{3}(V(M))\vert=O(p^{3d-7})$ by Corollary \ref{1:countingh}, the fact $\vert V(P)\cap \Gow_{3}(V(M))\vert>\d\vert \Gow_{3}(V(M))\vert$ implies that $\vert V(P(x,\phi_{x}(\cdot))\cap V(M(n+\phi_{x}(\cdot)))\vert>\d p^{d-3}$ for many $x\in \Gow_{1}(V(M))$. By Lemma \ref{1:bzt}, we have that $V(M(n+\phi_{x}(\cdot)))\subseteq V(P(x,\phi_{x}(\cdot)))$ for many of $x\in \Gow_{2}(V(M))$. 	By Proposition \ref{1:noloop} and a change of variable, if $G_{1,2,i,j}(x)\neq 0$, then $M(n+\phi_{x}(\cdot))$ divides $P(x,\phi_{x}(\cdot))$ with respect to the $B_{i,j}$-standard long division algorithm. 	So we may write
	\begin{equation}\label{1:erir1}
	P(x,\phi_{x}(w))=M(n+\phi_{x}(w))R_{x}(w)
	\end{equation}
	  for some polynomial $R_{x}$ of degree at most $s-2$.
	  Since every coefficient of $P(x,\phi_{x}(\cdot))$ can be written as a polynomial in $x$ of degree $O(s)$ divided by $\Delta_{1,2}(h_{1},h_{2})^{s}$ and every coefficient of $M(n+\phi_{x}(w))$ can be written as a polynomial in $x$ of degree $O(1)$ divided by $\Delta_{1,2}(h_{1},h_{2})^{2}$, we have that $R_{x}$ takes the form (assuming that $s\geq 2$)
	  $$R_{x}(w)=R(x,w)/(G_{1,2,i,j}(x)\Delta_{1,2}(h_{1},h_{2}))^{s}$$
	  for some polynomial $R$ of degree $O(s)$. By (\ref{1:erir1}), we have that 
	  \begin{equation}\label{1:erir3}
	(G_{1,2,i,j}(x)\Delta_{1,2}(h_{1},h_{2}))^{s}P(x,\phi_{x}(w))-M(n+\phi_{x}(w))R(x,w)=0
	\end{equation}
for many $x\in \Gow_{2}(V(M))$ for all $w\in\F_{p}^{d-2}$.
	  Note that the left hand side of (\ref{1:erir3}) can be written as $\sum_{i\in\N^{d-2}}C_{i}(x)w^{i}$ for some polynomial $C_{i}$. So $C_{i}(x)=0$ for many $x\in \Gow_{2}(V(M))$ for all $i\in\N^{d-2}$. By induction hypothesis, we have that $C_{i}\equiv 0$ and thus (\ref{1:erir3}) holds for all $x\in \Gow_{2}(V(M))$ and $w\in\F_{p}^{d-2}$.

	Note that  if $\Delta_{1,2}(h_{1},h_{2})\neq 0$, then
	 $$\frac{\begin{vmatrix} -h'_{1}\cdot h'_{3} & h_{1,2} \\
	-h'_{2}\cdot h'_{3} & h_{2,2}
	 \end{vmatrix}}{\Delta_{1,2}(h_{1},h_{2})}-h_{3,1}=\frac{\begin{vmatrix} -h_{1}\cdot h_{3} & h_{1,2} \\
	-h_{2}\cdot h_{3} & h_{2,2}
	 \end{vmatrix}}{\Delta_{1,2}(h_{1},h_{2})} \text{ and } \frac{\begin{vmatrix} h_{1,1} & -h'_{1}\cdot h'_{3} \\
	h_{2,1} & -h'_{2}\cdot h'_{3} 
	 \end{vmatrix}}{\Delta_{1,2}(h_{1},h_{2})}-h_{3,2}=\frac{\begin{vmatrix} h_{1,1} & -h_{1}\cdot h_{3} \\
	h_{2,1} & -h_{2}\cdot h_{3} 
	 \end{vmatrix}}{\Delta_{1,2}(h_{1},h_{2})}.$$
	 So   
	  $P(x,\phi_{x}(h'_{3}))-P(x,h_{3})$ can be written as $\Delta_{1,2}(h_{1},h_{2})^{-s}((h_{1}\cdot h_{3})P'_{1}(x,h_{3})+(h_{2}\cdot h_{3})P'_{2}(x,h_{3}))$ for some polynomials $P'_{1}$ and $P'_{2}$ of degree $O(s)$, and 
  $M(n+\phi_{x}(h'_{3}))-M(h_{3})$ can be written as $\Delta_{1,2}(h_{1},h_{2})^{-2}((h_{1}\cdot h_{3})M'_{1}(x,h_{3})+(h_{2}\cdot h_{3})M'_{2}(x,h_{3}))$ for some polynomials $P'_{1}$ and $P'_{2}$ of degree $O(1)$. By (\ref{1:erir3}), we have that 
  \begin{equation}\nonumber
	(G_{1,2,i,j}(x)\Delta_{1,2}(h_{1},h_{2}))^{s}P(x,h_{3})=M(n+h_{3})R_{0}(x,h_{3})+(h_{1}\cdot h_{3})R_{1}(x,h_{3})+(h_{2}\cdot h_{3})R_{2}(x,h_{3})
	\end{equation}
	 for some polynomials $R_{1},R_{2},R_{3}$ of degree $O(s)$ for all $x\in \Gow_{2}(V(M))$ with $\Delta_{1,2}(h_{1},h_{2})\neq 0$ and for all $h_{3}\in \V$. 
	 So 
	   \begin{equation}\label{1:erir21}
	(G_{1,2,i,j}(n,h_{1},h_{2},h_{3})\Delta_{1,2}(h_{1},h_{2}))^{s}P(n,h_{1},h_{2},h_{3})=0
	\end{equation}
	 for all $(n,h_{1},h_{2},h_{3})\in \Gow_{3}(V(M))$ with $\Delta_{1,2}(h_{1},h_{2})\neq 0$. Moreover, it is not hard to see that (\ref{1:erir21}) also holds if $\Delta_{1,2}(h_{1},h_{2})=0$.

	 \textbf{Step 3: removing the $(G_{1,2,i,j}(x)\Delta_{1,2}(h_{1},h_{2}))^{s}$ term using a parametrization trick.} It follows from (\ref{1:erir21}) that for all $(n,h_{1},h_{2},h_{3})\in \Gow_{3}(V(M))$, either $P(n,h_{1},h_{2},h_{3})=0$ or $G_{1,2,i,j}(x)\Delta_{1,2}(h_{1},h_{2})=0$ for all $3\leq i\leq j\leq d$.  
	  	 Since $M(n+\phi_{x}(\cdot))$ is non-degenerate, at least one of $G_{1,2,i,j}(x)$ is nonzero. So we have that for all $(n,h_{1},h_{2},h_{3})\in \Gow_{3}(V(M))$, either $P(n,h_{1},h_{2},h_{3})=0$ or $\Delta_{1,2}(h_{1},h_{2})=0$ for all $3\leq i\leq j\leq d$.
	We are close to the conclusion that $\Gow_{3}(V(M))\subseteq V(P)$. However, we need some more effort dealing with the $\Delta_{1,2}(h_{1},h_{2})$ term. 
	
	Denote $\Delta_{i,j}(h_{1},h_{2}):=\begin{vmatrix} 
	 h_{1,i} & h_{1,j} \\
	 h_{2,i} & h_{2,j} 
	 \end{vmatrix}$ for all $1\leq i<j\leq d$. Using a similar argument as above, we have that for all $(n,h_{1},h_{2},h_{3})\in \Gow_{3}(V(M))$, either $P(n,h_{1},h_{2},h_{3})=0$ or $\Delta_{i,j}(h_{1},h_{2})=0$ for all $1\leq i<j\leq d$. This is equivalent of saying that $P(n,h_{1},h_{2},h_{3})=0$ unless $h_{1}$ is parallel to $h_{2}$. To deal with the case when $h_{1}$ is parallel to $h_{2}$, we use a parametrization trick. 
	 Fix any such $(n,h_{1},h_{2},h_{3})$ and	define $f(t):=P(n,h_{1}+tv,h_{2},h_{3})$ for all $t\in\F_{p}$, where   $v\in\V$ is any vector not parallel to $h_{2}$. 
 Then $f\colon \F_{p}\to\F_{p}$ is a polynomial of degree at most $s$. Note that $h_{1}+tv,h_{2}$ is not parallel to $h_{2}$ unless $t=0$. So $f(t)=0$ for all $t\in\F_{p}\backslash\{0\}$. By Lemma \ref{1:ns}, this implies that $f\equiv 0$ and thus $f(0)=P(n,h_{1},h_{2},h_{3})=0$. This completes the proof of Proposition \ref{1:irrr0}.

\subsection{Some preliminary reductions}\label{1:s:94}

	 We start with some reductions before proving Theorem \ref{1:irrrg}.
 	In the rest of this section,
	if $f(n_{1},\dots,n_{K})$ is a function depending only on $n_{1},\dots,n_{i}$ for some $i\leq K$, then we write $f(n_{1},\dots,n_{i})=f(n_{1},\dots,n_{K})$ for convenience.
	Throughout we assume that $p\gg_{d,s} \d^{-O_{d,s}(1)}$.

	 By Propositions  \ref{1:yy3} (iii), \ref{1:iri4} and \ref{1:iso}, we may assume without loss of generality that $\Omega=V(\mathcal{J})$ for some nice and consistent $(M,K)$-family $\mathcal{J}$.  

	 Let $A$ be the matrix associated with $M$. 
		Since $A$ is  non-degenerate and symmetric,  there exists a $d\times d$ invertible matrix $B$ such that $BAB^{T}$ is symmetric, diagonal and non-degenerate. Note that if $F\colon(\V)^{K}\to\F_{p}$ is a nice $(M,K)$-integral quadratic function, then $F'(n_{1},\dots,n_{K}):=F(n_{1}B,\dots,n_{K}B)$ is a nice $(M',K)$-integral quadratic function, where $M'$ is any quadratic form associated with the matrix $BAB^{T}$. Moreover, the set of $(n_{1},\dots,n_{K})\in(\V)^{K}$ such that $(n_{1}B,\dots,n_{K}B)\in\Omega$ is a nice and consistent $M$-set of total co-dimension $r_{M}(\Omega)$. So by Proposition \ref{1:iso}, it suffices to prove  Theorem \ref{1:irrrg} under the assumption that $A$ is diagonal.

	By Lemma \ref{1:rep0}, $\Omega$ admits a nice and standard $M$-representation $(F_{i,j},1\leq i\leq K, 1\leq j\leq r_{i})$ for some $0\leq r_{i}\leq i$ with $r_{1}+\dots+r_{i}=r_{M}(\Omega)$.	
	Assume that 
	\begin{equation}\label{1:e:fi1}
	F_{i,1}(n_{1},\dots,n_{K})=c_{i}(n_{i}A)\cdot n_{i}+(n_{i}A)\cdot L_{i,1}(n_{1},\dots,n_{i-1})+u_{i,0}
	\end{equation}
	and that  
	\begin{equation}\label{1:e:fij}
	F_{i,j}(n_{1},\dots,n_{K})=(n_{i}A)\cdot L_{i,j}(n_{1},\dots,n_{i-1})+u_{i,j}
	\end{equation}
		for all $1\leq i\leq K$ and $2\leq j\leq r_{i}$
		for some $c_{i}\in\{0,1\}$, some $u_{i,j}\in\F_{p}$ and some $d$-integral linear transformations $L_{i,j}\colon(\V)^{i-1}\to \V$
		given by 
		\begin{equation}\nonumber
	L_{i,j}(n_{1},\dots,n_{i-1})=u_{i,j,i-1}n_{i-1}+\dots+u_{i,j,1}n_{1}
	\end{equation}
		for some $u_{i,j,\ell}\in\F_{p}$ such that the $r_{i}\times i$ matrix  
		$\begin{bmatrix}
		c_{i} & u_{i,1,i-1} & \dots & u_{i,1,1} \\
		0 & u_{i,2,i-1} & \dots & u_{i,2,1} \\
		\dots & \dots & \dots & \dots \\
		 0 & u_{i,r_{i},i-1} & \dots & u_{i,r_{i},1}
\end{bmatrix}$	
is in the reduced row echelon form and is of rank $r_{i}$.  In particular, $L_{i,2},\dots,L_{i,r_{i}}$ are linearly independent if $c_{i}=1$, and $L_{i,1},\dots,L_{i,r_{i}}$ are linearly independent if $c_{i}=0$.
For convenience we say that $(c_{1},\dots,c_{K})\in\{0,1\}^{K}$ is the \emph{type} of $\Omega$. 
	
	We first prove Part (i) of Theorem \ref{1:irrrg}.  For $1\leq k\leq K$, denote $\Omega_{k}:=\cap_{i=1}^{k}\cap_{j=1}^{r_{i}}V(F_{i,j})$.
	We prove by induction that Part (i) of Theorem \ref{1:irrrg} holds for $\Omega_{k}$ for all $1\leq k\leq K$.

Suppose we have proved that Part (i) of Theorem \ref{1:irrrg} holds for $\Omega_{k-1}$ for some $k\geq 1$. We prove that Part (i) of Theorem \ref{1:irrrg} holds for $\Omega_{k}$ (when $k=0$, we prove that Part (i) of Theorem \ref{1:irrrg} holds for $\Omega_{1}$ without any induction hypothesis).
		If $r_{k}=0$, then $\Omega_{k}=\Omega_{k-1}\times\V$ and so Part (i) of Theorem \ref{1:irrrg} follows from Proposition \ref{1:iri3} and the induction hypothesis. So we assume that $r_{k}\geq 1$.

	For convenience denote $r:=r_{M}(\Omega_{k})=r_{1}+\dots+r_{k}\leq r_{M}(\Omega)$.
	Write $\n'=(n_{1},\dots,n_{k-1})$ and $n_{i}=(n_{i,1},\dots,n_{i,d})$ for some $n_{i,j}\in\F_{p}$ for all $1\leq i\leq k$.  For $1\leq i\leq k$ and $I\subseteq \{1,\dots,d\}$, write $n_{i,I}:=(n_{i,j})_{j\in I}$.
		For $\n'\in(\V)^{k-1}$, let $\Omega_{k}(\n')$ denote the set of $n_{k}\in\V$ such that $F_{k,j}(\n',n_{k})=0$ for all $0\leq j\leq r_{k}$. It is clear that $(\n',n_{k})\in\Omega_{k}$ if and only if $\n'\in\Omega_{k-1}$ and $n_{k}\in \Omega_{k}(\n')$.	
		
		\subsection{Proof of Part (i) of Theorem \ref{1:irrrg} for the non-degenerate case}\label{1:s84}
		  We follow the steps similar to Proposition \ref{1:irrr0}.	
		Depending only whether $c_{k}=0$ or 1, the proofs are slightly different. We start with the case when $c_{k}=1$. In this case $r_{k}\leq k$.

	\textbf{Step 1: a parametrization for the last variable.}
	Write
	$$L_{k,j}(\n')A=(S_{k,j,1}(\n'),\dots,S_{k,j,d}(\n'))$$
	for some polynomials $S_{k,j,1},\dots,S_{k,j,d}\colon (\V)^{k-1}\to\F_{p}$ of degree at most 1. Since $A$ is diagonal, for all $1\leq \ell\leq d$, $S_{k,j,\ell}(\n')$ is dependent only on $n_{i,\ell}, 1\leq i\leq k-1$.
	 Then we may rewrite (\ref{1:e:fij}) as
		\begin{equation}\label{1:subsub0}
		F_{k,j}(\n',n_{k})=\sum_{\ell=1}^{d}n_{k,\ell}S_{k,j,\ell}(\n')+u_{k,j},
		\end{equation} 
 for all $2\leq j\leq r_{i}$.
 
	For $I\subseteq \{1,\dots,d\}$ with $\vert I\vert=r_{k}-1>0$, let $\Delta_{k,I}(\n')$   denote the determinant of the $(r_{k}-1)\times (r_{k}-1)$ matrix
	$(S_{k,j,\ell}(\n'))_{2\leq j\leq r_{k}, \ell\in I}$. If $r_{k}=1$ and $I=\emptyset$, then let $\Delta_{k,I}(\n')\equiv 1$. 
	Then 
	$\Delta_{k,I}$ is a polynomial from $(\V)^{k}$ to $\F_{p}$ of degree at most $r_{k}-1$  depending only on $n_{i,\ell}, 1\leq i\leq k-1, \ell\in I$ (since $A$ is diagonal). 
	For convenience denote $I^{c}:=\{1,\dots,d\}\backslash I$.
	Fix any $\n'\in \Omega_{k-1}$ with $\Delta_{k,I}(\n')\neq 0$. Then $L_{k,j}(\n')A, 2\leq j\leq r_{k}$ are linearly independent. By (\ref{1:subsub0}) and the knowledge of linear algebra,  $F_{k,2}(\n',n_{k})=\dots=F_{k,r_{k}}(\n',n_{k})=0$ if and only if
	for all $\ell'\in I$,
	\begin{equation}\label{1:subsub1}
	n_{k,\ell'}=\frac{R_{k,I,\ell'}(\n',n_{k,I^{c}})}{\Delta_{k,I}(\n')},
	\end{equation}	
	where $R_{k,I,\ell'}(\n',n_{k,I^{c}})$ is the determinant of the  $(r_{k}-1)\times (r_{k}-1)$ matrix  obtained by replacing the column $(S_{k,j,\ell'}(\n'))_{2\leq j\leq r_{k}}$ from $(S_{k,j,\ell''}(\n'))_{2\leq j\leq r_{k},\ell''\in I}$ by $(R_{k,I,j,\ell'}(\n',n_{k,I^{c}}))_{2\leq j\leq r_{k}}$ where
	\begin{equation}\label{1:subvomit}
	R_{k,I,j,\ell'}(\n',n_{k,I^{c}})=-u_{k,j}-\sum_{\ell\in I^{c}}n_{k,\ell}S_{k,j,\ell}(\n')=-F_{k,j}(\n',n_{k})+\sum_{\ell\in I}n_{k,\ell}S_{k,j,\ell}(\n').
	\end{equation}

Let $\Delta_{k,I,\ell,\ell'}(\n')$ be the determinant of the  $(r_{k}-1)\times (r_{k}-1)$ matrix  obtained by replacing the column $(S_{k,j,\ell'}(\n'))_{2\leq j\leq r_{k}}$ from $(S_{k,j,\ell''}(\n'))_{2\leq j\leq r_{k},\ell''\in I}$ by the column $(S_{k,j,\ell}(\n'))_{2\leq j\leq r_{k}}$,
	and let $Q_{k,I,\ell'}(\n')$ be the determinant of the  $(r_{k}-1)\times (r_{k}-1)$ matrix  obtained by replacing the column $(S_{k,j,\ell'}(\n'))_{2\leq j\leq r_{k}}$ from $(S_{k,j,\ell''}(\n'))_{2\leq j\leq r_{k},\ell''\in I}$ by the column $(-u_{k,j})_{2\leq j\leq r_{k}}$.
		 It follows from  (\ref{1:subvomit}) that
	\begin{equation}\label{1:meaningofpx}
	\begin{split}
	R_{k,I,\ell'}(\n',n_{k,I^{c}})=Q_{k,I,\ell'}(\n')-\sum_{\ell\in I^{c}}n_{k,\ell}\Delta_{k,I,\ell,\ell'}(\n').
	\end{split}
	\end{equation}
	
	For convenience we regard $\F_{p}^{d-r_{k}+1}$ as the set of $n_{k,I^{c}}=(n_{k,\ell})_{\ell\in I^{c}}, n_{k,\ell}\in\F_{p}$, i.e. the coordinates in $\F_{p}^{d-r_{k}+1}$ are not labeled by $\{1,\dots,d-r_{k}+1\}$, but by $I^{c}$.
	Let $\phi_{k,I,\n'}\colon \F_{p}^{d-r_{k}+1}\to\V$, be the map from $n_{k,I^{c}}$ to $(m_{\ell})_{1\leq \ell\leq d}$ given by $m_{\ell}=n_{k,\ell}$ for $\ell\in I^{c}$, and $m_{\ell}$ be the quantity $n_{k,\ell}$ defined in (\ref{1:subsub1}) for all $\ell\in I$. 
	Then $\phi_{k,I,\n'}$ is a polynomial of degree at most 1 of  with respect to $n_{k,\ell}$ for all $\ell\in I^{c}$. By (\ref{1:e:fi1}), we may denote
	\begin{equation}\label{1:fbbm}
	F_{I,k,1}(\n',m):=F_{k,1}(\n',\phi_{k,I,\n'}(m))=(\phi_{k,I,\n'}(m)A)\cdot \phi_{k,I,\n'}(m)+L_{k,1}(\n')A\cdot \phi_{k,I,\n'}(m)+u_{k,1}
	\end{equation}
	 for all $\n'\in \Omega_{k-1}$ with $\Delta_{k,I}(\n')\neq 0$ and $m\in \F_{p}^{d-r_{k}+1}$ (recall that $c_{k}=1$).
	Then for all $\n'\in \Omega_{k-1}$
	 with $\Delta_{k,I}(\n')\neq 0$ and $n_{k}\in \F_{p}^{d}$, 
	 \begin{equation}\label{1:meaningofp}
	 \text{$n_{k}\in\Omega_{k}(\n')$ $\Leftrightarrow$ $n_{k}=\phi_{k,I,\n'}(m)$ for some $m\in \F_{p}^{d-r_{k}+1}$ with $F_{I,k,1}(\n',m)=0$.}
	 \end{equation}

	Assume that the diagonal of the matrix $A$ is $(a_{1},\dots,a_{d})$. 
	It follows from (\ref{1:subsub1}), (\ref{1:meaningofpx}), and (\ref{1:fbbm}) that $F_{I,k,1}(\n',n_{k,I^{c}})$ equals to
	\begin{equation}\nonumber
	\begin{split}
	\sum_{\ell \in I^{c}}a_{\ell}n_{k,\ell}^{2}+\sum_{\ell'\in I}a_{\ell'}(\Delta_{k,I}(\n'))^{-2} \Bigl(Q_{k,I,\ell'}(\n')-\sum_{\ell\in I^{c}}n_{k,\ell}\Delta_{k,I,\ell,\ell'}(\n')\Bigr)^{2}
	\end{split}
	\end{equation}
	plus a polynomial which is of degree at most 1 in the variable $n_{k,I^{c}}$.
Since $S_{k,j,\ell}$ are of degree at most 1, we have that
$Q_{k,I,\ell'}$ and $\Delta_{k,I,\ell,\ell'}(\n')$ are polynomials of degrees at most $r_{k}-1$. Since $S_{k,j,\ell}(\n')$ is dependent only on $n_{i,\ell}$ for $1\leq i\leq k-1$, we have that $\Delta_{k,I,\ell,\ell'}(\n')$ is dependent only on $n_{i,j}, 1\leq i\leq k-1, j\in (I\cup\{\ell\})\backslash\{\ell'\}$. Therefore, for all $\ell,\ell'\in I^{c}$, the coefficient of the $n_{k,\ell}n_{k,\ell'}$ term of $F_{I,k,0}(\n',\cdot)$ can be written in the form
	\begin{equation}\label{1:pnmb}
	\begin{split}
	 \text{$2G_{k,I,\ell,\ell'}(\n')(\Delta_{I}(\n'))^{-2}$ if $\ell\neq \ell'$ and $G_{k,I,\ell,\ell}(\n')(\Delta_{I}(\n'))^{-2}$ if $\ell=\ell'$}
	\end{split}
	\end{equation}
	for some polynomial $G_{k,I,\ell,\ell'}$ of degree at most $2r_{k}$ which depends only on $n_{i,j}, 1\leq i\leq k-1, j\in I\cup\{\ell,\ell'\}$. If $I=\emptyset$, then we denote $G_{k,I,\ell,\ell'}\equiv 1$.

	Note that $\{\phi_{k,I,\n'}(m)\colon m\in\F_{p}^{d-r_{k}+1}\}$ is an affine subspace of $\V$ of co-dimension $r_{k}-1$, by Proposition \ref{1:iissoo}, $F_{I,k,1}(\n',\cdot)$ is a  quadratic form of rank  at least $d-2(r_{k}-1)\geq 3$.
	 Let $A_{k,I,\n'}$ be the $(d-r_{k}+1)\times (d-r_{k}+1)$ matrix associated to $F_{I,k,1}(\n',\cdot)$.	
	Let $i_{\ast}$ be the smallest element in $I^{c}$. For $i_{1}\in I^{c}$,
	let $\phi_{I,i_{1},i_{1}}\colon\F_{p}^{d-r_{k}+1}\to\F_{p}^{d-r_{k}+1}$ denote the bijective linear transformations given by switching the $i_{1}$-th and $i_{\ast}$-th components of $n_{k,I^{c}}$. For $i_{1},i_{2}\in I^{c}, i_{1}<i_{2}$, let   $\phi_{I,i_{1},i_{2}}\colon\F_{p}^{d-r_{k}+1}\to\F_{p}^{d-r_{k}+1}$ denote the bijective linear transformation
	which maps $n_{k,I^{c}}=(n_{k,\ell})_{\ell\in I^{c}}$ to a vector whose
  first component (i.e. the $i_{\ast}$-th component) equals to $\frac{1}{2}(n_{k,i_{1}}+n_{k,i_{2}})$,    second component equals to  $\frac{1}{2}(n_{k,i_{1}}-n_{k,i_{2}})$,
	and the remaining comments equal to $(n_{k,\ell})_{\ell\in  (I\cup \{i_{1},i_{2}\})^{c}}$ ordered in an arbitrary way.
	Let $B_{I,i_{1},i_{2}}$ be the $(d-r_{k}+1)\times (d-r_{k}+1)$ matrix induced by the bijective linear transformation $\phi_{I,i_{1},i_{2}}$.
	This construction is similar the $B_{i,j}$-standard matrix defined in Lemma \ref{1:2d2d}. We have that for all $i_{1},i_{2}\in I^{c}$,
	 $B_{I,i_{1},i_{2}}$ is an $(d-r_{k}+1)\times (d-r_{k}+1)$ invertible matrix 
	such that the $(i_{\ast},i_{\ast})$-th entry of  $B_{I,i_{1},i_{2}}A_{k,I,\n'}B^{T}_{k,I,i_{1},i_{2}}$ equals to the $(i_{1},i_{2})$-th entry of $A_{k,I,\n'}$. 
	   
	It is not hard to see from (\ref{1:pnmb}) that
	for all $i_{1},i_{2}\in I^{c}, i_{1}\leq i_{2}$,
	 the coefficient of the $(i_{\ast},i_{\ast})$-entry of  $B_{I,i_{1},i_{2}}A_{k,I,\n'}B^{T}_{I,i_{1},i_{2}}$ can be written as $G_{k,I,i_{1},i_{2}}(\n')/(\Delta_{I}(\n'))^{2}$.
	Since $f_{I,k,1}(\n',\cdot)$ is not constant zero (recall that $F_{I,k,1}(\n',\cdot)$ is of rank at least 3), we have:
	
	\begin{lem}\label{1:noog}
	   For every $1\leq k\leq K$, every $I\subseteq\{1,\dots,d\}$ with $\vert I\vert=r_{k}-1$ and every $\n'\in(\V)^{k-1}$,	   
	    there exist $i_{1},i_{2}\in I^{c}, i_{1}\leq i_{2}$ such that $G_{k,I,i_{1},i_{2}}(\n')\neq 0$.
	\end{lem}

	Let $Z_{k,I,i_{1},i_{2}}$ denote the set of $(\n',n_{k})\in(\V)^{k}$ such that $\Delta_{k,I}(\n'),G_{k,I,i_{1},i_{2}}(\n')\neq 0$. 
	We need the following estimate. 
	
	\begin{lem}\label{1:cl2}
	If $G_{k,I,i_{1},i_{2}}$ is not constant zero, then we have that $\vert {Z}^{c}_{k,I,i_{1},i_{2}}\cap\Omega_{k}\vert\leq \d p^{dk-r}/2$. 
	\end{lem}
	\begin{proof}
		Let $Z'$ denote the set of $\n'\in\Omega_{k-1}$ such that either $G_{k,I,i_{1},i_{2}}(\n')=0$ or $\Delta_{k,I}(\n')=0$.
		Let $Z$ be the set of $\n'\in\Omega_{k-1}$ such that $\vert\Omega_{k}(\n')\vert$ is not $p^{d-r_{k}}(1+O_{k}(p^{-1/2}))$. By Theorem \ref{1:ct} (ii), (setting $k'=k-1, r'_{k'}=r_{k}$ and $\rank(M)=d$), 
		we have that $\vert Z\vert=O_{k}(p^{d(k-2)+r_{k}-1})$. Again by Theorem \ref{1:ct}, we have that
	  \begin{equation}\nonumber
	  	\begin{split}
	  	&\quad\vert {Z}^{c}_{k,I,i_{1},i_{2}}\cap\Omega_{k}\vert
	  	=\sum_{\n'\in Z'\backslash Z}\vert\Omega_{k}(\n')\vert+\sum_{\n'\in Z'\cap Z}\vert\Omega_{k}(\n')\vert	
	  	\\&\leq \vert Z'\vert\cdot p^{d-r_{k}}(1+O_{k}(p^{-1/2}))+O_{k}(p^{d(k-1)+r_{k}-1}).	
	  	\end{split}	
	  \end{equation}	
	Since $d\geq 2r+1$, we have $d(k-1)+r_{k}-1<dk-r$.  		
		So it suffices to show that $\vert Z'\vert\leq \d p^{d(k-1)-(r-r_{k})}/4$.
		If not, then one of $\vert V(G_{k,I,i_{1},i_{2}})\cap \Omega_{k-1}\vert$ and $\vert V(\Delta_{k,I})\cap \Omega_{k-1}\vert$ is larger than $\d p^{d(k-1)-(r-r_{k})}/8$. Applying Theorem \ref{1:ct} to compute the cardinality of $\vert\Omega_{k-1}\vert$ and using the induction hypothesis that $\Omega_{k-1}$ is $\d'$-irreducible for all $\d'>0$ if $p\gg_{d,s} {\d'}^{-O_{d,s}(1)}$, we have that either $\Omega_{k-1}\subseteq  V(G_{k,I,i_{1},i_{2}})$ or  $\Omega_{k-1}\subseteq  V(\Delta_{k,I})$.
		
		Note that $\Delta_{k,I}(\n')$ depends only on $n_{i,\ell}, 1\leq i\leq k-1, \ell\in I$,    and certainly $\Delta_{k,I}$ is not constant 0. So we may pick some $n_{i,\ell}\in\F_{p}, 1\leq i\leq k-1, \ell\in I$ such that  $\Delta_{k,I}(\n')\neq 0$. 
 Let $\Omega'$ be the set of $(n_{i,\ell})_{1\leq i\leq k-1, \ell\in I^{c}}$ such that $(n_{i,\ell})_{1\leq i\leq k-1, 1\leq \ell\leq d}$ belongs to $\Omega_{k-1}$. 
		Let $M'\colon\F_{p}^{d-r_{k}+1}\to\F_{p}$ be the quadratic form given by
		$$M'((m_{\ell})_{\ell\in I^{c}}):=M((m'_{\ell})_{1\leq \ell\leq d}),$$
		where $m'_{\ell}=m_{\ell}$ for $\ell\in I^{c}$ and $m'_{\ell}=0$ for $\ell\in I$. Since $A$ is invertible and diagonal, we have that $M'$ is  non-degenerate.
		
		For $1\leq i\leq k-1, 1\leq j\leq r_{i}$, let $\tilde{F}_{i,j}\colon (\F_{p}^{d-r_{k}+1})^{k-1}\to\F_{p}$ be the function obtained from  $F_{i,j}$ (viewed as a function from $(\F_{p}^{d})^{k-1}$ to $\F_{p}$ where we forget the irrelevant variable $n_{k}$) by fixing the variables  $n_{i,\ell}\in\F_{p}, 1\leq i\leq k-1, \ell\in I$ chosen above. 
		Let $\Omega'$ be the set of $(n_{i,\ell})_{1\leq i\leq k-1, \ell\in I^{c}}$ such that $(n_{i,\ell})_{1\leq i\leq k-1, 1\leq \ell\leq d}$ belongs to $\Omega_{k-1}$. 
Then		$\Omega'=\cap_{i=1}^{k-1}\cap_{j=1}^{r_{i}}V(\tilde{F}_{i,j})$.
		By the construction of $F_{i,j}$, $\{F_{i,j}\colon 1\leq i\leq k-1, 1\leq j\leq r_{i}\}$ is an independent $(M',k)$-family. So $\{\tilde{F}_{i,j}\colon 1\leq i\leq k-1, 1\leq j\leq r_{i}\}$ is an independent $(M',k-1)$-family. 		Therefore $\Omega'\subseteq (\F_{p}^{d-r_{k}+1})^{k-1}$ is a consistent $M'$-set with total co-dimension $r_{1}+\dots+r_{k-1}$, which is the same as $\Omega_{k-1}$.
	Since $d-r_{k}+1\geq 2(r_{1}+\dots+r_{k-1})+1$, by Theorem \ref{1:ct},
		$\vert\Omega'\vert\geq p^{(d-r_{k}+1)(k-1)-(r_{1}+\dots+r_{k-1})}(1+O_{k}(p^{-1/2}))$. In particular, $\Omega'$ is non-empty. Therefore, $\Omega_{k-1}$ is not a subset of  $V(\Delta_{k,I})$.

	\

		We may argue similarly for $G_{k,I,i_{1},i_{2}}$. Let $J:=I\cup\{i_{1},i_{2}\}$.  
		Recall that $G_{k,I,i_{1},i_{2}}(\n')$ depends only on $n_{i,\ell}, 1\leq i\leq k-1, \ell\in J$.
		Since $G_{k,I,i_{1},i_{2}}$ is not constant zero, then we may pick some $n_{i,\ell}\in\F_{p}, 1\leq i\leq k-1, \ell\in J$ such that  $G_{k,I,i_{1},i_{2}}(\n')\neq 0$.
		 Let $\Omega'$ be the set of $(n_{i,\ell})_{1\leq i\leq k-1, \ell\in J^{c}}$ such that $(n_{i,\ell})_{1\leq i\leq k-1, 1\leq \ell\leq d}$ belongs to $\Omega_{k-1}$. 
		Let $M'\colon\F_{p}^{d-\vert J\vert}\to\F_{p}$ be the quadratic form given by
		$$M'((m_{\ell})_{\ell\in J^{c}}):=M((m'_{\ell})_{1\leq \ell\leq d}),$$
		where $m'_{\ell}=m_{\ell}$ for $\ell\in J^{c}$ and $m'_{\ell}=0$ for $\ell\in J$.	
	
	Similar to the argument in the previous case, we have that 		
	$M'$ is non-degenerate and that
	$\Omega'\subseteq (\F_{p}^{d-\vert J\vert})^{k-1}$ is a consistent $M'$-set with total co-dimension $r_{1}+\dots+r_{k-1}$. 	Since $d-\vert J\vert\geq d-(r_{k}+1)\geq 2(r_{1}+\dots+r_{k-1})+1$, by Theorem \ref{1:ct}, 
		$\vert\Omega'\vert\geq p^{(d-\vert J\vert)(k-1)-(r_{1}+\dots+r_{k-1})}(1+O_{k}(p^{-1/2}))$. In particular, $\Omega'$ is non-empty. Therefore, $\Omega_{k-1}$ is not a subset of  $V(G_{k,I,i_{1},i_{2}})$.	 
		 We now arrive at a contradiction and we are done. 	
	\end{proof}

	\textbf{Step 2: a factorization for the target polynomial.}
		This step is carried out  by the following proposition:

	\begin{prop}\label{1:cl1} 
		Let $s\in\N$, $1\leq k\leq K$, $I\subseteq \{1,\dots,d\}$ with $\vert I\vert=r_{k}-1$, $i_{1},i_{2}\in\{1,\dots,d\}\backslash I, i_{1}\leq i_{2}$, and $P\in\poly((\V)^{k}\to\F_{p})$ be a polynomial of degree at most $s$ with $\vert V(P)\cap \Omega_{k}\vert\geq \d p^{dk-r}$.	Then
	\begin{equation}\nonumber
	(G_{k,I,i_{1},i_{2}}(\n')\Delta_{k,I}(\n'))^{s+2}P(\n',n_{k})
	=\sum_{j=1}^{r_{k}}F_{k,j}(\n',n_{k})Q_{j}(\n',n_{k})
	\end{equation}
	for some polynomial $Q_{j}$ of degree at most $O_{k,s}(1)$ for all $\n'\in\Omega_{k-1}$ and $n_{k}\in\V$. 
	\end{prop}
	\begin{proof}
If  $G_{k,I,i_{1},i_{2}}$ is constant zero, then we may set all of $Q_{j}$ to be constant zero. So we assume that $G_{k,I,i_{1},i_{2}}$ is not constant zero.

	By Lemma \ref{1:cl2}, we have that $\vert V(P)\cap \Omega_{k}\cap  Z_{k,I,i_{1},i_{2}}\vert\geq \d p^{dk-r}/2$. 
	By the Pigeonhole Principle and Theorem \ref{1:ct} (ii), it is not hard to see that there exist a subset $U\subseteq\Omega_{k-1}$ with $\vert U\vert\gg\d p^{d(k-1)-(r-r_{k})}$ and a subset $U(\n')\subseteq \Omega_{k}(\n')$ of cardinality $\gg\d p^{d-r_{k}}$ for each $\n'\in U$ such that for all $\n'\in U$ and $n_{k}\in U(\n')$, we have that  $P(\n',n_{k})=0$, $\Delta_{k,I}(\n')\neq 0$ and   $G_{k,I,i_{1},i_{2}}(\n')\neq 0$.

	For simplicity denote $n'_{k}:=n_{k,I^{c}}=(n_{k,\ell})_{\ell\in I^{c}}$ and set
	 $P_{I}(\n',n'_{k}):=P(\n',\phi_{k,I,\n'}(n'_{k}))$. Then for all $n_{k}\in \F_{p}^{d}$, by (\ref{1:meaningofp}), we have that $(\n',n_{k})\in V(P_{I}(\n',\cdot))\cap \Omega_{k}(\n')$  if and only if $n_{k}=\phi_{k,I,\n'}(n'_{k})$ for some $n'_{k}\in\F_{p}^{d-r_{k}+1}$ with $P_{I}(\n',n'_{k})=F_{I,k,1}(\n',n'_{k})=0$. Since  $U(\n')$ is of cardinality $\gg\d p^{d-r_{k}}$, we have that $\vert V(P_{I}(\n',\cdot))\cap V(F_{I,k,1}(\n',\cdot))\vert\gg\d p^{d-r_{k}}$. 
	
	Recall that $F_{I,k,1}(\n',\cdot)$ is a non-degenerate quadratic form of rank at least 3. 
	Since  $G_{k,I,i_{1},i_{2}}(\n')\neq 0$, by  Proposition \ref{1:noloop}, 
	$F_{I,k,1}(\n',\cdot)$ divides $P_{I}(\n',\cdot)$ with respect to the $B_{i_{1},i_{2}}$-standard long division algorithm. Note that every coefficient of $P_{I}(\n',\cdot)$ can be written as $R(\n')/(\Delta_{I}(\n'))^{s}$ for some polynomial $R$ of degree at most $(r_{k}-1)s$, and every coefficient of   $f_{I,k,1}(\n',\cdot)$ can be written as $R(\n')/(\Delta_{I}(\n'))^{2}$ for some polynomial $R$ of degree at most $2(r_{k}-1)$.
	So the $B_{i_{1},i_{2}}$-standard long division algorithm gives us that 
	$$P_{I}(\n',n'_{k})=F_{I,k,1}(\n',n'_{k})R'(\n',n'_{k})/(G_{k,I,i_{1},i_{2}}(\n')\Delta_{I}(\n'))^{s}$$
	for some polynomial $R'$ of degree at most $O_{k,s}(1)$. 
	So
	\begin{equation}\label{1:subsub2}
	(G_{k,I,i_{1},i_{2}}(\n')\Delta_{k,I}(\n'))^{s}P_{I}(\n',n'_{k})-F_{I,k,1}(\n',n'_{k})R'(\n',n'_{k})=0
	\end{equation}
	for all $\n'\in U$ and $n'_{k}\in\F_{p}^{d-r_{k}+1}$.
	We may write the left hand side of (\ref{1:subsub2}) as $$\sum_{m\in \N^{d-r_{k}},\vert m\vert\leq s}C_{m}(\n'){n'_{k}}^{m}$$ for some polynomial $C_{m}(\n')$ of degree at most $O_{k,s}(1)$.
	Then we have that $C_{m}(\n')=0$ for all $\n'\in U$ and thus $\vert V(C_{m})\cap \Omega_{k-1}\vert\gg\d p^{d(k-1)-(r-r_{k})}$. By induction hypothesis, we have that $\Omega_{k-1}\subseteq V(C_{m})$. So (\ref{1:subsub2}) holds for all $\n'\in \Omega_{k-1}$ and $n'_{k}\in\F_{p}^{d-r_{k}+1}$.

		Let $(\phi_{k,I,\n'}(n'_{k}))_{\ell}$ denote the $\ell$-th entry of $\phi_{k,I,\n'}(n'_{k})$. 
		By (\ref{1:subsub1}) and (\ref{1:subvomit}), for all $\ell\in I$, 		$\Delta_{k,I}(\n')((\phi_{k,I,\n'}(n'_{k}))_{\ell}-n_{k,\ell})$ equals to
		the determinant of the $(r_{k}-1)\times (r_{k}-1)$ matrix  obtained by replacing the column $(S_{k,j,\ell}(\n'))_{2\leq j\leq r_{k}}$ from $(S_{k,j,\ell''}(\n'))_{2\leq j\leq r_{k},\ell''\in I}$ by the column 	$(q_{j,\ell})_{2\leq j\leq r_{k}}$, where
		$$q_{j,\ell}=R_{k,I,j,\ell}(\n',n'_{k})-n_{k,\ell}S_{k,j,\ell}(\n')=-F_{k,j}(\n',n_{k})+\sum_{\ell'\in I\backslash\{\ell\}}n_{k,\ell'}S_{k,j,\ell'}(\n')$$
		for all $2\leq j\leq r_{k}$, which, by linearity, equals to the determinant of the $(r_{k}-1)\times (r_{k}-1)$ matrix  obtained by replacing the column $(S_{k,j,\ell}(\n'))_{2\leq j\leq r_{k}}$ from $(S_{k,j,\ell''}(\n'))_{2\leq j\leq r_{k},\ell''\in I}$ by the column 	$(-F_{k,j}(\n',n_{k}))_{2\leq j\leq r_{k}}$.
This means that $\Delta_{k,I}(\n')((\phi_{k,I,\n'}(n'_{k}))_{\ell}-n_{k,\ell})$ can be written in the form
		\begin{equation}\label{1:repform}
		\sum_{j=2}^{r_{k}}F_{k,j}(\n',n_{k})Q_{j}(\n',n_{k})
		\end{equation}
		for some polynomial $Q_{j}\in\poly((\V)^{k}\to\F_{p})$ of degree at most $r_{k}-1$.
		Therefore,   $$\Delta_{k,I}(\n')^{s}(P(\n',n_{k})-P_{I}(\n',n'_{k}))=\Delta_{k,I}(\n')^{s}(P(\n',n_{k})-P(\n',(\phi_{k,I,\n'}(n'_{k})))$$
			can be written in the form (\ref{1:repform}) for some polynomial $Q_{j}\in\poly((\V)^{k}\to\F_{p})$ of degree at most $ks$,	and
		$$\Delta_{k,I}(\n')^{2}(F_{k,1}(\n',n_{k})-F_{I,k,1}(\n',n'_{k}))=\Delta_{k,I}(\n')^{2}(F_{k,1}(\n',n_{k})-F_{k,1}(\n',\phi_{k,I,\n'}(n'_{k})))$$	 
		can be writing in the form (\ref{1:repform}) for some polynomia $Q_{j}\in\poly((\V)^{k}\to\F_{p})$ of degree at most $2s$. Since (\ref{1:subsub2}) holds for all $\n'\in \Omega_{k-1}$ and $n'_{k}\in\F_{p}^{d-r_{k}+1}$, we have that 
		\begin{equation}\nonumber
		(G_{k,I,i_{1},i_{2}}(\n')\Delta_{I}(\n'))^{s+2}P(\n',n_{k})
		=\sum_{j=1}^{r_{k}}F_{k,j}(\n',n_{k})Q_{j}(\n',n_{k})
		\end{equation}
		for some polynomial $Q_{j}$ of degree at most $O_{k,s}(1)$ for all $\n'\in\Omega_{k-1}$ and $n_{k}\in\V$. 
	\end{proof}

	\textbf{Step 3: removing the $(G_{k,I,i_{1},i_{2}}(\n')\Delta_{k,I}(\n'))^{s+2}$ term using a parametrization trick.} 
	We are now ready to complete the proof of Part (i) of Theorem \ref{1:irrrg} for $\Omega_{k}$ for the case $c_{k}=1$.
	Let $P\in\poly((\V)^{k}\to\F_{p})$ be a polynomial of degree at most $s$ with $\vert V(P)\cap \Omega_{k}\vert\geq \d \vert \Omega_{k}\vert$.
	We wish to show that for all $\n=(\n',n_{k})\in \Omega_{k}$, we have that $P(\n',n_{k})=0$.

	We first consider the case when $n_{1},\dots,n_{k-1}$ are linearly independent. Since $L_{k,j}, 2\leq j\leq r_{k}$ are linearly independent,\footnote{Meaning that no nontrivial linear combination of $L_{k,j}, 2\leq j\leq r_{k}$ is zero.} we have that $L_{k,j}(\n'), 2\leq j\leq r_{k}$  are linearly independent.  
	Since $A$ is invertible, the vectors $L_{k,j}(\n')A=(S_{k,j,1}(\n'),\dots,S_{k,j,d}(\n')), 2\leq j\leq r_{k}$  are also linearly independent. By the knowledge of linear algebra, there exists a subset $I$ of $\{1,\dots,d\}$ of cardinality $r_{k}-1$ such that $\Delta_{k,I}(\n')\neq 0$.
	By Lemma \ref{1:noog},
	 there exist $i_{1},i_{2}\in I^{c}, i_{1}\leq i_{2}$ such that  $G_{k,I,i_{1},i_{2}}(\n')\neq 0$. On the other hand, by Theorem \ref{1:ct}, $\vert V(P)\cap\Omega_{k}\vert\geq  \d p^{dk-r}/2$ since $d\geq 2r+1$.
	Since $(\n',n_{k})\in \Omega_{k}$, we have that $(G_{k,I,i_{1},i_{2}}(\n')\Delta_{k,I}(\n'))^{s+2}P(\n',n_{k})=0$ by Proposition \ref{1:cl1}. Therefore,  $P(\n',n_{k})=0$.

	We now consider the general case.
	For $1\leq t\leq k-1$, we say that \emph{Property $t$} holds if for all $\n=(\n',n_{k})\in \Omega_{k}$ with 
	$n_{1},\dots,n_{t}$ being linearly independent, we have that $P(\n,n_{k})=0$.  By the above discussion, Property $k-1$ holds.
	Assume now that we have proven Property $t+1$ holds for some $0\leq t\leq k-2$. We show that Property $t$ holds.
	Pick any $\n=(\n',n_{k})\in \Omega_{k}$ such that $n_{1},\dots,n_{t}$  are linearly independent. If $n_{1},\dots,n_{t+1}$ are linearly independent, then $P(\n',n_{k})=0$ by induction hypothesis and we are done. So we assume that $n_{1},\dots,n_{t+1}$ are linearly dependent.
	
 	\begin{lem}\label{1:fvv3}
	There exists a subspace  $V$ of $\V$ dimension $3$ such that 
	\begin{itemize}
		\item for all $v\in V\backslash\{\bold{0}\}$, $n_{1},\dots,n_{t},v$ are linearly independent;
		\item for all $v\in V$ and $1\leq j\leq k$, we have that $(vA)\cdot n_{j}=0$;
		\item  we have that $V\cap V^{\pp}=\{\bold{0}\}$.
	\end{itemize}	
	\end{lem} 
	\begin{proof}
	 Let $W_{1}$ denote the set of $v\in\V$ such that $n_{1},\dots,n_{t},v$ are linearly dependent. 
	Then $\vert W_{1}\vert=p^{t}$.
		Let $W_{2}$ denote the set of $v\in\V$ such that 
	$(vA)\cdot n_{j}=0$
	for all $1\leq j\leq k$.
	Then $W_{2}$ is a subspace of $\V$ of dimension at least $d-k$. 
So there exists a subspace $W_{3}$ of $W_{2}$ with $W_{1}\cap W_{3}=\{\bold{0}\}$ of dimension 
$d'\geq d-k-t.$ 

Let $\phi\colon \F_{p}^{d'}\to W_{3}$ be any bijective linear transformation and let $M'\colon \F_{p}^{d'}\to\F_{p}$ be the quadratic form given by $M'=M\circ\phi$.
By Proposition \ref{1:iissoo} (iii), $\rank(M')=\rank(M\vert_{W_{3}})\geq d-2(k+t)\geq d-2(2k-2)\geq 3$. 
So by Lemma \ref{1:wfejpo}, there exists a subspace $U$ of $\F_{p}^{d'}$ of dimension 3 which is $M'$-non-isotropic. Then clearly $\phi(U)$ is $M$-non-isotropic. Moreover, $\phi(U)$ is a subspace of $W_{3}$ of dimension 3. We are done by taking $V=\phi(U)$.
\end{proof}

Let  $V$  be given by Lemma \ref{1:fvv3}, and $\phi\colon \F_{p}^{3}\to V$ be any bijective linear transformation. 
	Let $W$ be the set of $m\in\F_{p}^{3}$ such that $(n_{1},\dots,n_{t},n_{t+1}+\phi(m),n_{t+2},\dots,n_{k})\in\Omega_{k}$. Note that $\bold{0}\in W$.
	Since $(\phi(m)A)\cdot n_{j}=0$ for all $1\leq j\leq k$, by (\ref{1:e:fi1}) and (\ref{1:e:fij}),
	for all $1\leq i'\leq k$ and $1\leq j\leq r_{i'}$ with either $(i',j)\neq (t+1,1)$ or $(i',j)=(t+1,1)$ and $c_{t+1}=0$,
	 it is not hard to see that
	$$F_{i',j}(n_{1},\dots,n_{t},n_{t+1}+\phi(m),n_{t+2},\dots,n_{k})=F_{i',j}(n_{1},\dots,n_{t},n_{t+1},n_{t+2},\dots,n_{k}).$$
	Moreover, if $(i',j)=(t+1,1)$ and $c_{t+1}=1$, then by (\ref{1:e:fi1}),
	we have that
	$$M'(m):=F_{t+1,1}(n_{1},\dots,n_{t},n_{t+1}+\phi(m),n_{t+2},\dots,n_{k})$$
	is a quadratic form from $\F_{p}^{3}$ to $\F_{p}$. 	It is then not hard to see that 
	$W=V(M')$ if $c_{t+1}=1$ and $W=\F_{p}^{3}$ if $c_{t+1}=0$. Since $V\cap V^{\pp}=\{\bold{0}\}$, by Proposition \ref{1:iissoo} (ii), $\rank(M')=3$.
	So by  Propositions \ref{1:iri00} and \ref{1:iri0}, $W$ is $1/2$-irreducible.  
	
	On the other hand, for all $m\in\F_{p}^{3}\backslash\{\bold{0}\}$,
	since $n_{1},\dots,n_{t},\phi(m)$ are linearly independent and since $n_{1},\dots,n_{t+1}$ are linearly dependent, we have that $n_{1},\dots,n_{t},n_{t+1}+\phi(m)$ are linearly independent.
	By the induction hypothesis that Property $t+1$ holds, we have that 
	$$P'(m):=P(n_{1},\dots,n_{t},n_{t+1}+\phi(m),n_{t+2},\dots,n_{k})=0$$
	for all $m\in W\backslash\{\bold{0}\}$.
	In other words, $W\backslash\{\bold{0}\}\subseteq V(P')$. So $\vert V(P')\cap W\vert\geq \vert W\vert-1\geq\vert W\vert/2$.
	Since $W$ is $1/2$-irreducible,   we have that $W\subseteq V(P')$. In particular, since $\bold{0}\in W$, we have that 
	$P'(\bold{0})=P(\n',n_{k})=0.$
	
	In conclusion, we have that Property $t$ holds.  By induction, we deduce that Property 0 holds, meaning that  $\Omega_{k}\subseteq V(P)$. So $\Omega_{k}$ is $\d$-irreducible up to degree $s$.

		\subsection{Proof of Part (i) of Theorem \ref{1:irrrg} for the degenerate case}\label{1:s841}
	
	We  now consider the case when $c_{k}=0$. 
	In this case $r_{k}\leq k-1$. 	Since the proof is very similar to (actually slightly easier than) the case $c_{k}=1$, we only provide an outline for the proof of this case and leave the details to the interested readers.

	\textbf{Step 1: a parametrization for the last variable.}
	Write
	$$L_{k,j}(\n')A=(S_{k,j,1}(\n'),\dots,S_{k,j,d}(\n'))$$
	for some polynomials $S_{k,j,1},\dots,S_{k,j,d}\colon(\V)^{k-1}\to\F_{p}$ of degree at most 1. Since $A$ is diagonal, for all $1\leq \ell\leq d$, $S_{k,j,\ell}(\n')$ is dependent only on $n_{i,\ell}, 1\leq i\leq k-1$.
	 Then we may rewrite (\ref{1:e:fij}) as
		\begin{equation}\label{1:vsubsub0}
		F_{k,j}(\n',n_{k})=\sum_{\ell=1}^{d}n_{k,\ell}S_{k,j,\ell}(\n')+u_{k,j},
		\end{equation} 
 for all $1\leq j\leq r_{i}$.
 
	For $I\subseteq \{1,\dots,d\}$ with $\vert I\vert=r_{k}>0$, let $\Delta_{k,I}(\n')$   denote the determinant of the $r_{k}\times r_{k}$ matrix
	$(S_{k,j,\ell}(\n'))_{2\leq j\leq r_{k}, \ell\in I}$. 
	Then 
	$\Delta_{k,I}$ is a polynomial from $(\V)^{k}$ to $\F_{p}$ of degree at most $r_{k}$  depending only on $n_{i,\ell}, 1\leq i\leq k-1, \ell\in I$. 
	For convenience denote $I^{c}:=\{1,\dots,d\}\backslash I$.
	Fix any $\n'\in \Omega_{k-1}$ with $\Delta_{k,I}(\n')\neq 0$. Then $L_{k,j}(\n')A, 2\leq j\leq r_{k}$ are linearly independent. By (\ref{1:vsubsub0}) and the knowledge of linear algebra,  $F_{k,1}(\n',n_{k})=\dots=F_{k,r_{k}}(\n',n_{k})=0$ if and only if
	for all $\ell\in I$,
	\begin{equation}\label{1:vsubsub1}
	n_{k,\ell}=\frac{R_{k,I,\ell}(\n',n_{k,I^{c}})}{\Delta_{k,I}(\n')},
	\end{equation}	
	where $R_{k,I,\ell}(\n',n_{k,I^{c}})$ is the determinant of the  $r_{k}\times r_{k}$ matrix  obtained by replacing the column $(S_{k,j,\ell}(\n'))_{1\leq j\leq r_{k}}$ from $(S_{k,j,\ell'}(\n'))_{1\leq j\leq r_{k},\ell'\in I}$ by the column $(R_{k,I,j,\ell}(\n',n_{k,I^{c}}))_{1\leq j\leq r_{k}}$ where
	\begin{equation}\label{1:vsubvomit}
	R_{k,I,j,\ell}(\n',n_{k,I^{c}})=-u_{k,j}-\sum_{\ell'\in I^{c}}n_{k,\ell'}S_{k,j,\ell'}(\n')=-F_{k,j}(\n',n_{k})+\sum_{\ell'\in I}n_{k,\ell'}S_{k,j,\ell'}(\n').
	\end{equation}

	Let $Z_{k,I}$ denote the set of $(\n',n_{k})\in(\V)^{k}$ such that $\Delta_{k,I}(\n')\neq 0$. 
Similar to the proof of Lemma \ref{1:cl2}, we have (we omit the details):
	
	\begin{lem}\label{1:vcl2}
	We have that $\vert {Z}^{c}_{k,I}\cap\Omega_{k}\vert\leq \d p^{dk-r}/2$. 	
	\end{lem}

	\textbf{Step 2: a factorization for the target polynomial.}
		This step is carried out  by the following proposition:

	\begin{prop}\label{1:vcl1} 
		Let $s\in\N$, $1\leq k\leq K$, $I\subseteq \{1,\dots,d\}$ with $\vert I\vert=r_{k}$ and $P\in\poly((\V)^{k}\to\F_{p})$ be a polynomial of degree at most $s$ with $\vert V(P)\cap \Omega_{k}\vert\geq \d p^{dk-r}$.	Then
	\begin{equation}\nonumber
	\Delta_{k,I}(\n')^{s+2}P(\n',n_{k})
	=\sum_{j=1}^{r_{k}}F_{k,j}(\n',n_{k})Q_{j}(\n',n_{k})
	\end{equation}
	for some polynomial $Q_{j}$ of degree at most $O_{k,s}(1)$ for all $\n'\in\Omega_{k-1}$ and $n_{k}\in\V$. 
	\end{prop}
	\begin{proof}
	The proof is a simpler version of that of Proposition \ref{1:cl1}.
	By Lemma \ref{1:vcl2}, we have that $\vert V(P)\cap \Omega_{k}\cap  Z_{k,I}\vert\geq \d p^{dk-r}/2$. 
	By the Pigeonhole Principle and Theorem \ref{1:ct} (ii), it is not hard to see that there exist a subset $U\subseteq\Omega_{k-1}$ with $\vert U\vert\gg\d p^{d(k-1)-(r-r_{k})}$ and a subset $U(\n')\subseteq \Omega_{k}(\n')$ of cardinality $\gg\d p^{d-r_{k}}$ for each $\n'\in U$ such that for all $\n'\in U$ and $n_{k}\in U(\n')$, we have that  $P(\n',n_{k})=0$ and $\Delta_{k,I}(\n')\neq 0$.

	For simplicity denote $n'_{k}:=n_{k,I^{c}}=(n_{k,\ell})_{\ell\in I^{c}}$ and set
	 $P_{I}(\n',n'_{k}):=P(\n',\phi_{k,I,\n'}(n'_{k}))$. Then for all $n_{k}\in \F_{p}^{d}$,  we have that $(\n',n_{k})\in V(P_{I}(\n',\cdot))\cap \Omega_{k}(\n')$  if and only if $n_{k}=\phi_{k,I,\n'}(n'_{k})$ for some $n'_{k}\in\F_{p}^{d-r_{k}}$ with $P_{I}(\n',n'_{k})=0$. Since  $U(\n')$ is of cardinality $\gg\d p^{d-r_{k}}$, we have that $\vert V(P_{I}(\n',\cdot))\vert\gg\d p^{d-r_{k}}$.  
	 By Lemma \ref{1:ns}, we have that
	\begin{equation}\label{1:vsubsub2}
	P_{I}(\n',n'_{k})=0
	\end{equation}
	for all $\n'\in U$ and $n'_{k}\in\F_{p}^{d-r_{k}}$.
	We may write the left hand side of (\ref{1:vsubsub2}) as $$\sum_{m\in \N^{d-r_{k}},\vert m\vert\leq s}C_{m}(\n'){n'_{k}}^{m}$$ for some polynomial $C_{m}(\n')$ of degree at most $O_{k,s}(1)$.
	Then we have that $C_{m}(\n')=0$ for all $\n'\in U$ and thus $\vert V(C_{m})\cap \Omega_{k-1}\vert\gg\d p^{d(k-1)-(r-r_{k})}$. By induction hypothesis, we have that $\Omega_{k-1}\subseteq V(C_{m})$. So (\ref{1:vsubsub2}) holds for all $\n'\in \Omega_{k-1}$ and $n'_{k}\in\F_{p}^{d-r_{k}}$.

	Let $(\phi_{k,I,\n'}(n'_{k}))_{\ell}$ denote the $\ell$-th entry of $\phi_{k,I,\n'}(n'_{k})$.
		By (\ref{1:vsubsub1}) and (\ref{1:vsubvomit}), for all $\ell\in I$, 		$\Delta_{k,I}(\n')((\phi_{k,I,\n'}(n'_{k}))_{\ell}-n_{k,\ell})$ equals to 
		the determinant of the  $r_{k}\times r_{k}$ matrix  obtained by replacing the column $(S_{k,j,\ell}(\n'))_{1\leq j\leq r_{k}}$ from $(S_{k,j,\ell'}(\n'))_{1\leq j\leq r_{k},\ell'\in I}$ by the column $(q_{j,\ell})_{1\leq j\leq r_{k}}$, where	$$q_{j,\ell}=R_{k,I,j,\ell}(\n',n'_{k})-n_{k,\ell}S_{k,j,\ell}(\n')=-F_{k,j}(\n',n_{k})+\sum_{\ell'\in I\backslash\{\ell\}}n_{k,\ell'}S_{k,j,\ell'}(\n')$$
		for all $1\leq j\leq r_{k}$, which, by linearity, equals to 
		the determinant of the  $r_{k}\times r_{k}$ matrix  obtained by replacing the column $(S_{k,j,\ell}(\n'))_{1\leq j\leq r_{k}}$ from $(S_{k,j,\ell'}(\n'))_{1\leq j\leq r_{k},\ell'\in I}$ by the column $(-F_{k,j}(\n',n_{k}))_{1\leq j\leq r_{k}}$
 This means that $\Delta_{k,I}(\n')((\phi_{k,I,\n'}(n'_{k}))_{\ell}-n_{k,\ell})$ can be written in the form
		\begin{equation}\label{1:vrepform}
		\sum_{j=1}^{r_{k}}F_{k,j}(\n',n_{k})Q_{j}(\n',n_{k})
		\end{equation}
		for some polynomial $Q_{j}\in\poly((\V)^{k}\to\F_{p})$ of degree at most $r_{k}$.
		Therefore,   $$\Delta_{k,I}(\n')^{s}(P(\n',n_{k})-P_{I}(\n',n'_{k}))=\Delta_{k,I}(\n')^{s}(P(\n',n_{k})-P(\n',(\phi_{k,I,\n'}(n'_{k})))$$
			can be written in the form (\ref{1:vrepform}) for some polynomial $Q_{j}\in\poly((\V)^{k}\to\F_{p})$ of degree at most $ks$. Since (\ref{1:vsubsub2})  holds for all $\n'\in \Omega_{k-1}$ and $n'_{k}\in\F_{p}^{d-r_{k}}$, we have that 
		\begin{equation}\nonumber
		\Delta_{I}(\n')^{s+2}P(\n',n_{k})
		=\sum_{j=1}^{r_{k}}F_{k,j}(\n',n_{k})Q_{j}(\n',n_{k})
		\end{equation}
		for some polynomial $Q_{j}$ of degree at most $O_{k,s}(1)$ for all $\n'\in\Omega_{k-1}$ and $n_{k}\in\V$. 
	\end{proof}

	\textbf{Step 3: removing the $\Delta_{k,I}(\n')^{s+2}$ term using a parametrization trick.}
	We are now ready to complete the proof of Part (i) of Theorem \ref{1:irrrg} for $\Omega_{k}$ for the case $c_{k}=1$.
	Let $P\in\poly((\V)^{k}\to\F_{p})$ be a polynomial of degree at most $s$ with $\vert V(P)\cap \Omega_{k}\vert\geq \d \vert \Omega_{k}\vert$.
	We wish to show that for all $\n=(\n',n_{k})\in \Omega_{k}$, we have that $P(\n',n_{k})=0$.

	 We first consider the case when $n_{1},\dots,n_{k-1}$ are linearly independent. 
	Since $L_{k,j}, 1\leq j\leq r_{k}$ are linearly independent, we have that $L_{k,j}(\n'), 1\leq j\leq r_{k}$  are linearly independent.  
	Since $A$ is invertible, the vectors $L_{k,j}(\n')A=(S_{k,j,1}(\n'),\dots,S_{k,j,d}(\n')), 1\leq j\leq r_{k}$  are also linearly independent. By the knowledge of linear algebra, there exists a subset $I$ of $\{1,\dots,d\}$ of cardinality $r_{k}$ such that $\Delta_{k,I}(\n')\neq 0$.
On the other hand, by Theorem \ref{1:ct}, $\vert V(P)\cap\Omega_{k}\vert\geq  \d p^{dk-r}/2$ since $d\geq 2r+1$.
	Since $(\n',n_{k})\in \Omega_{k}$, we have that $\Delta_{k,I}(\n')^{s+2}P(\n',n_{k})=0$ by Proposition \ref{1:vcl1}. Therefore,  $P(\n',n_{k})=0$.
	
	If $n_{1},\dots,n_{k-1}$ are linearly dependent, the argument is similar to the case $c_{k}=1$. We omit the details.
	This completes the proof of Part (i) of Theorem \ref{1:irrrg}.

	\subsection{Hilbert Nullstellensatz properties for $M$-sets}
	
	Before  proving Parts (ii) and (iii) of Theorem \ref{1:irrrg},
we take a detour to study  Hilbert Nullstellensatz properties for $M$-sets. It is natural to ask the following question:

\begin{conj}[Hilbert Nullstellensatz for $M$-sets]\label{1:ccmmj}
	Let $d,k\in\N_{+}$, $s\in\N$, $p$ be a prime, $M\colon\V\to\F_{p}$ be a non-degenerate quadratic form, and $\Omega\subseteq(\V)^{k}$ be a consistent  $M$-set given by $\Omega=V(\mathcal{J})$ for some consistent $(M,k)$-family $\mathcal{J}=\{f_{1},\dots,f_{r}\}$ for some $r\in\N_{+}$.  If $d\gg_{k,r,s} 1$ and $p\gg_{d} 1$, then for any  polynomial $P\in\poly((\V)^{k}\to\F_{p})$ of degree at most $s$ with $\Omega\subseteq V(P)$, we have that
	\begin{equation}\nonumber
	P
	=\sum_{j=1}^{r}f_{j}Q_{j}
	\end{equation}
	for some polynomial $Q_{j}$ of degree $O_{k,r,s}(1)$. 
	\end{conj}

 Unfortunately
	we do not know the answer to Conjecture \ref{1:ccmmj}. However, we have a partial result which is sufficient for the purpose of this paper.
 We need to introduce some notations before stating the result. Let the notations be the same as Sections \ref{1:s:94}, \ref{1:s84} and \ref{1:s841}.
Assume that $\Omega$ is of type $(c_{1},\dots,c_{K})$.
 For all $1\leq k\leq K$, if $c_{k}=1$, pick some $I_{k}\subseteq \{1,\dots,d\}$ with $\vert I_{k}\vert=r_{k}-1$ and some $i_{k,1},i_{k,2}\in\{1,\dots,d\}\backslash I_{k}, i_{k,1}\leq i_{k,2}$;  if $c_{k}=0$, pick some $I_{k}\subseteq \{1,\dots,d\}$ with $\vert I_{k}\vert=r_{k}$ and set $i_{k,1}=i_{k,2}=0$. We say that $\mathfrak{I}:=(I_{k},i_{k,1},i_{k,2})_{k=1}^{K}$ is an \emph{admissible parametrization} of $\Omega$. Define  
	$$D_{\mathfrak{I}}(n_{1},\dots,n_{K}):=\prod_{k=1}^{K}G_{k,I_{k},i_{k,1},i_{k,2}}(n_{1},\dots,n_{k-1})\Delta_{k,I_{k}}(n_{1},\dots,n_{k-1}),$$
	where we set $\Delta_{k,I_{k}}\equiv 1$ if $I_{k}=\emptyset$ and set $G_{k,I_{k},i_{k,1},i_{k,2}}\equiv 1$ if either $I_{k}=\emptyset$ or $i_{k,1}=i_{k,2}=0$.	
We have the following partial answer to Conjecture \ref{1:ccmmj}:	
 
\begin{prop}\label{1:cl1s} 
Let the notations be the same as Sections \ref{1:s:94}, \ref{1:s84} and \ref{1:s841}.
Let $s\in\N$ and $\mathfrak{I}$ be an admissible parametrization of $\Omega$. 
Let $P\in\poly((\V)^{k}\to\F_{p})$ be a polynomial of degree at most $s$ with $\Omega\subseteq V(P)$. Then
	\begin{equation}\label{1:subsubfk8}
\begin{split}
	D_{\mathfrak{I}}^{T}P
	 =\sum_{k=1}^{K}\sum_{j=1}^{r_{k}}F_{k,j}Q_{k,j}
\end{split}
	\end{equation}
	for some $T=O_{K,s}(1)\in\N$ and some polynomial $Q_{k,j}$ of degree at most $O_{K,s}(1)$. 
	\end{prop}
\begin{proof}
Assume that $\mathfrak{I}:=(I_{k},i_{k,1},i_{k,2})_{k=1}^{K}$.
For $1\leq t\leq K$, we say that \emph{Property $t$} holds if
\begin{equation}\nonumber
\begin{split}
	&\quad\Bigl(\prod_{k=t}^{K}G_{k,I_{k},i_{k,1},i_{k,2}}(n_{1},\dots,n_{k-1})\Delta_{k,I_{k}}(n_{1},\dots,n_{k-1})\Bigr)^{T}P(n_{1},\dots,n_{K})
	\\&=\sum_{k=t}^{K}\sum_{j=1}^{r_{k}}F_{k,j}(n_{1},\dots,n_{k})Q_{k,j}(n_{1},\dots,n_{K})
\end{split}
	\end{equation}
for some $T=O_{K,s}(1)\in\N$ and some polynomial $Q_{k,j}$ of degree at most $O_{K,s}(1)$ for all $(n_{1},\dots,n_{t-1})\in\Omega_{t-1}$ and $n_{t},\dots,n_{K}\in\V$.

By Theorem \ref{1:ct}, Propositions \ref{1:cl1} and \ref{1:vcl1}, we have that Property $K$ holds. Suppose now that we have shown Property $(t+1)$ for some $1\leq t\leq K-1$, we prove that 
 Property $t$ holds. Suppose that 
\begin{equation}\label{1:subsubfk2}
\begin{split}
	&\quad\Bigl(\prod_{k=t+1}^{K}G_{k,I_{k},i_{k,1},i_{k,2}}(n_{1},\dots,n_{k-1})\Delta_{k,I_{k}}(n_{1},\dots,n_{k-1})\Bigr)^{T}P(n_{1},\dots,n_{K})
	\\&-\sum_{k=t+1}^{K}\sum_{j=1}^{r_{k}}F_{k,j}(n_{1},\dots,n_{k})Q_{k,j}(n_{1},\dots,n_{K})=0
\end{split}
	\end{equation}
for some $T=O_{K,s}(1)\in\N$ and some polynomial $Q_{k,j}$ of degree at most $O_{k,s}(1)$ for all $(n_{1},\dots,n_{t})\in\Omega_{t}$ and $n_{t+1},\dots,n_{K}\in\V$.
Let $Q$ be the polynomial on the left hand side of (\ref{1:subsubfk2}). Then $Q(n_{1},\dots,n_{K})=0$ for all $(n_{1},\dots,n_{t})\in\Omega_{t}$ and $n_{t+1},\dots,n_{K}\in\V$.
Writing 
\begin{equation}\label{1:subsubfk3}
\begin{split}
	Q(n_{1},\dots,n_{K})=\sum_{i_{t+1},\dots,i_{K}\in\N^{d},\vert i_{t+1}+\dots+i_{K}\vert\leq O_{K,s}(1)}C_{i_{t+1},\dots,i_{K}}(n_{1},\dots,n_{t})n_{t+1}^{i_{t+1}}\dots n_{K}^{i_{K}}
\end{split}
	\end{equation}
for some polynomials $C_{i_{t+1},\dots,i_{K}}$ of degrees at most $O_{k,s}(1)$, this implies that  $C_{i_{t+1},\dots,i_{K}}(n_{1},\dots,$ $n_{t})=0$ for all $i_{t+1},\dots,i_{K}\in\N^{d},\vert i_{t+1}+\dots+i_{K}\vert\leq O_{K,s}(1)$ and $(n_{1},\dots,n_{t})\in\Omega_{t}$. 
Applying Theorem \ref{1:ct},  Propositions \ref{1:cl1} and \ref{1:vcl1} to the set $\Omega_{t}$, we have that 
\begin{equation}\label{1:subsubfk4}
\begin{split}
	& \quad (G_{t,I_{t},i_{t,1},i_{t,2}}(n_{1},\dots,n_{t-1})\Delta_{t,I_{t}}(n_{1},\dots,n_{t-1}))^{T'}C_{i_{t+1},\dots,i_{K}}(n_{1},\dots,n_{t})
	\\&=\sum_{j=1}^{r_{t}}F_{t,j}(n_{1},\dots,n_{t})Q_{j;i_{t+1},\dots,i_{K}}(n_{1},\dots,n_{t})
\end{split}
	\end{equation}
	for some $T'=O_{K,s}(1)\in\N$ and some polynomial $Q_{j; i_{t+1},\dots,i_{K}}$ of degree at most $O_{K,s}(1)$ for all $(n_{1},\dots,n_{t-1})\in\Omega_{t-1}$ and $n_{t}\in\V$ for all $i_{t+1},\dots,i_{K}\in\N^{d},\vert i_{t+1}+\dots+i_{K}\vert\leq O_{K,s}(1)$. Property $t$ follows by combining (\ref{1:subsubfk2}), (\ref{1:subsubfk3}) and (\ref{1:subsubfk4}).

So by induction, we have that Property 1 holds and we are done.
\end{proof}

	Let $\tilde{D}_{\mathfrak{I}}, \tilde{F}_{k,j}\colon \Z^{d}\to\Z/p$ be regular liftings   of $D_{\mathfrak{I}}$ and $F_{k,j}$, respectively. By Lemmas \ref{1:ivie} and \ref{1:lifting}, we may lift Proposition \ref{1:cl1s} to the $\Z/p$ setting.

		\begin{coro}\label{1:cl3s} 
			Let the notations be the same as Sections \ref{1:s:94}, \ref{1:s84} and \ref{1:s841}.
Let $s\in\N$ and $\mathfrak{I}$ be an admissible parametrization of $\Omega$. 
			Let $P\in\poly_{p}((\Z^{d})^{k}\to \Z/p\vert \Z)$ be a polynomial of degree at most $s$ with $\iota^{-1}(\Omega)\subseteq V_{p}(P)$.	
			Then 
			there exist $T=O_{K,s}(1)\in\N$ and  $Q_{0},Q_{k,j}\in\poly((\Z^{d})^{K}\to \Z), 1\leq k\leq K, 0\leq j\leq r_{k}$ of degree at most $O_{K,s}(1)$ with $Q_{k,j}$ having integer coefficients  such that
\begin{equation}\nonumber
\begin{split}
	  (p\tilde{D}_{\mathfrak{I}})^{T}P
	 =Q_{0}+\sum_{k=1}^{K}\sum_{j=1}^{r_{k}}\tilde{F}_{k,j}Q_{k,j}.
\end{split}
	\end{equation}
		\end{coro}
	\begin{proof}
		Let $P'\in\poly((\V)^{k}\to\F_{p})$ be the polynomials induced by $P$. Then $\Omega\subseteq V(P')$. By 
		Proposition \ref{1:cl1s}, we have 
			\begin{equation}\label{1:subsub3}
			\begin{split}
			D_{\mathfrak{I}}^{T}P'
	 =\sum_{k=1}^{K}\sum_{j=1}^{r_{k}}F_{k,j}Q'_{k,j}
			\end{split}
			\end{equation}
		for some $T=O_{K,s}(1)\in\N$ and some 	$Q'_{k,j}$ of degree at most $O_{K,s}(1)$.  
		Let $\tilde{Q}_{k,j}$ be a regular lifting of $Q'_{k,j}$. 
By Lemma \ref{1:lifting} (iv), we have that $(p\tilde{D}_{\mathfrak{I}})^{T}P$
is a lifting of
$D_{\mathfrak{I}}^{T}P',$
and that $\sum_{k=1}^{K}\sum_{j=1}^{r_{k}}\tilde{F}_{k,j}p\tilde{Q}_{k,j}$
is a lifting of 
$\sum_{k=1}^{K}\sum_{j=1}^{r_{k}}F_{k,j}Q'_{k,j}.$
		So by Lemma \ref{1:lifting} (iii)  and (\ref{1:subsub3}),  the difference 
between	$(p\tilde{D}_{\mathfrak{I}})^{T}P$ and  $\sum_{k=1}^{K}\sum_{j=1}^{r_{k}}\tilde{F}_{k,j}p\tilde{Q}_{k,j}$		
		equals to an integer valued polynomial $Q_{0}$
of degree at most $O_{K,s}(1)$.
		Finally, since $\tilde{Q}_{k,j}$ takes values in $\Z/p$, $Q_{k,j}:=p\tilde{Q}_{k,j}$ takes values in $\Z$. Applying Lemma \ref{1:ivie} to modify $Q_{k,j}$ and $Q_{0}$ accordingly, we may require that all of $Q_{k,j}$ have integer coefficients, while $Q_{0}$ still takes integer values. 
		We are done.
	\end{proof}

	\subsection{Proof of parts (ii) and (iii) of Theorem \ref{1:irrrg}}
	We use the same notations as in Sections \ref{1:s:94} and \ref{1:s84}, except that now $\n,\n',n_{i},n_{i,j}$  denote elements in $\Z^{t}$ instead of $\F_{p}^{t}$ for some appropriate $t\in\N_{+}$, and that we use $r$ to denote $r_{M}(\Omega)$.
	Part (ii) follows from Part (i) and Proposition \ref{1:f2z}. So it remains to prove Part (iii).
	Our strategy is to use the $p$-expansion trick.
	Let $g\in\poly((\Z^{d})^{K}\to\Z/p^{s})$ be a polynomial of degree at most $s$ with $\vert V_{p}(g)\cap \iota^{-1}(\Omega)\cap ([p]^{d})^{K}\vert\geq \d p^{dK-r}$. Our goal is to show that $\iota^{-1}(\Omega)\subseteq V_{p}(g)$.

	Since $g$  takes values in $\Z/p^{s}$,
	by the multivariate polynomial interpolation, it is not hard to see that 	there exists $Q\in\N$, $p\nmid Q$ such that
	$$Qg=\sum_{i=0}^{s}\frac{g'_{i}}{p^{i}}$$ 
	for some 	integer valued polynomials $g'_{i}\colon(\Z^{d})^{K}\to\Z$ of degree at most $s$.

	It is convenient to reindex the polynomials $F_{k,j}, 1\leq k\leq K, 1\leq j\leq r_{k}$ as $f_{1},\dots,f_{r}$ in an arbitrary order. Let $\tilde{f}_{j}$ be a regular lifting of $f_{j}$.
		We say that a polynomial $F\colon(\Z^{d})^{K}\to\R$ is \emph{$t$-good} if 
		$$F=\sum_{i:=(i_{1},\dots,i_{r})\in\N^{r}, \vert i\vert\leq s/2}F_{i}\prod_{j=1}^{r}\tilde{f}_{j}^{i_{j}}$$
		for some  integer valued  
		polynomials $F_{i}\in\poly((\Z^{d})^{K}\to\Z)$ of degree at most $t-2\vert i\vert, 0\leq \vert i\vert\leq s/2$.  
		For convenience denote $a_{s}=s$ and let $a_{i}=O_{a_{i+1},K,s}(1)$ to be chosen later for all $0\leq i\leq s-1$. Then $a_{0}=O_{K,s}(1)<p$.

Let $\mathfrak{I}=(I_{k},i_{k,1},i_{k,2})_{k=1}^{K}$ be an admissible parametrization of $\Omega$.
	Then $p\tilde{D}_{\mathfrak{I}}$ is an integer valued polynomial of degree at most $O_{K,s}(1)$. Let $T=O_{K,s}(1)\in\N$  to be chosen later.
	Let $k'\in\N$ be the smallest integer such that 
	\begin{equation}\label{1:gg}
	(p\tilde{D}_{\mathfrak{I}})^{T(s-k')}Qg=\sum_{i=0}^{k'}\frac{g_{i}}{p^{i}}
	\end{equation}
	for some $Q\in\N$, $p\nmid Q$ and 
	$a_{k'}$-good polynomials $g_{i}$. 
	Obviously such $k'$ exists and is at most $s$. We first show that $k'=0$.	
	
	Suppose that 
	$k'>0$.
	Let $U:=V_{p}(g)\cap \tau(\Omega)$.
	 For all $\n\in U$  and $\m\in(\Z^{d})^{K}$, since $g(\n+p\m)\in\mathbb{Z}$, we have that
	$g_{k'}(\n+p\m)\in p\Z.$
	Suppose that 
		\begin{equation}\label{1:tempeeq1}
	g_{k'}=\sum_{i:=(i_{1},\dots,i_{r})\in\N^{r}, \vert i\vert\leq s/2}R_{i}\prod_{j=1}^{r}\tilde{f}_{j}^{i_{j}}
		\end{equation}
	for some integer valued polynomials $R_{i}\in\poly((\Z^{d})^{K}\to\mathbb{Z})$ of degree at most $a_{k'}-2\vert i\vert$. 
	
	\textbf{Claim.} There exists a subset $U'$ of $U$ of cardinality  at least $\d p^{dK-r}/2$ such that for all $\n\in U'$ and $i\in\N^{r}, \vert i\vert\leq s/2$, we have that $\frac{1}{p}R_{i}(\n)\in \mathbb{Z}$.
	
	By (\ref{1:e:fi1}) and (\ref{1:e:fij}), for all $1\leq j\leq r$, we may write 
	$$f_{j}(n_{1},\dots,n_{K})=\sum_{1\leq i\leq K}\frac{1}{2}c_{j,i,i}(n_{i}A)\cdot n_{i}+\sum_{1\leq i< i'\leq K}c_{j,i,i'}(n_{i}A)\cdot n_{i'}+c'_{j}$$
	for some $c_{j,i,i'},c'_{j}\in \F_{p}$ such that the vectors $v_{j}:=(c_{j,i,i'})_{1\leq i\leq i'\leq K}\in \F_{p}^{\binom{K+1}{2}}, 1\leq j\leq r$ are linearly independent. For convenience denote $c_{j,i,i'}:=c_{j,i',i}$ for $i'<i$. It is not hard to see that we may write
	$$\tilde{f}_{j}(\n+p\m)\equiv \tilde{f}_{j}(\n)+T_{j}(\n,\m)\mod p\Z$$
  for all $\m,\n\in(\Z^{d})^{K}$, where  $$T_{j}(\n,\m)=\sum_{i=1}^{K}(m_{i}\tau(A))\cdot\Bigl(\sum_{i'=1}^{K}\tau(c_{j,i,i'})n_{i'}\Bigr).$$
 
	Fix any $\n\in U$. Then 
	\begin{equation}\label{1:fnp}
	0\equiv g_{k'}(\n+p\m)\equiv\sum_{i:=(i_{1},\dots,i_{r})\in\N^{r}, \vert i\vert\leq s/2}R_{i}(\n)\prod_{j=1}^{r}(\tilde{f}_{j}(\n)+T_{j}(\n,\m))^{i_{j}}  \mod p\Z
	\end{equation}
	for all $m\in\Z^{d}$.
	Let $U''$ denote the set of $\n=(n_{1},\dots,n_{K})\in(\V)^{K}$ such that
	there exist $t_{1},\dots,t_{r}\in \F_{p}$ not all equal to 0 with
	\begin{equation}\nonumber
	\sum_{j=1}^{r}t_{j}\Bigl(\sum_{i'=1}^{K}c_{j,i,i'} n_{i'}\Bigr)=\bold{0} \text{ for all $1\leq i\leq K$.}
	\end{equation}
	 Then for all $\n\in U''$, there exist $t_{1},\dots,t_{r}\in \F_{p}$ such that
	 \begin{equation}\label{1:fnp2}
	\sum_{i'=1}^{K}u_{i,i'}n_{i'}=\bold{0} \text{ for all $1\leq i\leq K$,}
	\end{equation}	 
	where $u_{i,i'}:=\sum_{j=1}^{r}t_{j}c_{j,i,i'}$.
	Since $v_{1},\dots,v_{r}$ are linearly independent, 	at least one of $u_{i,i'}, 1\leq i,i'\leq K$ is non-zero. So at least one of the equations in (\ref{1:fnp2}) is nontrivial and thus one of $n_{1},\dots,n_{K}$
	  is uniquely determined by the rest. So 
	$$\vert U''\vert\leq Kp^{r}\cdot p^{d(K-1)}= Kp^{d(K-1)+r}.$$
	Let $U':=U\backslash \tau(U'')$.
	Since $\vert U\vert\geq \d p^{dK-r}$,  		
	and  $d\geq 2r+1$, if $p\gg_{d,s} \d^{-O_{d,s}(1)}$, then  $\vert U'\vert\geq \d p^{dK-r}/2$.

	On the other hand, for all $\n\in U'$, the vectors $(\sum_{i'=1}^{K}c_{j,i,i'}n_{i'})_{i=1}^{K}, 1\leq j\leq r$ (as elements in $\F_{p}^{dK}$) are linearly independent. Since $A$ is invertible,   we have that  $$\m \mod p(\Z^{d})^{K}\mapsto (\tilde{f}_{j}(\n)+T_{j}(\n,\m) \mod p\Z)_{1\leq j\leq r}$$ is a surjection from $(\Z/p\Z)^{dK}$ to $(\Z/p\Z)^{r}$. So it follows from (\ref{1:fnp}) that
	$$\sum_{i:=(i_{1},\dots,i_{r})\in\N^{r}, \vert i\vert\leq s/2}R_{i}(\n)\prod_{j=1}^{r}x_{j}^{i_{j}}\equiv 0 \mod p\Z$$
	for all $x_{1},\dots,x_{r}\in\Z$. By interpolation, this implies that $QR_{i}(\n)\in p\mathbb{Z}$ for all $i\in\N^{r}, \vert i\vert\leq s/2$  for some $Q\in\N, p\nmid Q$. So $R_{i}$ takes values in $\Z\cap \frac{p}{Q}\Z=p\Z$.
This completes the proof of the claim.
		
	\

	Let $U'$ be given by the claim. Since $\frac{1}{p}R_{i}$ belongs to $\poly((\Z^{d})^{K}\to \Z/p)$ and $\iota^{-1}(\Omega)$ is weakly $(\d,p)$-irreducible for all $\d>0$ as long as $p\gg_{d,s} \d^{-O_{d,s}(1)}$ by Part (ii), it follows from Theorem \ref{1:ct} that $\tau(\Omega)\subseteq V_{p}(R_{i})$.  By Corollary \ref{1:cl3s}, if $T=O_{a_{k'},K,s}(1)=O_{K,s}(1)$ is sufficiently large, then
	there exist  $Q_{i,j}\in\poly((\Z^{d})^{K}\to \Z), 0\leq j\leq r$ of degree at most $O_{K,s}(1)$ such that
\begin{equation}\label{1:tempeeq2}
\begin{split}
	\frac{1}{p}(p\tilde{D}_{\mathfrak{I}})^{T}R_{i}
	=Q_{i,0}+\sum_{j=1}^{r}\tilde{f}_{j}Q_{i,j}.
\end{split}
	\end{equation}
By (\ref{1:gg}), (\ref{1:tempeeq1}) and  (\ref{1:tempeeq2}),
we have that
$$(p\tilde{D}_{\mathfrak{I}})^{T(s-(k'-1))}Qg=\sum_{i=0}^{k'-2}\frac{(p\tilde{D}_{\mathfrak{I}})^{T}g_{i}}{p^{i}}+\frac{1}{p^{k'-1}}(p\tilde{D}_{\mathfrak{I}})^{T}(g_{k'-1}+\frac{g_{k'}}{p})$$
with $(p\tilde{D}_{\mathfrak{I}})^{T}(g_{k'-1}+\frac{g_{k'}}{p})$ being an $O_{K,s}(1)$-good polynomial. We arrive at a contradiction to the minimality of $k'$ provided that $a_{k'-1}\gg_{a_{k'},K,s} 1$.

	In conclusion, we have that $k'=0$ and thus 
		$$(p\tilde{D}_{\mathfrak{I}})^{Ts}Qg=g'$$
		for some $Q\in\N$, $p\nmid Q$ and some
		$a_{0}$-good polynomial $g'$. 
Then for every $\mathfrak{I}$ and every $\n\in\tau(\Omega)+p(\Z^{d})^{K},$ we have $g'(\n)\in\Z$ and thus
\begin{equation}\label{1:gg3}
(p\tilde{D}_{\mathfrak{I}}(\n))^{Ts}Qg(\n)\in\Z.
\end{equation}
	
	Our final step is to use a parametrization trick to get rid of the $(p\tilde{D}_{\mathfrak{I}}(\n))^{Ts}$ term in (\ref{1:gg3}). 
	The method is very similar to step 3 of the proof of Part (i) of Theorem \ref{1:irrrg}.

	We first pick $\n=(\n',n_{K})\in \iota^{-1}(\Omega)$ such that $n_{1},\dots,n_{K-1}$ are $p$-linearly independent. Fix $1\leq k\leq K$. If $c_{k}=1$, then
	recall that $r_{k}\leq k$ and $L_{k,j}, 2\leq j\leq r_{k}$ are linearly independent. Since $A$ is invertible, by the knowledge of linear algebra, there exists a subset $I_{k}$ of $\{1,\dots,d\}$ of cardinality $r_{k}-1$ such that $\Delta_{k,I_{k}}(\iota(\n'))\neq 0$. By Lemma \ref{1:noog},
	 there exist $i_{k,1},i_{k,2}\in I^{c}, i_{k,1}\leq i_{k,2}$ such that  $G_{k,I_{k},i_{k,1},i_{k,2}}(\iota(\n'))\neq 0$. 
	If $c_{k}=0$, then
	recall that $r_{k}\leq k-1$ and $L_{k,j}, 1\leq j\leq r_{k}$ are linearly independent. Since $A$ is invertible, by the knowledge of linear algebra, there exists a subset $I_{k}$ of $\{1,\dots,d\}$ of cardinality $r_{k}$ such that $\Delta_{k,I_{k}}(\iota(\n'))\neq 0$. Denote $i_{k,1}=i_{k,2}=0$. 

	Then $\mathfrak{I}=(I_{k},i_{k,1},i_{k,2})_{k=1}^{K}$ is an admissible parametrization of $\Omega$. Moreover, since $D_{\mathfrak{I}}(\iota(n))\neq 0$, we have that $p\tilde{D}_{\mathfrak{I}}(\n)\in \Z\backslash p\Z$. By (\ref{1:gg3}), we have $g(\n)\in\Z/Q'$ for some $Q'\in\N, p\nmid Q'$. On the other hand, since $g$ takes values in $\Z/p^{s}$, we have that $g(\n)\in\Z$.
	
	We now consider the case when $n_{1},\dots,n_{K-1}$ are not necessarily $p$-linearly independent. 
	For $1\leq t\leq K-1$, we say that \emph{Property $t$} holds if for all $\n=(\n',n_{K})\in \iota^{-1}(\Omega)$ with 
	$n_{1},\dots,n_{t}$ being $p$-linearly independent, we have that $g(\n)\in\Z$.  By the discussion above, Property $K-1$ holds.
	Assume now that we have proven Property $t+1$ holds for some $0\leq t\leq K-2$. We show that Property $t$ holds.
	Pick any $\n=(\n',n_{K})\in  \iota^{-1}(\Omega)$ such that $n_{1},\dots,n_{t}$  are $p$-linearly independent. If $n_{1},\dots,n_{t+1}$ are $p$-linearly independent, then $g(\n)\in\Z$ by induction hypothesis and we are done. So we assume that $n_{1},\dots,n_{t+1}$ are $p$-linearly dependent.

	By Lemma \ref{1:fvv3}, there exists a subspace  $V$ of $\V$ dimension $3$ such that 
	\begin{itemize}
		\item for all $v\in V\backslash\{\bold{0}\}$, $\iota(n_{1}),\dots,\iota(n_{t}),v$ are linearly independent;
		\item for all $v\in V$ and $1\leq j\leq K$, we have that $(vA)\cdot \iota(n_{j})=0$;
		\item  we have that $V\cap V^{\pp}=\{\bold{0}\}$.
	\end{itemize}	
	Let $\phi\colon \F_{p}^{3}\to V$ be any bijective linear transformation. 
	Let $W$ be the set of $m\in\F_{p}^{3}$ such that $(\iota(n_{1}),\dots,\iota(n_{t}),\iota(n_{t+1})+\phi(m),\iota(n_{t+2}),\dots,\iota(n_{K}))\in\Omega$. Note that $\bold{0}\in W$.
	Since $(\phi(m)A)\cdot \iota(n_{j})=0$ for all $1\leq j\leq K$, by (\ref{1:e:fi1}) and (\ref{1:e:fij}),
	for all $1\leq i'\leq K$ and $1\leq j\leq r_{i'}$ with either $(i',j)\neq (t+1,1)$ or $(i',j)=(t+1,1)$ and $c_{t+1}=0$, 	 it is not hard to see that
	$$F_{i',j}(\iota(n_{1}),\dots,\iota(n_{t}),\iota(n_{t+1})+\phi(m),\iota(n_{t+2}),\dots,\iota(n_{K}))=F_{i',j}(\iota(n_{1}),\dots,\iota(n_{K})).$$
	Moreover, if $(i',j)=(t+1,1)$ and $c_{t+1}=1$, then by (\ref{1:e:fi1}),
	we have that
	$$M'(m):=F_{t+1,1}(\iota(n_{1}),\dots,\iota(n_{t}),\iota(n_{t+1})+\phi(m),\iota(n_{t+2}),\dots,\iota(n_{K}))$$
	is a quadratic form from $\F_{p}^{3}$ to $\F_{p}$. 	It is then not hard to see that 
	$W=V(M')$ if $c_{t+1}=1$ and $W=\F_{p}^{3}$ if $c_{t+1}=0$.  Since	$V\cap V^{\pp}=\{\bold{0}\}$, by Proposition \ref{1:iissoo} (ii), $\rank(M')=3$.  
	So by  Propositions \ref{1:iri00} and \ref{1:iri0}, $\iota^{-1}(W)$ is strongly $(1/2,p)$-irreducible.
	
	On the other hand, for all $m\in\F_{p}^{3}\backslash\{\bold{0}\}$,
	since $A$  $\iota(n_{1}),\dots,\iota(n_{t}),\phi(m)$ are linearly independent  and since $\iota(n_{1}),\dots,\iota(n_{t+1})$ are linearly dependent, we have that $\iota(n_{1}),\dots,$ $\iota(n_{t}),\iota(n_{t+1})+\phi(m)$ are linearly independent.
	So $n_{1},\dots,n_{t},n_{t+1}+\tau\circ\phi(m)+pu$ are $p$-linearly independent for all $u\in\Z^{d}$.
	By the induction hypothesis that Property $t+1$ holds, we have that
	\begin{equation}\label{1:gg2}
	g(n_{1},\dots,n_{t},n_{t+1}+\tau\circ\phi(m)+pu,n_{t+2},\dots,n_{K})\in\Z
	\end{equation}
	for all $m\in W\backslash\{\bold{0}\}$ and $u\in\Z^{d}$.
	
		Let $\phi'\colon \Z^{3}\to\Z^{d}$ be the linear transformation induced by the restrictions $\phi'(e_{j}):=\tau(\phi(e_{j}))$ for $j=1,2,3$, where we slightly abuse the notation to let $e_{j}$ denote the $j$-th standard vector of both $\Z^{3}$ and $\F_{p}^{3}$. 
it is not hard to see that $\tau\circ \phi\equiv\phi'\circ\tau \mod p\Z^{d}$ and  $\phi'(x+py)-\phi'(x)\in p\Z^{d}$ for all $x,y\in\Z^{d}$. 
Let $g'\colon \Z^{3}\to \R$ be given by
	$$g'(x):=g(n_{1},\dots,n_{t},n_{t+1}+\phi'(x),n_{t+2},\dots,n_{K})$$
	for all $x_{1},\dots,x_{K-1}\in\Z^{3}$. 	Since $g\in\poly((\Z^{d})^{K}\to \Z/p^{s})$, we have that $g'\in\poly(\Z^{3}\to \Z/p^{s})$.

	On the other hand, for any $m\in W\backslash\{\bold{0}\}$ and $z\in\Z^{3}$,   
	by (\ref{1:gg2}), there exist $u\in\Z^{d}$ such that 
		\begin{equation}\nonumber
		\begin{split}
		&\quad g'(\tau(m)+pz)
		=g(n_{1},\dots,n_{t},n_{t+1}+\phi'(\tau(m)+pz),n_{t+2},\dots,n_{K})
		\\&=g(n_{1},\dots,n_{t},n_{t+1}+\tau\circ\phi'(m)+pu,n_{t+2},\dots,n_{K})\in\Z.
		\end{split}
		\end{equation}
This implies that
  $\tau(W)\backslash\{\bold{0}\}\subseteq V_{p}(g')$. So $\vert V_{p}(g')\cap \tau(W)\vert\geq \vert W\vert-1\geq\vert W\vert/2$.
	Since $\iota^{-1}(W)$ is strongly $(1/2,p)$-irreducible,   we have that $\iota^{-1}(W)\subseteq V_{p}(g')$. In particular, since $\bold{0}\in \tau(W)$, we have that 
	$g(\n)=g'(\bold{0})\in\Z$.
	
	In conclusion, we have that Property $t$ holds.  By induction, we deduce that Property 0 holds, meaning that  $\iota^{-1}(\Omega)\subseteq V_{p}(g)$. So $\iota^{-1}(\Omega)$ is strongly $(\d,p)$-irreducible up to degree $s$. This completes the proof of Part (iii) of Theorem \ref{1:irrrg}.

\section{Leibman dichotomies for nice and consistent $M$-sets}\label{1:s:AppD}
 
We present the proof of Theorem \ref{1:veryr} in this appendix. We restate this theorem below for convenience.

 	\begin{thm}[Nice and consistent  $M$-sets admit  Leibman dichotomies]\label{1:veryrr}
	Let $d,k\in\N_{+},s,r\in\N$ with $d\geq \max\{4r+1,4k+3,2k+s+11\}$, $C>0$ and $p$ be a prime. 
	There exists  $K:=O_{C,d}(1)$ such that for any  non-degenerate quadratic form $M\colon\V\to \F_{p}$, any nice and consistent $M$-set $\Omega\subseteq (\V)^{k}$ of total co-dimension $r$, and any $0<\d<1/2$,  if $p\gg_{C,d} \d^{-O_{C,d}(1)}$, 
	then  $\Omega$ admits a partially periodic $(\d,K\d^{-K})$-Leibman dichotomy up to step $s$ and complexity $C$.
\end{thm}

By Proposition \ref{1:yy3} (iii) and Lemma \ref{1:rrgg}, we may assume without loss of generality that $\Omega=V(\mathcal{J})$ for some nice and consistent $(M,k)$-family $\mathcal{J}$. By Lemma \ref{1:rep0},  $\mathcal{J}$
 admits a nice and standard $M$-representation $(F_{i,j}\colon 1\leq i\leq k, 1\leq j\leq r_{i})$ such that
\begin{equation}\label{1:6677}
F_{i,j}(n_{1},\dots,n_{k})=(n_{i}A)\cdot (a_{i,j,1}n_{1}+\dots+a_{i,j,i}n_{i})+u_{i,j}
\end{equation}
for some $a_{i,j,t}, u_{i,j}\in\F_{p}$  for all $1\leq i\leq k, 1\leq j\leq r_{i}$, such that the matrix $\begin{bmatrix}
a_{i,1,i} & \dots & a_{i,1,1}\\
\dots & \dots & \dots \\
a_{i,r_{i},i} & \dots & a_{i,r_{i},1}
\end{bmatrix}$
is of rank $r_{i}$  and is in the reduced row echelon form, where $0\leq r_{i}\leq i$ and $r_{1}+\dots+r_{k}=r$. In particular, $F_{i,j}$ is independent of $n_{i+1},\dots,n_{k}$ and depends nontrivially on $n_{1},\dots,n_{i}$.

Fix any $\d>0$.
	Throughout the proof, assume that $p\gg_{C,d} \d^{-O_{C,d}(1)}$.
	Let $G/\Gamma$ be an $s$-step $\N$-filtered nilmanifold  of complexity at most $C$ and 
	$g\in \poly_{p}(\Omega\to G_{\N}\vert\Gamma)$  be such that $(g(n)\Gamma)_{n\in \Omega}$ is not $\d$-equidistributed on $G/\Gamma$. 
 	
	The idea of the proof is inspired by Section 3 of \cite{GT14} (see also page 37 of \cite{Sun18} and Proposition 4.11 of \cite{Sun19} for similar methods), which we explain briefly. For any bijective $d$-integral linear transformation $\phi\colon (\V)^{k}\to (\V)^{k}$, one can show that for many $\x\in (\V)^{k-1}$, the sequence $(g\circ \phi^{-1}(\x,n)\Gamma)_{n\in\V\colon (\x,n)\in \phi(\Omega)}$ is not $O(\d)$-equidistributed on $G/\Gamma$. On the other hand, the set $\{n\in\V\colon (\x,n)\in \phi(\Omega)\}$ can be viewed as the set of zeros of some quadratic form induced by $M$. We may use Theorem \ref{1:sLei} to conclude that $\eta\circ g\circ \phi^{-1}(\x,\cdot) \mod \Z$ is a constant when restricted to $\Omega$ for some nontrivial  type-I horizontal character $\eta$ for many $\x$. If we take sufficiently many different $\phi$, then we will be able to combine these informations to conclude that $\eta\circ g \mod \Z$ is a constant on $\Omega$.

	We now turn to the rigorous proof of Theorem \ref{1:veryrr}.
	
	\textbf{Step 1: decomposing $\Omega$ using different linear transformations.}
		Let 
		$N:=\lceil p^{\alpha}\rceil$ for any fixed $\frac{k-2}{k-1}<\alpha<1$. 
		Denote $(a_{1},\dots,a_{k-1})\otimes n:=(a_{1}n,\dots,a_{k-1}n)\in(\V)^{k-1}$ for $a_{1},\dots,a_{k-1}\in \F_{p}$ and $n\in\V$.
		 For all $v\in \{0,\dots,N-1\}^{k-1}\subseteq\F_{p}^{k-1}$, let
		$\phi_{v}\colon(\V)^{k}\to (\V)^{k}$ denote the bijective $d$-integral linear transformation given by 
		$$\phi_{v}(\x,n):=(\x+v\otimes n,n) \text{ for all } \x\in (\V)^{k-1}, n\in\V.$$
	By Proposition  \ref{1:yy3} (iii), $\mathcal{J}_{v}:=\{F_{i,j}\circ\phi_{v}\colon 1\leq i\leq k',1\leq j\leq r_{i}\}$ is a nice and independent  $(M,k)$-family. So $\phi_{v}^{-1}(\Omega)$ is a nice and consistent $M$-set with total co-dimension $r$. By Lemma \ref{1:rep0}, $\phi_{v}^{-1}(\Omega)$ admits a nice and standard $M$-representation with some dimension vector $\bold{r}_{v}$ of dimension $r$. By the Pigeonhole Principle, there exist a subset $T$ of $\{0,\dots,N-1\}^{k-1}$ of cardinality $\gg_{k,r}N^{k-1}$ such that $\bold{r}_{v}$ equals to the same vector $\bold{r}'=(r'_{1},\dots,r'_{k})$ with $r'_{1}+\dots+r'_{k}=r$ for all $v\in T$. Since $\phi_{v}^{-1}(\Omega)$ is nice, we have $0\leq r'_{i}\leq i$.

	For convenience denote $D:=d(k-1)-r'_{1}-\dots-r'_{k-1}.$
	Let 
	$(\mathcal{J}'_{v},\mathcal{J}''_{v})$ be an $\{n_{1},\dots,n_{k-1}\}$-decomposition of $\mathcal{J}_{v}$ (recall Appendix \ref{1:s:dec} for the definition).
	Let $\Omega_{v}$ denote the set of $(n_{1},\dots,n_{k-1})\in (\V)^{k-1}$ such that $(n_{1},\dots,n_{k})\in V(\mathcal{J}'_{v})$ for all $n_{k}\in \V$, and for each $(n_{1},\dots,n_{k-1})\in (\V)^{k-1}$, let $\Omega_{v}(n_{1},\dots,n_{k-1})$ denote the set of $n_{k}\in \V$ such that $(n_{1},\dots,n_{k})$ $\in V(\mathcal{J}''_{v})$.
	Since $d\geq 2r+1$, by  Theorem \ref{1:ct} and the Pigeonhole Principle, for any $v\in T$, there exist a subset $U_{v}$ of $\Omega_{v}$ of cardinality $\gg_{d}\d p^{D}$  such that
	for all $\x\in U_{v}$, $\Omega_{v}(\x)$ has a standard $M$-representation with the one dimensional dimension vector $(r'_{k})$, and that	 $(g(n)\Gamma)_{n\in \Omega_{v}(\x)}$ is not $O_{d}(\d)$-equidistributed.
	 Let $J$ denote the set of $(v,\x)\in (\V)^{k-1}\times T$ such that $\x\in U_{v}$. Then $\vert J\vert\gg_{d}N^{k-1}\d p^{D}$.

	 Since $\Omega_{v}(\x)$ has a standard $M$-representation with the dimension vector $(r'_{k})$, we have that either 
	$\Omega_{v}(\x)$ is the intersection of the zeros of an $(M,1)$-integral quadratic function  and an affine subspace of $\V$  of co-dimension $r'_{k}-1$,  or $\Omega_{v}(\x)$ is an  affine subspace of $\V$  of co-dimension $r_{k}$.
	For the former case, since $d-2(r'_{k}-1)\geq d-2(k-1)\geq s+13$, 
	by Proposition \ref{1:iissoo} and  Theorem \ref{1:rLei}, for all $(v,x)\in J$, there exists a nontrivial type-I horizontal character $0<\Vert \eta_{v,\x}\Vert\leq O_{C,d}(1)\d^{-O_{C,d}(1)}$ such that $\eta_{v,\x}\circ g(\phi_{v}(\x,\cdot)) \mod \Z$ is a constant on $\Omega_{v}(\x)$. 
	For the later case, we have the same conclusion by Theorem \ref{1:Lei0}.
	
	By the Pigeonhole Principle, there exists a subset $J_{1}$ of $J$ of cardinality at least $\gg_{C,d}N^{k-1}\d^{O_{C,d}(1)} p^{D}$ such that $\eta_{v,\x}$ equals to the same $\eta$ for all $(v,\x)\in J_{1}$.  Again by the Pigeonhole Principle,  there exists a subset $J'$ of $T$ of cardinality $\gg_{C,d}\d^{O_{C,d}(1)}N^{k-1}$ such that for all $v\in J'$, the set $J'(v):=\{\x\colon (\x,v)\in J_{1}\}$ is of cardinality $\gg_{C,d}\d^{O_{C,d}(1)} p^{D}$.
	Since $N:=\lceil p^{\alpha}\rceil$, we have that $\vert J'\vert>p^{k-2}$.
Similar to the proof of Lemma \ref{1:iiddpp}, there exist  $v_{1},\dots,v_{k}\in J'$ such that $(1,v_{1}),\dots,(1,v_{k})\in\F_{p}^{k}$ are linearly independent. 
	
	Denote $F:=\eta\circ g$. We wish to show that $F \mod \Z$ is a constant on $\Omega$.

	\textbf{Step 2: deriving an invariance property for $F$.}
	It is convenient to transform $\Omega$ by the bijective $d$-integral linear transformation $\psi\colon (\V)^{k}\to(\V)^{k}$ given by
	$$\psi(n_{1},\dots,n_{k}):=(v_{1}\otimes n_{1}+\dots+v_{k}\otimes n_{k},n_{1}+\dots+n_{k}),$$
	and temporarily work with the function $F':=\eta\circ g\circ \psi$.
	Our next step is to combine the information we get for $F$ from the linear transformations $\phi_{v_{1}},\dots,\phi_{v_{k}}$. 
		For all $1\leq i\leq k$, 
		Let $\psi_{i}\colon (\V)^{k}\to(\V)^{k}$ be the bijective $d$-integral linear transformation given by
		$$\psi_{i}(n_{1},\dots,n_{k}):=\Bigl(\sum_{j=1}^{k}(v_{j}-v_{i})\otimes n_{j},n_{1}+\dots+n_{k}\Bigr).$$

		Let  $U_{i}\subseteq  \psi^{-1}_{i}(\Omega)$ be the set of $y\in \psi^{-1}_{i}(\Omega)$ such that $\psi_{i}(y)=(\x,n)$ for some $\x\in J'(v_{i})$ and $n\in \Omega_{v_{i}}(\x)$.
	For such  $y=(y_{1},\dots,y_{k})\in \psi^{-1}_{i}(\Omega)$ , we have  that $n=y_{1}+\dots+y_{k}$ and $\x=\sum_{j=1}^{k}(v_{j}-v_{i})\otimes y_{j}$. So
	\begin{equation}\label{1:sios}
		\begin{split}
		&\quad F'(y_{1},\dots,y_{k})
		=\eta\circ g(v_{1}\otimes y_{1}+\dots+v_{k}\otimes y_{k},y_{1}+\dots+y_{k})
		\\&=\eta\circ g(\x+v_{i}\otimes n,n)
		=\eta\circ g\circ\phi_{v_{i}}(\x,n).
		\end{split}
		\end{equation}	 
		
		For convenience denote $\y:=(y_{1},\dots,y_{i-1},y_{i},y_{i+1},\dots,y_{k})$, $\y':=(y_{1},\dots,y_{i-1},y'_{i},y_{i+1},\dots,$ $y_{k})$ and $\tilde{\y}:=(y_{1},\dots,y_{i-1},y_{i},y'_{i},y_{i+1},\dots,y_{k})$, where $y_{i}\in\V$.
		Let $\Omega'_{i}\subseteq (\V)^{k+1}$ denote the set of $\tilde{\y}\in (\V)^{k+1}$ such that $\y,\y'\in  \psi^{-1}(\Omega),$ and $U'_{i}\subseteq (\V)^{k+1}$ denote the set of $\tilde{\y}\in (\V)^{k+1}$ such that $\y,\y'\in U_{i}.$
		Pick any $\tilde{\y}\in U'_{i}$.  Assume that $\y=(\x,n)$ and $\y'=(\x',n')$ for some $\x,\x' \in J'(v_{i})$, $n\in \Omega_{v_{i}}(\x)$ and $n'\in \Omega_{v_{i}}(\x')$. Then $\x=\x'=\sum_{1\leq j\leq k, j\neq i}(v_{j}-v_{i})\otimes y_{j}$. Since $\eta\circ g(\phi_{v_{i}}(\x,\cdot)) \mod \Z$ is a constant on $\Omega_{v_{i}}(\x)$, we have that $$\eta\circ g\circ\phi_{v_{i}}(\x,n)\equiv\eta\circ g\circ\phi_{v_{i}}(\x,n')\mod \Z.$$
		So it follows from
	 (\ref{1:sios}) that 
		\begin{equation}\label{1:yy4}
		F'(\y)\equiv F'(\y') \mod \Z
		\end{equation}
 for all $\tilde{\y}\in U'_{i}$.
 Since  $\Omega=\cap_{t=1}^{k}\cap_{j=1}^{r_{t}}V(F_{t,j})$ and $F_{t,j}$ is independent of $n_{t+1},\dots,n_{k}$,   we have that 
 $$\psi(\Omega'_{i})=(\cap_{t=1}^{i-1}\cap_{j=1}^{r_{t}}V(F_{t,j,i}))\cap(\cap_{t=i}^{k}\cap_{j=1}^{r_{t}}(V(F_{t,j,i})\cap V(F'_{t,j,i}))),$$ where $F_{t,j,i},F'_{t,j,i}\colon (\V)^{k+1}\to \V$ are the $d$-integral linear transformations given by
 $$\text{$F_{t,j,i}(\tilde{\y}):=F_{t,j}(\y)$ and $F'_{t,j,i}(\tilde{\y}):=F_{t,j}(\y').$}$$
 By Proposition \ref{1:yy33} (iv) and (v),
 $\Omega_{i}'$ is a consistent $M$-set of total co-dimension at most $2r$. On the other hand, since $F_{t,j}$ are nice, so are $F_{t,j,i},F'_{t,j,i}$. So $\psi(\Omega_{i}')$ and thus $\Omega_{i}'$ is nice and consistent.
  Let $\kappa\gg_{C,d}\d^{O_{C,d}(1)}$ to be chosen later. By Theorem \ref{1:irrrg}, if $d\geq \max\{4r+1,4k+3\}$  and $p\gg_{C,d}\d^{-O_{d,C}(1)}\gg_{C,d}\kappa^{-O_{C,d}(1)}$,  then $\iota^{-1}(\Omega'_{i})$ is strongly $(\kappa, p)$-irreducible up to degree $s$.

Since $\Omega_{v_{i}}(\x)$ has an $M$-representation with dimension vector $(r'_{k})$ for all $\x\in U_{v_{i}}$, by Theorem \ref{1:ct},
$$\vert U'_{i}\vert=\sum_{\x\in J'(v_{i})}\vert \Omega_{v_{i}}(\x)\vert^{2}\gg_{C,d}\d^{-O_{C,d}(1)}p^{2(d-r'_{k})+D}$$
and
$$\vert \Omega'_{i}\vert=\sum_{\x\in \Omega_{v_{i}}}\vert \Omega_{v_{i}}(\x)\vert^{2}\gg_{C,d}\d^{-O_{C,d}(1)}p^{2(d-r'_{k})+D}.$$
 
 By Proposition \ref{1:BB}, 
 there exists $g'\in\poly_{p}(\psi^{-1}(\Omega)\to G_{\N}\vert\Gamma)$ such that $g'(n)\Gamma=g\circ\psi(n)\Gamma$ for all $n\in \psi^{-1}(\Omega)$.
 So by  Lemma \ref{1:goodcoordinates}, $F$ takes values in $\Z/p^{s}+C$ for some $C\in\R$. Since  $\iota^{-1}(\Omega'_{i})$ is strongly $(\kappa, p)$-irreducible up to degree $s$,    if $\kappa\gg_{C,d}\d^{O_{C,d}(1)}$ is chosen to be sufficiently small, then it follows from (\ref{1:yy4}) that 
  $F'(\y)\equiv F'(\y') \mod \Z$
  for all $\tilde{\y}\in \Omega'_{i}$. By the definition of $\Omega'_{i}$ and $F'$, we have that  
  \begin{equation}\label{1:fff}
  F(y_{1},\dots,y_{i-1},y_{i},y_{i+1},\dots,y_{k})\equiv F(y_{1},\dots,y_{i-1},y'_{i},y_{i+1},\dots,y_{k}) \mod \Z
  \end{equation}
	for all $1\leq i\leq k$ and $(y_{1},\dots,y_{k}),(y_{1},\dots,y_{i-1},y'_{i},y_{i+1},\dots,y_{k})\in \Omega$.

	\textbf{Step 3: showing that $F$ is a constant on $\Omega$.}
	If $\Omega=(\V)^{k}$, then (\ref{1:fff}) immediately implies that $F\mod \Z$ is a constant and we are done. However, in general this conclusion is not straightforward. In order to show that $F\mod \Z$ is a constant on $\Omega$, we need to change one coordinate at each time carefully.

		Fix any $(m_{1},\dots,m_{k}),(n_{1},\dots,n_{k})\in\Omega$ with $\sp_{\F_{p}}\{m_{1},\dots,m_{k}\}\cap \sp_{\F_{p}}\{n_{1},\dots,n_{k}\}=\{\bold{0}\}$.

		\textbf{Claim.} There exists
		$(z_{1},\dots,z_{k})\in\Omega$  such that $$\text{$(z_{1},\dots,z_{i},m_{i+1},\dots,m_{k}),(z_{1},\dots,z_{i},n_{i+1},\dots,n_{k})\in \Omega$ for all $0\leq i\leq k$.}$$
		
		Suppose we have chosen $z_{1},\dots,z_{i-1}$ for some $1\leq i\leq k$ such that  for all $0\leq i'\leq i-1$, we have that $z_{1},\dots,z_{i'},m_{i'+2},\dots,m_{k},n_{i'+2},\dots,n_{k}$ are linearly independent and that
		\begin{equation}\label{1:66a55}
	(z_{1},\dots,z_{i'},m_{i'+1},\dots,m_{k}),(z_{1},\dots,z_{i'},n_{i'+1},\dots,n_{k})\in \Omega
	\end{equation}
		(there is nothing to prove when $i=1$).
		We show that there exists $z_{i}\in\V$ such that $z_{1},\dots,z_{i},m_{i+2},\dots,m_{k},n_{i+2},\dots,n_{k}$ are linearly independent and that
		$$(z_{1},\dots,z_{i},m_{i+1},\dots,m_{k}),(z_{1},\dots,z_{i},n_{i+1},\dots,n_{k})\in \Omega.$$
		
		Let $W$ denote the span of $z_{1},\dots,z_{i-1},m_{i+2},\dots,m_{k},n_{i+2},\dots,n_{k}$ and $W'$ denote the set of  $z_{i}\in\V$  such that 
	\begin{equation}\nonumber
	F_{t,j}(z_{1},\dots,z_{i},m_{i+1},\dots,m_{k})=F_{t,j}(z_{1},\dots,z_{i},n_{i+1},\dots,n_{k})=0
	\end{equation}
	for all $1\leq t\leq k, 1\leq j\leq r_{t}$.
	It suffices to show that $\vert W'\vert>\vert W\vert$.
	
	By (\ref{1:6677}), the knowledge of linear algebra, and the construction of $F_{t,j}$, it is not hard to see that $\mathcal{J}$
 admits another nice and independent (but not necessarily standard) $M$-representation $(F'_{t,j}\colon 1\leq t\leq k, 1\leq j\leq r_{t})$ such that
 $F'_{t,j}=F_{t,j}$ for all $1\leq t\leq i-1, 1\leq j\leq r_{t}$, and that 
\begin{equation}\label{1:66a77}
F'_{t,j}(n_{1},\dots,n_{k})=(n_{t}A)\cdot (a'_{t,j,1}n_{1}+\dots+a'_{t,j,t}n_{t})+u'_{t,j}
\end{equation}
for some $a'_{t,j,t'}, u'_{t,j}\in\F_{p}$  for all $i\leq t\leq k, 1\leq j\leq r_{t}$, such that the matrix $\begin{bmatrix}
a'_{t,1,t} & \dots & a'_{t,1,1}\\
\dots & \dots & \dots \\
a'_{t,r_{t},t} & \dots & a'_{t,r_{t},1}
\end{bmatrix}$
is of rank $r_{t}$ (but is not necessarily in the row echelon form) and that $a'_{t,j,i}=0$ for all $j\geq 2$.	
Then $W'$ can be described as the set of  $z_{i}\in\V$  such that 
	\begin{equation}\label{1:6655}
	F'_{t,j}(z_{1},\dots,z_{i},m_{i+1},\dots,m_{k})=F'_{t,j}(z_{1},\dots,z_{i},n_{i+1},\dots,n_{k})=0
	\end{equation}
	for all $1\leq t\leq k, 1\leq j\leq r_{t}$. We remark that $r_{t}\leq t$.

	For convenience denote
	$\w:=(z_{1},\dots,z_{i-1},m_{i+1},\dots,m_{k},n_{i+1},\dots,n_{k}).$
	For $1\leq t\leq i-1, 1\leq j\leq r_{t}$, since 
	$F'_{t,j}=F_{t,j}$ is independent of $n_{t+1},\dots,n_{k}$, we have (\ref{1:6655}) for free  because of (\ref{1:66a55}).
	For  	$i+1\leq t\leq k, 1\leq j\leq r_{t}$, suppose that either $j=1, a'_{t,1,i}=0$ or $j\geq 2$ (in which case we still have $a'_{t,j,i}=0$). Then $F'_{t,j}(x_{1},\dots,x_{k})$ is independent of $x_{i}$.
	By (\ref{1:66a55}), we have that 
	$$F'_{t,j}(z_{1},\dots,z_{i},y_{i+1},\dots,y_{k})=F'_{t,j}(z_{1},\dots,z_{i-1},y_{i},\dots,y_{k})=0$$
	for $y=m$ or $n$. So we have (\ref{1:6655}) for free.
	For $i+1\leq t\leq k, j=1$ with $a'_{t,1,i}\neq 0$, by (\ref{1:66a77}), we have that (\ref{1:6655}) holds if and only if
	\begin{equation}\label{1:66eq1}
a'_{t,1,i}(z_{i}A)\cdot m_{t}+P_{t}(\w)=a'_{t,1,i}(z_{i}A)\cdot n_{t}+P'_{t}(\w)=0
\end{equation}	
	for some functions $P_{t},P'_{t}$.	
	For $t=i, 2\leq j\leq r_{t}$, since $a'_{i,j,i}=0$, by (\ref{1:66a77}), we have that (\ref{1:6655}) holds if and only if
	\begin{equation}\label{1:66eq2}
(z_{i}A)\cdot (a'_{i,j,1}z_{1}+\dots+a'_{i,j,i-1}z_{i-1})+u'_{i,j}=0.
\end{equation}	
	Finally, for $t=i, j=1$, by (\ref{1:66a77}), we have that (\ref{1:6655}) holds if and only if
	\begin{equation}\label{1:66eq3}
Q(z_{i}):=(z_{i}A)\cdot (a'_{i,1,1}z_{1}+\dots+a'_{i,1,i}z_{i})+u'_{i,1}=0.
\end{equation}	

	Since 
	$z_{1},\dots,z_{i-1},m_{i+1},\dots,m_{k},n_{i+1},\dots,n_{k}$ are linearly independent and the matrix $$\begin{bmatrix}
a'_{i,1,i} & \dots & a'_{i,1,1}\\
\dots & \dots & \dots \\
a'_{i,r_{i},i} & \dots & a'_{i,r_{i},1}
\end{bmatrix}$$
is of rank $r_{i}$, it is not hard to see that the following vectors are linearly independent:
$$\{a'_{i,j,1}z_{1}+\dots+a'_{i,j,i-1}z_{i-1}\colon 1\leq j\leq r_{i}\}\cup \{m_{t},n_{t}\colon i+1\leq t\leq k\}.$$
Since $A$ is invertible, by (\ref{1:66eq1}), (\ref{1:66eq2}) and  (\ref{1:66eq3}), if $a'_{i,1,i}=0$, then $W'$ is an affine subspace of $\V$ of co-dimension at most $r_{i}+2(k-i)$; if $a'_{i,1,i}\neq 0$, then $Q$ is a non-degenerate quadratic form and $W'$ is the intersection $V(M)$ and an affine subspace of $\V$ of co-dimension at most $r_{i}+2(k-i)-1\leq 2k-i-1$. Since $d-2(2k-i-1)\geq 3$,   by Corollary \ref{1:counting01} (ii),  in both cases we have that $\vert W'\vert=p^{d-(r_{i}+2(k-i))}(1+O(p^{-1/2}))$.
	 On the other hand, $\vert W\vert\leq p^{2k-i-3}$.  Since $d-(r_{i}+2(k-i))> 2k-i-3$,   there exists $z_{i}\in W'\backslash W$.
 By induction, we complete the proof of the claim. 
	 		
		\
		
		Let $(z_{1},\dots,z_{k})$ be given by the claim. By (\ref{1:fff}),
		\begin{equation}\nonumber
		\begin{split}
		&\quad F(m_{1},\dots,m_{k})\equiv F(z_{1},m_{2},\dots,m_{k})\equiv \dots\equiv F(z_{1},\dots,z_{k})
		\\&\equiv F(z_{1},\dots,z_{k-1},n_{k})\equiv\dots\equiv F(n_{1},\dots,n_{k}) \mod\Z.
		\end{split}
		\end{equation}	
		For general $(m_{1},\dots,m_{k}),(n_{1},\dots,n_{k})\in\Omega$, similar to the proof of Proposition \ref{1:iissoo},  the number of $(z_{1},\dots,z_{k})\in(\V)^{k}$ such that either $\sp_{\F_{p}}\{m_{1},\dots,m_{k}\}\cap \sp_{\F_{p}}\{z_{1},\dots,z_{k}\}$ or $\sp_{\F_{p}}\{n_{1},\dots,n_{k}\}\cap \sp_{\F_{p}}\{z_{1},\dots,z_{k}\}$ is nontrivial is at most $O_{k}(p^{d(k-1)+2k-1})$. On the other hand $\vert\Omega\vert=p^{dk-r}(1+O_{k}(p^{-1/2}))$ by Theorem \ref{1:ct}. Since $dk-r>d(k-1)+2k-1$,  there exists $(z_{1},\dots,z_{k})\in\Omega$ such that $\sp_{\F_{p}}\{m_{1},\dots,m_{k}\}\cap \sp_{\F_{p}}\{z_{1},\dots,z_{k}\}=\sp_{\F_{p}}\{n_{1},\dots,n_{k}\}\cap \sp_{\F_{p}}\{z_{1},\dots,z_{k}\}=\{\bold{0}\}$. So by the previous case,
		$$F(m_{1},\dots,m_{k})\equiv F(z_{1},\dots,z_{k})\equiv F(n_{1},\dots,n_{k}) \mod\Z.$$
		Therefore $\eta\circ g \mod \Z$ is a constant on $\Omega$. So $\Omega$ admits a  partially periodic $(\d,K\d^{-K})$-Leibman dichotomy up to step $s$ and complexity $C$ for some function $K$ depending only on $C,d$. This completes the proof of Theorem \ref{1:veryrr}.

\begin{rem}
Motivated by Theorem \ref{1:veryrr}, it is natural to ask if any consistent $M$-set admits a partially periodic Leibman dichotomy.  Unfortunately the approach used in Theorem \ref{1:veryrr} can not be applied directly to general consistent $M$-sets. This is because the equation (\ref{1:fff})  is insufficient to conclude that $F \mod \Z$ is a constant if $\Omega$ is not nice. For example, let $\Omega$ be the set of $(m_{1},\dots,m_{d},n_{1},\dots,n_{d})\in(\V)^{2}$ such that $m_{1}=n_{1}$, and let $F(m_{1},\dots,m_{d},n_{1},\dots,n_{d}):=\frac{1}{p}\tau(m_{1})$. Then  (\ref{1:fff})  holds since $F(m_{1},\dots,m_{d},n_{1},\dots,n_{d})=F(m_{1},\dots,m_{d},n'_{1},\dots,n'_{d})$ for all $(m_{1},\dots,m_{d},n_{1},\dots,n_{d}),(m_{1},\dots,m_{d},n'_{1},\dots,n'_{d})\in\V$. However, $F \mod \Z$ is not a constant on $\Omega$.

One possible way to avoid this issue is the do a change of variable under a linear transformation $L\colon (\V)^{k}\to \F_{p}^{t}$ for some $t\in \N$ to eliminate all the polynomials of degree 1 in the standard $M$-representation of $\Omega$ in the initial step. However, $L$ in general is not $d$-integral, and thus $L(\Omega)$ is no longer an $M$-set. Therefore, one needs to generalize the definition of $M$-sets (as well as all the results relevant to $M$-sets) to a larger class of sets in order to overcome this difficulty. Since Theorem \ref{1:veryrr} is good enough for our purposes, we do not pursuit this improvement.
\end{rem}

\bibliographystyle{plain}
\bibliography{swb}
\end{document}